\documentclass[12pt,english]{article}
\usepackage[T1]{fontenc}
\usepackage[latin9]{inputenc}
\usepackage{geometry}
\usepackage{color}
\usepackage{mathrsfs}
\usepackage{amsmath}
\usepackage{amsthm}
\usepackage{amssymb}
\PassOptionsToPackage{normalem}{ulem}
\usepackage{ulem}

\makeatletter

\numberwithin{equation}{section}
\theoremstyle{plain}
\newtheorem*{question*}{\protect\questionname}
\theoremstyle{plain}
\newtheorem*{thm*}{\protect\theoremname}
\theoremstyle{plain}
\newtheorem{thm}{\protect\theoremname}[section]
\theoremstyle{definition}
\newtheorem{example}[thm]{\protect\examplename}
\theoremstyle{definition}
\newtheorem{defn}[thm]{\protect\definitionname}
\theoremstyle{remark}
\newtheorem{rem}[thm]{\protect\remarkname}
\theoremstyle{plain}
\newtheorem{lem}[thm]{\protect\lemmaname}
\theoremstyle{plain}
\newtheorem{prop}[thm]{\protect\propositionname}
\theoremstyle{remark}
\newtheorem{notation}[thm]{\protect\notationname}

\usepackage[title]{appendix}
\usepackage{comment}
\usepackage{graphicx}
\usepackage{float}
\usepackage{hyperref}

\hypersetup{
     colorlinks   = true,
     citecolor    = blue,
     linkcolor    = blue
}

\date{}

\makeatother

\usepackage{babel}
\providecommand{\definitionname}{Definition}
\providecommand{\examplename}{Example}
\providecommand{\lemmaname}{Lemma}
\providecommand{\notationname}{Notation}
\providecommand{\propositionname}{Proposition}
\providecommand{\questionname}{Question}
\providecommand{\remarkname}{Remark}
\providecommand{\theoremname}{Theorem}

\begin{document}
\title{Expected Signature on a Riemannian Manifold and Its Geometric Implications}
\author{X. Geng, H. Ni and C. Wang}

\author{X. Geng\thanks{School of Mathematics and Statistics, University of Melbourne, Parkville VIC 3052, Australia. Email: xi.geng@unimelb.edu.au.},$\ $ H. Ni\thanks{Department of Mathematics, University College London, London WC1H 0AY, United Kingdom. Email: h.ni@ucl.ac.uk.}$\ \ $and C. Wang\thanks{Department of Mathematical Sciences, University of Bath, Bath BA2 7AY, United Kingdom. Email: cw2581@bath.ac.uk.}}

\maketitle
\begin{abstract}
On a compact Riemannian manifold $M,$ we show that the Riemannian distance
function $d(x,y)$ can be explicitly reconstructed from suitable asymptotics
of the expected signature of Brownian bridge from $x$ to $y$. In
addition, by looking into the asymptotic expansion of the fourth level
expected signature of the Brownian loop based at $x\in M$, one can explicitly
reconstruct both intrinsic (Ricci curvature) and extrinsic (second
fundamental form) curvature properties of $M$ at $x$. As independent interest, we also derive
the intrinsic PDE for the expected Brownian signature dynamics on $M$ from the perspective of the Eells-Elworthy-Malliavin horizontal lifting.

\tableofcontents{}

\end{abstract}

\section{Introduction and outline of main results}

In this section, we discuss the main motivation, introduce some background
information and outline the main results of the present article. 

\subsection{The signature transform}

The \textit{signature transform} (or simply the signature) of a multidimensional
path $\gamma:[0,T]\rightarrow\mathbb{R}^{d},$ which is defined by
the formal tensor series 
\begin{align*}
S(\gamma) & \triangleq\big(1,\gamma_{T}-\gamma_{0},\int_{0<s<t<T}d\gamma_{s}\otimes d\gamma_{t},\cdots,\int_{0<t_{1}<\cdots<t_{n}<T}d\gamma_{t_{1}}\otimes\cdots\otimes d\gamma_{t_{m}},\cdots\big)
\end{align*}
of iterated path integrals, provides an effective summary of the essential
information encoded in the original path $\gamma$. Such a transformation,
as well as its intrinsic form defined in terms of iterated integrals
against spatial one-forms, was originally introduced by the geometer
K.T. Chen \cite{Che54} in the 1950s for his purpose of constructing
a de Rham-type cohomology theory on loop spaces over manifolds. Similar
type of iterated integrals (with $\gamma$ being operator-valued paths)
was also used by the physicist F. Dyson \cite{Dys49} to obtain perturbative
expansions of Schr\"odinger equation (Dyson series).

The signature plays a fundamental role in the analysis of rough paths
and rough differential equations. Its mathematical properties have
been largely developed and better understood over the past two decades
by using modern techniques from the rough path theory, which was founded
by T. Lyons in his seminal work \cite{Lyo98} in 1998 and has led to far-reaching applications in stochastic analysis. Among others, a basic result is
the so-called \textit{signature uniqueness theorem}. It asserts
that every rough path is uniquely determined by its signature up to
treelike equivalence. This result was already obtained in Chen's originial
work \cite{Che58} in 1958 for piecewise smooth paths. It was extended
to the bounded variation case by Hambly-Lyons \cite{HL10} in 2010
and further to the general rough path case by Boedihardjo et al. \cite{BGLY16}
in 2016. The importance of the signature uniqueness theorem lies in
the fact that it potentially opens a gateway of studying geometric
properties of a rough path from its signature. In fact, the reconstruction
of a rough path from its signature at variaous quantitative levels
(the signature inversion problem) is among the most significant and
challenging open problems in rough path theory. Partial progress and
several different methods have been developed in the literature over
the past years to understand this problem, e.g. Lyons-Xu \cite{LX15}
by hyperbolic developments, Lyons-Xu \cite{LX18} by symmetrisation,
Chang et al. \cite{CDNX17} by probabilistic sampling, Chang-Lyons
\cite{CL19} by tensor insertions, Le Jan-Qian \cite{LQ12} and Geng
\cite{Gen17} by iterated integrals against suitable spatial one-forms
etc. If one is only interested in some particular quantitative properties
of a path instead of recovering the full trajectory, there could be
neat inversion formulae such as the relation (\ref{eq:LC}) below
which reconstructs the length of a smooth path from its signature
asymptotics in a rather simple and elegant way. 

On the other hand, if the underlying rough path is random (sample
paths of a stochastic process such as the Brownian motion) the signature
becomes a tensor-valued random variable. In this case, one naturally
considers the expectation of the signature (the \textit{expected signature}).
There is a probabilistic counterpart of the signature uniqueness theorem
due to Chevyrev-Lyons \cite{CL16} in 2016, which asserts that under
suitable conditions the law of a stochastic process is uniquely characterised
by its expected signature. To some extent, this can be viewed as an
infinite dimensional analogue of the classical moment problem for
random variables. There has been several works using a PDE approach to compute and analyse the expected signature of  Brownian motion \cite{LV04}, \cite{LN15}, diffusion processes \cite{N12} and L\'evy processes \cite{FS17}. For a general stochastic process, how one can explicitly and quantitatively reconstruct
distributional properties of the process from its expected
signature is widely unknown at the moment. 

On the applied side, more recently the applications of expected signature in machine learning are gaining attention. Specifically, Sig-MMD or Sig-Wasserstein-1 (Sig-W$_1$) distance based on the expected signature can serve as the maximum mean discrepancy (MMD) distance for general stochastic processes, which have wide applications in hypothesis testing \cite{chevyrev2022signature}, distributional regression \cite{LSDEL21} and generative models on sequential data \cite{ni2023conditional,HLSXWL21}. Compared with the vanilla MMD with vector-valued time series feature, Sig-MMD demonstrates superior performance in the two-sample hypothesis test of stochastic processes, as illustrated in \cite{chevyrev2022signature}. This metric can be kernelized to facilitate efficient computation \cite{SCFLY21}. Moreover, the Sig-WGAN model, which utilises Sig-W$_1$ based discriminator, has empirically shown to improve the accuracy and robustness of traditional GAN models for synthetic time series generation. This approach reduces the min-max game to supervised learning, hence significantly reducing the computational cost and yielding better performance, in particular, for the case of limited data  \cite{ni2023conditional,HLSXWL21}.

\subsection{Extended signature and the basic geometric question }

In this article, we take a different perspective of signature inversion.
Instead of trying to reconstruct a generic path from the knowledge
of its signature, we study signature dymanics in a geometric setting
and investigate the following basic question. 
\begin{question*}
Let $M$ be a Riemannian manifold. By observing the expected signature
of certain geometric dynamics on $M$ (e.g. Brownian bridges), can
one explicitly recover geometric properties (e.g. Riemannian distance,
curvature properties) of the underlying space? 
\end{question*}
Summarised in vague terms, the main finding of the present article
is that the expected signature of a Brownian bridge on $M$ with lifetime
$t$ encode rich geometric information (more precisely, metric and
curvature properties) about the underlying manifold $M$ in the asymptotics
when $t\rightarrow0^{+}$. How such information can be extracted explicitly
from the expected signature asymptotics is the main focus of the present
work. The precise statements of our main results are stated in Theorem
\ref{thm:ReconRD} (reconstruction of Riemannian distance) and Theorem
\ref{thm:RecCurv} (reconstruction of curvature properties) respectively. 

\vspace{2mm} Before explaining the essential ideas, one needs to
be careful about the notion of signature for manifold-valued paths
in the first place. Let $\gamma:[0,T]\rightarrow M$ be a smooth path
taking values in a differentiable manifold $M$. In the differential-geometric
setting, one cannot give intrinsic meanings to integrals like $\int d\gamma\otimes\cdots\otimes d\gamma$
without any additional structure. In fact, the intrinsic notion of
integration along $\gamma$ is the \textit{line integral against a
one-form} on $M$. More specifically, let $\phi_{1},\cdots,\phi_{n}$
be a given family of smooth one-forms on $M$. One can consider the
iterated line integral 
\[
\int_{0<t_{1}<\cdots<t_{n}<T}\phi_{1}(d\gamma_{t_{1}})\cdots\phi_{n}(d\gamma_{t_{n}}),
\]
which is globally well-defined due to the natural pairing between
cotangent and tangent vectors ($\phi(d\gamma_{t})$ is understood
as the pairing between $\phi(\gamma_{t})\in T_{\gamma_{t}}^{*}M$
and $\dot{\gamma}_{t}\in T_{\gamma_{t}}M$). By varying $n$ and the
test one-forms $\phi_{1},\cdots,\phi_{n}$, one obtains a family of
numbers associated with the given path $\gamma.$ This collection
of numbers, known as the \textit{extended signature} of $\gamma$
(cf. \cite{LQ12}), uniquely determines $\gamma$ up to tree-like
pieces. This geometric viewpoint is classical and was already well
understood back in K.T. Chen's original signature uniqueness theorem
\cite{Che58} in the 1950s. The extended signature (iterated line
integrals against one-forms) has been used in various contexts, e.g.
in reconstructing trajectories from the signature \cite{LQ12,Gen17}
and in the study of rough paths on manifolds \cite{CDL15}. In the
Euclidean case, the signature is a special case of iterated line integrals
(one simply takes $\phi_{i}=dx^{i}$) and the extended signature contains
exactly the same amount of information as the usual signature. Indeed,
it is well known that smooth one-forms can be approximated by polynomials
and iterated line integrals against polynomial forms can be expressed
as a suitable linear combination of signature coefficients (the shuffle
product formula). As a consequence, the extended signature is uniquely
determined by the usual signature. 

In our study, we will assume that a collection of one-forms $(\phi_{1},\cdots,\phi_{N})$
on $M$ (equivalently, an $\mathbb{R}^{N}$-valued one form $\phi$)
is given fixed and consider the Euclidean signature of the $\mathbb{R}^{N}$-valued
path $\int_{0}^{\cdot}\phi(d\gamma_{t})$. We call this the \textit{$\phi$-signature}
of $\gamma$ (of course, this is just part of the extended signature
of $\gamma$ given by iterated line integrals along combinations of
one-forms taken from the family $\phi=(\phi_{1},\cdots,\phi_{N})$).
A basic example that will be of our primary interest is the case when
$\phi$ is given by the exterior derivative of an embedding $F:M\rightarrow\mathbb{R}^{N}$.
In this case, the $\phi$-signature of $\gamma$ is just the Euclidean
signature of $F(\gamma)$ computed in the ambient space $\mathbb{R}^{N}$.

\subsection{Intrinsic PDE for the expected signature dynamics}

Let us now consider a Riemannian manifold $M$ and let $\phi$ be
an $\mathbb{R}^{N}$-valued one-form on $M$. For most of the time,
$\phi=dF$ where $F:M\rightarrow\mathbb{R}^{N}$ is an \textit{isometric}
embedding (this always exists by Nash's isometric embedding theorems).
We consider the expected $\phi$-signature of Brownian dynamics on
$M$. More specifically, let $W_{t}^{x}$ be a Brownian motion on
$M$ (i.e. a Markov process generated by $\Delta/2$ where $\Delta$
is the Laplace-Beltrami operator) starting at $x$. Let $\Psi(t,x)$
be the expected value of the $\phi$-signature (in the Stratonovich
sense) of $W^{x}$ up to time $t$. 

Our first main result establishes the intrinsic PDE governing the
dynamics of $\Psi(t,x)$ (cf. Theorem \ref{thm:BMSigPDE} ).
\begin{thm*}
The function $\Psi(t,x)$ satisfies the following tensor-algebra-valued
parabolic PDE on $M$:
\begin{equation}
\frac{\partial\Psi}{\partial t}=\frac{1}{2}\Delta\Psi+{\rm Tr}\big(\phi\otimes d\Psi+\frac{1}{2}(\nabla\phi+\phi\otimes\phi)\otimes\Psi\big)\label{eq:BMPDEIntro}
\end{equation}
\end{thm*}
Here $d$ is the exterior derivative operator and $\nabla$ is the
Riemannian connection. The definition of ${\rm Tr}(\cdot)$ is explained
in the remarks following the statement of Theorem \ref{thm:BMSigPDE}
below. The PDE for the expected signature of Euclidean Brownian motion
(and also of diffusion processes) was first established by Lyons-Ni
\cite{LN15} from the Markovian perspective. To derive the PDE (\ref{eq:BMPDEIntro}),
we take an intrinsic approach from the perspective of the Eells-Elworthy-Malliavin
lifting onto the orthornormal frame bundle. It is well known that
the lifted Brownian motion $\Xi_{t}$ on the bundle satisfies a canonical
SDE governed by the fundamental horizontal vector fields. As a result,
one can write down an SDE for the joint process $(\Xi_{t},\tilde{S}_{t})$
(cf. (\ref{eq:JointSDE})), where $\tilde{S}_{t}$ is the $\Phi$-signature
of $\Xi$ up to time $t$ and $\Phi$ is the pullback of $\phi$ onto
the bundle. Once the SDE for $(\Xi_{t},\tilde{S}_{t})$ is obtained,
it is standard to extract its generator and write down the associated
PDE governing the expectation of $\tilde{S}_{t}$ on the bundle (cf.
Lemma \ref{lem:PDEBundle}). It then remains to see how the
PDE on the bundle gets projected to the intrinsic PDE (\ref{eq:BMPDEIntro})
on the base manifold $M$. 

The PDE (\ref{eq:BMPDEIntro}) for the Brownian signature dynamics
easily yields an associated PDE for the expected signature of Brownian
bridge. Let $\{X_{s}^{t,x,y}:s\in[0,t]\}$ be a Brownian bridge from
$x$ to $y$ with lifetime $t$, i.e. the Brownian motion starting
at $x$ conditioned on reaching $y$ at time $t$. Let $\psi(t,x,y)$
be its expected $\phi$-signature. 
\begin{thm*}
The function $\psi(t,x,y)$ satisfies the following tensor-algebra-valued
parabolic PDE on $M$: 
\begin{align*}
\frac{\partial\psi}{\partial t} & =\frac{1}{2}\Delta_{x}\psi+\langle\nabla_{x}\log p(t,x,y),\nabla_{x}\psi+\phi\otimes\psi\rangle+{\rm Tr}(\phi\otimes d_{x}\psi)\\
 & \ \ \ +\frac{1}{2}{\rm Tr}\big((\nabla\phi+\phi\otimes\phi)\otimes\psi\big),
\end{align*}
where $p(t,x,y)$ is the heat kernel on $M$ and the subscript $x$
means differentiating with respect to the $x$-variable. 
\end{thm*}
As we will see, the function $\psi(t,x,y)$ encodes rich geometric
information about the manifold $M$ in the asymptotics as $t\rightarrow0^{+}$.
In what follows, we consider the particular case when $\phi=dF$, where $F:M\rightarrow E\triangleq\mathbb{R}^N$ is an isometric embedding. We will establish two types of reconstruction results:

\vspace{2mm}\noindent (i) {[}\textit{Reconstruction of Riemannian distance}{]}
Through a suitable procedure of combined asymptotics for $\psi(t,x,y)$
as $t\rightarrow0^{+}$ and $n\rightarrow\infty$ ($n$ is the degree
of the signature), one can explicitly recover the Riemannian distance
between $x$ and $y$.\

\vspace{2mm}\noindent (ii) {[}\textit{Reconstruction of curvature properties}{]} By considering the
asymptotic expansion of the fourth level component of $\psi(t,x,x)$
as $t\rightarrow0^{+}$, one can explicitly recover the Riemannian
metric tensor as well as both intrinsic (Ricci curvature) and extrinsic
(the second fundamental form) curvature properties at $x$. Here we
consider the fourth level projection because this can be shown to
be the first nonzero component in $\psi(t,x,x)$. 

\subsection{Reconstruction of Riemannian distance}

Recall that $\psi(t,x,y)$ is the expected $dF$-signature of the
Brownian bridge $X^{t,x,y}$ from $x$ to $y$ with lifetime $t$,
where $F:M\rightarrow\mathbb{R}^{N}$ is a given isometric embedding.
Our main idea of reconstructing the Riemannian distance $d(x,y)$
is largely inspired by a well-known open problem in rough path theory
which we shall first describe. 

Let $\gamma:[0,T]\rightarrow\mathbb{R}^{d}$ be a continuous path
with bounded total variation. By using the triangle inequality, it
is straight forward to see that 
\begin{equation}
\big(n!\big\| S_{n}(\gamma)\big\|\big)^{1/n}\leqslant L(\gamma)\ \ \ \forall n\geqslant1.\label{eq:LCUpBd}
\end{equation}
Here $S_{n}(\gamma)\in(\mathbb{R}^{d})^{\otimes n}$ is the $n$-th
level signature of $\gamma$, $\|\cdot\|$ is a suitable tensor norm
(e.g. the projective norm) and $L(\gamma)$ is the length of $\gamma$.
The remarkable (and rather surprising) point is that the simple estimate
(\ref{eq:LCUpBd}) becomes asymptotically sharp as $n\rightarrow\infty$.
More precisely, it was implicitly conjectured by Hambly-Lyons \cite{HL10}
which was made explicit in Chang-Lyons-Ni \cite{CLN18} that under
the projective tensor norm $\|\cdot\|_{{\rm proj}}$, 
\begin{equation}
\lim_{n\rightarrow\infty}\big(n!\big\| S_{n}(\gamma)\big\|_{{\rm proj}}\big)^{1/n}=L(\gamma)\label{eq:LC}
\end{equation}
for any tree-reduced (i.e. without treelike pieces) BV path $\gamma$.
In other words, the length of a tree-reduced path is recovered from
its normalised signature asymptotics. The formula (\ref{eq:LC}) was
established for ${\cal C}^{3}$ paths by Hambly-Lyons \cite{HL10}
and ${\cal C}^{1}$ paths by Lyons-Xu \cite{LX15}. It was extended
to the two-dimensional BV case by Boedihardjo-Geng \cite{BG22} without
any regularity assumption but under a stronger notion of tree-reducedness.
A stronger version of (\ref{eq:LC}) for ${\cal C}^{2}$-paths was
recently obtained by \cite{CMT23} generalising a corresponding result
in \cite{HL10}. To our best knowledge, proving (or disproving) the
conjectural formula (\ref{eq:LC}) for the general BV case remains
an open problem.

Returning to the geometric setting, our main idea of reconstructing
the Riemannian distance $d(x,y)$ from suitable asymptotics of the
expcted $dF$-signature $\psi(t,x,y)$ of the Brownian bridge $X^{t,x,y}$
can be summarised as follows.
\begin{enumerate}
\item When $t$ is small, one is essentially forcing the bridge $X^{t,x,y}$
to travel from $x$ to $y$ by a short amout of time and thus $X^{t,x,y}$
behaves like a minimising geodesic $\gamma^{x,y}$ joining $x$ to
$y$. This property is quantified through the large deviation principle
(LDP) of the Brownian bridge established by Hsu \cite{Hsu90}.
\item As a consequence of the first point, it is reasonable to expected
that $\pi_{n}\psi(t,x,y)$ (the $n$-th level projection) is close
to the $n$-th level $dF$-signature of the geodesic $\gamma^{x,y}$.
Note that the latter is just the usual Euclidean signature of $F(\gamma^{x,y})$
in the ambient space $\mathbb{R}^{N}.$
\item As a consequence of the length conjecture (\ref{eq:LC}) (which is
a proven fact due to the smoothness of geodesics), the length of $F(\gamma^{x,y})$
in $\mathbb{R}^{N}$ can be recovered from its normalised signature
asymptotics. Note that this length is precisely the Riemannian distance
$d(x,y)$ since $F$ is an isometric embedding.
\item In view of the second and third points, one naturally expects that
\[
\big(n!\|\pi_{n}\psi(t,x,y)\|\big)^{1/n}\approx d(x,y)
\]
provided that $t$ is small and $n$ is large. 
\end{enumerate}
Our main reconstruction result can be roughly stated as follows.
\begin{thm*}
The Riemannian distance $d(x,y)$ can be reconstructed from the following
asymptotic formula:
\begin{equation}
\lim_{n\rightarrow\infty}\big(n!\|\pi_{n}\psi(t_{n},x,y)\|\big)^{1/n}=d(x,y)\label{eq:ReconRDIntro}
\end{equation}
provided that $t_{n}<<n^{-6}$. 
\end{thm*}
The precise and quantitative statement of such a result, which requires
a few technical assumptions, is given by Theorem \ref{thm:ReconRD}.
Although the underlying idea is simple and natural, proving such a
relation mathematically requires a substantial amount of non-trivial
analysis. For instance, the above second point is not a direct consequence
of the LDP since the latter was proved under the uniform (instead
of rough path) topology and it is well-known that the signature is
not a continuous functional of the driving path with respect to the
uniform topology. As a result, one has to establish the signature
approximation separately and we do so by using the semimartingale
decomposition of Brownian bridge under normal coordinate charts. In
particular, one needs to separate out the geodesic component in the
decomposition and estimate the remainder effectively. 

We should mention a particularly subtle point among several other
technical challenges which will all be clear along the development
of the proof in Section \ref{subsec:RDPf}. For each fixed level $n$,
as long as the lifetime $t_{n}$ is chosen to be small enough one
naturally has 
\[
\pi_{n}\psi(t_{n},x,y)\approx n\text{-th level signature of }\gamma^{x,y}.
\]
If one applies standard signature moment estimates for semimartingales
without much care, one is led to choosing $t_{n}=o(e^{-Cn})$ which
is too small to be practically useful (e.g. in the context of simulating
a Brownian bridge with lifetime $t_{n}$). A main contribution of
our analysis is that the lifetime needs not be exponentially small
in $n$ to make the asymptotics (\ref{eq:ReconRDIntro}) valid. Indeed,
a \textit{polynomial} dependence $t_{n}=O(n^{-6})$ is sufficient
(cf. (\ref{eq:QuanRD}) for the quantitative reconstruction estimate
with explicit error bound). Improving from exponential to polynomial
dependence requires much finer signature estimates for the aforementioned
semimartingale decomposition. To achieve this, we make use of a deep
result of Kallenberg-Sztencel \cite{KS91} on dimension-free BDG inequalities
for Hilbert-space-valued martingales. 

\subsection{Reconstruction of metric and curvature properties}

The asymptotic formula (\ref{eq:ReconRDIntro}) provides a way of
recovering the Riemannian distance $d(x,y)$ from the expected signature
asymptotics of Brownian bridge. This can be viewed as an zeroth order
result. It then becomes reasonable to expect that curvature properties
would start to appear in the higher order asymptotic expansion of
the expected signature function $\psi(t,x,y)$.

To be specific, we consider the case when $x=y$ (the Brownian loop
based at $x$) and look at the function 
\[
\psi_{4}(t,x)\triangleq\pi_{4}\psi(t,x,x)\in(\mathbb{R}^{N})^{\otimes4}.
\]
Due to symmetry considerations, it can be shown that $\pi_{i}\psi_{4}(t,x,x)=0$
for $i=1,2,3$ (see the discussion at the start of Section \ref{subsec:MainThmRecCurv}).
Therefore, $\psi_{4}(t,x)$ is the first non-trivial component of
$\psi(t,x,x)$ which indeed carries rich geometric information. In
Propositions \ref{prop:t2Coef} and \ref{prop:t3Ceof} below, we compute
the expansion of $\psi_{4}(t,x)$ up to order $t^{3}$:
\[
\psi_{4}(t,x)=\hat{\Theta}_{x}t^{2}+\hat{\Xi}_{x}t^{3}+O(t^{4})
\]
with explicit expressions for the tensor coefficients $\hat{\Theta}_{x},\hat{\Xi}_{x}\in(\mathbb{R}^{N})^{\otimes4}$
in terms of metric and curvature coefficients as well as the embedding
$F$$.$ The actual expressions of $\hat{\Theta}_{x}$ and $\hat{\Xi}_{x}$
are too complicated to write down here. However, after applying suitable
tensor contraction $\mathfrak{C}:(\mathbb{R}^{N})^{\otimes4}\rightarrow(\mathbb{R}^{N})^{\otimes2}$
(cf. (\ref{eq:24Trace}) for its precise definition) their expressions
are simplified substantially and their intrinsic meanings in terms
of metric and curvature properties become clear. The main formulae
are given as follows. 
\begin{thm*}
The contracted tensors $\Theta_{x}\triangleq\mathfrak{C}\hat{\Theta}_{x}$
and $\Xi_{x}\triangleq\mathfrak{C}\hat{\Xi}_{x}$, being viewed as
symmetric bilinear forms on $T_{x}M$ (note that $T_{x}M$ is viewed
as a subspace of $\mathbb{R}^{N}$ under the embedding $F$), are
given by the following formulae:
\[
\Theta_{x}=\frac{d-1}{24}g_{x}
\]
and 
\[
\Xi_{x}=\frac{{\rm S}_{x}-18d^{2}|H_{x}|_{\mathbb{R}^{N}}^{2}}{8640}g_{x}+\frac{49d-20}{8640}{\rm Ric}_{x}+\frac{(5-4d)d}{480}\langle B_{x},H_{x}\rangle_{\mathbb{R}^{N}}.
\]
Here $d$ is the dimension of $M,$ $g$ is the Riemannian metric
tensor, ${\rm S}$ is the scalar curvature, ${\rm Ric}$ is the Ricci
curvature tensor, $B_{x}$ is the second fundamental form and $H$
is the mean curvature field with respect to $F$. 
\end{thm*}
The precise formulation of the theorem is given by Theorem \ref{thm:RecCurv}
below, where the relevant geometric concepts are recalled in Section
\ref{subsec:Curv}. As a consequence, the tensor $\Theta_{x}$ recovers
the metric tensor and $\Xi_{x}$ encodes both intrinsic and extrinsic
curvature properties. It is also seen from the formula (\ref{eq:t2Coef})
and Remark \ref{rem:RecMetric} that the tangent space $T_{x}M$ can
be reconstructed from the bilinear form $\Theta_{x}\in{\cal L}(E\times E;\mathbb{R})$.
In addition, by modifying the tensor contraction map $\mathfrak{C}$
it is actually possible to recover the quantities 
\[
{\rm Ric}_{x},{\rm S}_{x},\langle B_{x},H_{x}\rangle_{\mathbb{R}^{N}},|H_{x}|_{\mathbb{R}^{N}}
\]
separately. This is a simple linear algebra matter and is explained
in Remark \ref{rem:SepCurvRec} below.

\subsubsection*{Comparision with heat kernel asymptotics and potential applications}

The aforementioned results should be compared with the classical heat
kernel asymptotics in geometric analysis. It was a renowned theorem
of Varadhan \cite{Var67} that the Riemannian distance function can
be recovered from the small-time asymptotics of the heat kernel:
\begin{equation}
\lim_{t\rightarrow\infty}t\log p(t,x,y)=-\frac{1}{2}d(x,y)^{2}.\label{eq:VaradAsymp}
\end{equation}
In addition, it was shown by Minakshisundaram-Pleijel \cite{MP49}
that the heat kernel admits the following asymptotic expansion: 
\[
p(t,x,y)\sim\frac{1}{(2\pi t)^{d/2}}e^{-\frac{d(x,y)^{2}}{2t}}\sum_{k=0}^{\infty}u_{k}(x,y)t^{k}\ \ \ \text{as }t\rightarrow0^{+},
\]
where the coefficients $u_{k}(x,y)$ are certain $C^{\infty}$-functions
defined in a neighbourhood of the diagonal $\{(x,x):x\in M\}$. It
is known that $u_{0}(x,x)=1$, $u_{1}(x,x)={\rm S}_{x}/12$ (${\rm S}_{x}$
is the scalar curvature at $x$) and $u_{k}(x,x)$ are certain universal
polynomial functions of curvature coefficients and their covariant
derivatives at $x$. These asymptotic results indicate that some aspects
of the geometry / shape of the underlying space can be inferred from
the small-time asymptotics of the heat kernel. On the applied side,
the Varadhan asymptotics formula (\ref{eq:VaradAsymp}) provides the
theoretical foundation for the so-called \textit{heat method} in manifold
learning that are widely used in problems related to
learning the shape / structure of high-dimensional data sets (cf.
\cite{CWW17,ZCX20,LS23} and the references therein).

Our asymptotic results (Theorems \ref{thm:ReconRD} and \ref{thm:RecCurv})
provide a different perspective and mechanism for learning the shape
of the underlying manifold. It is based on the use of Brownian bridges
and their associated signature features instead of the heat kernel.
From the practical viewpoint, both the simulation of Brownian bridges
on manifolds and the effective computation of path signatures have
been largely developed over the past years by various groups (cf.
\cite{Jen22}, \textcolor{blue}{\cite{KL20}}
and the references therein). We therefore expect that our results
will potentially lead to new applications in manifold learning. As a mathematical
work on its own, we do not address applications in data sciences in
the current article and leave these exciting questions to future work. 

\section{Some geometric background and a notion of signature on manifolds}

We discuss the basic geometric set-up, introduce some necessary geometric
tools and define a natural notion of geometric signature that is used
in the current work.

\subsection{Notions from Riemannian geometry}

Throughout the rest of the article, let $(M,g)$ be a $d$-dimensional,
compact, oriented Riemannian manifold without boundary. We choose
to focus on the compact case to avoid non-essential technicalities;
we do not use global geometric properties (e.g. spectral decomposition)
and thus the main results extend naturally to the non-compact case under
extra technical conditions. We begin by reviewing some basics on Riemannian geometry. A standard reference for this part is \cite{DoC92}. 

\vspace{2mm}\noindent \textit{Convention}. Unless otherwise stated,
we always adopt Einstein's convention that repeated indices in an
expression are summed automatically over their range.

\subsubsection{The Laplace-Beltrami operator}

A fundamental analytic object on $M$ is the so-called \textit{Laplace-Beltrami
operator}. This is a second order differential operator which admits the
following local expression: 
\begin{equation}
\Delta f=\frac{1}{\sqrt{\det g}}\partial_{i}\big(\sqrt{\det g}g^{ij}\partial_{j}f\big).\label{eq:LapLocal}
\end{equation}
Here $g=(g_{ij})$ is the metric tensor and $(g^{ij})$ is the inverse
of $g$ under a local chart. The Laplacian $\Delta$ can be alternatively defined as follows.
Suppose that $M$ is isometrically embedding inside some Euclidean
space $\mathbb{R}^{N}$ with canonical basis $\{e_{1},\cdots,e_{N}\}$.
Let $V_{i}$ ($1\leqslant i\leqslant N$) be the vector field on $M$
defined by orthogonally projecting $e_{i}$ onto the tangent
space at every point of $M$. Then one has $\Delta=\sum_{i=1}^{N}V_{i}^{2}$,
where $V_{i}$ is equivalently viewed as a differential operator on
$M$ (defined by taking directional derivative along $V_{i}$).

The Laplace-Beltrami operator $\Delta$ is a non-positive definite,
unbounded, essentially self-adjoint operator on $L^{2}(M,dx)$ ($dx$
is the volume measure) which admits a kernel $p(t,x,y)$ known as
the \textit{heat kernel}. This is the fundamental solution to the
heat equation, i.e. the smallest positive solution to the equation
\begin{equation}
\begin{cases}
\partial_{t}p(t,x,y)=\frac{1}{2}\Delta_{x}p(t,x,y), & t>0,x,y\in M;\\
\lim_{t\downarrow0}p(t,x,\cdot)=\delta_{x},
\end{cases}\label{eq:HeatEq}
\end{equation}
where $\delta_{x}$ denotes the Dirac delta function on $M$. The
factor $1/2$ is introduced for its connection with Brownian motion;
the Brownian motion $B_{t}$ is generated by $\Delta/2$ and the heat
kernel $p(t,x,y)$ is also the transition density of Brownian motion:
\[
\mathbb{P}(B_{t}\in dy|B_{0}=x)=p(t,x,y)dy.
\]

\subsubsection{The Malliavin-Stroock heat kernel expansion}

Small-time expansions of the heat kernel are well studied in the geometric
analysis literature (cf. \cite{BGV04} and the references therein).
We are going to use the following expansion for the logarithmic derivative
of the heat kernel, which was due to Malliavin and Stroock \cite{MS96}.
Recall that the \textit{injective radius} at $x\in M$ is the largest
$\rho_{x}>0$ such that the exponential map $\exp_{x}:T_{x}M\rightarrow M$
is a diffeomorphism on the metric ball $B_{x}(\rho)\triangleq\{y:d(x,y)<\rho\}.$
Here $d(x,y)$ denotes the Riemannian distance function. The \textit{global
injective radius} $\rho_{M}$ is the minimum of $\rho_{x}$ over all
$x\in M$. This is a finite positive number due to the compactness
of $M$. For example, $\rho_M = \pi$ if $M=S^2$ (the unit sphere in $\mathbb{R}^3$). Define $$S_{M}\triangleq\{(x,y):d(x,y)<\rho_{M}\}.$$ It is
clear that any pair $(x,y)\in S_{M}$ can be joined by a unique minimising
geodesic.
\begin{thm}
The following expansion holds uniformly on each compact subset of
$S_{M}:$
\begin{equation}
u\nabla_{x}\log p(u,x,y)\sim\sum_{k=0}^{\infty}\nabla_{x}G_{k}(x,y)u^{k}\ \ \ \text{as }u\searrow0,\label{eq:MSExp}
\end{equation}
where $\nabla_{x}$ denotes the Riemannian gradient with respect to
the $x$-variable and $G_{k}(x,y)$ are intrinsic $C^{\infty}$-functions
on $S_{M}$.
\end{thm}

The following explicit expressions will be used in the sequel:
\begin{equation}
G_{0}(x,y)=-\frac{d(x,y)^{2}}{2},\ G_{1}(x,y)=-\frac{1}{2}\log\det(d\exp_{y})_{{\bf x}},\ \ \ (x,y)\in S_{M}.\label{eq:G0G1}
\end{equation}
Here $\exp_{y}:T_{y}M\rightarrow M$ denotes the exponential map,
${\bf x}\triangleq\exp_{y}^{-1}(x)$ and $(d\exp_{y})_{{\bf x}}$
is the differential of $\exp_{y}$ at ${\bf x}$, which is a linear
isomorphism between Euclidean spaces $T_{{\bf x}}T_{y}M\cong T_{y}M$
and $T_{x}M$. The functions $G_{k}(x,y)$ ($k\geqslant2$) can also
be determined explicitly in a recursive manner, but this is not needed
for us.

\subsubsection{\label{subsec:Curv}Ricci curvature and the second fundamental form }

Many fundamental geometric properties are related to curvature. Recall
that the \textit{Riemannian curvature tensor} is defined by 
\[
R(X,Y)Z\triangleq\nabla_{X}\nabla_{Y}Z-\nabla_{Y}\nabla_{X}Z-\nabla_{[X,Y]}Z,
\]
where $\nabla$ is the Levi-Civita connection and $X,Y,Z$ are smooth
vector fields. Note that $R$ is a type-$(1,3)$ tensor field on $M$
and for each $x\in M$, it gives rise to a multilinear map $R_{x}:T_{x}M\times T_{x}M\times T_{x}M\rightarrow T_{x}M$.
The \textit{Ricci curvature tensor} at $x$ is the symmetric bilinear
form ${\rm Ric}_{x}:T_{x}M\times T_{x}M\rightarrow\mathbb{R}$ defined
by 
\[
{\rm Ric}_{x}(v,w)\triangleq{\rm Tr}\big[u\mapsto R_{x}(u,v)w\big],\ \ \ v,w\in T_{x}M.
\]
The \textit{scalar curvature} at $x$ is the trace of ${\rm Ric}_{x}$
and is denoted as ${\rm S}_{x}$. Note that ${\rm Ric}$ is a type-$(0,2)$
tensor field and ${\rm S}$ is a smooth function on $M$. All these
curvature quantities are \textit{intrinsic} in the sense that their
definitions only depend on the metric tensor. 

On the other hand, one can also consider extrinsic curvature properties.
Let $F:M\rightarrow\mathbb{R}^{N}$ be a given isometric embedding.
For each $x\in M$, let $(T_{x}M)^{\perp}$ denote the orthogonal
complement of $T_{x}M$ inside $\mathbb{R}^{N}$ ($T_{x}M$ is viewed
as a subspace of $T_{F(x)}\mathbb{R}^{N}\cong\mathbb{R}^{N}$ through
the embedding). The\textit{ second fundamental form} at $x$ is the
vector valued bilinear form defined by 
\begin{equation}
B_{x}(v,w)\triangleq\tilde{\nabla}_{\tilde{X}}\tilde{Y}-\nabla_{X}Y,\ \ \ v,w\in T_{x}M.\label{eq:2FundForm}
\end{equation}
Here $X,Y$ are any vector fields on $M$ such that $X_{x}=v,Y_{x}=w$
and $\tilde{X},\tilde{Y}$ are any extensions of $X,Y$ to $\mathbb{R}^{N}$.
The operator $\tilde{\nabla}$ is the Levi-Civita connection in $\mathbb{R}^{N}$(which
is just usual Euclidean differentiation). It can be shown that the
expression (\ref{eq:2FundForm}) is well-defined (independent of choices
of $X,Y$ and their extensions) and $B_{x}$ is an $(T_{x}M)^{\perp}$-valued,
symmetric bilinear form on $T_{x}M\times T_{x}M$. In other words,
$B$ is a type-$(0,2)$ tensor field on $M$ taking values in the
normal bundle $(TM)^{\perp}$. It is also known that $\nabla_{X}Y$
is the tangential component of $\tilde{\nabla}_{\tilde{X}}\tilde{Y}$
(i.e. its projection onto $TM$). As a result, $B(X,Y)$ can also
be defined as the normal component of $\tilde{\nabla}_{\tilde{X}}\tilde{Y}$
(i.e. its projection onto $(TM)^{\perp}$). The \textit{mean curvature
vector} at $x\in M$ is the normal vector defined by 
\begin{equation}
H_{x}\triangleq d^{-1}{\rm Tr}B_{x}\in(T_{x}M)^{\perp},\label{eq:MeanCur}
\end{equation}
where $d$ is the dimension of $M$. By varying $x$, this gives rise
to a smooth section $H$ of the normal bundle $(TM)^{\perp}$. It
can be shown that 
\begin{equation}
\Delta F=d\cdot H,\label{eq:Lap=00003DH}
\end{equation}
where $\Delta$ is the Laplace-Beltrami operator on $M$. Note that the second
fundamental form is an \textit{extrinsic} curvature quantity as it
depends on the embedding $F$ (it describes how $M$ is embedded /
curved inside the ambient space).

\subsection{Brownian motion and bridge}

In this section, we recall the construction of Brownian motion and
Brownian bridge on $M$. The reader is referred to \cite{Hsu02} for
more details.

\subsubsection{The Eells-Elworthy-Malliavin construction of Brownian motion}

The Brownian motion on $M$ could just be defined as the Markov process
with generator $\Delta/2$ from the Markovian perspective. However,
the following intrinsic (SDE / pathwise) construction by Eells-Elworthy-Malliavin
will be useful for our derivation of the intrinsic PDE for the expected
signature dynamics in Section \ref{sec:PDE}.

A main difficulty in the intrinsic SDE construction of Brownian motion
is that the Laplacian may not always admit a decomposition $\Delta=\sum_{i}V_{i}^{2}$
with intrinsic vector fields $V_{i}$. This issue is overcome by lifting
the construction to the so-called \textit{orthonormal frame bundle}
(OFB) ${\cal O}(M)$ over $M$ defined by

\[
{\cal O}(M)\triangleq\bigcup_{x\in M}{\cal F}_{x}.
\]
Here ${\cal F}_{x}$ denotes the set of ONBs of $T_{p}M$. A generic
element in ${\cal O}(M)$ is given by a pair $(x,u)$, where $x\in M$
and $u=(e_{1},\cdots,e_{d})$ is an ONB of $T_{x}M$. The OFB ${\cal O}(M)$
is a principal bundle over $M$ with structure group $O(d)$ (the orthogonal group) acting
from the right:
\[
(x,u)\cdot Q\triangleq(x,u\cdot Q),\ \ \ x\in M,u=(e_{1},\cdots,e_{d})\in{\cal F}_{x}.
\]
In particular, it is a differentiable manifold of dimension
\[
D=\dim M+\dim O(d)=\frac{d^{2}+d}{2}.
\]
Let $\pi:{\cal O}(M)\rightarrow M$ denote the canonical projection.

An essential point of considering ${\cal O}(M)$ is that one can construct
intrinsic (horizontal) vector fields on it. Let $1\leqslant i\leqslant d$
be fixed. Given a point 
\[
\xi=(x,u=(e_{1},\cdots,e_{d}))\in{\cal O}(M),
\]
note that $e_{i}$ is a tangent vector of $M$ at $x$. Let $\gamma:(-\varepsilon,\varepsilon)\rightarrow M$
be a curve such that $\gamma_{0}=x$ and $\gamma'_{0}=e_{i}.$ By
parallel-translating the ONB $u$ along $\gamma$, one obtains a frame
of ONBs along $\gamma$, more precisely a curve $\xi_{t}\triangleq(\gamma_{t},u_{t})\in{\cal O}(M)$
where $u_{t}$ is an ONB of $T_{\gamma_{t}}M$. Define 
\[
H_{i}(\xi)\triangleq\xi'_{0}\in T_{\xi}{\cal O}(M).
\]
Since $\xi$ is arbitrary, one then obtains a global vector field
$H_{i}$ on ${\cal O}(M).$ The vector fields $\{H_{1},\cdots,H_{d}\}$
are called \textit{fundamental horizontal vector fields} on ${\cal O}(M).$ 

It is well known that the second order differential operator ${\cal L}\triangleq\sum_{i=1}^{d}H_{i}^{2}$
naturally projects to the Laplacian on $M$. This motivates the Eells-Elworthy-Malliavin
construction of Brownian motion given by the theorem below. The essential
idea is to first construct a horizontal Brownian motion on ${\cal O}(M)$
by solving an intrinsic SDE associated with the horizontal vector
fields $\{H_{i}\}$ and then to obtain an $M$-valued Brownian motion
through projection.
\begin{thm}\label{thm:HorBM}
Let $x\in M$ and $u=(e_{1},\cdots,e_{d})$ be a given fixed ONB of
$T_{x}M$. Let $\Xi_{t}$ be the solution to the following
Stratonovich SDE on $\mathcal{O}(M)$:
\begin{equation}
\begin{cases}
d\Xi_{t}=\sum_{i=1}^{d}H_{i}(\Xi_{t})\circ dB_{t}^{i},\\
\Xi_{0}=(x,u)\in{\cal O}(M),
\end{cases}\label{eq:HorBM}
\end{equation}
where $(B^{1},\cdots,B^{d})$ denotes a $d$-dimensional Euclidean
Brownian motion. Define $W_{t}\triangleq\pi(\Xi_{t})\in M.$ Then
the law of $W$ depends only on $x$ but not on the initial frame
$u$. By varying $x$ over $M$, one obtains a Markovian family whose
generator is $\Delta/2$ (thus a Brownian family on $M$).
\end{thm}

\subsubsection{The Brownian bridge}

Let $x,y\in M$ and $t\in(0,1]$ be given fixed. The \textit{Brownian
bridge} $(X_{s}^{t,x,y})_{0\leqslant s\leqslant t}$ from $x$ to
$y$ with lifetime $t$ is the Brownian motion starting at $x$ conditional
on reaching $y$ at time $t$:
\[
X_{s}^{t,x,y}=W_{s}^{x}|_{W_{t}^{x}=y},\ \ \ 0\leqslant s\leqslant t,
\]
where $W^{x}$ is a Brownian motion starting at $x$. Simple calculation
shows that this is a (time-inhomogeneous) Markov process with transition
density 
\begin{equation}
\mathbb{P}\big(X_{s_{2}}^{t,x,y}\in dz|X_{s_{1}}^{t,x,y}=w\big)=\frac{p(t-s_{2},z,y)p(s_{2}-s_{1},w,z)}{p(t-s_{1},w,y)}dz,\ \ \ 0\leqslant s_{1}<s_{2}<t\label{eq:BBTran}
\end{equation}
and generator
\[
({\cal L}f)(w)=\frac{1}{2}\Delta f(w)+\nabla_{w}\log p(t-s,w,y)\cdot\nabla f(w).
\]
Here $p(t,x,y)$ is the heat kernel for $\Delta/2$. One can also use the above Markovian perspective to construct the
Brownian bridge mathematically.

If the Brownian motion $W^{x}$ admits an SDE representation (e.g.
under some coordinate chart) 
\[
dW_{s}^{x}=b(W_{s}^{x})ds+\sigma(W_{s}^{x})dB_{s},\ W_{0}^{x}=x
\]
in $\mathbb{R}^{d},$ then the Brownian bridge $X^{t,x,y}$ has a
corresponding SDE representation
\begin{equation}
\begin{cases}
dX_{s}^{t,x,y}=\big(b(X_{s}^{t,x,y})+\nabla\log p(t-s,X_{s}^{t,x,y},y)\big)ds+\sigma(X_{s}^{t,x,y})dB_{s}, & 0\leqslant s<t;\\
X_{0}^{t,x,y}=x.
\end{cases}\label{eq:BBSDE}
\end{equation}
By abuse of notation the above $\nabla$ is its local expression under
the given chart (i.e. the vector $(g^{ij}\partial_{j})_{1\leqslant i\leqslant d}$)
. 
\begin{example}
In the Euclidean case $M=\mathbb{R}^{d}$, by using the explicit formula
for the heat kernel, one finds that 
\[
\nabla\log p(t-s,x,y)=\frac{y-x}{t-s}.
\]
The SDE (\ref{eq:BBSDE}) becomes 
\[
\begin{cases}
dX_{s}^{t,x,y}=-\frac{1}{t-s}X_{s}^{t,x,y}ds+dB_{s},\\
X_{0}^{t,x,y}=x-y=:\xi.
\end{cases}
\]
Its solution is explicitly given by 
\[
X_{s}^{t,x,y}=\frac{t-s}{t}\xi+(t-s)\int_{0}^{s}\frac{dB_{u}}{t-u},\ \ \ 0\leqslant s<t.
\]
\end{example}

\subsection{The $\phi$-signature on manifolds}

In this section, we introduce a natural notion of signature for paths
on differentiable manifolds (iterated integrals against one-forms).
This was first used by K.T. Chen in a geometric setting and also by
Le Jan-Qian in the study of the signature uniqueness theorem for Brownian
motion (they called it \textit{extended signature}).

We first recall the definition of path signature in the Euclidean space. Let $E=\mathbb{R}^{N}$
be equipped with the Euclidean metric. The (infinite) tensor algebra
over $E$ is the unital algebra 
\[
T((E))\triangleq\prod_{n=0}^{\infty}E^{\otimes n}=\big\{ a=(a_{0},a_{1},a_{2},\cdots):a_{i}\in E^{\otimes i}\ \forall i\in\mathbb{N}\big\},
\]
where addition $+$ is just the vector addition, multiplication $\otimes$
is defined by 
\[
(a\otimes b)_{m}\triangleq\sum_{k=0}^{m}a_{k}\otimes b_{m-k}\in(\mathbb{R}^{N})^{\otimes m},\ \ \ a=(a_{i}),b=(b_{i})\in T((E))
\]
and the unit is ${\bf 1}\triangleq(1,0,0,\cdots).$ The \textit{signature}
of a smooth path $\gamma:[0,T]\rightarrow E$ is the element in $T((E))$
defined by 
\[
S(\gamma)\triangleq\big(1,\gamma_{T}-\gamma_{0},\int_{0<s<t<T}d\gamma_{s}\otimes d\gamma_{t},\cdots,\int_{0<t_{1}<\cdots<t_{n}<T}d\gamma_{t_{1}}\otimes\cdots\otimes d\gamma_{t_{n}},\cdots\big).
\]
More generally, one could also define the \textit{signature path}
$t\mapsto S_{t}(\gamma)$ by integrating up to time $t$ instead ($S(\gamma)=S_{T}(\gamma)$).
The definition extends to the rough path setting according to Lyons'
extension theorem (cf. \cite[Theorem 2.2.1]{Lyo98}). If $\gamma$
is the Brownian motion, the above iterated integrals are equivalently understood
in the Stratonovich sense. It is standard that the signature path
$S_{t}(\gamma)$ satisfies the following differential equation on
$T((E))$:
\begin{equation}
dS_{t}(\gamma)=S_{t}(\gamma)\otimes d\gamma_{t}.\label{eq:SigDE}
\end{equation}

If $\gamma$ is a smooth path on a differentiable manifold $M$, one
cannot integrate along $\gamma$ intrinsically without any additional structure;
one needs to integrate $\gamma$ against a one-form. Recall that a
(smooth) \textit{one-form} is a smooth section $\phi$ of the cotangent bundle. In other words, it is a map that assigns to each location $x\in M$ a linear functional $\phi(x)\in T_x^*M$ in a smooth manner.  The integral 
\[
\int_{0}^{t}\phi(d\gamma)\triangleq\int_{0}^{t}{}_{T_{\gamma_{s}}^{*}M}\langle\phi(\gamma_{s}),\gamma'_{s}\rangle_{T_{\gamma_{s}}M}ds
\]
clearly has an intrinsic meaning and is called the \textit{line integral}
of $\gamma$ against $\phi$.
\begin{example}
Let $M$ be the torus $S^{1}\times S^{1}$ and let $\phi=(\alpha,\beta)$
be the two canonical generators of the first homology group of $M$. Then $\int_{0}^{t}\phi(d\gamma)\in\mathbb{R}^{2}$
describes the winding angles of $\gamma$ up to time $t$ with respect
to the vertical and horizontal circles.
\end{example}

Now let $E=\mathbb{R}^{N}$ and suppose that $\phi$ is a smooth $E$-valued
one-form on $M$. In other words, $\phi(x)$ is an $E$-valued linear
functional on $T_{x}M$. Equivalently, one can write $\phi=(\phi^{1},\cdots,\phi^{N})$
where each $\phi^{i}$ is a real-valued one-form.
\begin{defn}
Let $\gamma:[0,T]\rightarrow M$ be a smooth path. The $\phi$\textit{-signature}
of $\gamma$ up to time $t$ is defined to be
\[
S_{t}^{\phi}(\gamma)\triangleq S_{t}(\int_{0}^{\cdot}\phi(d\gamma))\in T((E)).
\]
This is the signature up to time $t$ of the $E$-valued path $\int_{0}^{\cdot}\phi(d\gamma)$
. We also write $S^{\phi}(\gamma)\triangleq S_{T}^{\phi}(\gamma)$.
\end{defn}

\begin{rem}
This notion of signature depends on the choice of $\phi$. Indeed,
the usual Euclidean signature is a special case of this: the signature
of an $\mathbb{R}^{N}$-valued path is the $\phi$-signature with
respect to the $\mathbb{R}^{N}$-valued one-form $\phi=(dx^{1},\cdots,dx^{N}).$
\end{rem}

\begin{rem}
The definition extends naturally to the rough path case. We choose
not to introduce any technicalities from general rough path theory
since they are not essential for most of our analysis. In our study,
$\gamma$ will either be a smooth path (a geodesic) or an $M$-valued
semimartingale where the Stratonovich calculus for $\gamma$ is classical
(e.g. via embedding $M$ into an ambient Euclidean space and perform
the usual stochastic calculus over there). The reader is referred
to \cite{Hsu02} for an excellent introduction to stochastic analysis
on manifolds.
\end{rem}

\begin{example}
An important example which is of our primary interest is the case
when $M$ is embedded inside some Euclidean space $\mathbb{R}^{N}$
and the one-form $\phi=dF$ ($F$ is the embedding map). In this case,
one has $\int_{0}^{t}\phi(d\gamma_{s})=F(\gamma_{t})-F(\gamma_{0})$
and the $\phi$-signature of $\gamma$ is just the Euclidean
signature of $F(\gamma_{t})$ in $\mathbb{R}^{N}$.
\end{example}

\section{\label{sec:PDE}The expected Brownian signature dynamics}

In this section, we consider the expected $\phi$-signature of Brownian
motion and Brownian bridge on $M$. We derive intrinsic PDEs describing
their dynamics from the perspective of the Eells-Elworthy-Malliavin
lifting. Throughout the rest, $E=\mathbb{R}^{N}$ and $\phi$ is a
given fixed $E$-valued one-form on $M$.

\subsection{Intrinsic PDE for the expected $\phi$-signature }

Let $W^{x}=\{W_{t}^{x}:t\geqslant0\}$ be a Brownian motion in $M$
starting at $x\in M$. Consider the expected $\phi$-signature of
$W^{x}$ up to time $t$ defined by
\begin{equation}
\Psi(t,x)\triangleq\mathbb{E}[S_{t}^{\phi}(W^{x})]\in T((E)),\ \ \ (t,x)\in[0,\infty)\times M.\label{eq:DefESig}
\end{equation}
Note that $\Psi$ is well-defined due to the compactness of $M$;
by embedding $M$ into an Euclidean space $V$ one can view $W^{x}$
as the solution to an SDE on $V$ with $C_{b}^{\infty}$-coefficients
and $\phi$ also extends to a $C_{b}^{\infty}$-form on $V$. We first
derive the intrinsic PDE governing the dynamics of $\Psi(t,x)$.
\begin{thm}
\label{thm:BMSigPDE}The function $\Psi(t,x)$ satisfies the following
parabolic PDE on $M$:

\begin{equation}
\frac{\partial\Psi}{\partial t}=\frac{1}{2}\Delta\Psi+{\rm Tr}\big(\phi\otimes d\Psi+\frac{1}{2}(\nabla\phi+\phi\otimes\phi)\otimes\Psi\big)\label{eq:PhiSigPDE}
\end{equation}
with initial condition $\Psi(0,\cdot)={\bf 1}.$
\end{thm}

Before getting into the proof, some explanation about the notation
in (\ref{eq:PhiSigPDE}) is needed.

\vspace{2mm}\noindent (i) The PDE (\ref{eq:PhiSigPDE}) is $T((E))$-valued.
Its projection onto the first $m$ degrees becomes a finite dimensional
coupled PDE system. \\
(ii) Since $\Psi$ is an $T((E))$-valued function on $M$ (for fixed
time $t$), $d\Psi$ is a $T((E))$-valued one-form. The product $\phi\otimes d\Psi$
is the $T((E))$-valued bilinear form defined by 
\[
\phi\otimes d\Psi(U,V)\triangleq\langle\phi,U\rangle\otimes\langle d\Psi,V\rangle
\]
where $U,V$ are vector fields on $M$ and the tensor product $\otimes$
is the multiplication in $T((E))$. The trace ${\rm Tr}(\phi\otimes d\Psi)$
becomes a $T((E))$-valued function on $M$. At each point $x\in M$,
one has 
\[
{\rm Tr}(\phi\otimes d\Psi)(x)=\sum_{i=1}^{d}\langle\phi,e_{i}\rangle(x)\otimes\langle d\Psi(t,x),e_{i}\rangle,
\]
where $\{e_{1},\cdots,e_{d}\}$ is any ONB of $T_{x}M.$ The above
expression is clearly independent of the choice of the ONB.\\
(iii) $\nabla\phi$ is the covariant derivative of $\phi$ with respect
to the Levi-Civita connection. The zeroth order term in (\ref{eq:PhiSigPDE})
is the $T((E))$-valued function defined by 
\[
{\rm Tr}\big((\nabla\phi+\phi\otimes\phi)\otimes\Psi\big)(x)=\sum_{i=1}^{d}\big(\langle\nabla_{e_{i}}\phi,e_{i}\rangle(x)+\langle\phi,e_{i}\rangle(x)\otimes\langle\phi,e_{i}\rangle(x)\big)\otimes\Psi(t,x).
\]
This expression is also independent of the choice of the ONB $\{e_{1},\cdots,e_{d}\}$
of $T_{x}M$.
\begin{example}
In the Euclidean case $M=\mathbb{R}^{d}$ with $\phi=(dx^{1},\cdots,dx^{d}),$
the term $\nabla\phi+\phi\otimes\phi$ appearing in the PDE (\ref{thm:BMSigPDE})
is constant in space. As a result, the PDE must also be spatially
homogeneous and thus reduces to the ODE 
\begin{equation}
\frac{d\Psi}{dt}=\frac{1}{2}\sum_{i=1}^{d}E_{i}\otimes E_{i}\otimes\Psi,\ \ \ \Psi(0,\cdot)={\bf 1},\label{eq:SigODERd}
\end{equation}
where $\{E_{1},\cdots,E_{d}\}$ is the canonical basis of $\mathbb{R}^{d}$.
Its explicit solution is the expected signature of Euclidean Brownian
motion:
\begin{equation}
\Psi(t,x)=\exp\big(\frac{t}{2}(E_{1}\otimes E_{1}+\cdots+E_{d}\otimes E_{d})\big).\label{eq:ESigRd}
\end{equation}
\end{example}

\subsubsection*{Proof of Theorem \ref{thm:BMSigPDE}}

We now proceed to prove Theorem \ref{thm:BMSigPDE}. The main observation
is that the horizontal Brownian motion on the OFB ${\cal O}(M)$ (cf.
Theorem \ref{thm:HorBM}) coupled with its signature process satisfies
an intrinsic SDE on the ${\cal O}(M)\times T((E))$. A benefit from
this is that one can easily identify the generator from the SDE and
write down the associated Kolmorogov's forward equation on the bundle
${\cal O}(M)$ (Feynman-Kac representation). The PDE on ${\cal O}(M)$
then naturally projects to the desired intrinsic PDE (\ref{eq:PhiSigPDE})
on the base manifold $M$. Some technical care is needed to implement
this idea precisely and we develop the steps carefully in what follows. 

First of all, one has the following standard fact. 
\begin{lem}
\label{lem:PDEUniqueness}The PDE (\ref{eq:PhiSigPDE}) with initial
condition $\Psi(0,\cdot)={\bf 1}$ has a unique smooth solution denoted
as $\hat{\Psi}(t,x)$ (i.e. every component of $\hat{\Psi}$ is smooth
in $(t,x)$).
\end{lem}

\begin{proof}
Let $f_{0}\in{\cal C}(M)$ and $g\in{\cal C}([0,\infty)\times M)$
be given functions. Consider the following inhomogeneous (scalar)
Cauchy problem on $M$:
\begin{equation}
\begin{cases}
\frac{\partial u}{\partial t}=\frac{1}{2}\Delta u+g,\\
u(0,\cdot)=f_{0}.
\end{cases}\label{eq:StandardCauchy}
\end{equation}
Since $M$ is compact, it is standard that the above PDE has a unique
solution $u(t,x)$ which is smooth for all positive times. In fact,
the solution is explicitly given by 
\[
u(t,x)=\int_{M}p(t,x,y)f_{0}(y)dy+\int_{0}^{t}\int_{M}p(t-s,x,y)g(s,y)dyds.
\]
The integrals are well-defined due to the compactness of $M$ and
the continuity of $f_{0},g.$ The smoothness of $u$ follows from
the smoothness of the heat kernel $p(t,x,y)$. 

To prove the lemma, it is enough to observe that the PDE (\ref{eq:PhiSigPDE})
is completely decoupled into the (scalar) form of (\ref{eq:StandardCauchy})
when one rewrites it in terms of coordinate components of $\Psi$. In fact, one can solve the PDE (\ref{eq:StandardCauchy}) inductively on the
degree of signature where $g$ is given by lower degree components
which is presumed to be known by induction. 
\end{proof}

Next, we lift the solution $\hat{\Psi}(t,x)$ given by Lemma \ref{lem:PDEUniqueness}
to the bundle ${\cal O}(M)$. Let us define 
\begin{equation}
f(t,\xi)\triangleq\hat{\Psi}(t,\pi(\xi)),\ \ \ t\geqslant0,\xi\in{\cal O}(M),\label{eq:f(t,xi)}
\end{equation}
where $\pi:{\cal O}(M)\rightarrow M$ denotes the bundle projection.
We also denote $\Phi\triangleq\pi^{*}\phi$ as the pullback of $\phi$
on the bundle. Then one can establish the PDE for $f(t,\xi)$ on ${\cal O}(M)$.
\begin{lem}
\label{lem:PDEBundle}The function $f(t,\xi)$ satisfies the following
parabolic PDE:
\begin{equation}
\frac{\partial f}{\partial t}=\frac{1}{2}\sum_{i=1}^{d}H_{i}^{2}f+\sum_{i=1}^{d}\langle\Phi,H_{i}\rangle\otimes H_{i}f+\frac{1}{2}\sum_{i=1}^{d}\big(H_{i}\langle\Phi,H_{i}\rangle+\langle\Phi,H_{i}\rangle^{\otimes2}\big)\otimes f\label{eq:HorPDE}
\end{equation}
with initial condition $f(0,\xi)={\bf 1}.$
\end{lem}

\begin{proof}
Since $\hat{\Psi}(t,x)$ satisfies the PDE (\ref{eq:PhiSigPDE}) on
$M$ and $f$ is the pullback of $\hat{\Psi}$ on the bundle ${\cal O}(M)$,
it is sufficient to prove that the pullback of the right hand side
of (\ref{eq:PhiSigPDE}) is precisely the right hand side of (\ref{eq:HorPDE}).
We look at each individual term separately. 

\vspace{2mm}\noindent (i) The horizontal Laplacian $\sum_{i=1}^{d}H_{i}^{2}$
on ${\cal O}(M)$ is the pullback of the Laplacian $\Delta$ on $M$
(cf. \cite[Proposition 3.1.2]{Hsu02}):
\[
\sum_{i=1}^{d}H_{i}^{2}f=\sum_{i=1}^{d}H_{i}^{2}\pi^{*}\hat{\Psi}=\pi^{*}\Delta\hat{\Psi}.
\]
(ii) We claim that 
\[
\sum_{i=1}^{d}\langle\Phi,H_{i}\rangle\otimes H_{i}f=\pi^{*}{\rm Tr}(\phi\otimes d\hat{\Psi}).
\]
Indeed, let $\xi=(x,u)\in{\cal O}(M)$ with $u=\{e_{1},\cdots,e_{d}\}$.
Then one has 
\[
\langle\Phi,H_{i}\rangle(\xi)=\langle\pi^{*}\phi,H_{i}\rangle(\xi)=\langle\phi,(\pi_{*})_{\xi}H_{i}\rangle(x)=\langle\phi,e_{i}\rangle(x).
\]
Similarly, 
\[
(H_{i}f)(\xi)=\langle df,H_{i}\rangle(\xi)=\langle d\hat{\Psi},e_{i}\rangle(x).
\]
It follows that
\[
\big(\sum_{i=1}^{d}\langle\Phi,H_{i}\rangle\otimes H_{i}f\big)(\xi)=\big(\sum_{i=1}^{d}\langle\phi,e_{i}\rangle\otimes\langle d\hat{\Psi},e_{i}\rangle\big)(x)={\rm Tr}((\phi\otimes d\hat{\Psi}))(x).
\]
Note that the trace, as a bilinear form on the inner product space
$T_{x}M$, is independent of the choice of the ONB $u$. In a similar
way, one also has 
\[
\sum_{i=1}^{d}\langle\Phi,H_{i}\rangle^{\otimes2}=\pi^{*}{\rm Tr}(\phi\otimes\phi).
\]
(iii) It remains to show that 
\begin{equation}
\sum_{i=1}^{d}H_{i}\langle\Phi,H_{i}\rangle=\pi^{*}{\rm Tr}(\nabla\phi).\label{eq:0OrderTr}
\end{equation}
As before, let $\xi=(x,u)$ be given fixed with $u=\{e_{1},\cdots,e_{d}\}$.
Let $\gamma:(-\varepsilon,\varepsilon)\rightarrow M$ be a curve in
$M$ such that $\gamma_{0}=x$ and $\gamma'_{0}=e_{i}.$ Let $\xi_{t}=(\gamma_{t},u_{t})$
be the horizontal lifting of $\gamma$ with $\xi_{0}=\xi.$ Then one
has 
\begin{align*}
H_{i}\langle\Phi,H_{i}\rangle(\xi) & =\frac{d}{dt}\big|_{t=0}\langle\Phi,H_{i}\rangle(\xi_{t})=\frac{d}{dt}\big|_{t=0}\langle\phi,(\pi_{*})_{\xi_{t}}H_{i}\rangle(\gamma_{t})\\
 & =\frac{d}{dt}\big|_{t=0}\langle\phi(\gamma_{t}),e_{i}(t)\rangle=\langle\frac{D\phi(\gamma_{t})}{dt}\big|_{t=0},e_{i}\rangle+\langle\phi(x),\frac{De_{i}(t)}{dt}\big|_{t=0}\rangle,
\end{align*}
where $e_{i}(t)$ denotes the $i$-th component of $u_{t}$ and $\frac{D}{dt}$
is the covariant derivative along $\gamma$. Since $\xi_{t}$ is horizontal,
one has 
\[
\frac{De_{i}(t)}{dt}=0\ \ \ \forall t.
\]
As a result, one obtains that 
\[
H_{i}\langle\Phi,H_{i}\rangle(\xi)=\langle\nabla_{e_{i}}\phi,e_{i}\rangle(x)
\]
and the claim (\ref{eq:0OrderTr}) thus follows.
\end{proof}

Now we derive the Feynman-Kac representation for the PDE (\ref{eq:HorPDE})
(with no surprise, it is given in terms of the joint process defined
by the horizontal Brownian motion and its $\Phi$-signature process).
Recall that $\Xi^{\xi}$ is the horizontal Brownian motion defined
by the horizontal SDE (\ref{eq:HorBM}) on the bundle and let $\tilde{S}_{t}^{\Phi}(\Xi^{\xi})\in T((E))$
denote the signature path of the $E$-valued path 
\[
t\mapsto\int_{0}^{t}\Phi(\circ d\Xi_{s}^{\xi}).
\]

\begin{lem}
\label{lem:FKRepBundle}One has the following representation: 
\begin{equation}
f(t,\xi)=\mathbb{E}\big[\tilde{S}_{t}^{\Phi}(\Xi^{\xi})\big],\ \ \ (t,\xi)\in[0,\infty)\times{\cal O}(M).\label{eq:FKRepBundle}
\end{equation}
\end{lem}

\begin{proof}
Consider the Markov family 
\[
(\Xi,\tilde{S})\triangleq\big\{\big(\Xi_{t}^{\xi},g\otimes\tilde{S}_{t}^{\Phi}(\Xi^{\xi})\big):t\geqslant0,(\xi,g)\in{\cal O}(M)\times T((E))\big\}.
\]
This family is governed by the Stratonovich SDE (cf. (\ref{eq:SigDE}))
\begin{equation}
\begin{cases}
d\Xi_{t}=\sum_{i=1}^{d}H_{i}(\Xi_{t})\circ dB_{t}^{i},\\
d\tilde{S}_{t}=\sum_{i=1}^{d}\tilde{S}_{t}\otimes\langle\Phi,H_{i}\rangle(\Xi_{t})\circ dB_{t}^{i}
\end{cases}\label{eq:JointSDE}
\end{equation}
on the state space ${\cal O}(M)\times T((E))$. Let ${\cal L}$
be the generator of $(\Xi,\tilde{S})$. It is standard that ${\cal L}=\frac{1}{2}\sum_{i=1}^{d}{\cal V}_{i}^{2}$
where $\{{\cal V}_{1},\cdots,{\cal V}_{d}\}$ are the vector fields
for the SDE (\ref{eq:JointSDE}) and they are viewed as differential
operators on $C^{\infty}\big({\cal O}(M)\times T((E))\big)$. Respectively,
we also denote 
\[
\hat{{\cal L}}F\triangleq\frac{1}{2}\sum_{i=1}^{d}H_{i}^{2}F+\sum_{i=1}^{d}\langle\Phi,H_{i}\rangle\otimes H_{i}F+\frac{1}{2}\sum_{i=1}^{d}\big(H_{i}\langle\Phi,H_{i}\rangle+\langle\Phi,H_{i}\rangle^{\otimes2}\big)\otimes F
\]
for $F\in C^{\infty}\big({\cal O}(M)\times T((E))\big)$. Note that
the situation here is essentially finite dimensional since one can
consistently look at the truncations of the SDE (\ref{eq:JointSDE})
and PDE (\ref{eq:HorPDE}) up to any fixed level. 

Consider the function
\begin{equation}
F(t,(\xi,g))\triangleq g\otimes f(t,\xi),\label{eq:ExtendedF}
\end{equation}
where we recall that $f(t,\xi)$ is defined by (\ref{eq:f(t,xi)}).
We claim that ${\cal L}F=\hat{{\cal L}}F$. Indeed, note that the
vector fields for the SDE (\ref{eq:JointSDE}) are given by 
\[
{\cal V}_{i}\triangleq V_{i}\oplus W_{i},\ \ \ 1\leqslant i\leqslant d,
\]
where
\begin{equation}
V_{i}(\xi,g)=H_{i}(\xi),\ W_{i}(\xi,g)=g\otimes\langle\Phi,H_{i}\rangle(\xi).\label{eq:WVecField}
\end{equation}
We first compute ${\cal V}_{i}F$. By the definition (\ref{eq:ExtendedF})
of $F$, it is clear that 
\begin{equation}
(V_{i}F)(t,(\xi,g))=g\otimes(H_{i}f)(t,\xi).\label{eq:Vi}
\end{equation}
In addition, by the definition (\ref{eq:WVecField}) of $W_{i}$ one
has (recall $E=\mathbb{R}^{N}$)
\[
(W_{i}F)(t,(\xi,g))=\sum_{m\geqslant1}\sum_{i_{1},\cdots,i_{m}=1}^{N}g^{i_{1},\cdots,i_{m-1}}\langle\Phi^{i_{m}},H_{i}\rangle(\xi)\partial_{i_{1},\cdots,i_{m}}(g\otimes f(t,\xi)),
\]
where $g^{i_{1},\cdots,i_{m-1}}$ are the coordinate components of
the tensor $g$, $\Phi^{j}$ is the $j$-th component of $\Phi=(\Phi^{1},\cdots,\Phi^{N})$
as an $\mathbb{R}^{N}$-valued one-form and $\partial_{i_{1},\cdots,i_{m}}$
is the partial derivative in $g$ with respect to the coordinate $g^{i_{1},\cdots,i_{m}}.$
For any given word $J=(j_{1},\cdots,j_{l})$, the $J$-th component
of $W_{i}F$ is thus given by 
\begin{align*}
 & (W_{i}F)^{J}(t,(\xi,g))\\
 & =\sum_{m\geqslant1}\sum_{i_{1},\cdots,i_{m}=1}^{N}g^{i_{1},\cdots,i_{m-1}}\langle\Phi_{i_{m}},H_{i}\rangle(\xi)\partial_{i_{1},\cdots,i_{m}}\big(\sum_{s=1}^{l}g^{j_{1},\cdots,j_{s}}f(t,\xi)^{j_{s+1},\cdots,j_{l}}\big)\\
 & =\sum_{s=1}^{l}g^{j_{1},\cdots,j_{s-1}}\langle\Phi_{j_{s}},H_{i}\rangle(\xi)f(t,\xi)^{j_{s+1},\cdots,j_{l}}\\
 & =\big(g\otimes\langle\Phi,H_{i}\rangle\otimes f(t,\xi)\big)^{J}.
\end{align*}
In its equivalent tensor form, 
\begin{equation}
(W_{i}F)(t,(\xi,g))=g\otimes\langle\Phi,H_{i}\rangle(\xi)\otimes f(t,\xi).\label{eq:Wi}
\end{equation}
Combining (\ref{eq:Vi}) and (\ref{eq:Wi}), one finds that 
\[
({\cal V}_{i}F)(t,(\xi,g))=g\otimes\big((H_{i}f)(t,\xi)+\langle\Phi,H_{i}\rangle(\xi)\otimes f(t,\xi)\big).
\]
By applying ${\cal V}_{i}$ again, one arrives at
\begin{align}
{\cal V}_{i}^{2}F & =g\otimes H_{i}\big(H_{i}f+\langle\Phi,H_{i}\rangle\otimes f\big)+g\otimes\langle\Phi,H_{i}\rangle\otimes\big(H_{i}f+\langle\Phi,H_{i}\rangle\otimes f\big)\nonumber \\
 & =g\otimes\big(H_{i}^{2}f+H_{i}\langle\Phi,H_{i}\rangle\otimes f+2\langle\Phi,H_{i}\rangle\otimes H_{i}f+\langle\Phi,H_{i}\rangle^{\otimes2}\otimes f\big).\label{eq:GenSigBund}
\end{align}
This gives the relation ${\cal L}F=\hat{{\cal L}}F.$

It follows from the above relation and Lemma \ref{lem:PDEBundle}
that $F$ satisfies the PDE $\frac{\partial F}{\partial t}={\cal L}F$
with initial condition $F(0,\cdot)=\rho(\cdot),$ where $\rho:{\cal O}(M)\times T((E))\rightarrow T((E))$
denotes the projection onto the second component. It is now a standard
application of the Feynman-Kac representation that 
\[
g\otimes f(t,\xi)=F(t,(\xi,g))=\mathbb{E}\big[\rho\big(\Xi_{t}^{\xi},g\otimes\tilde{S}_{t}^{\Phi}(\Xi^{\xi})\big)\big]=g\otimes\mathbb{E}\big[\tilde{S}_{t}^{\Phi}(\Xi^{\xi})\big].
\]
 By cancelling out the $g$ on both sides, one arrives at the desired
relation (\ref{eq:FKRepBundle}).
\end{proof}
Finally, the lemma below completes the proof of Theorem \ref{thm:BMSigPDE}.
Recall that $\Psi(t,x)$ is the expected $\phi$-signature of the
Brownian motion (cf. (\ref{eq:DefESig})).
\begin{lem}
\label{lem:LiftPsi}$\Psi(t,x)=\hat{\Psi}(t,x)$ for all $t\geqslant0$
and $x\in M$. 
\end{lem}

\begin{proof}
Fix $x\in M$ and let $\xi\in\pi^{-1}(x).$ According to Theorem \ref{thm:HorBM},
$W_{t}\triangleq\pi(\Xi_{t}^{\xi})$ is a Brownian motion starting
at $x$. In addition, by the definition of $\Phi$ one has 
\[
\langle\Phi,\circ d\Xi_{t}^{\xi}\rangle=\langle\pi^{*}\phi,\circ d\Xi_{t}^{\xi}\rangle=\langle\phi,\circ\pi_{*}d\Xi_{t}^{\xi}\rangle=\langle\phi,\circ dW_{t}\rangle.
\]
It follows from Lemma \ref{lem:FKRepBundle} that 
\[
\hat{\Psi}(t,x)=f(t,\xi)=\mathbb{E}\big[\tilde{S}_{t}^{\Phi}(\Xi^{\xi})\big]=\mathbb{E}[S_{t}^{\phi}(W)]=\Psi(t,x).
\]
Note that the above relation is independent of the choice of $\xi\in\pi^{-1}(x)$
since the law of $W$ is. 
\end{proof}
\begin{rem}
\label{rem:GenPhi}The proof of Lemma \ref{lem:FKRepBundle} does
not use the fact that $\Phi$ is the lifting of the one-form $\phi$.
Indeed, let $\Phi$ be any given $E$-valued one-form on ${\cal O}(M)$.
For each $\xi\in{\cal O}(M)$, let $\Xi_{t}^{\xi}$ be the horizontal
Brownian motion starting at $\xi$. Then one can show that the function
\[
(t,\xi)\mapsto\mathbb{E}[\tilde{S}_{t}^{\Phi}(\Xi^{\xi})]
\]
satisfies the PDE (\ref{eq:HorPDE}).
\end{rem}

\subsection{The Brownian bridge case}

Now let us consider the expected $\phi$-signature of the Brownian
bridge defined by 
\[
\psi(t,x,y)\triangleq\mathbb{E}\big[S^{\phi}(X^{t,x,y})\big],\ \ \ t\geqslant0,\ x,y\in M,
\]
where $X^{t,x,y}$ is the Brownian bridge from $x$ to $y$ with lifetime
$t$. By using the PDE (\ref{eq:PhiSigPDE}), it is not hard to derive
the associated PDE for $\psi(t,x,y)$.
\begin{thm}
\label{thm:BBSigPDE}For each fixed $y\in M,$ the function 
\[
(0,\infty)\times M\ni(t,x)\mapsto\psi(t,x,y)
\]
satisfies the following PDE:
\begin{align}
\frac{\partial\psi}{\partial t} & =\frac{1}{2}\Delta_{x}\psi+\langle\nabla_{x}\log p(t,x,y),\nabla_{x}\psi+\phi\otimes\psi\rangle+{\rm Tr}(\phi\otimes d_{x}\psi)\nonumber \\
 & \ \ \ +\frac{1}{2}{\rm Tr}\big((\nabla\phi+\phi\otimes\phi)\otimes\psi\big),\label{eq:BridgePDE}
\end{align}
where $p(t,x,y)$ is the heat kernel.
\end{thm}

\begin{rem}
The term $\langle\nabla\log p,\phi\otimes\psi\rangle$ is understood
as $_{TM}\langle\nabla\log p,\phi\rangle_{T^{*}M}\otimes\psi\in T((E))$. 
\end{rem}

\begin{proof}
By the definition of Brownian bridge, one has 
\[
\psi(t,x,y)=\frac{\mathbb{E}[S_{t}^{\phi}(W^{x})\delta_{y}(W_{t}^{x})]}{\mathbb{E}[\delta_{y}(W_{t}^{x})]}=\frac{\mathbb{E}[S_{t}^{\phi}(W^{x})\delta_{y}(W_{t}^{x})]}{p(t,x,y)},
\]
where $W^{x}$ is the Brownian motion starting at $x$. Given any
smooth function $f:M\rightarrow\mathbb{R},$ exactly the same argument
as in the proof of Theorem \ref{thm:BMSigPDE} shows that the function
\[
(0,\infty)\times M\ni(t,x)\mapsto\mathbb{E}\big[f(W_{t}^{x})S_{t}^{\phi}(W^{x})\big]
\]
satisfies the same PDE (\ref{eq:PhiSigPDE}). In particular, the function
($y$ being fixed)
\[
\bar{\psi}(t,x)\triangleq\mathbb{E}[S_{t}^{\phi}(W^{x})\delta_{y}(W_{t}^{x})]
\]
also satisfies (\ref{eq:PhiSigPDE}).

Now let us write $\psi=\bar{\psi}/p$. By using the PDE for $\bar{\psi}$
and the heat equation (\ref{eq:HeatEq}) for $p$, one finds that
\begin{align*}
\frac{\partial\psi}{\partial t} & =\frac{1}{p}\frac{\partial\bar{\psi}}{\partial t}+\bar{\psi}\big(-\frac{1}{p^{2}}\frac{\partial p}{\partial t}\big)\\
 & =\frac{1}{p}\big(\frac{1}{2}\Delta\bar{\psi}+{\rm Tr}\big(\phi\otimes d\bar{\psi}+\frac{1}{2}(\nabla\phi+\phi\otimes\phi)\otimes\bar{\psi}\big)-\frac{\bar{\psi}}{2p}\Delta p\big)\\
 & =\frac{1}{p}\big(\frac{1}{2}\Delta(p\psi)+{\rm Tr}\big(\phi\otimes d(p\psi)\big)+\frac{1}{2}{\rm Tr}\big((\nabla\phi+\phi\otimes\phi)\otimes\psi\big)-\frac{\psi}{2p}\Delta p.
\end{align*}
Since
\[
\Delta(p\psi)=\Delta p\cdot\psi+p\cdot\Delta\psi+2\langle\nabla p,\nabla\psi\rangle,
\]
it follows that 
\[
\frac{\partial\psi}{\partial t}=\frac{1}{2}\Delta\psi+\frac{1}{p}\langle\nabla p,\nabla\psi\rangle+\frac{1}{p}{\rm Tr}(\phi\otimes\psi dp)+{\rm Tr}\big(\phi\otimes d\psi+\frac{1}{2}(\nabla\phi+\phi\otimes\phi)\otimes\psi\big).
\]
The result follows by further noting that 
\begin{align*}
{\rm Tr}(\phi\otimes\psi dp) & =\sum_{i=1}^{d}\langle\nabla p,e_{i}\rangle\phi(e_{i})\otimes\psi=\phi(\nabla p)\otimes\psi,
\end{align*}
where $\{e_{1},\cdots,e_{d}\}$ is any ONB of $T_{x}M$ and $p^{-1}\nabla p=\nabla\log p$.
\end{proof}
A basic point to make is that rich geometric information about the
underlying space $M$ is encoded in the asymptotic behaviour of the
function $\psi(t,x,y)$ as $t\rightarrow0^{+}$. This motivates our study in the next two sections.

\section{\label{sec:ReconRD}Signature asymptotics and reconstruction of Riemannian
distance}

In this section, we develop a method of reconstructing the Riemannian
distance $d(x,y)$ from the expected signature of the Brownian bridge
$X^{t,x,y}$ through an explicit asymptotics procedure, at least when
$d(x,y)$ is not too far from each other. 

\subsection{The main theorem }

Suppose that $F:M\rightarrow E=\mathbb{R}^{N}$ is a given fixed isometric
embedding (which always exists by Nash's embedding theorem). For each
$n\geqslant1$, we equip $E^{\otimes n}$ with the Hilbert-Schmidt
tensor norm $\|\cdot\|_{{\rm HS}}$. In other words, viewing $E$
as an Euclidean space, the space $E^{\otimes n}$ is equipped with an
inner product structure induced by 
\[
\langle v_{1}\otimes\cdots\otimes v_{n},w_{1}\otimes\cdots\otimes w_{n}\rangle_{{\rm HS}}\triangleq\langle v_{1},w_{1}\rangle_{E}\cdots\langle v_{n},w_{n}\rangle_{E},\ \ \ v_{i},w_{j}\in E.
\]
Recall that $X^{t,x,y}$ is the Brownian bridge from $x$ to $y$
with lifetime $t$ and $S^{dF}(X^{t,x,y})$ is its $\phi$-signature
with $\phi=dF$. Our main result in this section is stated as follows.
\begin{thm}
\label{thm:ReconRD}Let $x,y\in M$ be given fixed such that $d(x,y)<\rho_{M}/2$
where $\rho_{M}$ is the global injective radius of $M$. Then
there exist geometric constants $C,\kappa$ depending on $x,y,M$,
such that 
\begin{equation}
\big|\big(n!\big\|\mathbb{E}\big[\pi_{n}S^{dF}(X^{\kappa n^{-6},x,y})\big]\big\|_{{\rm HS}}\big)^{1/n}-d(x,y)\big|\leqslant\frac{C}{n},\ \ \ n\geqslant1.\label{eq:QuanRD}
\end{equation}
In particular, with $t_{n}=o(n^{-6})$ one has 
\begin{equation}
\lim_{n\rightarrow\infty}\big(n!\big\|\mathbb{E}\big[\pi_{n}S^{dF}(X^{t_{n},x,y})\big]\big\|_{{\rm HS}}\big)^{1/n}=d(x,y).\label{eq:ReconRDAsym}
\end{equation}
\end{thm}

\begin{rem}
The assumption that $\phi=dF$ with some isometric embedding $F$
plays no essential role; the only relevant fact is that $d(x,y)$
is precisely the Euclidean length of the path $\int_{0}^{\cdot}\phi(d\gamma^{x,y})$
($\gamma^{x,y}$ is the unique minimising geodesic from $x$ to $y$).
For a general $\phi$, the theorem remains valid as long as one replaces
$d(x,y)$ with the length of the path $\int_{0}^{\cdot}\phi(d\gamma^{x,y})$
in $E$.
\end{rem}

\begin{rem}
We expect that the same type of result holds (possibly with a different
condition on $t_{n}$) even if $y$ is on the cut-locus of $x.$ This
may cause extra technical difficulties in the current argument due
to the non-uniqueness of minimising geodesics connecting $x,y$. 
\end{rem}

\subsection{\label{subsec:RDPf}Proof of Theorem \ref{thm:ReconRD}}

As we explained in the introduction, the asymptotics formula (\ref{eq:ReconRDAsym})
is largely inspired by the length conjecture (which is an actual theorem
for smooth paths such as the geodesic $\gamma^{x,y}$) as well as
the fact that $X^{t,x,y}\approx\gamma^{x,y}$ when $t$ is small.
However, there are several technical challenges to implement such
an idea precisely. Essentially, given a degree $n$ one wants to identify
a lifetime scale $t_{n}$ such that the (normalised) $n$-th level
expected signature of the Brownian bridge with lifetime $t_{n}$ yields
the desired limit $d(x,y)$ as $n\rightarrow\infty$. If one applies
standard signature and semimartingale estimates without much care,
one is led to requiring that $t_{n}=o(e^{-C n}),$ which is way too small
to be useful in practice. Improving it from exponential to polynomial
scale $n^{-6}$ is the most delicate point in the argument (cf. Section
\ref{subsec:KeyEst} below). 

In what follows, we develop the main ingredients for proving Theorem
\ref{thm:ReconRD} in a precise mathematical way. Basically, we localise
the problem on a nice coordinate chart (a normal chart) in which Euclidean
stochastic calculus can be applied. The part that the Brownian bridge
exits the chart before its lifetime is negligible (due to Hsu's large
deviation principle), while the part within the chart yields the main
contribution. 

Throughout the rest, $\phi\triangleq dF$ and $x,y\in M$ with $d(x,y)<\rho_{M}/2$
are all given fixed. 

\subsubsection{\label{subsec:NormChart}Computations under normal chart}

We will perform local calculations under normal coordinates.
For this purpose, we first collect some useful geometric properties
of normal charts.

We fix an orthonormal basis (ONB) $\{e_{i}(0):i=1,\cdots,d\}$ of
$T_{y}M$. Let $U\subseteq T_{y}M$ be a fixed open ball centered
at the origin with radius $r\leqslant\rho_{y}$ (the injective radius
at $y$) so that the exponential map $\exp_{y}:U\rightarrow M$ is
a diffeomorphism. This gives rise to a local parametrisation 
\[
{\bf x}=(x^{1},\cdots,x^{d})\mapsto\exp_{y}\big(x^{i}e_{i}(0)\big)
\]
of the ball $B(y,r)=\exp_{y}U=:V,$ which is known as the \textit{normal
coordinate chart} based at $y$. As a convention, a bold letter always
refers to an Euclidean vector (e.g. ${\bf x}=(x^{1},\cdots,x^{d})$)
that is also identified with an element in $U$, while a regular letter
(e.g. $x$) refers to a point on $M$. For instance, one can legally
write $x=\exp_{y}({\bf x}).$

The natural coordinate vector fields on $V$ are denoted as $\{\partial_{i}:i=1,\cdots,d\}$
as usual. It is useful to introduce another local frame of vector
fields $\{e_{i}:1,\cdots,d\}$ as follows. Given ${\bf x}\in U$,
we define $e_{i}({\bf x})$ to be the parallel translation of $e_{i}(0)$
along the geodesic $t\mapsto\exp_{y}(t{\bf x})$ to the endpoint ${\bf x}$
(at $t=1$). By varying ${\bf x}$, one obtains a vector field $e_{i}$
on $V$. It is clear that $\{e_{1}({\bf x}),\cdots,e_{d}({\bf x})\}$
is an ONB of $T_{\exp_{y}{\bf x}}M$ for every ${\bf x}\in U.$ Note
that $e_{i}(0)=\partial_{i}|_{{\bf x}=0}$. In general, one can write
$e_{i}=\sigma_{i}^{j}\partial_{j}$ where $(\sigma_{i}^{j})_{1\leqslant i,j\leqslant d}$
is a smooth matrix-valued function on $U$ that is everywhere invertible.

Recall that the metric tensor under the normal chart $V$ is defined
by $g_{ij}\triangleq\langle\partial_{i},\partial_{j}\rangle$. The
following fact will be useful to us later on (cf. \cite[Proposition 1.27 (iii)]{BGV04}).
\begin{lem}
\label{lem:NormChartLem}On the normal chart $V$, one has 
\[
x^{i}g_{ij}=x^{j},\ x^{i}g^{ij}=x^{j}.
\]
\end{lem}

It is well known that under the normal chart near the origin, the
Riemannian metric $g_{ij}$ agrees with Euclidean metric $\delta_{ij}$
up to the second order with error given in terms of curvature coefficients
(cf. \cite[Proposition 1.28]{BGV04}).
\begin{prop}
\label{prop:MetricExp}On the normal chart $V$, one has 
\[
g_{ij}({\bf x})=\delta_{ij}-\frac{1}{3}R_{ikjl}({\bf 0})x^{k}x^{l}+O(|{\bf x}|^{3})\ \ \ \text{as\ }{\bf x}=(x^{1},\cdots,x^{d})\rightarrow0,
\]
where $R_{ijkl}\triangleq\langle R(\partial_{k},\partial_{l})\partial_{j},\partial_{i}\rangle$
are the curvature coefficients in $V$.
\end{prop}

We now express $\Delta/2$ on $V$ as 
\begin{equation}
\frac{1}{2}\Delta=\frac{1}{2}a^{ij}\partial_{ij}^{2}+b^{i}\partial_{i}.\label{eq:LapChart}
\end{equation}
From the local expression (\ref{eq:LapLocal}), it is plain to check
that 
\begin{equation}
a^{ij}=g^{ij},\ b^{i}=\frac{g^{ij}}{4\det g}\partial_{j}\det g+\frac{1}{2}\partial_{j}g^{ij}.\label{eq:abLocal}
\end{equation}
We also recall that $G_{0}(x,y),G_{1}(x,y)$ are the first two terms
appearing in the Malliavin-Stroock expansion (\ref{eq:MSExp}). The
following explicit formulae for the local expressions of these quantities
will be needed for us. 
\begin{lem}
\label{lem:abGNormChart}On the normal chart $V$, one has 
\begin{equation}
a^{ij}({\bf 0})=\delta^{ij},\ b^{i}({\bf 0})=0;\label{eq:abInitial}
\end{equation}
\begin{equation}
\partial_{p}a^{ij}({\bf 0})=0,\ \partial_{i}b^{j}({\bf 0})=\frac{1}{3}R_{jppi},\ \partial_{pp}^{2}a^{ij}({\bf 0})=\frac{2}{3}R_{ipjp};\label{eq:abDer}
\end{equation}
\begin{equation}
\nabla_{x}G_{0}({\bf x},{\bf 0})=-x^{j}\partial_{j};\label{eq:G0Der}
\end{equation}
\begin{equation}
\partial_{j}G_{1}({\bf 0},{\bf 0})=0,\ \partial_{ij}^{2}G_{1}({\bf 0},{\bf 0})=\frac{1}{6}R_{pipj}.\label{eq:G1Local}
\end{equation}
Here $R_{ijkl}$ is the curvature coefficient as in Proposition \ref{prop:MetricExp}
and the derivatives in the last line are taken with respect to the
first variable in $G_{1}.$ 
\end{lem}

\begin{proof}
The curvature coefficients satisfy the following symmetries:
\begin{equation}
R_{ijkl}=-R_{jikl}=-R_{ijlk}=R_{klij}.\label{eq:CurSym}
\end{equation}
The relations (\ref{eq:abInitial}) and (\ref{eq:abDer}) follow directly
from Proposition \ref{prop:MetricExp} and (\ref{eq:abLocal}). In
addition, by the local expression of $\nabla$ and Lemma \ref{lem:NormChartLem}
one has 
\[
\nabla_{x}G_{0}({\bf x},{\bf 0})=-x^{i}g^{ij}\partial_{j}=-x^{j}\partial_{j}.
\]
This gives the relation (\ref{eq:G0Der}). Finally, recall that $G_{1}$
is defined by (\ref{eq:G0G1}). Standard geometric consideration shows
that 
\begin{equation}
\det(d\exp_{y})_{{\bf x}}=\sqrt{\det g({\bf x})}\label{eq:Derdexp}
\end{equation}
for any ${\bf x}=\exp_{y}^{-1}(x)\in U$ ($x\in V$). The relation
(\ref{eq:G1Local}) then easily follows from Proposition \ref{prop:MetricExp}
and (\ref{eq:Derdexp}).
\end{proof}
\begin{rem}
For the proof of Theorem \ref{thm:ReconRD}, only (\ref{eq:abInitial})
and (\ref{eq:G0Der}) are relevant. The relations (\ref{eq:abDer})
and (\ref{eq:G1Local}) will only be needed for the proof of Theorem
\ref{thm:RecCurv} in the next section. 
\end{rem}

\subsubsection{\label{subsec:KeyEst} Localisation I: the main contribution}

Our analysis is essentially localised on a normal chart around $y$.
Throughout the rest, we take $V\triangleq B(y,\rho_{M}/2-\varepsilon)$
where $\varepsilon\in(0,\rho_{M}/2-d(x,y))$ is fixed (note that $x\in V$).
Define 
\[
\tau\triangleq\inf\{s\in[0,t]:X_{s}^{t,x,y}\notin V\}
\]
and we set $\tau=t$ if such $s$ does not exist. We shall consider
the decomposition 
\[
\mathbb{E}\big[\pi_{n}S^{\phi}(X^{t,x,y})\big]=\mathbb{E}\big[\pi_{n}S^{\phi}(X^{t,x,y});\tau=t\big]+\mathbb{E}\big[\pi_{n}S^{\phi}(X^{t,x,y});\tau<t\big].
\]
One naturally expects that the first term provides the main contribution
and the second term is negligible. In this section, we establish the
main estimate for the first term. The second term will be handled
in the next section. 

Note that the term $\mathbb{E}\big[\pi_{n}S^{\phi}(X^{t,x,y});\tau=t\big]$
corresponds to the situation where $X^{t,x,y}$ does not leave the
normal chart $B$ around $y$ during its lifetime. For this part,
one can apply Euclidean stochastic calculus to the SDE (\ref{eq:BBSDE})
of the Brownian bridge. We break down the analysis into several steps.

\subsubsection*{A representation of $X^{t,x,y}$ }

In the Malliavin-Stroock expansion (\ref{eq:MSExp}), recall that
$G_{0}(x,y)=-d(x,y)^{2}/2$. By applying Lemma \ref{lem:abGNormChart}
on the normal chart $V$ defined previously, one can write (using
Euclidean variables on $U\triangleq\exp_{y}^{-1}V$)
\begin{equation}
\nabla_{z}\log p(u,{\bf z},{\bf 0})=-\frac{{\bf z}}{u}+Q(u,{\bf z}),\label{eq:Q}
\end{equation}
where $\nabla_{z}$ denotes the Euclidean vector $(g^{ij}\partial_{j})_{1\leqslant i\leqslant d}$
and $$Q(u,{\bf z})\triangleq\sum_{k=0}^{\infty}\nabla_{z}G_{k+1}({\bf z},{\bf 0})u^{k}$$
satisfies 
\[
\sup_{u\in[0,1],{\bf z}\in\bar{U}}|Q(u,z)|<\infty.
\]
The SDE (\ref{eq:BBSDE}) of $X^{t,x,y}$ on $V$ is thus expressed
as 
\begin{equation}
\begin{cases}
dX_{s}=\big(-\frac{X_{s}}{t-s}+b(X_{s})+Q(t-s,X_{s})\big)ds+\sigma(X_{s})dB_{s}, & s<\tau;\\
X_{0}={\bf x}\triangleq\exp_{y}^{-1}(x).
\end{cases}\label{eq:BridgeSDESimplified}
\end{equation}
where $(\sigma,b)$ are local coefficients of $\Delta/2$ satisfying
the relation (\ref{eq:LapChart}) with $a=\sigma^{T}\sigma.$ We give
a representation of $X_{s}$ that will be used frequently in the sequel.
\begin{lem}
\label{lem:ExpBridge}On the event $\{\tau=t\}$, one has 
\begin{equation}
\frac{X_{tr}}{1-r}={\bf x}+t\int_{0}^{tr}\frac{\sigma(X_{u})dB_{u}}{t-u}+t\int_{0}^{r}\frac{1}{1-\eta}\big(b(X_{t\eta})+Q(t(1-\eta),X_{t\eta})\big)d\eta,\ \ \ r\in[0,1).\label{eq:ExpBridge}
\end{equation}
\end{lem}

\begin{proof}
We introduce the integrating factor $x_{s}\triangleq t-s$ ($s\in[0,t]$),
which satisfies 
\[
dx_{s}=-\frac{1}{t-s}x_{s}ds\in\mathbb{R}.
\]
Setting $Z_{s}\triangleq x_{s}^{-1}X_{s},$ one finds that 
\[
dZ_{s}=\frac{1}{t-s}\big(b(X_{s})+Q(t-s,X_{s})\big)ds+\frac{1}{t-s}\sigma(X_{s})dB_{s}.
\]
By integrating both sides from $0$ to $s$ and noting that $Z_{0}={\bf x}/t,$
one has 
\[
\frac{X_{s}}{t-s}=\frac{{\bf x}}{t}+\int_{0}^{s}\frac{\sigma(X_{u})dB_{u}}{t-u}+\int_{0}^{s}\frac{1}{t-u}\big(b(X_{u})+Q(t-u,X_{u})\big)du.
\]
The result follows by taking $s=tr$. 
\end{proof}

\subsubsection*{A semimartingale decomposition for $\phi(\circ dX_{tr})$}

On the event $\{\tau=t\}$, the signature component $\pi_{n}S^{\phi}(X^{t,x,y})$
is given by 
\begin{align}
\pi_{n}S^{\phi}(X^{t,x,y}) & =\int_{0<s_{1}<\cdots<s_{n}<t}\phi(\circ dX_{s_{1}})\otimes\cdots\otimes\phi(\circ dX_{s_{n}})\nonumber \\
 & =\int_{0<r_{1}<\cdots<r_{n}<1}\phi(\circ dX_{tr_{1}})\otimes\cdots\otimes\phi(\circ dX_{tr_{n}}),\label{eq:PinSig}
\end{align}
where $X_{s}$ satisfies the SDE (\ref{eq:BridgeSDESimplified}).
To relate it with the $\phi$-signature of the minimising geodesic
\[
\gamma_{r}^{x,y}\triangleq\exp((1-r){\bf x}),\ \ \ 0\leqslant r\leqslant1,
\]
an important step is to factor out the geodesic component in the semimartingale
decomposition of $\phi(\circ dX_{tr}).$ We summarise the main structure
in the lemma below. Let $(\Gamma_{r})_{0\leqslant r\leqslant1}$ be
the $E$-valued path defined by $\Gamma_{r}\triangleq\int_{0}^{r}\phi(d\gamma^{x,y})$. 
\begin{lem}
\label{lem:SemDecomp}On the event $\{\tau=t\}$, one has the following
decomposition of $\phi(\circ dX_{tr})$: 
\begin{equation}
\phi(\circ X_{tr})=\dot{\Gamma}_{r}dr+\sqrt{t}\big(A_{r}^{t}dr+dM_{r}^{t}\big),\ \ \ r\in[0,1),\label{eq:PhidXDecomp}
\end{equation}
where $\{A_{r}^{t},M_{r}^{t}:0\leqslant r<1\}$ are $E$-valued
stochastic processes satisfying the following properties:

\vspace{2mm}\noindent (i) $\{A_{r}^{t}\}$ is a continuous, $\{{\cal F}_{tr}^{B}\}$-adapted
semimartingale such that 
\begin{equation}
\|A_{r}^{t}\|_{p}\leqslant\frac{C\sqrt{p}}{\sqrt{1-r}},\ \ \ \forall r\in[0,1),\ p\geqslant2,\label{eq:ABound}
\end{equation}
where $\|\cdot\|_{p}$ denotes the $L^{p}$-norm.\\
(ii) $\{M_{r}^{t}\}$ is an $\{{\cal F}_{tr}^{B}\}$-martinagle
whose components satisfy
\begin{equation}
\langle M^{t,i}\rangle_{r}\leqslant Cr,\ \ \ \forall r\in[0,1),\ i=1,\cdots,N.\label{eq:QVMBound}
\end{equation}
In the above inequalities, $C$ denotes a geometric constant depending
only on $M,\phi$ and the localisation $V$.
\end{lem}

\begin{proof}
We write $\phi=\phi_{i}dx^{i}$ on $V$ and also set $\hat{Q}(u,{\bf z})\triangleq b({\bf z})+Q(u,{\bf z})$
to ease notation. According to the SDE (\ref{eq:BridgeSDESimplified})
and the It\^o-Stratonovich conversion, one has 
\begin{align*}
\phi(\circ dX_{s}) & =-\phi_{i}(X_{s})\frac{X_{s}^{i}}{t-s}ds+\big(\phi_{i}(X_{s})\hat{Q}^{i}(t-s,X_{s})\\
 & \ \ \ +\frac{1}{2}\partial_{j}\phi_{i}(X_{s})a^{ij}(X_{s})\big)ds+\phi_{i}(X_{s})\sigma_{\alpha}^{i}(X_{s})dB_{s}^{\alpha}
\end{align*}
for $s\in[0,\tau).$ By taking $s=tr$ and substituting (\ref{eq:ExpBridge})
into the first term, it is easily seen that the decomposition (\ref{eq:PhidXDecomp})
holds with
\begin{align}
A_{r}^{t}\triangleq & -\phi_{i}(X_{\tau\wedge tr})\big(\sqrt{t}\int_{0}^{\tau\wedge tr}\frac{\sigma_{\alpha}^{i}(X_{u})dB_{u}^{\alpha}}{t-u}+\sqrt{t}\int_{0}^{(\tau/t)\wedge r}\frac{1}{1-\eta}\hat{Q}^{i}(t(1-\eta),X_{t\eta})d\eta\big)\label{eq:ALine1}\\
 & \ \ \ -\big(\phi_{i}(X_{\tau\wedge tr})-\phi_{i}((1-r){\bf x})\big)x^{i}+\sqrt{t}\big(\phi_{i}(X_{\tau\wedge tr})\hat{Q}^{i}(t(1-r),X_{\tau\wedge tr})\nonumber \\
 & \ \ \ +\frac{1}{2}\partial_{j}\phi_{i}(X_{\tau\wedge tr})a^{ij}(X_{\tau\wedge tr})\big)\label{eq:ALine2}
\end{align}
and 
\[
M_{r}^{t}\triangleq\int_{0}^{(\tau/t)\wedge r}\phi_{i}(X_{t\eta})\sigma_{\alpha}^{i}(X_{t\eta})dB_{\eta}^{t,\alpha},
\]
where $B_{\eta}^{t,\alpha}\triangleq B_{t\eta}^{\alpha}/\sqrt{t}$
is a rescaled Brownian motion. Since $\phi_{i},\sigma_{\alpha}^{i}\in C_{b}^{\infty}(\bar{U})$,
it is clear that the bound (\ref{eq:QVMBound}) holds. 

We now estimate each term of $A_{r}^{t}.$ Since $\hat{Q}^{i}$
is also uniformly bounded on $\bar{U}$, the entire expression of
(\ref{eq:ALine2}) is uniformly bounded by some deterministic constant
which is in turn trivially enlarged to the form of (\ref{eq:ABound}).
To estimate the first term of (\ref{eq:ALine1}), one notes from the
Burkholder-Davis-Gundy (BDG) inequality (cf. \cite[Theorem A]{CK91}
with BDG constant $2\sqrt{p}$) that 
\[
\big\|\sqrt{t}\int_{0}^{\tau\wedge tr}\frac{(\sigma dB)_{u}^{i}}{t-u}\big\|_{p}\leqslant2\sqrt{p}\cdot\big\|\sqrt{t\int_{0}^{tr}\frac{C}{(t-u)^{2}}du}\big\|_{p}=C'\sqrt{p}\cdot\sqrt{\frac{r}{1-r}},
\]
which is dominated by (\ref{eq:ABound}) since $r\leqslant1.$ The
second term of (\ref{eq:ALine1}) is uniformly bounded by $C|\log(1-r)|$
which is also trivially enlarged to (\ref{eq:ABound}). The desired
bound (\ref{eq:ABound}) thus follows.
\end{proof}

\subsubsection*{The signature remainder estimate}

In the decomposition (\ref{eq:PhidXDecomp}), let us denote 
\[
dY_{r}^{0}\triangleq\dot{\Gamma}_{r}dr,\ dY_{r}^{1}\triangleq A_{r}^{t}dr+dM_{r}^{t},
\]
where we omitted the script $t$ to ease notation. Then on the event
$\{\tau=t\}$ one can write 
\begin{align}
 & \int_{0<r_{1}<\cdots<r_{n}<1}\phi(\circ dX_{tr_{1}})\otimes\cdots\otimes\phi(\circ dX_{tr_{n}})\nonumber \\
 & =\sum_{I=(\omega_{1},\cdots,\omega_{n})}t^{\frac{n-|I|_{0}}{2}}\int_{0<r_{1}<\cdots<r_{n}<1}\circ dY_{r_{1}}^{\omega_{1}}\otimes\cdots\otimes\circ dY_{r_{n}}^{\omega_{n}},\label{eq:SigDecomp}
\end{align}
where the summation is taken over all words $I=(\omega_{1},\cdots,\omega_{n})$
with $\omega_{i}=0,1$ and $|I|_{0}$ denotes the number of zero's
in $I$ (the number of geodesic components). 

Given such a word $I,$ we set 
\begin{equation}
J_{n}(\rho;I)\triangleq\int_{0<r_{1}<\cdots<r_{n}<\rho}\circ dY_{r_{1}}^{\omega_{1}}\otimes\cdots\otimes\circ dY_{r_{n}}^{\omega_{n}},\ \ \ \rho\in[0,1].\label{eq:JnDef}
\end{equation}
A key step for proving Theorem \ref{thm:ReconRD} is a proper estimate
of $J_{n}(\rho;I)$ that respects the number of geodesic components
in $I$. The main result for this purpose is stated below. In what
follows, $|\cdot|$ denotes the Hilbert-Schmidt tensor norm and $\|\cdot\|_{p}\triangleq\mathbb{E}[|\cdot|^{p};\tau=t]^{1/p}$
denotes the $L^{p}$-norm on $\{\tau=t\}$ for tensor-valued random
variables. 
\begin{prop}
\label{prop:MainSigEst}There exists a positive constant $\Lambda$
depending only on $M,\phi,N$ and the localisation $V$, such that
the following estimate holds true
\begin{equation}
\big\| J_{n}(\rho;I)\big\|_{p}\leqslant\frac{(\Lambda p)^{n-k}}{\sqrt{(n-k)!}}\frac{\|\Gamma\|_{1\text{-var};[0,\rho]}^{k}}{k!}(1-\sqrt{1-\rho})^{\frac{n-k}{2}}\label{eq:MainSigEst}
\end{equation}
 for any $p\geqslant2,\rho\in[0,1],$ any $n\in\mathbb{N}$, $1\leqslant k\leqslant n$
and any word $I=(\omega_{1},\cdots,\omega_{n})$ with $|I|_{0}=k$. 
\end{prop}

\begin{rem}
A delicate point in the proof is to make sure that one only introduces
an exponential factor like $C^{n-k}\|\Gamma\|_{1\text{-var}}^{k}$
(as seen in (\ref{eq:MainSigEst})) instead of something like $C^{n}.$
This is the key point for reducing the lifetime scale $t_{n}$ to
polynomial dependence in $n$ in Theorem \ref{thm:ReconRD}; a more
``standard'' signature estimate would introduce a factor $C^{n}$
forcing $t_{n}$ to decay exponentially in order to make (\ref{eq:ReconRDAsym})
work. 
\end{rem}

To prove Proposition \ref{prop:MainSigEst}, we first recall a basic
result of Kallenberg-Sztencel (cf. \cite[Theorem 3.1]{KS91}) which
leads to a dimension-free BDG inequality. This also avoids the introduction
of $C^{n}$ when estimating moments of $E^{\otimes n}$-valued martingales. 
\begin{lem}
\label{lem:KaSz}Let $S=(S^{1},\cdots,S^{l})$ be a collection of
continuous local martingales. Then there exists a two-dimensional
martingale $T=(T^{1},T^{2})$ defined possibly on an extended filtered
probability space, such that 
\[
|S|=|T|,\ \langle S\rangle=\langle T\rangle,
\]
where
\[
|S|\triangleq\sqrt{(S^{1})^{2}+\cdots+(S^{l})^{2}},\ \langle S\rangle\triangleq\langle S^{1}\rangle+\cdots+\langle S^{l}\rangle
\]
and similarly for $T.$
\end{lem}

Next, we present a lemma that provide the basic estimates for proving
Proposition \ref{prop:MainSigEst} by induction. 
\begin{lem}
\label{lem:IndJEst}There exist positive constants $K_{1},K_{2},K_{3}$
depending only on $M,\phi,N$ and the localisation $V$, such that
the following estimates hold true for any $p\geqslant2,n\in\mathbb{N}$
and any word $I=(\omega_{1},\cdots,\omega_{n+1})$:

\vspace{2mm}\noindent (i) If $\omega_{n+1}=0$, then one has 
\begin{equation}
\|J_{n+1}(\rho;I)\|_{p}\leqslant\int_{0}^{\rho}\|J_{n}(\rho;I')\|_{p}|\dot{\Gamma}_{\eta}|d\eta,\ \ \ \rho\in[0,1].\label{eq:LastGeoEst}
\end{equation}
(ii) If $\omega_{n+1}=1$, then one has 
\begin{align}
\|J_{n+1}(\rho;I)\|_{p} & \leqslant K_{1}\sqrt{r}\int_{0}^{\rho}\|J_{n}(\eta;I')\|_{q}\frac{1}{\sqrt{1-\eta}}d\eta+K_{2}\sqrt{p}\big(\int_{0}^{\rho}\|J_{n}(\eta;I')\|_{p}^{2}d\eta\big)^{1/2}\nonumber \\
 & \ \ \ +K_{3}\int_{0}^{\rho}\|J_{n-1}(\eta;I'')\|_{p}d\eta,\label{eq:LastNotGeoEst}
\end{align}
where $q,r$ are any positive numbers such that $1/p=1/q+1/r$ and
$I'$ (resp. $I''$) is the word obtained by removing the last entry
(resp. last two entries) of $I$. 
\end{lem}

\begin{proof}
If $\omega_{n+1}=0,$ one has
\[
J_{n+1}(\rho;I)=\int_{0}^{\rho}J_{n}(\eta;I')\otimes\dot{\Gamma}_{\eta}d\eta
\]
and the estimate (\ref{eq:LastGeoEst}) follows immediately. If $\omega_{n+1}=1$,
one has 
\begin{align*}
J_{n+1}(\rho;I) & =\int_{0}^{\rho}J_{n}(\eta;I')\otimes A_{\eta}^{t}d\eta+\int_{0}^{\rho}J_{n}(\eta;I')\otimes\circ dM_{\eta}^{t}\\
 & =\int_{0}^{\rho}J_{n}(\eta;I')\otimes A_{\eta}^{t}d\eta+\int_{0}^{\rho}J_{n}(\eta;I')\otimes\cdot dM_{\eta}^{t}\\
 & \ \ \ +\frac{1}{2}\int_{0}^{\rho}\cdot dJ_{n}(\eta;I')\otimes\cdot dM_{\eta}^{t}\\
 & =:I_{n}^{1}(\rho)+I_{n}^{2}(\rho)+I_{n}^{3}(\rho).
\end{align*}
where $\cdot dM$ denotes It\^o's integral. We now estimate each
term on the right hand side separately. Note that the last term comes
from the It\^o-Stratonovich correction and is present only when $\omega_{n}=1$. 

\vspace{2mm}\noindent \textit{\uline{Estimation of }}\uline{\mbox{$I_{n}^{1}.$}}
According to (\ref{eq:ABound}) and Young's inequality, one has 
\[
\|I_{n}^{1}(\rho)\|_{p}\leqslant\int_{0}^{\rho}\|J_{n}(\eta;I')\|_{q}\|A_{\eta}^{t}\|_{r}d\eta\leqslant K_{1}\sqrt{r}\int_{0}^{\rho}\|J_{n}(\eta;I')\|_{q}\frac{1}{\sqrt{1-\eta}}d\eta.
\]
\textit{\uline{Estimation of}}\uline{ \mbox{$I_{n}^{2}.$}} We
are going to applying Lemma \ref{lem:KaSz} to the multidimensional
martingale 
\[
I_{n}^{2}(\rho)\triangleq\int_{0}^{\rho}J_{n}(\eta;I')\otimes\cdot dM_{\eta}^{t}\in(\mathbb{R}^{N})^{\otimes(n+1)}.
\]
By the definition of the Hilbert-Schmidt norm, one has 
\[
|I_{n}^{2}(\rho)|=\big(\sum_{i_{1},\cdots,i_{n+1}=1}^{N}\big(\int_{0}^{\rho}J_{n}(\eta;I)^{i_{1},\cdots,i_{n}}dM_{\eta}^{t,i_{n+1}}\big)^{2}\big)^{1/2}.
\]
According to (\ref{eq:QVMBound}), one also has 
\begin{align*}
\langle I_{n}^{2}\rangle_{\rho} & =\sum_{i_{1},\cdots,i_{n+1}=1}^{N}\int_{0}^{\rho}\big(J_{n}(\eta;I)^{i_{1},\cdots,i_{n}}\big)^{2}d\langle M^{t,i_{n+1}}\rangle_{\eta}\\
 & \leqslant NC\int_{0}^{\rho}\sum_{i_{1},\cdots,i_{n}=1}^{N}\big(J_{n}(\eta;I)^{i_{1},\cdots,i_{n}}\big)^{2}d\eta=NC\int_{0}^{\rho}\big|J_{n}(\eta;I)\big|^{2}d\eta.
\end{align*}
It follows from Lemma \ref{lem:KaSz} as well as the 2D BDG-inequality
that 
\[
\|I_{n}^{2}(\rho)\|_{p}\leqslant4\sqrt{2}\sqrt{p}\|\langle I_{n}^{2}\rangle_{\rho}^{1/2}\|_{p}\ \ \ \forall p\geqslant2.
\]
Here the constant $4\sqrt{2}\sqrt{p}$ is easily derived from the
1D case whose BDG constant is $2\sqrt{p}$ (cf. \cite[Theorem A]{CK91}).
Therefore, one finds that 
\[
\|I_{n}^{2}(\rho)\|_{p}\leqslant4\sqrt{2NC}\sqrt{p}\big\|\big(\int_{0}^{\rho}\big|J_{n}(\eta;I)\big|^{2}d\eta\big)^{1/2}\big\|_{p}.
\]
In addition, note that 
\[
\big\|\big(\int_{0}^{\rho}\big|J_{n}(\eta;I)\big|^{2}d\eta\big)^{1/2}\big\|_{p}=\big\|\int_{0}^{\rho}\big|J_{n}(\eta;I')\big|^{2}d\eta\big\|_{p/2}^{1/2}\leqslant\big(\int_{0}^{\rho}\big\| J_{n}(\eta;I')\big\|_{p}^{2}d\eta\big)^{1/2}.
\]
Consequently, one arrives at
\begin{equation}
\|I_{n}^{2}(\rho)\|_{p}\leqslant K_{2}\sqrt{p}\big(\int_{0}^{\rho}\big\| J_{n}(\eta;I')\big\|_{p}^{2}d\eta\big)^{1/2}\label{eq:MartEst}
\end{equation}
with $K_{2}\triangleq4\sqrt{2NC}.$

\vspace{2mm}\noindent \textit{\uline{Estimation of }}\uline{\mbox{$I_{n}^{3}$}.}
This is only relevant when $\omega_{n}=\omega_{n+1}=1$. In this case,
by (\ref{eq:QVMBound}) one has 
\[
\big|I_{n}^{3}(\rho)\big|\leqslant\frac{1}{2}\int_{0}^{\rho}\big|J_{n-1}(\eta;I'')\big|\cdot|d\langle M^{t},\otimes M^{t}\rangle_{\eta}\big|\leqslant K_{3}\int_{0}^{\rho}\big|J_{n-1}(\eta;I'')\big|d\eta.
\]
It follows that
\[
\|I_{n}^{3}(\rho)\|_{p}\leqslant K_{3}\int_{0}^{\rho}\big\| J_{n-1}(\eta;I'')\big\|_{p}d\eta.
\]
By adding up the previous three inequalities, one obtains the desired
estimate (\ref{eq:LastNotGeoEst}) and concludes the proof of the
lemma. 
\end{proof}
We are now in a position to prove Proposition \ref{prop:MainSigEst}.

\begin{proof}[Proof of Proposition \ref{prop:MainSigEst}]

We are going to prove the inequality (\ref{eq:MainSigEst}) by induction
on $n$. Consider first the base case $n=1$. If $I=(0)$, then one
has 
\[
\big|J_{1}(\rho;I)\big|\leqslant\int_{0}^{\rho}\big|\dot{\Gamma}_{\eta}\big|d\eta=\|\Gamma\|_{1\text{-var};[0,\rho]}
\]
and the estimate (\ref{prop:MainSigEst}) holds trivially with $n=k=1$.
If $I=(1)$, then 
\[
J_{1}(\rho;I)=\int_{0}^{\rho}A_{\eta}^{t}d\eta+M_{\rho}^{t}.
\]
The first term is estimated by (\ref{eq:ABound}) as 
\[
\big\|\int_{0}^{\rho}A_{r}^{t}dr\big\|_{p}\leqslant C\sqrt{p}\int_{0}^{\rho}\frac{d\eta}{\sqrt{1-\eta}}=2C\sqrt{p}(1-\sqrt{1-\rho}).
\]
The second term is estimated by the BDG inequality and (\ref{eq:QVMBound})
as 
\[
\|M_{\rho}^{t}\|_{p}\leqslant\sqrt{2NCp}\sqrt{\rho}.
\]
It follows that 
\[
\|J_{1}(\rho;I)\|_{p}\leqslant\big(2C(1-\sqrt{1-\rho})+\sqrt{2NC}\sqrt{\rho}\big)p.
\]
To make the estimate (\ref{eq:MainSigEst}) work with $n=1,k=0$,
it suffices to require $\Lambda$ to satisfy
\[
2C(1-\sqrt{1-\rho})+\sqrt{2NC}\sqrt{\rho}\leqslant\Lambda\big(1-\sqrt{1-\rho}\big)^{1/2}\ \ \ \forall\rho\in[0,1].
\]
Such a $\Lambda$ clearly exists. 

In what follows, we fix the choice of 
\begin{equation}
\Lambda\triangleq\max\big\{\sup_{\rho\in[0,1]}\frac{2C(1-\sqrt{1-\rho})+\sqrt{2NC}\sqrt{\rho}}{\big(1-\sqrt{1-\rho}\big)^{1/2}},4eK_{1}+\sqrt{2}K_{2}+4K_{3}\big\},\label{eq:LamDef}
\end{equation}
where $K_{1},K_{2},K_{3}$ are the constants appearing in Lemma \ref{lem:IndJEst}.
We just showed the base case $n=1$ and are going to establish the
induction step with the same $\Lambda$. Suppose that the estimate
(\ref{eq:MainSigEst}) is valid for all words with length $\leqslant n$.
Now let $I=(\omega_{1},\cdots,\omega_{n+1})$ be a given word of length
$n+1$ and let $k$ be the number of zero's in $I$. We write 
\[
L_{k}(\rho)\triangleq\int_{0<r_{1}<\cdots<r_{k}<\rho}\big|\dot{\Gamma}_{r_{1}}\big|\cdots\big|\dot{\Gamma}_{r_{k}}\big|dr_{1}\cdots dr_{k}=\frac{\|\Gamma\|_{1\text{-var};[0,\rho]}^{k}}{k!}.
\]
\textit{\uline{Case I: }}\uline{\mbox{$\omega_{n+1}=0.$} }

\vspace{2mm} According to (\ref{eq:LastGeoEst}) and the induction
hypothesis, one has
\begin{align*}
\|J_{n+1}(\rho;I)\|_{p} & \leqslant\int_{0}^{\rho}\frac{(\Lambda p)^{n-k+1}}{\sqrt{(n-k+1)!}}(1-\sqrt{1-\eta})^{\frac{n-k+1}{2}}L_{k}(\eta)\cdot\big|\dot{\Gamma}_{\eta}\big|d\eta\\
 & \leqslant\frac{(\Lambda p)^{n-k+1}}{\sqrt{(n-k+1)!}}(1-\sqrt{1-\rho})^{\frac{n-k+1}{2}}\int_{0}^{\rho}L_{k}(\eta)\cdot\big|\dot{\Gamma}_{\eta}\big|d\eta\\
 & =\frac{(\Lambda p)^{n-k+1}}{\sqrt{(n-k+1)!}}(1-\sqrt{1-\rho})^{\frac{n-k+1}{2}}L_{k+1}(\rho).
\end{align*}
\textit{\uline{Case II: \mbox{$\omega_{n+1}=1.$}}}

\vspace{2mm} In this case, we apply the induction hypothesis to each
term on the right hand side of (\ref{eq:LastNotGeoEst}). Recall that
$I'$ (resp. $I''$) is the word obtained by removing the last entry
(resp. last two entries) of $I$. Firstly, one has 
\begin{align*}
{\rm I} & \triangleq K_{1}\sqrt{r}\int_{0}^{\rho}\big\| J_{n}(\eta;I')\big\|_{q}\frac{1}{\sqrt{1-\eta}}d\eta\\
 & \leqslant K_{1}\sqrt{r}\int_{0}^{\rho}\frac{(\Lambda q)^{n-k}}{\sqrt{(n-k)!}}(1-\sqrt{1-\eta})^{\frac{n-k}{2}}L_{k}(\eta)\frac{1}{\sqrt{1-\eta}}d\eta\\
 & \leqslant K_{1}\sqrt{r}\frac{(\Lambda q)^{n-k}}{\sqrt{(n-k)!}}L_{k}(\rho)\int_{0}^{\rho}(1-\sqrt{1-\eta})^{\frac{n-k}{2}}\frac{d\eta}{\sqrt{1-\eta}}.
\end{align*}
By evaluating the last integral explicitly, it is easily seen that
\begin{equation}
{\rm I}\leqslant\frac{4K_{1}\sqrt{r}(\Lambda q)^{n-k}}{\sqrt{(n-k)!}(n+2-k)}(1-\sqrt{1-\rho})^{\frac{n+1-k}{2}}L_{k}(\rho).\label{eq:IEstInd}
\end{equation}
We now choose $r=p(n+1-k)$ and $q=\frac{p(n+1-k)}{n-k}$ (so that
$1/p=1/q+1/r$). It follows that 
\begin{align*}
\sqrt{r}q^{n-k} & =\sqrt{p(n+1-k)}\big(\frac{p(n+1-k)}{n-k}\big)^{n-k}\\
 & =p^{n-k+1/2}\sqrt{n+1-k}\big(1+\frac{1}{n-k}\big)^{n-k}\leqslant e\sqrt{n+1-k}p^{n+1-k}.
\end{align*}
By substituting this into (\ref{eq:IEstInd}), one obtains that 
\begin{equation}
{\rm I}\leqslant4eK_{1}\frac{\Lambda^{n-k}p^{n+1-k}}{\sqrt{(n+1-k)!}}(1-\sqrt{1-\rho})^{\frac{n+1-k}{2}}L_{k}(\rho).\label{eq:IEst}
\end{equation}

Next, we estimate the second term in (\ref{eq:LastNotGeoEst}); one
has 
\begin{align*}
{\rm II} & \triangleq K_{2}\sqrt{p}\big(\int_{0}^{\rho}\big\| J_{n}(\eta;I')\big\|_{p}^{2}d\eta\big)^{1/2}\\
 & \leqslant K_{2}\sqrt{p}\big(\int_{0}^{\rho}\big(\frac{(\Lambda p)^{n-k}}{\sqrt{(n-k)!}}(1-\sqrt{1-\eta})^{\frac{n-k}{2}}L_{k}(\eta)\big)^{2}d\eta\big)^{1/2}\\
 & \leqslant K_{2}\sqrt{p}\frac{(\Lambda p)^{n-k}}{\sqrt{(n-k)!}}L_{k}(\rho)\big(\int_{0}^{\rho}(1-\sqrt{1-\eta})^{n-k}d\eta\big)^{1/2}.
\end{align*}
The last integral is estimated as 
\begin{align*}
\int_{0}^{\rho}(1-\sqrt{1-\eta})^{n-k}d\eta & \leqslant\int_{0}^{\rho}(1-\sqrt{1-\eta})^{n-k}\frac{1}{\sqrt{1-\eta}}d\eta\\
 & =\frac{2}{n+1-k}(1-\sqrt{1-\rho})^{n+1-k}.
\end{align*}
Therefore, one obtains that 
\begin{align}
{\rm II} & \leqslant K_{2}\sqrt{p}\frac{(\Lambda p)^{n-k}}{\sqrt{(n-k)!}}L_{k}(\rho)\times\frac{\sqrt{2}}{\sqrt{n+1-k}}(1-\sqrt{1-\rho})^{\frac{n+1-k}{2}}\nonumber \\
 & \leqslant\sqrt{2}K_{2}\frac{\Lambda^{n-k}p^{n+1-k}}{\sqrt{(n+1-k)!}}(1-\sqrt{1-\rho})^{\frac{n+1-k}{2}}L_{k}(\rho).\label{eq:IIEst}
\end{align}

Finally, we estimate the third term in (\ref{eq:LastNotGeoEst}),
which is denoted as 
\[
{\rm III}\triangleq K_{3}\int_{0}^{\rho}\big\| J_{n-1}(\eta;I'')\big\|_{p}d\eta.
\]
Note that this term comes from It\^o-Stratonovich correction and
is present only when $\omega_{n}=\omega_{n+1}=1$. In particular,
there are $k$ geodesic components in $I''$ for this scenario. As
a result, by the induction hypothesis one has
\begin{align}
{\rm III} & \leqslant K_{3}\int_{0}^{\rho}\frac{(\Lambda p)^{n-1-k}}{\sqrt{(n-1-k)!}}(1-\sqrt{1-\eta})^{\frac{n-1-k}{2}}L_{k}(\eta)d\eta\nonumber \\
 & \leqslant4K_{3}\frac{\Lambda^{n-k}p^{n+1-k}}{\sqrt{(n+1-k)!}}(1-\sqrt{1-\rho})^{\frac{n+1-k}{2}}L_{k}(\rho).\label{eq:IIIEst}
\end{align}

By putting (\ref{eq:IEst}), (\ref{eq:IIEst}), (\ref{eq:IIIEst})
together, one sees that 
\[
\big\| J_{n+1}(\rho;I)\big\|_{p}\leqslant(4eK_{1}+\sqrt{2}K_{2}+4K_{3})\frac{\Lambda^{n-k}p^{n+1-k}}{\sqrt{(n+1-k)!}}(1-\sqrt{1-\rho})^{\frac{n+1-k}{2}}L_{k}(\rho).
\]
In view of the choice of $\Lambda$ given by (\ref{eq:LamDef}), the
right hand side of the above inequality is further bounded by 
\[
\frac{(\Lambda p)^{n+1-k}}{\sqrt{(n+1-k)!}}(1-\sqrt{1-\rho})^{\frac{n+1-k}{2}}L_{k}(\rho).
\]
This completes the induction step. 

\end{proof}

\subsubsection{\label{subsec:Remainder}Localisation II: the remainder }

Our aim here is to estimate the term $\mathbb{E}\big[\pi_{n}S^{\phi}(X^{t,x,y});\tau<t\big]$
and the main result is stated as follows.
\begin{prop}
\label{prop:LocEst}There exist constants $C_{1},C_{2},\delta>0$
depending only on $x,y,M$, such that 
\begin{equation}
\big\|\mathbb{E}\big[\pi_{n}S^{\phi}(X^{t,x,y});\tau<t\big]\big\|_{{\rm HS}}\leqslant\frac{C_{1}^{n}}{\sqrt{n!}}e^{-C_{2}/t}\label{eq:LocEst}
\end{equation}
for all $n\geqslant1$ and $t\in(0,\delta).$
\end{prop}

The main idea of the proof is to estimate the probability of the event
$\{\tau<t\}$ and moments of $\pi_{n}S^{\phi}(X^{t,x,y})$ separately.
For the first part, we rely on Hsu's large deviation principle for
the Brownian bridge. For the second part, we make use of several heat
kernel estimates in geometric analysis. In what follows, we
divide our proof of Proposition \ref{prop:LocEst} into several steps. 

\subsubsection*{The exit probability}

We first estimate the probability that the Brownian bridge exits the
chart $V$ before its lifetime $t$.
\begin{lem}
\label{lem:TauEst}There are constants $\delta,K>0$ depending only
on $x,y,$ such that 
\[
\mathbb{P}(\tau<t)\leqslant e^{-K/t}\ \ \ \forall t\in(0,\delta).
\]
\end{lem}

\begin{proof}
Let $\mathbb{Q}_{x,y}^{t}$ denote the law of the process $\{X_{st}^{t,x,y}:s\in[0,1]\}$
on the path space 
\[
\Omega_{x,y}\triangleq\big\{ w:[0,1]\rightarrow M\big|w\ \text{continuous, }w_{0}=x,w_{1}=y\big\}.
\]
It is obvious that 
\[
\mathbb{P}(\tau<t)\leqslant\mathbb{Q}_{x,y}^{t}\big(\{w\in\Omega_{x,y}:D(w,y)\geqslant\rho_{M}/2-\varepsilon\}\big)
\]
where $D(w,y)\triangleq\sup\{d(w_{s},y):s\in[0,1]\}.$ According to
the large deviation principle for $\{\mathbb{Q}_{x,y}^{t}\}$ (cf.
\cite[Theorem 2.2]{Hsu90}), one has 
\begin{align*}
\underset{t\rightarrow0}{\overline{\lim}}\ t\log\mathbb{P}(\tau t) & \leqslant\underset{t\rightarrow0}{\overline{\lim}}\ t\log\mathbb{Q}_{x,y}^{t}\big(\{w\in\Omega_{x,y}:D(w,y)\geqslant\rho_{M}/2-\varepsilon\}\big)\\
 & \leqslant-\inf_{\substack{w\in\Omega_{x,y}\\
D(w,y)\geqslant\rho_{M}/2-\varepsilon
}
}J_{x,y}(w),
\end{align*}
where 
\[
J_{x,y}(w)\triangleq\begin{cases}
\dfrac{1}{2}\big(\int_{0}^{1}|\dot{w}_{s}|^{2}ds-d(x,y)^{2}\big), & |\dot{w}|\in L^{2}([0,1]);\\
\infty, & \text{otherwise}.
\end{cases}
\]
For those $w$ with $|\dot{w}|\in L^{2}([0,1])$ and $D(w,y)\geqslant\rho_{M}/2-\varepsilon$,
one has 
\[
\int_{0}^{1}|\dot{w}_{s}|^{2}ds\geqslant\big(\int_{0}^{1}|\dot{w}_{s}|ds\big)^{2}\geqslant(\rho_{M}/2-\varepsilon)^{2}.
\]
It follows that 
\[
\underset{t\rightarrow0}{\overline{\lim}}\ t\log\mathbb{P}(\tau<t)\leqslant\frac{1}{2}d(x,y)^{2}-\frac{1}{2}(\rho_{M}/2-\varepsilon)^{2}.
\]
The result follows by noting that the right hand side is a negative
number. 
\end{proof}

\subsubsection*{Signature moment estimates of Brownian bridge}

Next, we estimate moments of $\pi_{n}S^{\phi}(X^{t,x,y})$. The main
result is stated below. Here we also use the Hilbert-Schmidt norm
on tensors. 
\begin{lem}
\label{lem:SigMoment}There exist positive constants $\Lambda,L$
depending only on $M$ and $\phi$, such that 
\[
\big\|\pi_{n}S^{\phi}(X^{t,x,y})\big\|_{p}\leqslant\frac{(\Lambda p)^{n/2}}{\sqrt{n!}}e^{L/pt}
\]
for all $p\geqslant1,n\in\mathbb{N},t\in(0,1]$ and $x,y\in M$. 
\end{lem}

The argument is quite similar to the proof of Proposition \ref{prop:MainSigEst}.
Since the estimate here is global, we take an extrinsic perspective
and make use of a well-known estimate for the logarithmic derivative
of the heat kernel.

\subsubsection*{I. Extrinsic construction of the Brownian bridge}

Let $M$ be isometrically embedding in some Euclidean space $\mathbb{R}^{m}$
(which is always possible due to Nash's embedding theorem). A \textit{tubular
neighbourhood} of $M$ in $\mathbb{R}^{m}$ is defined by the normal
bundle 
\[
{\cal B}_{\varepsilon}\triangleq\bigsqcup_{x\in M}B^{\perp}(x,\varepsilon).
\]
Here $B^{\perp}(x,\varepsilon)\triangleq\{x'\in\mathbb{R}^{m}:\overrightarrow{xx'}\in(T_{x}M)^{\perp},|x'-x|_{\mathbb{R}^{m}}<\varepsilon\}$.
Since $M$ is compact, it is a standard fact in differential topology
(cf. \cite[Chap. 10]{Lee06}) that when $\varepsilon$ is small, ${\cal B}_{\varepsilon}$
is diffeomorphic to the $\varepsilon$-neighbourhood of $M$ in $\mathbb{R}^{m}.$
We fix such an $\varepsilon.$ There is a canonical projection $\pi:{\cal B}_{\varepsilon}\rightarrow M$
taking $x\in{\cal B}_{\varepsilon}$ onto its unique closest point
$\pi(x)$ in $M$. 

Functions on $M$ can be extended to $\mathbb{R}^{m}$ by using a
bump function on ${\cal B}_{\varepsilon}$. More precisely, let $\eta:[0,\infty)\rightarrow\mathbb{R}$
be a smooth bump function such that $0\leqslant\eta\leqslant1$, $\eta\equiv1$
on $[0,\varepsilon/2]$ and $\eta\equiv0$ on $(\varepsilon,\infty)$.
Given any function $f\in C^{\infty}(M),$ we define 
\[
\bar{f}(x)\triangleq\eta(d_{\mathbb{R}^{m}}(x,\pi(x)))f(\pi(x)),\ \ \ x\in{\cal B}_{\varepsilon}.
\]
It is clear that $\bar{f}\in C_{c}^{\infty}(\mathbb{R}^{m})$. Vector
fields and one-forms on $M$ are extended to $\mathbb{R}^{m}$ in
a similar way. We always use the notation $\bar{\cdot}$ to denote
the extended object. 

For each $x\in M$ and $1\leqslant\alpha\leqslant m,$ let $V_{\alpha}(x)$
be the tangent vector to $M$ defined by orthogonally projecting the
$\alpha$-th canonical basis vector of $\mathbb{R}^{m}$ onto $T_{x}M$.
Then $V_{\alpha}$ defines a smooth vector field on $M$ and we take
its extension $\bar{V}_{\alpha}$ to $\mathbb{R}^{m}$ as before.
Note that the Laplacian on $M$ can be expressed as $\Delta_{M}=\sum_{\alpha=1}^{m}V_{\alpha}^{2}$.
Given fixed $t\in(0,1]$ and $x,y\in\mathbb{R}^{m}$, let $\{\bar{X}_{s}^{t,x,y}:0\leqslant s<t\}$
be the solution to the SDE 
\[
d\bar{X}_{s}=\sum_{\alpha=1}^{m}\bar{V}_{\alpha}(X_{s})dB_{s}+\bar{\nabla}_{x}\bar{p}(t-s,\bar{X}_{s},y)ds,\ \ \ \bar{X}_{0}=x.
\]
Here $\bar{\nabla}$ denotes the Euclidean gradient in $\mathbb{R}^{m}$.
Equivalently, $\bar{X}_{\cdot}^{t,\cdot,y}$ is a Markov family with
generator 
\[
\bar{{\cal L}}_{s}^{t,y}=\frac{1}{2}\sum_{\alpha=1}^{m}\bar{V}_{\alpha}^{2}+\bar{\nabla}_{x}\bar{p}(t-s,\cdot,y)\cdot\bar{\nabla}.
\]
The following fact is of no surprise.
\begin{lem}
\label{lem:ExtBrid}Suppose that $x,y\in M$. Then with probability
one, the process $\bar{X}^{t,x,y}$ lives on $M$ for all time. In
this case, $\bar{X}_{\cdot}^{t,x,y}$ has the same law as the Brownian
bridge from $x$ to $y$ on $M$ with lifetime $t$. 
\end{lem}

\begin{proof}
Consider the function $F(z)\triangleq d_{\mathbb{R}^{m}}(z,M)^{2}$
on $\mathbb{R}^{m}$. For $x\in{\cal B}_{2\varepsilon},$ one has
$F(x)=0$ iff $x\in M$. Define 
\[
\sigma\triangleq\inf\{s\in[0,t):\bar{X}_{s}^{t,x,y}\notin{\cal B}_{\varepsilon/2}\}.
\]
By our choices of extensions, it is easily seen that $(\bar{{\cal L}}_{s}^{t,y}F)(\bar{X}_{s}^{t,x,y})=0$
on $[0,\sigma]$. It follows from the martingale characterisation
that 
\[
\mathbb{E}\big[F(\bar{X}_{s\wedge\sigma}^{t,x,y})\big]=F(x)=0\implies\bar{X}_{s\wedge\sigma}^{t,x,y}\in M\ \ \ \text{a.s. }
\]
By continuity, the process lives on $M$ for all time before $\sigma\wedge t$
which also indicates that $\sigma=t$ (i.e. no exiting ${\cal B}_{\varepsilon/2}$
occurs). For the second part of the lemma, one first notes that the
process 
\[
s\mapsto F(\bar{X}_{s}^{t,x,y})-F(x)-\int_{0}^{s}\bar{{\cal L}}_{r}^{t,y}F(\bar{X}_{r}^{t,x,y})dr
\]
is a martingale for all $F\in C_{c}^{2}(\mathbb{R}^{m})$. In addition,
the law of the Brownian bridge $X^{t,x,y}$ on $M$ is characterised
by the fact that 
\[
s\mapsto f(X_{s}^{t,x,y})-f(x)-\int_{0}^{s}\big(\frac{1}{2}\Delta_{M}f(X_{r}^{t,x,y})+\nabla_{x}^{M}p(t-r,X_{r}^{t,x,y},y)\cdot\nabla^{M}f(X_{r}^{t,x,y})\big)dr
\]
is a martingale for all $f\in C^{2}(M)$. The result follows from
the fact that $\bar{X}_{r}^{t,x,y}\in M$ as well as the relation
that 
\[
\bar{{\cal L}}_{r}^{t,y}\bar{f}\big|_{M}=\frac{1}{2}\Delta_{M}f+\nabla_{x}^{M}p(t-r,\cdot,y)\cdot\nabla^{M}f\ \ \ \forall f\in C^{2}(M).
\]
The proof of the lemma is thus complete. 
\end{proof}

\subsubsection*{II. Heat kernel bounds and a moment estimate for the Brownian bridge}

Our proof of Lemma \ref{lem:SigMoment} relies on a few heat kernel
estimates which we now recall.
\begin{prop}
Let $M$ be a compact Riemannian manifold.

\vspace{2mm} \noindent (i) (cf. \cite[Theorem 5.3.4]{Hsu02})
There are positive constants $K_{1},K_{2},K_{3}$ such that 
\begin{equation}
K_{1}e^{-K_{2}/t}\leqslant p(t,x,y)\leqslant\frac{K_{3}}{t^{d-1/2}}e^{-d(x,y)^{2}/2t}\ \ \ \forall(t,x,y)\in(0,1]\times M\times M.\label{eq:HKBound}
\end{equation}
(ii) (cf. \cite[Inequality (1.1)]{ST98}) There exists a positive
constant $K_{4},$ such that 
\begin{equation}
\big|\nabla_{x}\log p(t,x,y)\big|\leqslant K_{4}\big(\frac{1}{\sqrt{t}}+\frac{d(x,y)}{t}\big)\ \ \ \forall(t,x,y)\in(0,1]\times M\times M.\label{eq:STEst}
\end{equation}
\end{prop}

The following moment estimate of the Brownian bridge will also be
needed for us. It is not in the sharpest form but is enough for our
purpose. 
\begin{lem}
\label{lem:BBMoment}Let $\beta\in(0,1/2)$ be a given fixed parameter.
There exist positive constants $K_{5},K_{6}$ such that 
\[
\|d(X_{s}^{t,x,y},y)\|_{p}\leqslant K_{5}\sqrt{p}e^{K_{2}/pt}(t-s)^{\beta}
\]
for all $p\geqslant1$, $s<t\in(0,1]$ and $x,y\in M.$ Here $K_{2}$
is the same constant appearing in (\ref{eq:HKBound}).
\end{lem}

\begin{proof}
According to the transition density formula (\ref{eq:BBTran}), one
has 
\begin{align*}
\mathbb{E}\big[d(X_{s}^{t,x,y},y)^{p}\big] & =\frac{1}{p(t,x,y)}\int_{M}d(z,y)^{p}p(s,x,z)p(t-s,z,y)dz\\
 & =\frac{1}{p(t,x,y)}\big(\int_{\{z:d(z,y)\leqslant\sqrt{p}(t-s)^{\beta}\}}+\int_{\{z:d(z,y)\geqslant\sqrt{p}(t-s)^{\beta}\}}\big)\\
 & \ \ \ d(z,y)^{p}p(s,x,z)p(t-s,z,y)dz\\
 & \leqslant\sqrt{p}^{p}(t-s)^{\beta p}+D_{M}^{p}\cdot K_{1}^{-1}e^{K_{2}/t}\cdot\frac{K_{3}}{(t-s)^{d-1/2}}e^{-p(t-s)^{2\beta-1}/2},
\end{align*}
where $D_{M}$ is the diameter of $M$ and we also used the semigroup
property as well as the estimate (\ref{eq:HKBound}) to reach the
last inequality. One can assume without loss of generality that $K_{1}<1$
and $K_{3}>1$. Since $p\geqslant1,$ one obtains that 
\[
\big\| d(X_{s}^{t,x,y},y)\big\|_{p}\leqslant\sqrt{p}(t-s)^{\beta}+D_{M}K_{1}^{-1}K_{3}e^{K_{2}/pt}\cdot\frac{e^{-(t-s)^{2\beta-1}/2}}{(t-s)^{d-1/2}}.
\]
The desired estimate follows by noting that 
\[
\frac{e^{-(t-s)^{2\beta-1}/2}}{(t-s)^{d-1/2}}\leqslant C_{d,\beta}(t-s)^{\beta}
\]
for all $s<t\in(0,1].$
\end{proof}

\subsubsection*{III. Proof of Lemma \ref{lem:SigMoment}}

We now proceed to prove Lemma \ref{lem:SigMoment}. The argument is
similar to (but not as fine as) the proof of Proposition \ref{prop:MainSigEst}
and some repeated calculations will thus be omitted. Recall that $\beta\in(0,1/2)$
is a fixed number appearing in Lemma \ref{lem:BBMoment}. For simplicity,
we just write $X=X^{t,x,y}$ (which is also the process $\bar{X}\triangleq\bar{X}^{t,x,y}$
by Lemma \ref{lem:ExtBrid}). We postulate the following estimate
on signature moments of $\bar{X}$ and will prove it by induction
on $n$:
\begin{equation}
\big\|\int_{0<r_{1}<\cdots<r_{n}<\rho}\bar{\phi}(\circ d\bar{X}_{tr_{1}})\otimes\cdots\otimes\bar{\phi}(\circ d\bar{X}_{tr_{n}})\big\|_{p}\leqslant\frac{(\Lambda p)^{n/2}}{\sqrt{n!}}e^{K_{2}/pt}\big(1-(1-\rho)^{\beta}\big)^{n/2}\label{eq:IndHypoExt}
\end{equation}
for all $p\geqslant1,n\in\mathbb{N},t\in(0,1],$ $\rho\in[0,1].$
Here $K_{2}$ is the same constant appearing in Part II and $\Lambda$
is a universal constant that will be determined in the induction argument.
Lemma \ref{lem:SigMoment} follows immediately by taking $\rho=1.$
We set 
\[
J_{n}(\rho)\triangleq\int_{0<r_{1}<\cdots<r_{n}<\rho}\bar{\phi}(\circ d\bar{X}_{tr_{1}})\otimes\cdots\otimes\bar{\phi}(\circ d\bar{X}_{tr_{n}}).
\]
\textit{\uline{Semi-martingale decomposition}}:

\vspace{2mm}From the extrinsic perspective in Part I, one can write
\begin{align}
\bar{\phi}(\circ d\bar{X}_{tr}) & =\sqrt{t}\bar{\phi}_{i}(\bar{X}_{tr})\bar{V}^{i}(\bar{X}_{tr})dB_{r}^{t}+\frac{t}{2}\partial_{j}\bar{\phi}_{i}(\bar{X}_{tr})\bar{a}^{ij}(\bar{X}_{tr})dr\nonumber \\
 & \ \ \ +t\bar{\phi}_{i}(\bar{X}_{tr})\partial_{i}\log\bar{p}(t(1-r),\bar{X}_{tr},y)dr,\ \ \ r\in[0,1],\label{eq:SemDecompExt}
\end{align}
where $B_{r}^{t}\triangleq B_{tr}/\sqrt{t}$ and $\bar{a}\triangleq\bar{V}\cdot\bar{V}^{T}.$
Note that $\bar{\phi},\bar{V},\bar{a}\in C_{c}^{\infty}$ by our choice
of extension. In what follows, we will use $C_{i}$ to denote constants
depending only on $M,\phi,\beta$ and the embedding.

\vspace{2mm}\noindent \textit{\uline{Base step}}: 

\vspace{2mm}By using the BDG inequality, one has 
\begin{equation}
\big\|\sqrt{t}\int_{0}^{\rho}\bar{\phi}_{i}(\bar{X}_{tr})\bar{V}^{i}(\bar{X}_{tr})dB_{r}^{t}\big\|_{p}\leqslant C_{1}\sqrt{p}\cdot\sqrt{\rho}.\label{eq:I1a}
\end{equation}
It is also obvious that 
\begin{equation}
\big\|\int_{0}^{\rho}\frac{t}{2}\partial_{j}\bar{\phi}_{i}(\bar{X}_{tr})\bar{a}^{ij}(\bar{X}_{tr})dr\big\|_{p}\leqslant C_{2}\rho.\label{eq:I1b}
\end{equation}
For the last term of (\ref{eq:SemDecompExt}), by using the Stroock-Turestsky
estimate (\ref{eq:STEst}) and Lemma \ref{lem:BBMoment}, one has
\begin{align*}
 & \big\|\int_{0}^{\rho}t\bar{\phi}_{i}(\bar{X}_{tr})\partial_{i}\log\bar{p}(t(1-r),\bar{X}_{tr},y)dr\big\|_{p}\\
 & \leqslant C_{3}t\int_{0}^{\rho}K_{4}\big(\frac{1}{\sqrt{t(1-r)}}+\frac{\|d(\bar{X}_{tr},y)\|_{p}}{t(1-r)}\big)dr\\
 & \leqslant C_{3}K_{4}\int_{0}^{\rho}\frac{dr}{\sqrt{1-r}}+C_{3}K_{4}\int_{0}^{\rho}\frac{K_{5}\sqrt{p}e^{K_{2}/pt}(t(1-r))^{\beta}}{1-r}dr.
\end{align*}
After evaluating the above integrals explicitly and combining with
(\ref{eq:I1a}, \ref{eq:I1b}), one finds that 
\begin{equation}
\|J_{1}(\rho)\|_{p}\leqslant C_{4}\sqrt{p}e^{K_{2}/pt}\big(1-(1-\rho)^{\beta}\big)^{1/2}\label{eq:J1Est}
\end{equation}
with suitably chosen constant $C_{4}$.

\vspace{2mm}\noindent \textit{\uline{Induction step}}:

\vspace{2mm}Suppose that the induction hypothesis (\ref{eq:IndHypoExt})
holds. According to the decomposition (\ref{eq:SemDecompExt}), one
can write 
\begin{align*}
J_{n+1}(\rho) & =\int_{0}^{\rho}J_{n}(r)\otimes\bar{\phi}(\circ d\bar{X}_{tr})\\
 & =\sqrt{t}\int_{0}^{\rho}J_{n}(r)\otimes\bar{\phi}_{i}(\bar{X}_{tr})\bar{V}^{i}(\bar{X}_{tr})dB_{r}^{t}+\frac{t}{2}\int_{0}^{\rho}J_{n}(r)\otimes\partial_{j}\bar{\phi}_{i}(\bar{X}_{tr})\bar{a}^{ij}(\bar{X}_{tr})dr\\
 & \ \ \ +\frac{t}{2}\int_{0}^{\rho}J_{n-1}(r)\otimes\bar{\phi}_{i}(\bar{X}_{tr})\otimes\bar{\phi}_{j}(\bar{X}_{tr})\bar{a}^{ij}(\bar{X}_{tr})dr\\
 & \ \ \ +t\int_{0}^{\rho}J_{n}(r)\otimes\bar{\phi}_{i}(\bar{X}_{tr})\partial_{i}\log p(t(1-r),\bar{X}_{tr},y)dr\\
 & =:I_{n}^{1}(\rho)+I_{n}^{2}(\rho)+I_{n}^{3}(\rho)+I_{n}^{4}(\rho).
\end{align*}
By using the dimension-free BDG inequality (the same way leading to
(\ref{eq:MartEst})), one has
\[
\|I_{n}^{1}(\rho)\|_{p}\leqslant C_{5}\sqrt{p}\big(\int_{0}^{\rho}\|J_{n}(r)\|_{p}^{2}dr\big)^{1/2}.
\]
By substituting the estimate (\ref{eq:IndHypoExt}) into the right
hand side and evaluating the resulting $dr$-integral explicitly,
one obtains that 
\begin{equation}
\|I_{n}^{1}(\rho)\|_{p}\leqslant\frac{C_{6}\Lambda^{n/2}p^{\frac{n+1}{2}}}{\sqrt{(n+1)!}}\cdot e^{K_{2}/pt}\big(1-(1-\rho)^{\beta}\big)^{\frac{n+1}{2}}.\label{eq:I1Ext}
\end{equation}
The estimation of $I_{n}^{2}(\rho)$ and $I_{n}^{3}(\rho)$ follows
the same route and one finds that 
\begin{equation}
\|I_{n}^{2}(\rho)\|_{p}\vee\|I_{n}^{3}(\rho)\|_{p}\leqslant\frac{C_{7}(\Lambda p)^{n/2}}{\sqrt{(n+1)!}}\cdot e^{K_{2}/pt}\big(1-(1-\rho)^{\beta}\big)^{\frac{n+1}{2}}.\label{eq:I23Ext}
\end{equation}

We now consider $I_{n}^{4}(\rho).$ By using the gradient estimate
(\ref{eq:STEst}), one has 
\begin{equation}
\|I_{n}^{4}(\rho)\|_{p}\leqslant C_{8}\big(\int_{0}^{\rho}\frac{1}{\sqrt{1-r}}\|J_{n}(r)\|_{p}dr+\int_{0}^{\rho}\frac{1}{1-r}\|d(\bar{X}_{tr},y)\cdot J_{n}(r)\|_{p}dr\big).\label{eq:I4a}
\end{equation}
After applying the induction hypothesis (\ref{eq:IndHypoExt}), the
first integral is majorised as
\[
\frac{2}{\beta}\frac{(\Lambda p)^{n/2}}{\sqrt{(n+1)!}}e^{K_{2}/pt}\big(1-(1-\rho)^{\beta}\big)^{\frac{n+1}{2}}.
\]
For the second integral in (\ref{eq:I4a}), one applies H\"older's
inequality
\[
\|d(\bar{X}_{tr},y)\cdot J_{n}(r)\|_{p}\leqslant\|d(\bar{X}_{tr},y)\|_{p_{1}}\cdot\|J_{n}(r)\|_{p_{2}}
\]
with $p_{1}\triangleq p(n+1)$ and $p_{2}\triangleq p(n+1)/n$. After
further applying Lemma \ref{lem:BBMoment} and the induction hypothesis,
it follows that
\[
\int_{0}^{\rho}\frac{1}{1-r}\|d(\bar{X}_{tr},y)\cdot J_{n}(r)\|_{p}dr\leqslant\frac{2K_{5}\sqrt{e}}{\beta}\cdot\frac{\Lambda^{n/2}p^{\frac{n+1}{2}}}{\sqrt{(n+1)!}}\cdot e^{K_{2}/pt}\big(1-(1-\rho)^{\beta}\big)^{\frac{n+1}{2}}.
\]
Therefore, one arrives at 
\begin{equation}
\|I_{n}^{4}(\rho)\|_{p}\leqslant\frac{C_{9}\Lambda^{n/2}p^{\frac{n+1}{2}}}{\sqrt{(n+1)!}}\cdot e^{K_{2}/pt}\big(1-(1-\rho)^{\beta}\big)^{\frac{n+1}{2}}.\label{eq:I4Ext}
\end{equation}

Putting the estimates (\ref{eq:I1Ext}), (\ref{eq:I23Ext}), (\ref{eq:I4Ext})
together, one concludes that 
\[
\|J_{n+1}(\rho)\|_{p}\leqslant(C_{6}+2C_{7}+C_{9})\cdot\frac{\Lambda^{n/2}p^{\frac{n+1}{2}}}{\sqrt{(n+1)!}}\cdot e^{K_{2}/pt}\big(1-(1-\rho)^{\beta}\big)^{\frac{n+1}{2}}.
\]
By taking the base step estimate (\ref{eq:J1Est}) into account, one
only needs to choose 
\[
\Lambda=\max\{C_{4}^{2},C_{6}+2C_{7}+C_{9}\}.
\]
It then follows that 
\[
\|J_{1}(\rho)\|_{p}\leqslant(\Lambda p)^{1/2}e^{K_{2}/pt}\big(1-(1-\rho)^{\beta}\big)^{1/2}
\]
and 
\[
\|J_{n+1}(\rho)\|_{p}\leqslant\frac{(\Lambda p)^{\frac{n+1}{2}}}{\sqrt{(n+1)!}}\cdot e^{K_{2}/pt}\big(1-(1-\rho)^{\beta}\big)^{\frac{n+1}{2}}.
\]
This finishes the induction argument. 

\subsubsection*{Proof of Proposition \ref{prop:LocEst}}

We now complete the proof of Proposition \ref{prop:LocEst}. Let $p,q>1$
be such that $1/p+1/q=1.$ According to Lemma \ref{lem:TauEst}, Lemma
\ref{lem:SigMoment} and H\"older's inequality, one has 
\begin{align*}
\big\|\mathbb{E}\big[\pi_{n}S^{\phi}(X^{t,x,y});\tau<t\big]\big\|_{{\rm HS}} & \leqslant\big\|\pi_{n}S^{\phi}(X^{t,x,y})\big\|_{p}\cdot\mathbb{P}(\tau<t)^{1/q}\\
 & \leqslant\frac{(\Lambda p)^{n/2}}{\sqrt{n!}}\cdot e^{L/pt}\cdot e^{-K/qt}=\frac{(\Lambda p)^{n/2}}{\sqrt{n!}}e^{-(K/q-L/p)/t}.
\end{align*}
One can choose (and fix) $p,q$ so that $K/q-L/p>0.$ The desired
estimate (\ref{eq:LocEst}) thus follows with $C_{1}\triangleq\sqrt{\Lambda p}$
and $C_{2}\triangleq K/q-L/p.$ 

\subsubsection{Completing the proof of Theorem \ref{thm:ReconRD}}

Putting together all previously developed ingredients, we are now
in a position to prove Theorem \ref{thm:ReconRD}. We will continue
to use the notation introduced earlier. 

Let us define 
\[
b_{n}(t)\triangleq n!\big\|\mathbb{E}\big[\pi_{n}S^{\phi}(X^{t,x,y})\big]\big\|_{{\rm HS}},\ L\triangleq\|\Gamma\|_{1\text{-var};[0,1]}.
\]
Recall that in the isometric embedding context (i.e. $\phi=dF$) one
has $L=d(x,y)$. Our goal is to estimate $b_{n}(t)^{1/n}-L$. To this
end, one writes 
\begin{equation}
b_{n}(t)^{1/n}-L=L\big(\exp\big(\frac{1}{n}\log\frac{b_{n}(t)}{L^{n}}\big)-1\big).\label{eq:Bn-L}
\end{equation}
The proof of Theorem \ref{thm:ReconRD} will be completed as long
as the following lemma is established. 
\begin{lem}
\label{lem:FinalLem}There exist positive constants $\kappa,K$, $R_{1},R_{2}$,
such that 
\[
\frac{b_{n}(t)}{L^{n}}\in[R_{1},R_{2}]
\]
for all $n\geqslant K$ and $t\in(0,\kappa n^{-6}].$ 
\end{lem}

\begin{proof}
One begins by expressing
\begin{equation}
\frac{b_{n}(t)}{L^{n}}=\frac{n!\big\|\pi_{n}S^{\phi}(\gamma^{x,y})\big\|_{{\rm HS}}}{L^{n}}+\frac{b_{n}(t)-n!\big\|\pi_{n}S^{\phi}(\gamma^{x,y})\big\|_{{\rm HS}}}{L^{n}}.\label{eq:DecompQuot}
\end{equation}
Since $\gamma^{x,y}$ is smooth, the following asymptotics for the
first term is a direct consequence of \cite[Theorem 8]{HL10} (see
also \cite[Theorem 22]{CMT23}): 

\begin{equation}
\lim_{n\rightarrow\infty}\frac{n!\big\|\pi_{n}S^{\phi}(\gamma^{x,y})\big\|_{{\rm HS}}}{L^{n}}=c(\Gamma),\label{eq:HLAsymp}
\end{equation}
where $c(\Gamma)>0$ is some explicit constant depending on the geometry
of $\Gamma$. 

Our task is now reduced to estimating the second term on the right
hand side of (\ref{eq:DecompQuot}). Trivially, one has 
\begin{align}
\big|\frac{b_{n}(t)-n!\big\|\pi_{n}S^{\phi}(\gamma^{x,y})\big\|_{{\rm HS}}}{L^{n}}\big| & \leqslant\frac{n!}{L^{n}}\big\|\mathbb{E}\big[\pi_{n}S^{\phi}(X^{t,x,y})\big]-\pi_{n}S^{\phi}(\gamma^{x,y})\big\|_{{\rm HS}}\nonumber \\
 & =:A_{n}(t)+B_{n}(t),\label{eq:En}
\end{align}
where 
\begin{align*}
A_{n}(t) & \triangleq\frac{n!}{L^{n}}\big\|\mathbb{E}\big[\pi_{n}S^{\phi}(X^{t,x,y})-\pi_{n}S^{\phi}(\gamma^{x,y});\tau<t\big]\big\|_{{\rm HS}},\\
B_{n}(t) & \triangleq\frac{n!}{L^{n}}\big\|\mathbb{E}\big[\pi_{n}S^{\phi}(X^{t,x,y})-\pi_{n}S^{\phi}(\gamma^{x,y});\tau=t\big]\big\|_{{\rm HS}}.
\end{align*}
According to Proposition \ref{prop:LocEst} and Lemma \ref{lem:TauEst},
one has 
\begin{align}
A_{n}(t) & \leqslant\frac{n!}{L^{n}}\big\|\mathbb{E}\big[\pi_{n}S^{\phi}(X^{t,x,y});\tau<t\big]\big\|_{{\rm HS}}+\frac{n!}{L^{n}}\big\|\pi_{n}S^{\phi}(\gamma^{x,y})\big\|_{{\rm HS}}\cdot\mathbb{P}(\tau<t)\nonumber \\
 & \leqslant\frac{n!}{L^{n}}\cdot\frac{C_{1}^{n}}{\sqrt{n!}}e^{-C_{2}/t}+\frac{n!}{L^{n}}\cdot\frac{L^{n}}{n!}\cdot e^{-C_{3}/t}\nonumber \\
 & \leqslant\big(\sqrt{n!}(C_{1}/L)^{n}+1\big)e^{-C_{4}/t}\label{eq:E'Est}
\end{align}
for all $n\geqslant1$ and $t\in(0,\delta)$ with some positive constant
$\delta.$ 

To estimate $B_{n}(t),$ we first recall from (\ref{eq:SigDecomp})
that on the event $\{\tau=t\}$ one has the decomposition 
\[
\pi_{n}S^{\phi}(X^{t,x,y})=\pi_{n}S^{\phi}(\gamma^{x,y})+\sum_{k=0}^{n-1}\sum_{I\in{\cal I}_{n}(k)}t^{\frac{n-k}{2}}J_{n}(1;I),
\]
where ${\cal I}_{n}(k)$ denotes the set of length-$n$ words $I$
with $|I|_{0}=k$ and $J_{n}(1;I)$ is defined by (\ref{eq:JnDef}).
According to Proposition \ref{prop:MainSigEst} and (\ref{eq:SigDecomp}),
one finds that 
\begin{align*}
B_{n}(t) & \leqslant\frac{n!}{L^{n}}\sum_{k=0}^{n-1}\sum_{I\in{\cal I}_{n}(k)}t^{\frac{n-k}{2}}\big\| J_{n}(1;I)\big\|_{2}\\
 & \leqslant\frac{n!}{L^{n}}\sum_{k=0}^{n-1}\sum_{I\in{\cal I}_{n}(k)}\frac{(2\Lambda)^{n-k}}{\sqrt{(n-k)!}}\cdot\frac{L^{k}}{k!}\cdot t^{\frac{n-k}{2}}=n!\sum_{k=0}^{n-1}{n \choose k}\frac{1}{\sqrt{(n-k)!}k!}\big(\frac{2\Lambda\sqrt{t}}{L}\big)^{n-k}.
\end{align*}
To analyse the last expression, let us define 
\[
F_{n}(x)\triangleq n!\sum_{k=0}^{n-1}{n \choose k}\frac{x^{n-k}}{\sqrt{(n-k)!}k!}=\sum_{r=1}^{n}{n \choose r}\frac{n!x^{r}}{(n-r)!\sqrt{r!}}
\]
where one takes $x=2\Lambda\sqrt{t}/L$ in our context. Setting 
\[
a_{r}\triangleq{n \choose r}\frac{n!x^{r}}{(n-r)!\sqrt{r!}},
\]
one easily sees that 
\[
\frac{a_{r+1}}{a_{r}}=x\cdot\frac{(n-r)^{2}}{(r+1)^{3/2}}\leqslant n^{2}x
\]
for all $r=1,\cdots,n-1$. By requiring 
\begin{equation}
n^{2}x=\frac{2\Lambda n^{2}\sqrt{t}}{L}\leqslant1\iff t\leqslant\big(\frac{L}{2\Lambda}\big)^{2}n^{-4},\label{eq:TReq1}
\end{equation}
one concludes that $\{a_{r}\}_{r=1}^{n}$ is a decreasing sequence,
thus yielding that 
\[
F_{n}(x)\leqslant na_{1}=n^{3}x=\frac{2\Lambda n^{3}\sqrt{t}}{L}.
\]
Therefore, one has 
\begin{equation}
B_{n}(t)\leqslant\frac{2\Lambda n^{3}\sqrt{t}}{L}\label{eq:E''Est}
\end{equation}
for all $n$ and $t$ satisfying the relation (\ref{eq:TReq1}). 

By substituting the estimates (\ref{eq:E'Est}) and (\ref{eq:E''Est})
into (\ref{eq:En}), one obtains that 
\[
\big|\frac{b_{n}(t)-n!\big\|\pi_{n}S^{\phi}(\gamma^{x,y})\big\|_{{\rm HS}}}{L^{n}}\big|\leqslant\big(\sqrt{n!}(C_{1}/L)^{n}+1\big)e^{-C_{4}/t}+\frac{2\Lambda n^{3}\sqrt{t}}{L}
\]
for all $n,t$ provided that (\ref{eq:TReq1}) holds. Note that the
right hand side is of constant magnitude if $t=O(n^{-6})$. More specifically,
there exist positive constants $\kappa,K$ such that 
\begin{equation}
\big|\frac{b_{n}(t)-n!\big\|\pi_{n}S^{\phi}(\gamma^{x,y})\big\|_{{\rm HS}}}{L^{n}}\big|\leqslant\frac{1}{2}c(\Gamma)\label{eq:RemainEst}
\end{equation}
for all $n\geqslant K$ and $t\in(0,\kappa n^{-6}]$ (the relation
(\ref{eq:TReq1}) is also satisfied with such choice of $t$). Here
$c(\Gamma)$ is the Hambly-Lyons limit appearing in (\ref{eq:HLAsymp}).
Now the conclusion of the lemma follows from the decomposition (\ref{eq:DecompQuot}),
the asymptotics (\ref{eq:HLAsymp}) and the inequality (\ref{eq:RemainEst})
(with enlarged $K$ if necessary). 
\end{proof}
\begin{proof}[Proof of Theorem \ref{thm:ReconRD}]

Under the situation of Lemma \ref{lem:FinalLem}, one has 
\[
\big|\log\frac{b_{n}(t)}{L^{n}}\big|\leqslant\max\big\{|\log R_{1}|,|\log R_{2}|\big\}.
\]
for all $n,t$ such that $n\geqslant K,t\leqslant\kappa n^{-6}.$
It follows from the relation (\ref{eq:Bn-L}) that
\[
\big|b_{n}(t)^{1/n}-L\big|=L\big|\exp\big(\frac{1}{n}\log\frac{b_{n}(t)}{L^{n}}\big)-1\big|\leqslant\frac{C}{n}
\]
with a suitable constant $C$. The proof of Theorem \ref{thm:ReconRD}
is now complete.  

\end{proof}

\section{\label{sec:ReconK}Small-time expansion and its connection with curvature
properties}

Theorem \ref{thm:ReconRD} shows that the Riemannian distance can
be recovered from suitable signature asymptotics of the Brownian bridge.
In this section, we are going to show that curvature properties of
the manifold $M$ (both intrinsic and extrinsic) can also be recovered
explicitly from the small-time expansion of the signature. 

\subsection{\label{subsec:MainThmRecCurv}The main theorem}

From now on, we consider the Brownian loop based at $x\in M$ (i.e.
the Brownian bridge $X^{t,x,x}$). Recall that an isometric embedding
$F:M\rightarrow E=\mathbb{R}^{N}$ is given fixed. For each $n\geqslant1$,
we define (with $\phi=dF$ here)
\[
\psi_{n}(t,x)\triangleq\mathbb{E}\big[\pi_{n}S^{\phi}(X^{t,x,x})\big]\in E^{\otimes n},\ \ \ t>0,x\in M.
\]
This is the $n$-th level expected $\phi$-signature of the Brownian
loop $X^{t,x,x}.$ 

Let us denote $S_{n}\triangleq\pi_{n}S^{\phi}(X^{t,x,x})$. Since
$\int_{0}^{s}\phi(dX^{t,x,x})=F(X_{s}^{t,x,x})-F(x)$, it is clear
that 
\[
S_{1}=F(X_{t}^{t,x,x})-F(x)=F(x)-F(x)=0\implies\psi_{1}(t,x)=0.
\]
From the symmetry of the heat kernel, it is easily seen that $X^{t,x,x}$
has the same law as its time reversal $s\mapsto\overleftarrow{X}_{s}^{t,x,x}\triangleq X_{t-s}^{t,x,x}$.
As a result, the paths $\Gamma_{\cdot}\triangleq F(X_{\cdot}^{t,x,x})-F(x)$
and its reversal $\overleftarrow{\Gamma}$ have the same law. Since
the signature of $\overleftarrow{\Gamma}$ is the inverse of the signature
of $\Gamma$ (in the tensor algebra $T((E))$), it follows that 
\begin{equation}
S^{\phi}\big(X^{t,x,x}\big)^{-1}\stackrel{{\rm law}}{=}S^{\phi}(X^{t,x,x}).\label{eq:RevSym}
\end{equation}
By taking its second level projection, one has 
\[
-S_{2}+S_{1}\otimes S_{1}\stackrel{{\rm law}}{=}S_{2}.
\]
But one already knows $S_{1}=0.$ Therefore, $\psi_{2}(t,x)=\mathbb{E}[S_{2}]=0.$
Similarly, the third level projection of (\ref{eq:RevSym}) gives
\[
-S_{3}=-S_{3}+S_{1}\otimes S_{2}+S_{2}\otimes S_{1}-S_{1}^{\otimes3}\stackrel{{\rm law}}{=}S_{3}\implies\psi_{3}(t,x)=\mathbb{E}[S_{3}]=0.
\]

The above calculation suggests that, in order to extract nontrivial geometric information,
one has to at least consider $\psi_{4}(t,x)$. A similar calculation
shows that 
\begin{equation}
\psi_{4}(t,x)=\mathbb{E}[S_{4}]=\frac{1}{2}\mathbb{E}[S_{2}\otimes S_{2}].\label{eq:2MomRel}
\end{equation}
This needs not be a zero tensor; indeed it contains the second moment
of L\'evy areas of the path $\Gamma$.

In this section, we study the small-time expansion of the function
$\psi_{4}(t,x)$. We first define a contraction on $4$-tensors to
reduce dimension. Let $\xi\in E^{\otimes4}$ be a given tensor. Being
viewed as a $4$-linear map $\xi:E\times E\times E\times E\rightarrow\mathbb{R}$,
we define $\mathfrak{C}\xi:E\times E\rightarrow\mathbb{R}$ to be
the $2$-tensor obtained from $\xi$ by taking trace on the $(2,4)$-positions.
More precisely, 
\begin{equation}
(\mathscr{\mathfrak{C}}\xi)(\cdot,\cdot)\triangleq\sum_{\alpha=1}\xi(\cdot,\varepsilon_{\alpha},\cdot,\varepsilon_{\alpha}),\label{eq:24Trace}
\end{equation}
where $\{\varepsilon_{\alpha}:\alpha=1,\cdots,N\}$ is any ONB of
$E$. One also needs to recall the basic curvature quantities introduced
in Section \ref{subsec:Curv}. Our main theorem for this part is stated
as follows. 
\begin{thm}
\label{thm:RecCurv}One has the following small-time expansion: 
\[
\mathfrak{C}\psi_{4}(t,x)=\Theta_{x}\cdot t^{2}+\Xi_{x}\cdot t^{3}+O(t^{4})\ \ \ \text{as }t\rightarrow0^{+}.
\]
In addition, the coefficients $\Theta_{x},\Xi_{x}\in E^{\otimes2}$
are symmetric $2$-tensors that admit the following explicit representation.

\vspace{2mm}\noindent (i) The tensor $\Theta_{x}$ recovers the Riemannian
metric tensor $g_{x}$. More precisely, one has
\begin{equation}
\Theta_{x}(v,w)=\frac{d-1}{24}\langle\pi_{x}v,\pi_{x}w\rangle_{T_{x}M},\ \ \ v,w\in E,\label{eq:t2Coef}
\end{equation}
where $\pi_{x}:E\rightarrow T_{x}M$ denotes the orthogonal projection
onto the tangent space $T_{x}M.$ 

\vspace{2mm}\noindent (ii) The tensor $\Xi_{x}$ encodes both intrinsic
(Ricci curvature) and extrinsic (second fundamental form) curvature
properties of $M$ at $x$. More precisely, its restriction to $T_{x}M\times T_{x}M$
as a symmetric bilinear form on $T_{x}M$ is explicitly given by 
\begin{equation}
\Xi_{x}|_{T_{x}M\times T_{x}M}=\frac{{\rm S}_{x}-18d^{2}|H_{x}|_{E}^{2}}{8640}g_{x}+\frac{49d-20}{8640}{\rm Ric}_{x}+\frac{(5-4d)d}{480}\langle B_{x},H_{x}\rangle_{E}.\label{eq:t3Coef}
\end{equation}
Here $g_{x}$ is the metric tensor, ${\rm Ric}_{x},{\rm S}_{x}$ are
the Ricci tensor and scalar curvature, $B_{x},H_{x}$ are the second
fundamental form and the mean curvature vector, and $\langle\cdot,\cdot\rangle_{E}$
is the Euclidean inner product in $E$. 
\end{thm}

\begin{rem}
It can be shown that the $4$-tensor $\psi_{4}(t,x)$ is anti-symmetric
on the first-two components and also on the last-two components. Therefore,
the operation (\ref{eq:24Trace}) is basically the only nontrivial
way of contraction into a $2$-tensor. 
\end{rem}

To prove Theorem \ref{thm:RecCurv}, we will actually compute the third-order expansion of the $4$-tensor
\begin{equation}
\psi_{4}(t,x)=\hat{\Theta}_{x}t^{2}+\hat{\Xi}_{x}t^{3}+O(t^{4})\label{eq:Psi4Exp}
\end{equation}
with explicit tensor coefficients $\hat{\Theta}_{x},\hat{\Xi}_{x}$
(cf. Propositions \ref{prop:t2Coef}, \ref{prop:t3Ceof} below). By
some elementary algebraic manipulation, it is possible to recover
all the quantities ${\rm Ric}_{x},{\rm S}_{x},\langle B_{x},H_{x}\rangle_{E}$
and $|H_{x}|_{E}$ in (\ref{eq:t3Coef}) separately from the coefficients
$\hat{\Theta}_{x},\hat{\Xi}_{x}$ (cf. Remark \ref{rem:SepCurvRec}
for a discussion). But the corresponding formulae are more involved.
For simplicity, we choose to present the formula in the form of (\ref{eq:t3Coef})
through the contraction $\mathfrak{C}\psi_{4}(t,x)$. 

In the following sections, we develop the proof of Theorem \ref{thm:RecCurv}.
Our strategy follows the same spirit as the proof of Theorem \ref{thm:ReconRD};
we perform local calculations in the normal chart around $x$ based
on the local SDE representation of $X^{t,x,x}$. The actual calculation
turns out to be so involved that it is  impossible to achieve entirely
by hand. Therefore, our full calculation is computer-assisted (by Wolfram
Mathematica). We will develop the explicit analysis for a few representative
cases in the computation of the coefficients $\hat{\Theta}_{x},\hat{\Xi}_{x}$
and leave all other parallel cases to the appendix with documented
Mathematica codes provided in GitHub. 

\subsection{Some basic expansions}

As in the proof of Theorem \ref{thm:ReconRD}, we will localise the
analysis in a normal chart $V$ around $x$ (cf. Section \ref{subsec:NormChart}
for the construction with $x=y$ in the current situation). As we
have seen in Section \ref{subsec:Remainder}, the case that $X^{t,x,x}$
leaves the chart $V$ before its lifetime gives a negligible contribution
(it is of order $e^{-C/t}$ which is smaller than any power of $t$;
cf. Proposition \ref{prop:LocEst}). Therefore, one only needs consider
the situation that $X^{t,x,x}$ remains in the chart $V$ for all
time (and we will always assume this is the case). In this case, the
analysis reduces to Euclidean stochastic calculus. 

We will use the relation (\ref{eq:2MomRel}) to compute $\psi_{4}(t,x).$
Namely, we write 
\begin{equation}
\psi_{4}(t,x)=\frac{1}{2}\mathbb{E}[\Pi_{t}^{x}\otimes\Pi_{t}^{x}],\label{eq:Psi4}
\end{equation}
 where 
\begin{align*}
\Pi_{t}^{x} & \triangleq\int_{0}^{t}\big(\int_{0}^{v}\phi(\circ dX_{u}^{t,x,x})\big)\otimes\circ\phi\big(\circ dX_{v}^{t,x,x}\big)\\
 & =\int_{0}^{1}\big(\int_{0}^{\rho}\phi(\circ dX_{tr}^{t,x,x})\big)\otimes\circ\phi(\circ dX_{t\rho}^{t,x,x}),
\end{align*}
and we made the change of variables $v=t\rho,u=tr$ to normalise the
processes on the unit interval. To evaluate the expectation (\ref{eq:Psi4})
(in particular, its expansion in $t$), one needs to compute the semimartingale
decomposition of $\phi(\circ dX_{tr}^{t,x,x})$ in the chart $V$. 

\subsubsection{SDE representation of $X^{t,x,x}$}

Recall that the SDE representation of $X^{t,x,x}$ in $V$ (in terms
of coordinates in $U\triangleq\exp_{x}^{-1}V$) is given by (\ref{eq:BridgeSDESimplified})
with $Q(u,{\bf z})$ defined in terms of the Malliavin-Stroock expansion
(\ref{eq:MSExp}) (cf. (\ref{eq:Q})). For our purpose of computing
small-time expansion, one needs to extract one more term from $Q(u,{\bf z})$,
i.e. by writing
\begin{equation}
\nabla_{z}\log p(u,{\bf z},0)=-\frac{{\bf z}}{u}+\nabla_{z}G_{1}({\bf z},0)+\bar{Q}(u,{\bf z}),\label{eq:BarQ}
\end{equation}
where $\bar{Q}(u,{\bf z})\triangleq\sum_{k=1}^{\infty}\nabla_{z}G_{k+1}({\bf z},{\bf 0})u^{k}$
satisfies 
\begin{equation}
\big|\bar{Q}(u,{\bf z})\big|\leqslant Cu\ \ \ \forall u\in[0,1],{\bf z}\in\bar{U}\label{eq:UEstBarQ}
\end{equation}
with some geometric constant $C$ depending on the localisation $V.$
Correspondingly, the SDE (\ref{eq:BridgeSDESimplified}) for $X^{t,x,x}$
is rewritten as
\begin{equation}
dX_{s}=\big(-\frac{X_{s}}{t-s}+\bar{b}(X_{s})+\bar{Q}(t-s,X_{s})\big)ds+\sigma(X_{s})dB_{s},\label{eq:SDEBL}
\end{equation}
where 
\begin{equation}
\bar{b}({\bf z})\triangleq b({\bf z})+\nabla_{z}G_{1}({\bf z},{\bf 0}).\label{eq:barb}
\end{equation}

\subsubsection{Decomposition of $\Pi_{t}$}

By using the SDE (\ref{eq:SDEBL}), one can easily compute the semimartingale
decomposition of $\phi(\circ dX^{t,x,x})$ and thus an associated
decomposition of $\Pi_{t}$. This is summarised in the lemma below.
Recall that $\phi=\phi_{i}dx^{i}$ and $a=\sigma\sigma^{T}.$ 
\begin{lem}
\label{lem:PiDecomp}Under the chart $V$, one has
\[
\phi(\circ dX_{tr})=dI_{tr}+dJ_{tr}+dK_{tr}+dL_{tr},
\]
where 
\begin{equation}
dI_{tr}\triangleq-\phi_{i}(X_{tr})\frac{X_{tr}^{i}}{1-r}dr,\ dJ_{tr}\triangleq t\varphi(X_{tr})dr,\label{eq:IJDef}
\end{equation}
\[
dK_{tr}\triangleq\sqrt{t}\phi_{i}(X_{tr})\sigma_{\alpha}^{i}(X_{tr})dB_{r}^{t,\alpha},\ dL_{tr}\triangleq t\bar{Q}(t(1-r),X_{tr})dr
\]
with
\begin{equation}
\varphi({\bf z})\triangleq\phi_{i}({\bf z})\bar{b}^{i}({\bf z})+\frac{1}{2}a^{ij}({\bf z})\partial_{j}\phi_{i}({\bf z})\label{eq:varphi}
\end{equation}
and $B_{r}^{t}\triangleq B_{tr}/\sqrt{t}$ being the rescaled Brownian
motion. In addition, 
\begin{align}
\Pi_{t}^{x} & =\int_{0}^{1}\big(\int_{0}^{\rho}dI_{tr}+dJ_{tr}+dK_{tr}+dL_{tr}\big)\otimes\cdot\big(dI_{t\rho}+dJ_{t\rho}+dK_{t\rho}+dL_{t\rho}\big)\nonumber \\
 & \ \ \ +\frac{t}{2}\int_{0}^{1}\phi_{i}(X_{t\rho})\otimes\phi_{j}(X_{t\rho})a^{ij}(X_{t\rho})d\rho,\label{eq:PiDecomp}
\end{align}
where the ``$\cdot$'' in the above double integral means It\^o's
integral (of course, only the $dK$-term is relevant).
\end{lem}

\begin{proof}
This is explicit It\^o's calculus based on the semimartingale decomposition
(\ref{eq:SDEBL}) together with the change of variables $v=t\rho,u=tr$.
\end{proof}

\subsubsection{Expansions of $\frac{X_{tr}}{1-r}$ and $f(X_{tr})$}

In order to compute $\mathbb{E}[\Pi_{t}^{x}\otimes\Pi_{t}^{x}]$,
in views of the decomposition (\ref{eq:PiDecomp}) one needs to know
suitable expansions of $\frac{X_{tr}}{1-r}$ and $f(X_{tr})$ with
respect to $t$. These are summarised in the lemma below.
\begin{lem}
\label{lem:BasicExp}Under the notation in Lemma \ref{lem:PiDecomp},
let us define 
\[
G_{r}^{t}\triangleq\sqrt{t}\int_{0}^{r}\frac{1}{1-\eta}\sigma(X_{t\eta})dB_{\eta}^{t},\ H_{r}^{t}\triangleq t\int_{0}^{r}\frac{1}{1-\eta}\bar{b}(X_{t\eta})d\eta.
\]
Then one has 
\begin{equation}
\frac{X_{tr}}{1-r}=G_{r}^{t}+H_{r}^{t}+{\cal E}^{t}(r)\label{eq:TypeIExp}
\end{equation}
and for any $f\in C_{b}^{\infty}(\bar{U})$ one has 
\begin{align}
f(X_{tr}) & =f({\bf 0})+\partial_{k}f({\bf 0})\cdot(1-r)\big(G_{r}^{t,k}+H_{r}^{t,k}\big)\nonumber \\
 & \ \ \ \frac{1}{2}\partial_{kl}^{2}f({\bf 0})\cdot(1-r)^{2}G_{r}^{t,k}G_{r}^{l,k}+{\cal E}_{f}^{t}(r).\label{eq:TypeIIExp}
\end{align}
Here the remainders ${\cal E}^{t}(r),{\cal E}_{f}^{t}(r)$ satisfy
the following estimates:
\begin{equation}
\sup_{0\leqslant r\leqslant1}\|{\cal E}^{t}(r)\|_{L^{p}}\leqslant Ct^{2},\ \sup_{0\leqslant r\leqslant1}\|{\cal E}_{f}^{t}(r)\|_{L^{p}}\leqslant Ct^{3/2}\label{eq:ExpError}
\end{equation}
for all $p\geqslant1,$ $t\in(0,1]$. The above constant $C$ depends
on $p,f$ and the localisation $V$. 
\end{lem}

\begin{proof}
The first expansion (\ref{eq:TypeIExp}) follows easily from Lemma
\ref{lem:ExpBridge}; in fact, one explicitly has 
\[
{\cal E}^{t}(r)=t\int_{0}^{r}\frac{1}{1-\eta}\bar{Q}(t(1-\eta),X_{t\eta})d\eta,
\]
where we recall that $\bar{Q}$ is defined through (\ref{eq:BarQ}).
The first estimate in (\ref{eq:ExpError}) follows immediately from
(\ref{eq:UEstBarQ}). The second expansion (\ref{eq:TypeIIExp}) follows
from the standard Taylor approximation of $f$, where 
\begin{align*}
\big|{\cal E}_{f}^{t}(r)\big| & \leqslant C\|f\|_{C_{b}^{3}(\bar{U})}\big[|{\cal E}^{t}(r)|+(1-r)^{2}\big(|G_{r}^{t}|+|H_{r}^{t}|\big)\cdot|H_{r}^{t}|\\
 & \ \ \ +(1-r)^{3}\big(|G_{r}^{t}|+|H_{r}^{t}|+|{\cal E}^{t}(r)|\big)\big].
\end{align*}
Note that 
\begin{align}
(1-r)\|G_{r}^{t}\|_{L^{p}} & =(1-r)\big\|\sqrt{t}\int_{0}^{r}\frac{1}{1-\eta}\sigma(X_{t\eta})dB_{\eta}^{t}\big\|_{L^{p}}\nonumber \\
 & \leqslant C_{p}\sqrt{t}(1-r)\sqrt{\int_{0}^{r}\frac{d\eta}{(1-\eta)^{2}}}\leqslant C_{p}'\sqrt{t}\label{eq:GOrder}
\end{align}
and 
\[
(1-r)|H_{r}^{t}|=(1-r)\big|t\int_{0}^{r}\frac{1}{1-\eta}\bar{b}(X_{t\eta})d\eta\big|\leqslant Ct(1-r)|\log(1-r)|\leqslant C't.
\]
The second estimate in (\ref{eq:ExpError}) thus follows. 
\end{proof}
\begin{rem}
\label{rem:OrderCount}The point of Lemma \ref{lem:BasicExp} is to
collect those specific terms in the expansions of $\frac{X_{tr}}{1-r}$
and $f(X_{tr})$ that have order $\leqslant t$ (this is enough for
our purpose of computing the expansion of $\psi_{4}(t,x)$ up to order
$t^{3}$). Note that the term $G_{r}^{t}$ has order $\sqrt{t}$ and
the term $H_{r}^{t}$ has order $t$. Both ${\cal E}^{t}(r),{\cal E}_{f}^{t}(r)$
are remainders that will be ignored in the asymptotic evaluation.
We also remark that although $G_{r}^{t}$ and $H_{r}^{t}$ contain
singularities as $r\nearrow1$, they will always be eliminated after
performing integration in (\ref{eq:Psi4}).
\end{rem}

\subsubsection{Summary of structure and order-counting}

The computation of the coefficients $\hat{\Theta}_{x},\hat{\Xi}_{x}$
in (\ref{eq:Psi4Exp}) essentially boils down to selecting different
terms in the decomposition (\ref{eq:PiDecomp}) with specific orders
in $t$. For this purpose, it is useful to recapture the basic structure
we already derived. First of all, the decomposition of $\Pi_{t}^{x}\otimes\Pi_{t}^{x}$
has the structure 
\begin{align}
\Pi\otimes\Pi= & \big[\int\big(\int dI+dJ+dK+dL\big)\otimes(dI+dJ+dK+dL)+P\big]\nonumber \\
 & \ \ \ \otimes\big[\int\big(\int dI+dJ+dK+dL\big)\otimes(dI+dJ+dK+dL)+P\big],\label{eq:PiPiExp}
\end{align}
where $I,J,K,L$ are as in Lemma \ref{lem:PiDecomp} and 
\[
P=P_{t}\triangleq\frac{t}{2}\int_{0}^{1}\phi_{i}(X_{t\rho})\otimes\phi_{j}(X_{t\rho})a^{ij}(X_{t\rho})d\rho
\]
is the last term (the It\^o-Stratonovich correction) in (\ref{eq:PiDecomp}). 

The key observation is that the individual terms $I,J,K,L,P$ have
the following $t$-orders respectively:
\begin{equation}
I\sim\sqrt{t},\ J\sim t,\ K\sim\sqrt{t},\ L\sim t^{2},\ P\sim t.\label{eq:IJKLPOrder}
\end{equation}
This is clear from Lemma \ref{lem:PiDecomp} and Lemma \ref{lem:BasicExp}.
As a result, in order to compute $\hat{\Theta}_{x}$ and $\hat{\Xi}_{x}$
one only needs to collect and combine terms in the expansion of (\ref{eq:PiPiExp})
that have total orders of $t^{2}$ and $t^{3}$ respectively. The
principle here is easy and mechanical, but the computation is extremely
huge (we have to rely on computer assistance for computing $\hat{\Xi}_{x}$). 

\subsection{Computation of the $t^{2}$-coefficient $\hat{\Theta}_{x}$ }

In view of (\ref{eq:PiPiExp}) and (\ref{eq:IJKLPOrder}), the $t^{2}$-term
in the expansion of $\psi_{4}(t,x)$ precisely comes from the following
$25$ combinations:
\begin{equation}
\begin{array}{c}
(II;II),(II;IK),(II;KI),(II;KK),\\
(IK;II),(IK;IK),(IK;KI),(IK;KK)\\
(KI;II),(KI;IK),(KI;KI),(KI;KK),\\
(KK;II),(KK;IK),(KK;KI),(KK;KK)\\
(II;P),(IK;P),(KI;P),(KK;P)\\
(P;II),(P;IK),(P;KI),(P;KK),(P;P).
\end{array}\label{eq:t2Cases}
\end{equation}
Here e.g. $(IK;KI)$ means picking $dI\otimes dK$ in the first line
of (\ref{eq:PiPiExp}) (i.e. the first $\Pi$ in $\Pi\otimes\Pi$)
and $dK\otimes dI$ in the second line of (\ref{eq:PiPiExp}) (i.e.
the second $\Pi$ in $\Pi\otimes\Pi$); this produces a term $dI\otimes dK\otimes dK\otimes dI$
of order $\sqrt{t}^{4}=t^{2}$. Similarly, $(P;II)$ means picking
$P$ in the first $\Pi$ and $dI\otimes dI$ in the second $\Pi$;
this produces a term of order $t\times\sqrt{t}^{2}=t^{2}.$ The $(P;P)$
means picking $P$ in both $\Pi$'s which produces a term of order
$t\times t=t^{2}$. 
\begin{notation}
\label{Not:t2Eq}In what follows, we write $X_{t}\stackrel{2}{=}Y_{t}$
for random variables $X_{t},Y_{t}$ if $\mathbb{E}[X_{t}-Y_{t}]=o(t^{2})$
as $t\rightarrow0^{+}.$ In fact, in our problem it will always be
the case that $|\mathbb{E}[X_{t}-Y_{t}]|\leqslant Ct^{3}$ (one does
not see the $t^{5/2}$-term because the expectation of an It\^o integral
is zero) and $\|X_{t}-Y_{t}\|_{L^{p}}\leqslant C_{p}t^{5/2}$. This
notation is used to only keep track of terms up to the desired order
$t^{2}$ and ignore all higher order terms in the expansion. 
\end{notation}

\subsubsection{The $(II;II)$ term}

Here we demonstrate the computation of the $(II;II)$-term in detail.
According to Lemma \ref{lem:PiDecomp} (the expression of $dI_{tr}$),
the $t^{2}$-coefficient of this particular term comes from the expectation
of 
\begin{align}
A_{II;II} & \triangleq\int_{0}^{1}\big(\int_{0}^{\rho}\phi_{i}(X_{tr})\frac{X_{tr}^{i}}{1-r}dr\big)\otimes\phi_{j}(X_{t\rho})\frac{X_{t\rho}^{j}}{1-\rho}d\rho\nonumber \\
 & \ \ \ \otimes\int_{0}^{1}\big(\int_{0}^{\theta}\phi_{k}(X_{t\delta})\frac{X_{t\delta}^{k}}{1-\delta}d\delta\big)\otimes\phi_{l}(W_{t\theta})\frac{X_{t\theta}^{l}}{1-\theta}d\theta.\label{eq:A1111}
\end{align}
Since we are extracting terms of order $t^{2}$, in the expansions
(\ref{eq:TypeIExp}, \ref{eq:TypeIIExp}) for the $\phi$'s and $X$'s
the only possibility is freezing all the $\phi$'s at the origin and
taking the $G^{t}$-term for the $X$'s (recall from (\ref{eq:GOrder})
and Remark \ref{rem:OrderCount} that $G^{t}$ is of order $\sqrt{t}$).
Namely, using Notation \ref{Not:t2Eq} one has 
\begin{align}
A_{II;II} & \stackrel{2}{=}\phi_{i}({\bf 0})\otimes\phi_{j}({\bf 0})\otimes\phi_{k}({\bf 0})\otimes\phi_{l}({\bf 0})\nonumber \\
 & \ \ \ \times\int_{0}^{1}\big(\int_{0}^{\rho}G_{r}^{t,i}dr\big)G_{\rho}^{t,j}d\rho\times\int_{0}^{1}\big(\int_{0}^{\theta}G_{\delta}^{t,k}d\delta\big)G_{\theta}^{t,l}d\theta\nonumber \\
 & =\phi_{i}({\bf 0})\otimes\phi_{j}({\bf 0})\otimes\phi_{k}({\bf 0})\otimes\phi_{l}({\bf 0})\nonumber \\
 & \ \ \ \times\int_{\substack{0<r<\rho<1\\
0<\delta<\theta<1
}
}G_{r}^{t,i}G_{\rho}^{t,j}G_{\delta}^{t,k}G_{\theta}^{t,l}drd\rho d\delta d\theta.\label{eq:4IntA1111}
\end{align}
To compute the $t^{2}$-coefficient of $\mathbb{E}[A_{II;II}]$, one
has to evaluate $\mathbb{E}\big[G_{r}^{t,i}G_{\rho}^{t,j}G_{\delta}^{t,k}G_{\theta}^{t,l}\big].$
This is contained in the lemma below.
\begin{lem}
\label{lem:UniversalF}Consider the function 
\[
F(a,b,c,d;i,j,k,l)\triangleq\mathbb{E}\big[G_{a}^{t,i}G_{b}^{t,j}G_{c}^{t,k}G_{d}^{t,l}\big]
\]
where $0\leqslant a\leqslant b\leqslant c\leqslant d\leqslant1$ and
$i,j,k,l$ are arbitrary coordinate indices. Then one has 
\begin{align}
F(a,b,c,d;i,j,k,l) & =t^{2}\big[\frac{ab}{(1-a)(1-b)}\big(\delta^{ij}\delta^{kl}+\delta^{ik}\delta^{jl}+\delta^{il}\delta^{jk}\big)\nonumber \\
 & \ \ \ +\frac{a}{1-a}\frac{c-b}{(1-b)(1-c)}\delta^{ij}\delta^{kl}\big]+o(t^{2})\ \ \ \text{as }t\rightarrow0^{+}.\label{eq:UniversalF}
\end{align}
\end{lem}

\begin{proof}
Since we are only considering the $t^{2}$-coefficient and $\sigma({\bf 0})={\rm Id}$
(cf. (\ref{eq:abInitial}) and also note that $a=\sigma^{T}\sigma$),
one has 
\begin{align*}
G_{a}^{t,i}G_{b}^{t,j}G_{c}^{t,k}G_{d}^{t,l}\stackrel{2}{=} & \sqrt{t}^{4}\big(\int_{0}^{a}\frac{1}{1-\eta}dB_{\eta}^{t,i}\big)\big(\int_{0}^{b}\frac{1}{1-\eta}dB_{\eta}^{t,j}\big)\\
 & \ \ \ \times\big(\int_{0}^{c}\frac{1}{1-\eta}dB_{\eta}^{t,k}\big)\big(\int_{0}^{d}\frac{1}{1-\eta}dB_{\eta}^{t,l}\big).
\end{align*}
The right hand side is just the product of four Gaussian variables
whose covariance function is easily calculated from the relation that
\[
\mathbb{E}\big[\big(\int_{0}^{\rho}\frac{1}{1-\eta}dB_{\eta}^{t,p}\big)\big(\int_{0}^{\theta}\frac{1}{1-\eta}dB_{\eta}^{t,q}\big)\big]=\delta^{pq}\int_{0}^{\rho\wedge\theta}\frac{d\eta}{(1-\eta)^{2}}=\frac{\rho\wedge\theta}{1-\rho\wedge\theta}\delta^{pq}.
\]
The result (\ref{eq:UniversalF}) follows immediately from the standard
formula for $4$-th moments of Gaussian vectors. 
\end{proof}
\begin{lem}
One has 
\begin{align}
\mathbb{E}[A_{II;II}] & =t^{2}\phi_{i}({\bf 0})\otimes\phi_{j}({\bf 0})\otimes\phi_{k}({\bf 0})\otimes\phi_{l}({\bf 0})\nonumber \\
 & \ \ \ \times\big(\frac{1}{4}\delta^{ij}\delta^{kl}+\frac{1}{3}\delta^{ik}\delta^{jl}+\frac{1}{6}\delta^{il}\delta^{jk}\big)+o(t^{2})\ \ \ \text{as }t\rightarrow0^{+}.\label{eq:A1111Coef}
\end{align}
\end{lem}

\begin{proof}
This is obtained by further decomposing the integral (\ref{eq:4IntA1111})
according to the actual orderings of $(r,\rho,\delta,\theta)$ and
applying Lemma \ref{lem:UniversalF} to each scenario. For instance,
the integral over the region $r<\delta<\theta<\rho$ gives 
\begin{align}
 & \int_{0<r<\delta<\theta<\rho<1}F(r,\delta,\theta,\rho;i,k,l,j)\nonumber \\
 & =t^{2}\int_{0<r<\delta<\theta<\rho<1}\big(\frac{r\delta}{(1-r)(1-\delta)}\epsilon^{ijkl}+\frac{r}{1-r}\frac{\theta-\delta}{(1-\delta)(1-\theta)}\delta^{ik}\delta^{jl}\big)+o(t^{2})\nonumber \\
 & =t^{2}\big(\frac{1}{36}\epsilon^{ijkl}+\frac{1}{24}\delta^{ik}\delta^{jl}\big)+o(t^{2}),\label{eq:A1111OneOrder}
\end{align}
where $\epsilon^{ijkl}\triangleq\delta^{ij}\delta^{kl}+\delta^{ik}\delta^{jl}+\delta^{il}\delta^{jk}.$
Other scenarios corresponding to different orderings of $(r,\rho,\delta,\theta)$
are obtained from (\ref{eq:A1111OneOrder}) by symmetry. Combining
and simplifying all these expressions gives the relation (\ref{eq:A1111Coef}).
\end{proof}

\subsubsection{Final result}

All the remaining $24$ cases in (\ref{eq:t2Cases}) are treated by
very similar kind of calculations (elementary It\^o-calculus). We
will not present the separate results here and leave them to Appendix
\ref{sec:t2}. The final result is summarised as follows. To ease
notation, we denote 
\begin{equation}
\phi_{ijkl}\triangleq\phi_{i}({\bf 0})\otimes\phi_{j}({\bf 0})\otimes\phi_{k}({\bf 0})\otimes\phi_{l}({\bf 0}).\label{eq:PhiijklNot}
\end{equation}
 
\begin{prop}
\label{prop:t2Coef}The $t^{2}$-coefficient $\hat{\Theta}_{x}$ in
the expansion of (\ref{eq:Psi4Exp}) is given by 
\[
\hat{\Theta}_{x}=\frac{1}{24}(\phi_{ijij}-\phi_{ijji})\in E^{\otimes4}.
\]
\end{prop}

\begin{proof}
By summing up the expressions (\ref{eq:t21}--\ref{eq:t214}) in
Appendix \ref{sec:t2} as well as (\ref{eq:A1111Coef}), one finds
that 
\begin{equation}
\mathbb{E}[\Pi_{t}^{x}\otimes\Pi_{t}^{x}]=\frac{t^{2}}{12}(\phi_{ijij}-\phi_{ijji})+o(t^{2}).\label{eq:PPt2Coef}
\end{equation}
The result thus follows from the relation (\ref{eq:Psi4}).
\end{proof}
\begin{rem}
\label{rem:IntrIntt2}Although $\hat{\Theta}_{x}$ is computed in
terms of local coordinates, it is apparently an intrinsic quantity
(since $\psi_{4}(t,x)$ is). Its intrinsic meaning is described as
follows. Since $\phi$ is an $E$-valued one-form on $M$, $\phi(x)\otimes\phi(x)\otimes\phi(x)\otimes\phi(x)$
can be viewed as an $E^{\otimes4}$-valued $4$-tensor on $T_{x}M$.
Consider the $E^{\otimes4}$-valued bilinear form on $(T_{x}M)^{\otimes2}$
defined by
\begin{align*}
{\cal T}_{x}: & (T_{x}M)^{\otimes2}\times(T_{x}M)^{\otimes2}\rightarrow E^{\otimes4}\\
 & (u\otimes v,w\otimes z)\mapsto\langle\phi(x),u\rangle\otimes\langle\phi(x),v\rangle\otimes[\langle\phi(x),w\rangle,\langle\phi(x),z\rangle],
\end{align*}
where $\langle\cdot,\cdot\rangle$ denotes the cotangent-tangent pairing
and $[a,b]\triangleq a\otimes b-b\times a$ for $a,b\in E$. Then
$\hat{\Theta}_{x}$ is the trace of ${\cal T}_{x}$ with respect to
the Hilbert-Schmidt structure on $(T_{x}M)^{\otimes2}$ induced from
the Riemannian metric on $T_{x}M$. The geometric significance of
$\hat{\Theta}_{x}$ will become clearer in the case when $\phi=dF$ and
after applying the contraction operator $\mathfrak{C}$ defined by
(\ref{eq:24Trace}) (in fact, $\hat{\Theta}_{x}$ encodes the entire
Riemannian metric tensor at $x$; cf. equation (\ref{eq:t2Coef})
and Remark \ref{rem:RecMetric} below).
\end{rem}

\subsection{Computation of the $t^{3}$-coefficient $\hat{\Xi}_{x}$}

From Proposition \ref{prop:t2Coef} and Remark \ref{rem:IntrIntt2},
the $4$-tensor $\hat{\Theta}_{x}$ encodes information about the
metric tensor at the location $x$ (the first fundamental form at
$x$). However, it does not encode any curvature properties as the
expression does not involve differentiation of the metric tensor.
To recover curvature quantities, one has to look at higher order terms
in the expansion. 

The basic principle of calculating higher order coefficients in the
$t$-expansion of $\psi_{4}(t,x)$ is the same as the computation
of $\hat{\Theta}_{x}$, which is in turn based on the decomposition
(\ref{eq:PiPiExp}) and order counting. However, the calculation becomes
much lengthier and at some point unmanageable by hand (we thus have
to rely on computer assistance). It is quite straight forward that
the $t^{5/2}$-coefficient of $\mathbb{E}[\Pi_{t}^{x}\otimes\Pi_{t}^{x}]$
is exactly zero for the obvious reason that the expected value of
an It\^o integral is zero. Therefore, our next task is to compute
the $t^{3}$-coefficient $\hat{\Xi}_{x}$ of the expansion of $\psi_{4}(t,x)$.
As we will see, $\hat{\Xi}_{x}$ encodes an interesting amount of
explicit information about both intrinsic (Ricci) and extrinsic (second
fundamental form) curvature properties at $x$. 

In the decomposition (\ref{eq:PiPiExp}) of $\Pi_{t}^{x}\otimes\Pi_{t}^{x}$,
we recall that the (lowest) $t$-orders of the terms $I,J,K,L,P$
are given by (\ref{eq:IJKLPOrder}) respectively. Accordingly, we
just denote their degrees to be 
\[
\deg I=\deg K=0.5,\ \deg J=\deg P=1,\ \deg L=2.
\]
Note that each of them also contain higher order terms; for instance
one can further expand $\phi_{i}(X_{tr})$ in $I$ (cf. (\ref{eq:IJDef}))
by using (\ref{eq:TypeIIExp}) to get terms in $I$ that have orders
$t,t^{3/2},\cdots.$ When one expands the product from (\ref{eq:PiPiExp}),
in order to extract the $t^{3}$-term in the asymptotic expansion
there are three main scenarios to consider:
\begin{enumerate}
\item \uline{Combinations with total degree \mbox{$3$}}. In this case,
one extracts leading coefficients and there is no need to further
expand any of the individual terms. For instance, consider the combination
$JI;P$. This means picking $dJ\otimes dI$ in the first line of (\ref{eq:PiPiExp})
and $P$ in the second. The resulting total degree is $1+0.5+1=3$.
Therefore, one only needs to compute the leading coefficient in the
corresponding expectation without further expanding any of the $I,J,K$'s.
\item \uline{Combinations with total degree \mbox{$2.5$}}. In this case,
one needs to further extract the $\sqrt{t}$-term in the resulting
product (because an extra order of $\sqrt{t}$ is needed to yield
a total order of $t^{3}$). For instance, consider the combination
$JI;IK$ whose total degree is $2.5$. One then needs to compute its
$\sqrt{t}$-order term which can come from a further $\sqrt{t}$-expansion
from any of the $I,J,K$'s; e.g. taking the $\sqrt{t}$-order term
\[
\partial_{k}\varphi({\bf 0})\cdot(1-r)\big(G_{r}^{t,k}+H_{r}^{t,k}\big)
\]
in the expansion of $\varphi(X_{tr})$ in $J$ (cf. (\ref{eq:IJDef}),
(\ref{eq:varphi}) and (\ref{eq:TypeIIExp})) will yield such a term
and there are several other possibilities. 
\item \uline{Combinations with total degree \mbox{$2$}}. In this case,
one needs to further extract the $t$-term in the resulting product.
For instance, consider the combination $II;IK$ whose total degree
is $2$. One then needs to compute its $t$-order term which can come
from many possibilities, e.g. by expanding the first $I$ to order
$t^{3/2}$ (note that $I$ itself already has order $\sqrt{t}$) or
by expanding the second $I$ and last $K$ each to the order of $t$
but there are many other possibilities as well. 
\end{enumerate}
In view of the above three scenarios, we aim at obtaining an expansion
of the form 
\begin{equation}
\mathbb{E}[\Pi_{t}^{x}\otimes\Pi_{t}^{x}]=\frac{t^{2}}{12}(\phi_{ijij}-\phi_{ijji})+({\cal S}_{1}+{\cal S}_{2}+{\cal S}_{3})t^{3}+o(t^{3}),\label{eq:PPt23Exp}
\end{equation}
where the $t^{2}$-coefficient in (\ref{eq:PPt23Exp}) was derived
in (\ref{eq:PPt2Coef}) and ${\cal S}_{i}$ is the total coefficient
to be computed under the above $i$-th scenario ($i=1,2,3$). We now
proceed to derive the explicit formulae for ${\cal S}_{1},{\cal S}_{2},{\cal S}_{3}$.
Since the analysis is quite elementary and mechanical, as before we
will only discuss the details in one representative combination in
each scenario and leave the case-by-case discussion to the appendix.
We first introduce two convenient sets of notation.
\begin{notation}
\label{Not:t3Eq}We write $X_{t}\stackrel{3}{=}Y_{t}$ for random
variables $X_{t},Y_{t}$ if $\mathbb{E}[X_{t}]$ and $\mathbb{E}[Y_{t}]$
have the same $t^{3}$-terms in their $t$-expansions as $t\rightarrow0^{+}.$ 
\end{notation}

\begin{notation}
\label{not:PhiNot}To state the final results in a more compact form,
we will introduce notation like
\begin{align*}
\phi_{ijkl} & \triangleq\phi_{i}({\bf 0})\otimes\phi_{j}({\bf 0})\otimes\phi_{k}({\bf 0})\otimes\phi_{l}({\bf 0}),\\
\phi_{i,p|jk|l,q} & \triangleq\partial_{p}\phi_{i}({\bf 0})\otimes\phi_{j}({\bf 0})\otimes\phi_{k}({\bf 0})\otimes\partial_{q}\phi_{l}({\bf 0})\\
\phi_{ij|k,pq|l} & \triangleq\phi_{i}({\bf 0})\otimes\phi_{j}({\bf 0})\otimes\partial_{pq}^{2}\phi_{k}({\bf 0})\otimes\phi_{l}({\bf 0})\ \text{etc.}
\end{align*}
This kind of notation (without the derivatives) was already used in
the statement of Proposition \ref{prop:t2Coef} before. We will also
write $\partial_{i}f\triangleq\partial_{i}f({\bf 0}).$
\end{notation}

\subsubsection{Total degree $=3$ }

All possible combinations in the decomposition (\ref{eq:PiPiExp})
for this scenario is listed below: 
\footnotesize
\begin{equation}
\begin{array}{c}
(JI;JI),(IJ;JI),(JI;IJ),(IJ;IJ),(JI;JK),(IJ;JK),\\
(JI;KJ),(IJ;KJ),(JK;JI),(KJ;JI),(JK;IJ),(KJ;IJ),\\
(JK;JK),(KJ;JK),(JK;KJ),(KJ;KJ),(JJ;II),(JJ;IK),\\
(JJ;KI),(JJ;KK),(II;JJ),(IK;JJ),(KI;JJ),(KK;JJ)(JJ;P),(P;JJ).
\end{array}\label{eq:D3List}
\end{equation}
\normalsize

\subsubsection*{A representative case: $JI;JI$}

As a representative case, let us consider the combination $JI;JI$
in (\ref{eq:D3List}). The resulting product is 
\begin{align*}
B_{JI;JI} & \triangleq\int_{0}^{1}\big(t\int_{0}^{\rho}\varphi(X_{tr})dr\big)\otimes\phi_{i}(X_{t\rho})\frac{X_{t\rho}^{i}}{1-\rho}d\rho\\
 & \ \ \ \otimes\int_{0}^{1}\big(t\int_{0}^{\theta}\varphi(X_{t\delta})d\delta\big)\otimes\phi_{j}(X_{t\theta})\frac{X_{t\theta}^{i}}{1-\theta}d\theta,
\end{align*}where $\varphi$ is the function defined by (\ref{eq:varphi}).
Since $B_{JI;JI}$ already has order $t^{3}$, to extract its coefficient
the only possibility is freezing $\varphi,\phi_{i},\varphi,\phi_{j}$
at the origin and replacing $\frac{X_{t\rho}^{i}}{1-\rho},\frac{X_{t\theta}^{j}}{1-\theta}$
by their leading terms $G_{\rho}^{t,i},G_{\theta}^{t,j}$ from the
expansion (\ref{eq:TypeIExp}). This leads to 
\begin{align*}
B_{JI;JI} & \stackrel{3}{=}\varphi({\bf 0})\otimes\phi_{i}({\bf 0})\otimes\varphi({\bf 0})\otimes\phi_{j}({\bf 0})\\
 & \ \ \ \times\int_{0}^{1}\big(t\int_{0}^{\rho}dr\big)G_{\rho}^{t,i}d\rho\times\int_{0}^{1}\big(t\int_{0}^{\theta}d\delta\big)G_{\theta}^{t,j}d\theta.
\end{align*}
It is straight forward to check that 
\begin{equation}
\mathbb{E}[G_{\rho}^{t,i}G_{\theta}^{t,j}]=\frac{\rho\wedge\theta}{1-\rho\wedge\theta}\delta^{ij}t+o(t)\ \ \ \text{as }t\rightarrow0^{+}.\label{eq:2ProdG}
\end{equation}
Therefore, one has
\begin{align*}
 & \mathbb{E}\big[\int_{0}^{1}\big(t\int_{0}^{\rho}dr\big)G_{\rho}^{t,i}d\rho\times\int_{0}^{1}\big(t\int_{0}^{\theta}d\delta\big)G_{\theta}^{t,j}d\theta\big]\\
 & =t^{2}\int_{[0,1]^{2}}\rho\theta\mathbb{E}[G_{\rho}^{t,i}G_{\theta}^{t,j}]d\rho d\theta=\frac{7}{12}\delta^{ij}t^{3}+o(t^{3}).
\end{align*}
In view of the definition (\ref{eq:varphi}) of $\varphi$ and Lemma
\ref{lem:abGNormChart}, one also has $\varphi({\bf 0})=1/2\partial_{k}\phi_{k}.$
As a consequence, with Notation \ref{not:PhiNot} in mind one finds
that 
\[
\mathbb{E}[B_{JI;JI}]=\frac{7t^{3}}{48}\phi_{i,i|j|k,k|j}+o(t^{3}).
\]

\subsubsection*{Summary of result }

The results for all other cases in (\ref{eq:D3List}) are summarised
in Appendix \ref{sec:t3}. By adding up all these expressions, one arrives
at the following formula. 
\begin{lem}
\label{lem:Total3}The total $t^{3}$-coefficient ${\cal S}_{1}$
(cf. (\ref{eq:PPt23Exp}) for the notation) of $\mathbb{E}[\Pi_{t}^{x}\otimes\Pi_{t}^{x}]$
coming from all the $26$ cases in (\ref{eq:D3List}) is equal to
\[
{\cal S}_{1}=\frac{1}{48}\phi_{i,i|j|k,k|j}+{\color{red}{\color{black}\frac{1}{48}}}{\color{black}\phi_{i|j,j|i|k,k}-\frac{1}{48}\phi_{i,i|jj|k,k}-\frac{1}{48}}\phi_{i|j,j|k,k|i}.
\]
\end{lem}

\subsubsection{Total degree $=2.5$}

All possible combinations in the decomposition (\ref{eq:PiPiExp})
for this scenario include the following $20$ cases \begin{equation}
\begin{array}{c}
(JI;P),(IJ;P),(JK;P),(KJ;P),\\
(JI;II),(IJ;II),(JK;II),(KJ;II),\\
(JI;IK),(IJ;IK),(JK;IK),(KJ;IK),\\
(JI;KI),(IJ;KI),(JK;KI),(KJ;KI),\\
(JI;KK),(IJ;KK),(JK;KK),(KJ;KK),
\end{array}\label{eq:D2.5List}
\end{equation}
as well as the corresponding $20$ cases obtained by swapping the $(1,2)$ and $(3,4)$ tensor slots for each case in (\ref{eq:D2.5List}) (e.g. $(P;JI)$ etc.)

\subsubsection*{A representative case: $JI;P$}

We consider one representative case from (\ref{eq:D2.5List}): the
combination $JI;P$. The resulting product is 
\begin{align*}
C_{JI;P} & \triangleq-t^{2}\int_{0}^{1}\big(\int_{0}^{\rho}\varphi(X_{tr})dr\big)\otimes\phi_{i}(X_{t\rho})\frac{X_{t\rho}^{i}}{1-\rho}d\rho\\
 & \ \ \ \otimes\frac{1}{2}\int_{0}^{1}\phi_{k}(X_{t\theta})\otimes\phi_{l}(X_{t\theta})a^{kl}(X_{t\theta})d\theta.
\end{align*}
The order of $C_{JI;P}$ is $t^{2.5}$ and the extra order of $\sqrt{t}$
comes from exactly one of the following expansions:

\vspace{2mm}\noindent (i) \uline{Take the \mbox{$\sqrt{t}$}-order
term (namely, \mbox{$\partial_{p}\varphi({\bf 0})(1-r)G_{r}^{t,p}$})
in the expansion of \mbox{$\varphi(X_{tr})$} (and the leading order
term in each of the remaining terms in their expansions, namely \mbox{$\phi_{i}({\bf 0}),$}
\mbox{$G_{\rho}^{t,i}$}, \mbox{$\phi_{k}({\bf 0})$}, \mbox{$\phi_{l}({\bf 0})$}
and \mbox{$a^{kl}({\bf 0})$} respectively).}

\vspace{2mm} According to Lemma \ref{lem:BasicExp}, the resulting
integral for this case is given by 
\begin{align*}
C_{JI;P}^{(1)} & \triangleq-\frac{1}{2}t^{2}\partial_{p}\varphi({\bf 0})\otimes\phi_{i}({\bf 0})\otimes\phi_{k}({\bf 0})\otimes\phi_{l}({\bf 0})\delta^{kl}\\
 & \ \ \ \times\int_{0}^{1}\big(\int_{0}^{\rho}(1-r)G_{r}^{t,p}dr\big)G_{\rho}^{t,i}d\rho\times\int_{0}^{1}d\theta.
\end{align*}
For $r<\rho,$ we recall from the relation (\ref{eq:2ProdG}) that
\[
\mathbb{E}\big[G_{r}^{t,p}G_{\rho}^{t,i}\big]=\frac{tr}{1-r}\delta^{pi}+o(t).
\]
As a result, one has 
\begin{align*}
\mathbb{E}\big[C_{JI;P}^{(1)}\big] & =-\frac{1}{2}t^{2}\partial_{p}\varphi({\bf 0})\otimes\phi_{i}({\bf 0})\otimes\phi_{k}({\bf 0})\otimes\phi_{l}({\bf 0})\delta^{kl}\\
 & \ \ \ \times\int_{0<r<\rho<1}(1-r)\big(\frac{tr}{1-r}\delta^{pi}+o(t)\big)drd\rho\\
 & =-\frac{t^{3}}{12}\partial_{i}\varphi({\bf 0})\otimes\phi_{i}({\bf 0})\otimes\phi_{k}({\bf 0})\otimes\phi_{k}({\bf 0})+o(t^{3}).
\end{align*}
On the other hand, it is seen from Lemma \ref{lem:abGNormChart} that
\[
\partial_{i}\varphi({\bf 0})=\partial_{i}\bar{b}^{j}({\bf 0})\phi_{j}({\bf 0})+\frac{1}{2}\partial_{ij}^{2}\phi_{j}({\bf 0}).
\]
By using Notation \ref{not:PhiNot}, one obtains that 
\[
\mathbb{E}\big[C_{JI;P}^{(1)}\big]=\big(-\frac{1}{12}\partial_{i}\bar{b}^{j}\phi_{jikk}-\frac{1}{24}\phi_{j,ij|ikk}\big)t^{3}+o(t^{3}).
\]
(ii) \uline{Take the \mbox{$\sqrt{t}$}-order term in \mbox{$\phi_{i}(X_{t\rho})$}.}

\vspace{2mm} Similar to Case (i), the resulting integral is 
\begin{align*}
C_{JI;P}^{(2)} & \triangleq-\frac{1}{2}t^{2}\varphi({\bf 0})\otimes\partial_{p}\phi_{i}({\bf 0})\otimes\phi_{k}({\bf 0})\otimes\phi_{k}({\bf 0})\\
 & \ \ \ \times\int_{0}^{1}\big(\int_{0}^{\rho}dr\big)(1-\rho)G_{\rho}^{t,p}G_{\rho}^{t,i}d\rho\times\int_{0}^{1}d\theta.
\end{align*}
As before, with the relation $\varphi({\bf 0})=1/2\partial_{j}\phi_{j}({\bf 0})$
one computes that 
\begin{align*}
\mathbb{E}\big[C_{JI;P}^{(2)}\big] & =-\frac{1}{2}t^{2}\varphi({\bf 0})\otimes\partial_{p}\phi_{i}({\bf 0})\otimes\phi_{k}({\bf 0})\otimes\phi_{k}({\bf 0})\\
 & \ \ \ \times\int_{0}^{1}\rho(1-\rho)\big(\frac{t\rho}{1-\rho}\delta^{pi}+o(t)\big)d\rho\times1\\
 & =-\frac{t^{3}}{12}\phi_{i,i|j,j|kk}+o(t^{3}).
\end{align*}
(iii) \uline{Take the \mbox{$t$}-order term in \mbox{$\frac{X_{t\rho}^{i}}{1-\rho}$}
(note that the leading term of \mbox{$\frac{X_{t\rho}}{1-\rho}$}
is \mbox{$G_{\rho}^{t}$} which has order \mbox{$\sqrt{t}$}).}

\vspace{2mm} In the expansion (\ref{eq:TypeIExp}) of $\frac{X_{t\rho}}{1-\rho}$,
it is not hard to see from Lemma \ref{lem:abGNormChart} that the
term $H_{\rho}^{t}$ actually has order $t^{3/2}$ (because $\bar{b}({\bf 0})=0$).
As a result, this particular case gives zero contribution: $\mathbb{E}[C_{JI;P}^{(3)}]=o(t^{3}).$ 

\vspace{2mm}\noindent (iv) \uline{Take the \mbox{$\sqrt{t}$}-order
term in either \mbox{$\Phi(X_{t\theta})\triangleq\phi_{k}(X_{t\theta})\otimes\phi_{l}(X_{t\theta})a^{kl}(X_{t\theta})$}.}

\vspace{2mm}The resulting integral for this case is 
\begin{align*}
C_{JI;P}^{(4)} & \triangleq-\frac{1}{2}t^{2}\varphi({\bf 0})\otimes\phi_{i}({\bf 0})\otimes\partial_{p}\Phi({\bf 0})\\
 & \ \ \ \times\int_{0}^{1}\big(\int_{0}^{\rho}dr\big)G_{\rho}^{t,i}d\rho\times\int_{0}^{1}(1-\theta)G_{\theta}^{t,p}d\theta.
\end{align*}
According to the relation (\ref{eq:2ProdG}) and Lemma \ref{lem:abGNormChart},
its expectation is given by 
\begin{align*}
\mathbb{E}\big[C_{JI;P}^{(4)}\big] & =-\frac{1}{2}t^{2}\times\frac{1}{2}\partial_{j}\phi_{j}({\bf 0})\otimes\phi_{i}({\bf 0})\otimes\big(\partial_{p}\phi_{k}({\bf 0})\otimes\phi_{k}({\bf 0})+\phi_{k}({\bf 0})\otimes\partial_{p}\phi_{k}({\bf 0})\big)\\
 & \ \ \ \times\int_{0}^{1}\int_{0}^{1}\rho(1-\theta)\big(\frac{\rho\wedge\theta}{1-\rho\wedge\theta}t\delta^{ip}+o(t)\big)d\rho d\theta\\
 & =-\frac{t^{3}}{24}\big(\phi_{j,j|i|k,i|k}+\phi_{j,j|i|k|k,i}\big)+o(t^{3}).
\end{align*}
(v) \uline{First take the leading parts of all terms and then expand
the expectation to the order of \mbox{$\sqrt{t}$}. }

\vspace{2mm} This means taking the $\sqrt{t}$-order term for the
function 
\begin{align*}
t & \mapsto-\frac{1}{2}t^{2}\varphi({\bf 0})\otimes\phi_{i}({\bf 0})\otimes\phi_{k}({\bf 0})\otimes\phi_{l}({\bf 0})\delta^{kl}\\
 & \ \ \ \times\int_{0}^{1}\big(\int_{0}^{\rho}dr\big)\mathbb{E}\big[G_{\rho}^{t,i}\big]d\rho\times\int_{0}^{1}d\theta.
\end{align*}
But this function is identically zero since $G_{\rho}^{t}$ is an
It\^o integral. Therefore, this particular case gives zero contribution:
$\mathbb{E}[C_{JI;P}^{(5)}]=0.$

\vspace{2mm}\noindent To summarise, by adding up the above results
one concludes that the $t^{3}$-coefficient of $\mathbb{E}[C_{JI;P}]$
is given by 
\[
-\frac{1}{12}\partial_{i}\bar{b}^{j}\phi_{jikk}-\frac{1}{24}\phi_{j,ij|ikk}-\frac{1}{12}\phi_{i,i|j,j|kk}-\frac{1}{24}\phi_{j,j|i|k,i|k}-\frac{1}{24}\phi_{j,j|i|k|k,i}.
\]

\subsubsection*{Summary of result}

The results for all other cases in (\ref{eq:D2.5List}) are summarised
in Appendix \ref{sec:t3}. By adding up all these expressions, one arrives
at the following formula. 
\begin{lem}
\label{lem:Total2.5}The total $t^{3}$-coefficient ${\cal S}_{2}$
(cf. (\ref{eq:PPt23Exp}) for the notation) of $\mathbb{E}[\Pi_{t}^{x}\otimes\Pi_{t}^{x}]$
coming from all the $40$ cases in (\ref{eq:D2.5List}) is equal to
\begin{align*}
{\cal S}_{2} & =-\frac{1}{48}\phi_{i,i|j,k|jk}+\frac{1}{48}\phi_{i,i|j,k|kj}-\frac{1}{48}\phi_{i,j|k,k|ij}+\frac{1}{48}\phi_{i,j|k,k|ji}\\
 & \ \ \ -\frac{1}{48}\phi_{i,i|j|k,k|j}+\frac{1}{48}\phi_{i,i|jj|k,k}+\frac{1}{48}\phi_{i|j,j|k,k|i}-\frac{1}{48}\phi_{i|j,j|i|k,k}\\
 & \ \ \ -\frac{1}{48}\phi_{ij|i,j|k,k}+\frac{1}{48}\phi_{ij|j,i|k,k}-\frac{1}{48}\phi_{ij|k,k|i,j}+\frac{1}{48}\phi_{ij|k,k|j,i}.
\end{align*}
\end{lem}

\subsubsection{Total degree $=2$}

All possible combinations in the decomposition (\ref{eq:PiPiExp})
for this scenario is listed below:
\footnotesize{
\begin{equation}
\begin{array}{c}
(II;II),(II;IK),(II;KI),(II;KK),\ (IK;II),(IK;IK),(IK;KI),(IK;KK),\\
(KI;II),(KI;IK),(KI;KI),(KI;KK),\ (KK;II),(KK;IK),(KK;KI),(KK;KK),\\
(II;P),(IK;P),(KI;P),(KK;P),\ (P;II),(P;IK),(P;KI),(P;KK),\ (PP).
\end{array}\label{eq:D2List}
\end{equation}
}
\normalsize 

\noindent We consider a representative case given by the combination of $II;II$
in the decomposition (\ref{eq:PiPiExp}). The resulting product is
\begin{align*}
D_{II;II} & =\int_{0}^{1}\big(\int_{0}^{\rho}\phi_{i}(X_{tr})\frac{X_{tr}^{i}}{1-r}dr\big)\otimes\phi_{j}(X_{t\rho})\frac{X_{t\rho}^{j}}{1-\rho}d\rho\\
 & \ \ \ \otimes\int_{0}^{1}\big(\int_{0}^{\theta}\phi_{k}(X_{t\delta})\frac{X_{t\delta}^{k}}{1-\delta}d\delta\big)\otimes\phi_{l}(X_{t\theta})\frac{X_{t\theta}^{l}}{1-\theta}d\theta.
\end{align*}
The order of $D_{II;II}$ is $t^{2}$. There are five possible ways
of expansion to produce a $t^{3}$-term:

\vspace{2mm}\noindent (i) Take the $t$-order term in exactly one
of the $\phi$'s (and the leading order terms for all remaining terms).\\
(ii) Take the $\sqrt{t}$-order terms in exactly two of the $\phi$'s.\\
(iii) Take the $t^{3/2}$-order term in one of the $\frac{W_{tr}}{1-r}$'s.\\
(iv) First only take the $\sqrt{t}$-order term in exactly one of
the $\phi$'s, then expand the expectation of the resulting integral
to an extra order of $\sqrt{t}$.\\
(v) First take the leading terms for all terms and then expand the
expectation of the resulting integral to order $t$. 

\vspace{2mm}\noindent The computation here is essentially the same
as before. However, the complexity gets substantially higher and we
have to rely on computer-assistance (also for all other cases in (\ref{eq:D2List})).
We directly present the final result for the current scenario; the
case-by-case computations are given in Appendix \ref{sec:t3}.
\begin{lem}
\label{lem:Total2}The total $t^{3}$-coefficient ${\cal S}_{3}$
(cf. (\ref{eq:PPt23Exp}) for the notation) of $\mathbb{E}[\Pi_{t}^{x}\otimes\Pi_{t}^{x}]$
coming from all the $25$ cases in (\ref{eq:D2List}) is equal to
\[
{\cal S}_{3}={\cal S}_{3}^{(1)}+{\cal S}_{3}^{(2)}+{\cal S}_{3}^{(3)}+{\cal S}_{3}^{(4)},
\]
where the above four quantities are defined by the following expressions
respectively:
\begin{align*}
{\cal S}_{3}^{(1)}\triangleq & \frac{1}{60}\phi_{i,jk|jik}-\frac{1}{60}\phi_{i.jk|jki}+\frac{1}{120}\phi_{i,kk|jij}-\frac{1}{120}\phi_{i,kk|jji}\\
 & -\frac{1}{60}\phi_{i|j,ik|jk}+\frac{1}{60}\phi_{i|j,ik|kj}+\frac{1}{120}\phi_{i|j,kk|ij}-\frac{1}{120}\phi_{i|j,kk|ji}\\
 & +\frac{1}{60}\phi_{ij|i,jk|k}+\frac{1}{120}\phi_{ij|i,kk|j}-\frac{1}{60}\phi_{ij|j,ik|k}-\frac{1}{120}\phi_{ij|j,kk|i}\\
 & +\frac{1}{120}\phi_{iji|j,kk}-\frac{1}{120}\phi_{ijj|i,kk}-\frac{1}{60}\phi_{ijk|i,jk}+\frac{1}{60}\phi_{ijk|j,ik},
\end{align*}
\begin{align*}
{\cal S}_{3}^{(2)}\triangleq & -\frac{1}{120}\phi_{i,i|j,k|jk}+\frac{1}{120}\phi_{i,i|j,k|kj}-\frac{1}{240}\phi_{i,j|i,k|jk}+\frac{1}{240}\phi_{i,j|j,k|ik}-\frac{1}{240}\phi_{i,j|j,k|ki}\\
 & +\frac{1}{240}\phi_{i,k|i,j|jk}-\frac{1}{240}\phi_{i,k|j,i|jk}+\frac{1}{240}\phi_{i,k|j,i|kj}+\frac{1}{120}\phi_{i,k|j,j|ik}-\frac{1}{120}\phi_{i,k|j,j|ki}\\
 & +\frac{1}{80}\phi_{i,k|j,k|ij}-\frac{1}{80}\phi_{i,k|j,k|ji}-\frac{1}{360}\phi_{i,i|j|j,k|k}+\frac{1}{720}\phi_{i,i|j|k,j|k}+\frac{1}{720}\phi_{i,i|j|k,k|j}\\
 & -\frac{1}{360}\phi_{i,j|i|j,k|k}+\frac{1}{180}\phi_{i,j|j|i,k|k}-\frac{1}{360}\phi_{i,j|j|k,i|k}-\frac{1}{360}\phi_{i,j|j|k,k|i}+\frac{1}{720}\phi_{i,k|i|j,j|k}\\
 & +\frac{1}{720}\phi_{i,k|i|j,k|j}+\frac{1}{45}\phi_{i,k|j|i,j|k}+\frac{1}{45}\phi_{i,k|j|i,k|j}-\frac{1}{90}\phi_{i,k|j|j,i|k}-\frac{1}{90}\phi_{i,k|j|j,k|i}\\
 & -\frac{1}{90}\phi_{i,k|j|k,i|j}-\frac{1}{90}\phi_{i,k|j|k,j|i}-\frac{1}{720}\phi_{i,i|jj|k,k}+\frac{1}{360}\phi_{i,i|jk|j,k}-\frac{1}{720}\phi_{i,i|jk|k,j}\\
 & -\frac{1}{720}\phi_{i,j|ij|k,k}+\frac{1}{360}\phi_{i,j|ji|k,k}-\frac{1}{180}\phi_{i,j|jk|i,k}+\frac{1}{360}\phi_{i,j|jk|k,i}-\frac{1}{720}\phi_{i,k|ij|j,k}\\
 & +\frac{1}{360}\phi_{i,k|ij|k,j}+\frac{1}{90}\phi_{i,k|ji|j,k}+\frac{1}{90}\phi_{i,k|ji|k,j}-\frac{1}{45}\phi_{i,k|jj|i,k}+\frac{1}{90}\phi_{i,k|jj|k,i}\\
 & -\frac{1}{45}\phi_{i,k|jk|i,j}+\frac{1}{90}\phi_{i,k|jk|j,i}+\frac{1}{360}\phi_{i|i,j|j,k|k}-\frac{1}{720}\phi_{i|i,k|j,j|k}-\frac{1}{720}\phi_{i|i,k|j,k|j}\\
 & -\frac{1}{180}\phi_{i|j,i|j,k|k}+\frac{1}{360}\phi_{i|j,i|k,j|k}+\frac{1}{360}\phi_{i|j,i|k,k|j}+\frac{1}{360}\phi_{i|j,j|i,k|k}-\frac{1}{720}\phi_{i|j,j|k,i|k}\\
 & -\frac{1}{720}\phi_{i|j,j|k,k|i}+\frac{1}{90}\phi_{i|j,k|i,j|k}+\frac{1}{90}\phi_{i|j,k|i,k|j}-\frac{1}{45}\phi_{i|j,k|j,i|k}-\frac{1}{45}\phi_{i|j,k|j,k|i}\\
 & +\frac{1}{90}\phi_{i|j,k|k,i|j}+\frac{1}{90}\phi_{i|j,k|k,j|i}+\frac{1}{720}\phi_{i|i,j|j|k,k}+\frac{1}{720}\phi_{i|i,k|j|j,k}-\frac{1}{360}\phi_{i|i,k|j|k,j}\\
 & -\frac{1}{360}\phi_{i|j,i|j|k,k}+\frac{1}{180}\phi_{i|j,i|k|j,k}-\frac{1}{360}\phi_{i|j,i|k|k,j}+\frac{1}{720}\phi_{i|j,j|i|k,k}-\frac{1}{360}\phi_{i|j,j|k|i,k}\\
 & +\frac{1}{720}\phi_{i|j,j|k|k,i}+\frac{1}{45}\phi_{i|j,k|i|j,k}-\frac{1}{90}\phi_{i|j,k|i|k,j}-\frac{1}{90}\phi_{i|j,k|j|i,k}-\frac{1}{90}\phi_{i|j,k|j|k,i}\\
 & -\frac{1}{90}\phi_{i|j,k|k|i,j}+\frac{1}{45}\phi_{i|j,k|k|j,i}+\frac{1}{120}\phi_{ij|i,j|k,k}+\frac{1}{80}\phi_{ij|i,k|j,k}+\frac{1}{240}\phi_{ij|i,k|k,j}\\
 & -\frac{1}{120}\phi_{ij|j,i|k,k}-\frac{1}{80}\phi_{ij|j,k|i,k}-\frac{1}{240}\phi_{ij|j,k|k,i}+\frac{1}{240}\phi_{ij|k,i|j,k}-\frac{1}{240}\phi_{ij|k,i|k,j}\\
 & -\frac{1}{240}\phi_{ij|k,j|i,k}+\frac{1}{240}\phi_{ij|k,j|k,i}-\frac{1}{120}\phi_{ij|k,k|i,j}+\frac{1}{120}\phi_{ij|k,k|j,i},
\end{align*}
\begin{align}
{\cal S}_{3}^{(3)}\triangleq & \frac{1}{144}\partial_{k}\bar{b}^{j}\phi_{ijik}+\frac{1}{144}\partial_{j}\bar{b}^{k}\phi_{ijik}-\frac{1}{144}\partial_{k}\bar{b}^{i}\phi_{ijjk}-\frac{1}{144}\partial_{i}\bar{b}^{k}\phi_{ijjk}\nonumber \\
 & -\frac{1}{144}\partial_{k}\bar{b}^{j}\phi_{ijki}-\frac{1}{144}\partial_{j}\bar{b}^{k}\phi_{ijki}+\frac{1}{144}\partial_{k}\bar{b}^{i}\phi_{ijkj}+\frac{1}{144}\partial_{i}\bar{b}^{k}\phi_{ijkj},\label{eq:bTerm}
\end{align}
\begin{align}
{\cal S}_{3}^{(4)}\triangleq & \frac{11}{1440}\partial_{ll}a^{jk}\phi_{ijik}-\frac{11}{1440}\partial_{ll}a^{ik}\phi_{ijjk}-\frac{11}{1440}\partial_{ll}a^{jk}\phi_{ijki}+\frac{11}{1440}\partial_{ll}a^{ik}\phi_{ijkj}\nonumber \\
 & +\frac{1}{120}\partial_{jl}a^{ik}\phi_{ijkl}-\frac{1}{120}\partial_{jk}a^{il}\phi_{ijkl}-\frac{1}{120}\partial_{il}a^{jk}\phi_{ijkl}+\frac{1}{120}\partial_{ik}a^{jl}\phi_{ijkl}.\label{eq:aTerm}
\end{align}
\end{lem}

\subsubsection{Final result}

To summarise, we have obtained the following result. 
\begin{prop}
\label{prop:t3Ceof}The $t^{3}$-coefficient $\hat{\Xi}_{x}$ in the
expansion of (\ref{eq:Psi4Exp}) is given by 
\[
\hat{\Xi}_{x}=\frac{1}{2}({\cal S}_{1}+{\cal S}_{2}+{\cal S}_{3})\in E^{\otimes4},
\]
where ${\cal S}_{1},{\cal S}_{2},{\cal S}_{3}$ are explicitly given
by Lemmas \ref{lem:Total3}, \ref{lem:Total2.5}, \ref{lem:Total2}
respectively. 
\end{prop}

\begin{rem}
Although the expressions we give here are in terms of local quantities,
the intrinsic meaning of $\hat{\Xi}_{x}$ (in particular, its explicit
connection with curvature properties) will be clear after suitable
geometric reduction (cf. Proposition \ref{prop:XiFinal} below).
\end{rem}

\subsection{Completing the proof of Theorem \ref{thm:RecCurv}: geometric reduction }

We now specialise in the situation where $\phi=dF$ and $F:M\rightarrow E=\mathbb{R}^{N}$
is an isometric embedding. Despite the complicated expressions
of $\hat{\Theta}_{x}$ and $\hat{\Xi}_{x}$, in this case they simplify
substantially after performing suitable dimension reduction. More
precisely, we are going to consider the contraction $\mathfrak{C}:E^{\otimes4}\rightarrow E^{\otimes2}$
defined by (\ref{eq:24Trace}) (i.e. taking trace over $E$ with respect
to the $(2,4)$-position). 

\subsubsection{Reduction of $\hat{\Theta}_{x}$}

Recall that 
\[
\Theta_{x}\triangleq\mathfrak{C}\hat{\Theta}_{x}\in E^{\otimes2}\cong{\cal L}(E\times E;\mathbb{R})
\]
is the corresponding $t^{2}$-coefficient of $\mathfrak{C}\psi_{4}(t,x)$.
Since $F$ is an isometric embedding, the metric tensor admits the
following local expression under the previous normal chart $U$:
\begin{equation}
g_{ij}({\bf x})=\langle\partial_{i}F({\bf x}),\partial_{j}F({\bf x})\rangle_{E},\ \ \ {\bf x}\in U.\label{eq:MetricExtrinsic}
\end{equation}
According to Proposition \ref{prop:MetricExp}, one has $\langle\partial_{i}F({\bf 0}),\partial_{j}F({\bf 0})\rangle=\delta_{ij}.$
In particular, $\{\partial_{i}F({\bf 0})\}_{1\leqslant i\leqslant d}$
is an ONB of $T_{x}M$. We also note that $\phi_{i}({\bf x})=\partial_{i}F({\bf x})$
in the current setting. 

\begin{proof}[Proof of Theorem \ref{thm:RecCurv}: expression of $\Theta_x$]

According to Proposition \ref{prop:t2Coef}, one has 
\[
\Theta_{x}=\frac{1}{24}\mathfrak{C}\big(\partial_{i}F\otimes\partial_{j}F\otimes\partial_{i}F\otimes\partial_{j}F-\partial_{i}F\otimes\partial_{j}F\otimes\partial_{j}F\otimes\partial_{i}F\big),
\]
where all the derivatives are evaluated at ${\bf x}={\bf 0}$. Since
$\{\partial_{i}F\}_{1\leqslant i\leqslant d}$ is an ONB of $T_{x}M$,
one finds that 
\begin{align*}
\Theta_{x} & =\frac{1}{24}\sum_{k=1}^{d}\big(\langle\partial_{j}F,\partial_{k}F\rangle^{2}\partial_{i}F\otimes\partial_{i}F\\
 & \ \ \ -\langle\partial_{j}F,\partial_{k}F\rangle\langle\partial_{i}F,\partial_{k}F\rangle\partial_{i}F\otimes\partial_{j}F\big)=\frac{d-1}{24}\partial_{i}F\otimes\partial_{i}F.
\end{align*}
On the other hand, let $\pi_{x}:E\rightarrow T_{x}M$ denote the orthogonal
projection. Then for any $v,w\in E,$ one has
\[
\big(\partial_{i}F\otimes\partial_{i}F\big)(v,w)=\langle v,\partial_{i}F\rangle_{E}\langle w,\partial_{i}F\rangle_{E}=\langle\pi_{x}v,\pi_{x}w\rangle_{T_{x}M}.
\]
The relation (\ref{eq:t2Coef}) thus follows. 

\end{proof}
\begin{rem}
\label{rem:RecMetric}The $2$-tensor $\Theta_{x}$ allows one to
reconstruct the tangent space $T_{x}M$; indeed it is clear from (\ref{eq:t2Coef})
that 
\[
v\in(T_{x}M)^{\perp}\iff\Theta_{x}(v,w)=0\ \forall w\in E.
\]
As a result, the formula (\ref{eq:t2Coef}) shows that the Riemannian
metric tensor $g$ can be explicitly reconstructed from the $2$-tensor
$\Theta_{x}$. 
\end{rem}

\subsubsection{Reduction of $\hat{\Xi}_{x}$}

Recall that the $t^{3}$-coefficient $\Xi_{x}$ is given by 
\begin{equation}
\Xi_{x}=\mathfrak{C}\hat{\Xi}_{x}=\frac{1}{2}\big(\mathfrak{C}{\cal S}_{1}+\mathfrak{C}{\cal S}_{2}+\mathfrak{C}{\cal S}_{3}\big)\in{\cal L}(E\times E;\mathbb{R}).\label{eq:XiRaw}
\end{equation}
To simplify its expression, we are going to write 
\[
\Xi_{x}=\Xi_{x}^{T}+\Xi_{x}^{\perp},
\]
where $\Xi_{x}^{T}$ and $\Xi_{x}^{\perp}$ are defined by taking
tangential and vertical $(2,4)$-traces of $\hat{\Xi}_{x}$ respectively.
Namely, let $\{\varepsilon_{1},\cdots,\varepsilon_{d}\}$ be an ONB
basis of $T_{x}M$ and let $\{\varepsilon_{d+1},\cdots,\varepsilon_{N}\}$
be an ONB basis of $(T_{x}M)^{\perp}$ respectively. Then one defines
\[
\Xi_{x}^{T}(\cdot,\cdot)\triangleq\sum_{i=1}^{d}\hat{\Xi}_{x}(\cdot,\varepsilon_{i},\cdot,\varepsilon_{i}),\ \Xi_{x}^{\perp}(\cdot,\cdot)\triangleq\sum_{j=d+1}^{N}\hat{\Xi}_{x}(\cdot,\varepsilon_{i},\cdot,\varepsilon_{i}).
\]
In what follows, we shall compute $\Xi_{x}^{T}$, $\Xi_{x}^{\perp}$
and obtain their intrinsic forms respectively. 

First of all, one has the following basic observation. 
\begin{lem}
\label{lem:Basis}The family $\{\partial_{i}F({\bf 0}):1\leqslant i\leqslant d\}$
is an ONB of $T_{x}M$. In addition, one has $\partial_{ij}^{2}F({\bf 0})\perp T_{x}M$
for any fixed pair of indices $i,j$.
\end{lem}

\begin{proof}
The first part was already seen before. For the second part, let $i,j,k$
be given fixed indices. Then one has the following relations on the
normal chart $U$:
\begin{equation}
\begin{cases}
\partial_{k}\langle\partial_{i}F,\partial_{j}F\rangle_{E}({\bf x})=\langle\partial_{ik}^{2}F,\partial_{j}F\rangle_{E}({\bf x})+\langle\partial_{i}F,\partial_{jk}^{2}F\rangle_{E}({\bf x}),\\
\partial_{i}\langle\partial_{j}F,\partial_{k}F\rangle_{E}({\bf x})=\langle\partial_{ij}^{2}F,\partial_{k}F\rangle_{E}({\bf x})+\langle\partial_{j}F,\partial_{ik}^{2}F\rangle_{E}({\bf x}),\\
\partial_{j}\langle\partial_{i}F,\partial_{k}F\rangle_{E}({\bf x})=\langle\partial_{ij}^{2}F,\partial_{k}F\rangle_{E}({\bf x})+\langle\partial_{i}F,\partial_{jk}^{2}F\rangle_{E}({\bf x}).
\end{cases}\label{eq:2ndDerNormal}
\end{equation}
According to (\ref{eq:MetricExtrinsic}) and Proposition \ref{prop:MetricExp},
the left hand side of (\ref{eq:2ndDerNormal}) is zero at ${\bf x}={\bf 0}$.
Therefore, 
\[
\langle\partial_{ij}^{2}F,\partial_{k}F\rangle_{E}({\bf 0})=\langle\partial_{ik}^{2}F,\partial_{j}F\rangle_{E}({\bf 0})=\langle\partial_{jk}^{2}F,\partial_{i}F\rangle_{E}({\bf 0})=0.
\]
Since $T_{x}M$ is spanned by $\{\partial_{k}F({\bf 0}):1\leqslant k\leqslant d\}$,
one concludes that $\partial_{ij}^{2}F({\bf 0})\perp T_{x}M.$
\end{proof}
Next, we compute intermediate expressions of $\Xi_{x}^{T}$, $\Xi_{x}^{\perp}$
in the lemma below. Throughout the rest, unless otherwise stated all
quantities are evaluated at ${\bf x}={\bf 0}$ and this will be omitted
to ease notation (e.g. $\partial_{ij}^{2}F$ means $\partial_{ij}^{2}F({\bf 0})$).
We will also omit the subscript $E$ for the inner product on $E$. 
\begin{lem}
\label{lem:InterXiTP}One has 
\begin{align}
\Xi_{x}^{T}= & \big(\frac{1}{8640}{\rm S}_{x}+\frac{1}{120}\langle\partial_{j}F,\partial_{jkk}^{3}F\rangle\big)\partial_{i}F\otimes\partial_{i}F\nonumber \\
 & \ \ \ +\big(\frac{d+34}{8640}{\rm Ric}_{ij}-\frac{1}{240}\langle\partial_{i}F,\partial_{jkk}^{3}F\rangle-\frac{1}{240}\langle\partial_{j}F,\partial_{ikk}^{3}F\rangle\big)\partial_{i}F\otimes\partial_{j}F\nonumber \\
 & \ \ \ +\frac{d-2}{1440}\partial_{ii}^{2}F\otimes\partial_{jj}^{2}F+\frac{8d-7}{1440}\partial_{ij}^{2}F\otimes\partial_{ij}^{2}F\nonumber \\
 & \ \ \ +\frac{d-1}{240}\big(\partial_{i}F\otimes\partial_{ijj}^{3}F+\partial_{ijj}^{3}F\otimes\partial_{i}F\big),\label{eq:XiT}
\end{align}
where ${\rm Ric}_{ij}$ are the Ricci curvature coefficients and ${\rm S}_{x}$
is the scalar curvature at $x$. Respectively, one also has
\begin{align}
\Xi_{x}^{\perp}= & \big(\frac{1}{1440}\langle\partial_{jj}^{2}F,\partial_{kk}^{2}F\rangle+\frac{1}{180}\langle\partial_{jk}^{2}F,\partial_{jk}^{2}\rangle\big)\partial_{i}F\otimes\partial_{i}F\nonumber \\
 & \ \ \ -\big(\frac{7}{1440}\langle\partial_{ik}^{2}F,\partial_{jk}^{2}F\rangle+\frac{1}{720}\langle\partial_{ij}^{2}F,\partial_{kk}^{2}F\rangle\big)\partial_{i}F\otimes\partial_{j}F.\label{eq:XiPerp}
\end{align}
\end{lem}

\begin{proof}
This follows from explicit calculation based on the formulae for ${\cal S}_{1},{\cal S}_{2},{\cal S}_{3}$
provided by Lemmas \ref{lem:Total3}, \ref{lem:Total2.5}, \ref{lem:Total2}
respectively as well as Lemma \ref{lem:Basis}. The curvature quantities
appear because there are terms like $\partial_{i}b,\partial_{ij}^{2}G_{1},\partial_{ij}^{2}a$
in ${\cal S}_{3}^{(3)},$${\cal S}_{3}^{(4)}$ inside ${\cal S}_{3}$
(cf. (\ref{eq:bTerm}), (\ref{eq:aTerm}) for their definitions and
also recall (\ref{eq:barb}) for the definition of $\bar{b}$). These
quantities are explicitly related to curvature coefficients by Lemma
\ref{lem:abGNormChart}. We will leave the routine and lengthy details
to the patient reader. 
\end{proof}
To further simplify the expressions of $\Xi_{x}^{T}$ and $\Xi_{x}^{\perp}$,
we first prepare a lemma which computes the inner products of certain
derivatives of $F.$
\begin{lem}
\label{lem:InnerDerF} (i) One has 
\begin{equation}
\langle\partial_{jj}^{2}F,\partial_{kk}^{2}F\rangle=d^{2}|H_{x}|^{2}\label{eq:LapNorm}
\end{equation}
and 
\begin{equation}
\langle\partial_{jk}^{2}F,\partial_{jk}^{2}F\rangle=-{\rm S}_{x}+d^{2}|H_{x}|^{2},\ \langle\partial_{j}F,\partial_{jkk}^{3}F\rangle=\frac{2}{3}{\rm S}_{x}-d^{2}|H_{x}|^{2},\label{eq:InnerDer1}
\end{equation}
where $H_{x}$ is the mean curvature vector at $x$ (cf. (\ref{eq:MeanCur})). 

\vspace{2mm}\noindent (ii) One also has 
\begin{equation}
\langle\partial_{ij}^{2}F,\partial_{kk}^{2}F\rangle=d\cdot\langle B_{x}(\partial_{i}F,\partial_{j}F),H_{x}\rangle,\label{eq:InnerDer2}
\end{equation}
\begin{equation}
\langle\partial_{i}F,\partial_{jkk}^{3}F\rangle=\frac{2}{3}{\rm Ric}_{ij}-d\cdot\langle B_{x}(\partial_{i}F,\partial_{j}F),H_{x}\rangle,\label{eq:InnerDer3}
\end{equation}
\begin{equation}
\langle\partial_{ik}^{2}F,\partial_{jk}^{2}F\rangle=-{\rm Ric}_{ij}+d\cdot\langle B_{x}(\partial_{i}F,\partial_{j}F),H_{x}\rangle,\label{eq:InnerDer4}
\end{equation}
where $B_{x}$ is the second fundamental form at $x$ (cf. (\ref{eq:2FundForm})).
\end{lem}

\begin{proof}
(i) Let us denote 
\[
X\triangleq\langle\partial_{j}F,\partial_{jkk}^{3}F\rangle,\ Y\triangleq\langle\partial_{jk}^{2}F,\partial_{jk}^{2}F\rangle,\ Z\triangleq\langle\partial_{jj}^{2}F,\partial_{kk}^{2}F\rangle.
\]
Direct calculation shows that 
\begin{equation}
\partial_{kk}^{2}\langle\partial_{j}F,\partial_{j}F\rangle=2X+2Y,\ \partial_{jk}^{2}\langle\partial_{j}F,\partial_{k}F\rangle=2X+Y+Z.\label{eq:InnerDF}
\end{equation}
According to Proposition \ref{prop:MetricExp}, one has (at ${\bf x}={\bf 0}$)
\[
\partial_{kk}^{2}\langle\partial_{j}F,\partial_{j}F\rangle=\partial_{kk}^{2}g_{jj}=-\frac{2}{3}R_{jkjk}=-\frac{2}{3}{\rm S}_{x}
\]
and 
\[
\partial_{jk}^{2}\langle\partial_{j}F,\partial_{k}F\rangle=\partial_{jk}^{2}g_{jk}=-\frac{1}{3}R_{jkkj}=\frac{1}{3}{\rm S}_{x}.
\]
By substituting this back into (\ref{eq:InnerDF}), one finds that
\[
X=\frac{2}{3}{\rm S}_{x}-Z,\ Y=-{\rm S}_{x}+Z.
\]
The relation (\ref{eq:InnerDer1}) thus follows from the fact that
$Z=|\Delta F|^{2}=d^{2}|H_{x}|^{2}$ (cf. (\ref{eq:Lap=00003DH})).

\vspace{2mm}\noindent (ii) To obtain the relation (\ref{eq:InnerDer2}),
recall that $\partial_{kk}^{2}F({\bf 0})=d\cdot H_{x}\in(T_{x}M)^{\perp}$.
It follows that 
\begin{align}
\langle\partial_{ij}^{2}F,\partial_{kk}^{2}F\rangle & =\langle\tilde{\nabla}_{\partial_{i}F}^{E}\partial_{j}F,d\cdot H_{x}\rangle\nonumber \\
 & =\langle\tilde{\nabla}_{\partial_{i}F}^{E}\partial_{j}F-\nabla_{\partial_{i}F}^{M}\partial_{j}F,d\cdot H_{x}\rangle\nonumber \\
 & =\langle B_{x}(\partial_{i}F,\partial_{j}F),d\cdot H_{x}\rangle,\label{eq:XiTInd3-1}
\end{align}
where $\tilde{\nabla}^{E}$ (respectively, $\nabla^{M}$) is the Euclidean
connection on $E$ (respectively, Riemannian connection on $M$) and
the last equality follows from the definition (\ref{eq:2FundForm})
of the second fundamental form. This gives (\ref{eq:InnerDer2}).

For the relation (\ref{eq:InnerDer3}), one first writes 
\begin{equation}
\langle\partial_{i}F,\partial_{jkk}^{3}F\rangle=\partial_{j}\langle\partial_{i}F,\partial_{kk}^{2}F\rangle-\langle\partial_{ij}^{2}F,\partial_{kk}^{2}F\rangle.\label{eq:InnerDerPf1}
\end{equation}
The second term is just given by (\ref{eq:InnerDer2}). By using the
relation 
\[
\langle\partial_{i}F,\partial_{kk}^{2}F\rangle=\partial_{k}g_{ik}-\frac{1}{2}\partial_{i}g_{kk}
\]
as well as Proposition \ref{prop:MetricExp}, the first term on the
right hand side of (\ref{eq:InnerDerPf1}) is computed as
\[
\partial_{j}\langle\partial_{i}F,\partial_{kk}^{2}F\rangle=\frac{2}{3}{\rm Ric}_{ij}.
\]
The relation (\ref{eq:InnerDer3}) thus follows. 

Similarly, for the last relation (\ref{eq:InnerDer4}) one first writes
\begin{equation}
\langle\partial_{ik}^{2}F,\partial_{jk}^{2}F\rangle=\partial_{k}\langle\partial_{i}F,\partial_{jk}^{2}\rangle-\langle\partial_{i}F,\partial_{jkk}^{3}F\rangle.\label{eq:InnerDerPf2}
\end{equation}
The second term is given by (\ref{eq:InnerDer3}). To compute the
first term, one has 
\[
\langle\partial_{i}F,\partial_{jk}^{2}F\rangle=\frac{1}{2}\big(\partial_{k}g_{ij}+\partial_{j}g_{ik}-\partial_{i}g_{jk}\big).
\]
It then follows from Proposition \ref{prop:MetricExp} that (at ${\bf x}={\bf 0}$)
\begin{align*}
\partial_{k}\langle\partial_{i}F,\partial_{jk}^{2}\rangle & =\frac{1}{2}\big(\partial_{k}^{2}g_{ij}+\partial_{jk}^{2}g_{ik}-\partial_{ik}^{2}g_{jk}\big)\\
 & =\frac{1}{2}\big(-\frac{2}{3}{\rm Ric}_{ij}+\frac{1}{3}{\rm Ric}_{ij}-\frac{1}{3}{\rm Ric}_{ij}\big)=-\frac{1}{3}{\rm Ric}_{ij}.
\end{align*}
By substituting this and (\ref{eq:InnerDer3}) into (\ref{eq:InnerDerPf2}),
one obtains the relation (\ref{eq:InnerDer4}).
\end{proof}
Finally, we consider the restriction of $\Xi_{x}^{T}$ and $\Xi_{x}^{\perp}$
on $T_{x}M\times T_{x}M$ and derive their intrinsic expressions based
on Lemma \ref{lem:InterXiTP}. 
\begin{prop}
\label{prop:XiFinal}The restricted $2$-tensors $\Xi_{x}^{T},\Xi_{x}^{\perp}\in{\cal L}(T_{x}M\times T_{x}M,\mathbb{R})$
are given by 
\begin{equation}
\Xi_{x}^{T}=\frac{49d-62}{8640}{\rm Ric}_{x}+\big(\frac{49{\rm S}_{x}}{8640}-\frac{d^{2}}{120}|H_{x}|^{2}\big)g_{x}-\frac{(d-2)d}{120}\langle B_{x},H_{x}\rangle\label{eq:XiTFinal}
\end{equation}
and 
\begin{equation}
\Xi_{x}^{\perp}=\frac{9d^{2}|H_{x}|^{2}-8S_{x}}{1440}g_{x}+\frac{7}{1440}{\rm Ric}_{x}-\frac{d}{160}\langle B_{x},H_{x}\rangle.\label{eq:XiPerpFinal}
\end{equation}
\end{prop}

\begin{proof}
Let $v,w\in T_{x}M$ be given vectors. We first evaluate $\Xi_{x}^{T}(v,w)$
by using (\ref{eq:XiT}). This is a simple task based on Lemma \ref{lem:InnerDerF}
and the following basic relations. 

\vspace{2mm}\noindent (i) It is obvious that 
\[
(\partial_{i}F\otimes\partial_{i}F)(v,w)=\langle v,\partial_{i}F\rangle\langle w,\partial_{i}F\rangle=\langle v,w\rangle.
\]
In other words, one has $\partial_{i}F\otimes\partial_{i}F=g_{x}.$
In addition, one also has 
\[
{\rm Ric}_{ij}\partial_{i}F\otimes\partial_{j}F(v,w)={\rm Ric}(v,w).
\]
(ii) Since $v,w\in T_{x}M,$ it is a direct consequence of Lemma \ref{lem:Basis}
that 
\[
(\partial_{ii}F\otimes\partial_{jj}F)(v,w)=(\partial_{ij}F\otimes\partial_{ij}F)(v,w)=0.
\]
(iii) According to Lemma \ref{lem:InnerDerF}, one has 
\[
\langle\partial_{i}F,\partial_{jkk}^{3}F\rangle(\partial_{i}F\otimes\partial_{j}F)(v,w)=\frac{2}{3}{\rm Ric}(v,w)-d\cdot\langle B_{x}(v,w),H_{x}\rangle
\]
and
\[
(\partial_{i}F\otimes\partial_{ijj}^{3}F)(v,w)=\frac{2}{3}{\rm Ric}(v,w)-d\cdot\langle B_{x}(v,w),H_{x}\rangle.
\]
By substituting all the above relations as well as (\ref{eq:InnerDer1})
into (\ref{eq:XiT}), one obtains the expression (\ref{eq:XiTFinal})
for $\Xi_{x}^{T}$. The derivation of (\ref{eq:XiPerpFinal}) for
$\Xi_{x}^{\perp}$ is similar and is left to the patient reader. 
\end{proof}
\begin{proof}[Proof of Theorem \ref{thm:RecCurv}: expression of $\Xi_x|_{T_xM\times T_xM}$]

The relation (\ref{prop:t3Ceof}) follows by adding up the two expressions
(\ref{eq:XiTFinal}) and (\ref{eq:XiPerpFinal}).

\end{proof}

Now the proof of Theorem \ref{thm:RecCurv} is complete. 
\begin{rem}
\label{rem:SepCurvRec}Recall that 
\[
{\rm Tr}g_{x}=d,\ {\rm Tr}{\rm Ric}_{x}={\rm S}_{x},\ {\rm Tr}B_{x}=d\cdot H_{x},
\]
where ${\rm Tr}$ here means taking trace over $T_{x}M$. After taking
trace on both (\ref{eq:XiTFinal}) and (\ref{eq:XiPerpFinal}) one
obtains a linear system for the variables ${\rm S}_{x}$ and $|H_{x}|^{2}$.
By solving such a system, one obtains the representations of ${\rm S}_{x}$
and $|H_{x}|$ in terms of ${\rm Tr}\Xi_{x}^{T}$ and ${\rm Tr}\Xi_{x}^{\perp}$.
Once they become known, one can then view (\ref{eq:XiTFinal}) and
(\ref{eq:XiPerpFinal}) as a linear system for the (tensor) variables
${\rm Ric}_{x}$ and $\langle B_{x},H_{x}\rangle$ (recall that $g_{x}$
is already known from $\Theta_{x}$). Solving this system gives corresponding
representations of these two tensors. As a consequence, the quantities
\[
g_{x},{\rm Ric}_{x},{\rm S}_{x},\langle B_{x},H_{x}\rangle,|H_{x}|
\]
can all be reconstructed explicitly from the the tensors $\Theta_{x},\Xi_{x}^{T},\Xi_{x}^{\perp}$.
We will not present these formulae here. 
\end{rem}

\begin{rem}
One can of course consider higher level signatures and higher order
expansions. We stopped at level four and order $t^{3}$ because (i)
the computation is already highly involved and (ii) the encoded curvature
properties are simple and rich enough to reveal in this case. Higher
order expansions will contain more complicated mixtures of covariant derivatives
of the curvature tensor and of the embedding map $F$, whose information
becomes harder to unwind. It is not unreasonable to expect that
the entire Riemnannian curvature tensor and the second fundamental
form are both encoded in the small-time expansion of the expected
signature $\mathbb{E}[S^{dF}(X^{t,x,x})]$. This is an interesting question to investigate, which is not obvious at all due to the complicated
nature of the current computation. 
\end{rem}

\addtocontents{toc}{\protect\setcounter{tocdepth}{1}}

\begin{appendices}

\section{Remaining cases for $t^2$-coefficient}\label{sec:t2}

Here we present the results for all remaining cases listed in (\ref{eq:t2Cases}).
We continue to use Notation \ref{Not:t2Eq} and (\ref{eq:PhiijklNot}).
Also recall that $B_{r}^{t}\triangleq B_{tr}/\sqrt{t}$ which is again
a Brownian motion. 

\subsection{The $(II;IK)$ and $(IK;II)$ terms}

We first consider 
\begin{align*}
A_{II;IK} & \triangleq-\sqrt{t}\int_{0}^{1}\big(\int_{0}^{\rho}\phi_{i}(X_{tr})\frac{X_{tr}^{i}}{1-r}dr\big)\otimes\phi_{j}(X_{t\rho})\frac{X_{t\rho}^{j}}{1-\rho}d\rho\\
 & \ \ \ \otimes\int_{0}^{1}\big(\int_{0}^{\theta}\phi_{k}(X_{t\delta})\frac{X_{t\delta}^{k}}{1-\delta}d\delta\big)\otimes\phi_{l}(X_{t\theta})(\sigma(X_{t\theta})dB_{\theta}^{t})^{l}.
\end{align*}
By applying the expansions (\ref{eq:TypeIExp}) and (\ref{eq:TypeIIExp}),
one has 
\begin{align}
A_{II;IK} & \stackrel{2}{=}-\sqrt{t}^{4}\phi_{ijkl}\times\int_{0<r<\rho<1}\big(\int_{0}^{r}\frac{dB_{u}^{t,i}}{1-u}\big)\big(\int_{0}^{\rho}\frac{dB_{v}^{t,j}}{1-v}\big)drd\rho\nonumber \\
 & \ \ \ \times\int_{0}^{1}\big(\int_{0}^{\theta}\big(\int_{0}^{\delta}\frac{dB_{\eta}^{t,k}}{1-\eta}\big)d\delta\big)dB_{\theta}^{t,l}.\label{eq:AIIIK}
\end{align}
Computing the expectation of the right hand side of (\ref{eq:AIIIK})
is just routine It\^o calculus. One finds that

\begin{equation}
\mathbb{E}[A_{II;IK}]=\big(-\frac{5}{24}\delta^{ik}\delta^{jl}-\frac{1}{24}\delta^{il}\delta^{jk}\big)\phi_{ijkl}t^{2}+o(t^{2}).\label{eq:ResultIIIK}
\end{equation}

To compute the $(IK;II)$ case, one can make use of the symmetry

\[
A_{IK;II}=P(A_{II;IK}),
\]
where $P:E^{\otimes4}\rightarrow E^{\otimes4}$ is the tensor permutation
induced by 
\[
P(v_{1}\otimes v_{2}\otimes w_{1}\otimes w_{2})\triangleq w_{1}\otimes w_{2}\otimes v_{1}\otimes v_{2},\ \ \ w_{i},v_{j}\in E.
\]
By applying this permutation to (\ref{eq:ResultIIIK}), one immediately
obtains that (after suitable renaming of indices)
\begin{equation}
\mathbb{E}[A_{IK;II}]=\mathbb{E}[A_{II;IK}]=\big(-\frac{5}{24}\delta^{ik}\delta^{jl}-\frac{1}{24}\delta^{il}\delta^{jk}\big)\phi_{ijkl}t^{2}+o(t^{2}).\label{eq:t21}
\end{equation}

\subsection{The $(II;KI)$ and $(KI;II)$ terms}

Target:

\begin{align*}
A_{II;KI} & \triangleq-\sqrt{t}\int_{0}^{1}\big(\int_{0}^{\rho}\phi_{i}(X_{tr})\frac{X_{tr}^{i}}{1-r}dr\big)\otimes\phi_{j}(X_{t\rho})\frac{X_{t\rho}^{j}}{1-\rho}d\rho\\
 & \ \ \ \otimes\int_{0}^{1}\big(\int_{0}^{\theta}\phi_{k}(X_{t\delta})(\sigma(X_{t\delta})dB_{\delta}^{t})^{k}\big)\otimes\phi_{l}(X_{t\theta})\frac{X_{t\theta}^{l}}{1-\theta}d\theta.
\end{align*}
Reduction:
\begin{align*}
A_{II;KI}\stackrel{2}{=} & -\sqrt{t}^{4}\phi_{ijkl}\times\int_{0<r<\rho<1}\big(\int_{0}^{r}\frac{dB_{u}^{t,i}}{1-u}\big)\big(\int_{0}^{\rho}\frac{dB_{v}^{t,j}}{1-v}\big)drd\rho\times\int_{0}^{1}B_{\theta}^{t,k}\big(\int_{0}^{\theta}\frac{dB_{\eta}^{t,l}}{1-\eta}\big)d\theta.
\end{align*}
\textit{\uline{Result}}:
\begin{equation}
\mathbb{E}[A_{II;KI}]=\mathbb{E}[A_{KI;II}]=\big(-\frac{1}{2}\delta^{ij}\delta^{kl}-\frac{11}{24}\delta^{ik}\delta^{jl}-\frac{7}{24}\delta^{il}\delta^{jk}\big)\phi_{ijkl}t^{2}+o(t^{2}).\label{eq:t22}
\end{equation}

\subsection{The $(II;KK)$ and $(KK;II)$ terms}

Target:

\begin{align*}
A_{II;KK} & \triangleq\sqrt{t}^{2}\int_{0}^{1}\big(\int_{0}^{\rho}\phi_{i}(X_{tr})\frac{X_{tr}^{i}}{1-r}dr\big)\otimes\phi_{j}(X_{t\rho})\frac{X_{t\rho}^{j}}{1-\rho}d\rho\\
 & \ \ \ \otimes\int_{0}^{1}\big(\int_{0}^{\theta}\phi_{k}(X_{t\delta})(\sigma(X_{t\delta})dB_{\delta}^{t})^{k}\big)\otimes\phi_{l}(X_{t\theta})(\sigma(X_{t\theta})dB_{\theta}^{t})^{l}.
\end{align*}
Reduction:
\begin{align*}
A_{II;KK}\stackrel{2}{=} & \sqrt{t}^{4}\phi_{ijkl}\times\int_{0<r<\rho<1}\big(\int_{0}^{r}\frac{dB_{u}^{t,i}}{1-u}\big)\big(\int_{0}^{\rho}\frac{dB_{v}^{t,j}}{1-v}\big)drd\rho\times\int_{0}^{1}B_{\theta}^{t,k}dB_{\theta}^{t,l}.
\end{align*}
\textit{\uline{Result}}:
\begin{equation}
\mathbb{E}[A_{II;KK}]=\mathbb{E}[A_{KK;II}]=\big(\frac{3}{8}\delta^{ik}\delta^{jl}+\frac{1}{8}\delta^{il}\delta^{jk}\big)\phi_{ijkl}t^{2}+o(t^{2}).\label{eq:t23}
\end{equation}

\subsection{The $(IK;KI)$ and $(KI;IK)$ terms}

Target:

\begin{align*}
A_{IK;KI} & \triangleq\sqrt{t}^{2}\int_{0}^{1}\big(\int_{0}^{\rho}\phi_{i}(X_{tr})\frac{X_{tr}^{i}}{1-r}dr\big)\otimes\phi_{j}(X_{t\rho})(\sigma(X_{t\rho})dB_{\rho}^{t})^{j}\\
 & \ \ \ \otimes\int_{0}^{1}\big(\int_{0}^{\theta}\phi_{k}(X_{t\delta})(\sigma(X_{t\delta})dB_{\delta}^{t})^{k}\big)\otimes\phi_{l}(X_{t\theta})\frac{X_{t\theta}^{l}}{1-\theta}d\theta.
\end{align*}
Reduction:
\begin{align*}
A_{IK;KI}\stackrel{2}{=} & \sqrt{t}^{4}\phi_{ijkl}\times\int_{0}^{1}\big(\int_{0}^{\rho}\big(\int_{0}^{r}\frac{dB_{u}^{t,i}}{1-u}\big)dr\big)dB_{\rho}^{t,j}\times\int_{0}^{1}B_{\theta}^{t,k}\big(\int_{0}^{\theta}\frac{dB_{v}^{t,l}}{1-v}\big)d\theta.
\end{align*}
\textit{\uline{Result}}:
\begin{equation}
\mathbb{E}[A_{IK;KI}]=\mathbb{E}[A_{KI;IK}]=\big(\frac{1}{4}\delta^{ik}\delta^{jl}+\frac{1}{12}\delta^{il}\delta^{jk}\big)\phi_{ijkl}t^{2}+o(t^{2}).\label{eq:t24}
\end{equation}

\subsection{The $(IK;KK)$ and $(KK;IK)$ terms}

Target:

\begin{align*}
A_{IK;KK} & \triangleq-\sqrt{t}^{3}\int_{0}^{1}\big(\int_{0}^{\rho}\phi_{i}(X_{tr})\frac{X_{tr}^{i}}{1-r}dr\big)\otimes\phi_{j}(X_{t\rho})(\sigma(X_{t\rho})dB_{\rho}^{t})^{j}\\
 & \ \ \ \otimes\int_{0}^{1}\big(\int_{0}^{\theta}\phi_{k}(X_{t\delta})(\sigma(X_{t\delta})dB_{\delta}^{t})^{k}\big)\otimes\phi_{l}(X_{t\theta})(\sigma(X_{t\theta})dB_{\theta}^{t})^{l}.
\end{align*}
Reduction:
\begin{align*}
A_{IK;KK}\stackrel{2}{=}- & \sqrt{t}^{4}\phi_{ijkl}\times\int_{0}^{1}\big(\int_{0}^{\rho}\big(\int_{0}^{r}\frac{dB_{u}^{t,i}}{1-u}\big)dr\big)dB_{\rho}^{t,j}\times\int_{0}^{1}B_{\theta}^{t,k}dB_{\theta}^{t,l}.
\end{align*}
\textit{\uline{Result}}:
\begin{equation}
\mathbb{E}[A_{IK;KK}]=\mathbb{E}[A_{KK;IK}]=-\frac{1}{4}\delta^{ik}\delta^{jl}\phi_{ijkl}t^{2}+o(t^{2}).\label{eq:t25}
\end{equation}

\subsection{The $(KI;KK)$ and $(KK;KI)$ terms}

Target:

\begin{align*}
A_{KI;KK} & \triangleq-\sqrt{t}^{3}\int_{0}^{1}\big(\int_{0}^{\rho}\phi_{i}(X_{tr})(\sigma(X_{tr})dB_{r}^{t})^{i}\big)\otimes\phi_{j}(X_{t\rho})\frac{X_{t\rho}^{j}}{1-\rho}d\rho\\
 & \ \ \ \otimes\int_{0}^{1}\big(\int_{0}^{\theta}\phi_{k}(X_{t\delta})(\sigma(X_{t\delta})dB_{\delta}^{t})^{k}\big)\otimes\phi_{l}(X_{t\theta})(\sigma(X_{t\theta})dB_{\theta}^{t})^{l}.
\end{align*}
Reduction:
\begin{align*}
A_{KI;KK}\stackrel{2}{=}- & \sqrt{t}^{4}\phi_{ijkl}\times\int_{0}^{1}B_{\rho}^{t,i}\big(\int_{0}^{\rho}\frac{dB_{r}^{t,j}}{1-r}\big)d\rho\times\int_{0}^{1}B_{\theta}^{t,k}dB_{\theta}^{t,l}.
\end{align*}
\textit{\uline{Result}}:
\begin{equation}
\mathbb{E}[A_{KI;KK}]=\mathbb{E}[A_{KK;KI}]=\big(-\frac{1}{2}\delta^{ik}\delta^{jl}-\frac{1}{4}\delta^{il}\delta^{jk}\big)\phi_{ijkl}t^{2}+o(t^{2}).\label{eq:t26}
\end{equation}

\subsection{The $(IK;IK)$ term}

Target:

\begin{align*}
A_{IK;IK} & \triangleq\sqrt{t}^{2}\int_{0}^{1}\big(\int_{0}^{\rho}\phi_{i}(X_{tr})\frac{X_{tr}^{i}}{1-r}dr\big)\otimes\phi_{j}(X_{t\rho})(\sigma(X_{t\rho})dB_{\rho}^{t})^{j}\\
 & \ \ \ \otimes\int_{0}^{1}\big(\int_{0}^{\theta}\phi_{k}(X_{t\delta})\frac{X_{t\delta}^{k}}{1-\delta}d\delta\big)\otimes\phi_{l}(X_{t\theta})(\sigma(X_{t\theta})dB_{\theta}^{t})^{l}.
\end{align*}
Reduction:
\begin{align*}
A_{IK;IK}\stackrel{2}{=} & \sqrt{t}^{4}\phi_{ijkl}\times\int_{0}^{1}\big(\int_{0}^{\rho}\big(\int_{0}^{r}\frac{dB_{u}^{t,i}}{1-u}\big)dr\big)dB_{\rho}^{t,j}\\
 & \ \ \ \times\int_{0}^{1}\big(\int_{0}^{\theta}\big(\int_{0}^{\delta}\frac{dB_{v}^{t,k}}{1-v}\big)d\delta\big)dB_{\theta}^{t,l}.
\end{align*}
\textit{\uline{Result}}:
\begin{equation}
\mathbb{E}[A_{IK;IK}]=\frac{1}{6}\delta^{ik}\delta^{jl}\phi_{ijkl}t^{2}+o(t^{2}).\label{eq:t27}
\end{equation}

\subsection{The $(KI;KI)$ term}

Target:

\begin{align*}
A_{KI;KI} & \triangleq\sqrt{t}^{2}\int_{0}^{1}\big(\int_{0}^{\rho}\phi_{i}(X_{tr})(\sigma(X_{tr})dB_{r}^{t})^{i}\big)\otimes\phi_{j}(X_{t\rho})\frac{X_{t\rho}^{j}}{1-\rho}d\rho\\
 & \ \ \ \otimes\int_{0}^{1}\big(\int_{0}^{\theta}\phi_{k}(X_{t\delta})(\sigma(X_{t\delta})dB_{\delta}^{t})^{k}\big)\otimes\phi_{l}(X_{t\theta})\frac{X_{t\theta}^{l}}{1-\theta}d\theta.
\end{align*}
Reduction:
\begin{align*}
A_{KI;KI}\stackrel{2}{=} & \sqrt{t}^{4}\phi_{ijkl}\times\int_{0}^{1}B_{\rho}^{t,i}\big(\int_{0}^{\rho}\frac{dB_{u}^{t,j}}{1-u}\big)d\rho\times\int_{0}^{1}B_{\theta}^{t,k}\big(\int_{0}^{\theta}\frac{dB_{v}^{t,l}}{1-v}\big)d\theta.
\end{align*}
\textit{\uline{Result}}:
\begin{equation}
\mathbb{E}[A_{KI;KI}]=\big(\delta^{ij}\delta^{kl}+\frac{2}{3}\delta^{ik}\delta^{jl}+\frac{1}{2}\delta^{il}\delta^{jk}\big)\phi_{ijkl}t^{2}+o(t^{2}).\label{eq:t28}
\end{equation}

\subsection{The $(KK;KK)$ term}

Target:

\begin{align*}
A_{KK;KK} & \triangleq\sqrt{t}^{4}\int_{0}^{1}\big(\int_{0}^{\rho}\phi_{i}(X_{tr})(\sigma(X_{tr})dB_{r}^{t})^{i}\big)\otimes\phi_{j}(X_{t\rho})(\sigma(X_{t\rho})dB_{\rho}^{t})^{j}\\
 & \ \ \ \otimes\int_{0}^{1}\big(\int_{0}^{\theta}\phi_{k}(X_{t\delta})(\sigma(X_{t\delta})dB_{\delta}^{t})^{k}\big)\otimes\phi_{l}(X_{t\theta})(\sigma(X_{t\theta})dB_{\theta}^{t})^{l}.
\end{align*}
Reduction:
\begin{align*}
A_{KK;KK}\stackrel{2}{=} & \sqrt{t}^{4}\phi_{ijkl}\times\int_{0}^{1}B_{\rho}^{t,i}dB_{\rho}^{t,j}\times\int_{0}^{1}B_{\theta}^{t,k}dB_{\theta}^{t,l}.
\end{align*}
\textit{\uline{Result}}:
\begin{equation}
\mathbb{E}[A_{KK;KK}]=\frac{1}{2}\delta^{ik}\delta^{jl}\phi_{ijkl}t^{2}+o(t^{2}).\label{eq:t29}
\end{equation}

\subsection{The $(II;P)$ and $(P;II)$ terms}

Target: 
\begin{align*}
A_{II;P} & \triangleq\frac{1}{2}t\int_{0}^{1}\big(\int_{0}^{\rho}\phi_{i}(X_{tr})\frac{X_{tr}^{i}}{1-r}dr\big)\otimes\phi_{j}(X_{t\rho})\frac{X_{t\rho}^{j}}{1-\rho}d\rho\\
 & \ \ \ \otimes\int_{0}^{1}\phi_{k}(X_{t\theta})\otimes\phi_{l}(X_{t\theta})a^{kl}(X_{t\theta})d\theta.
\end{align*}
Reduction:
\[
A_{II;P}\stackrel{2}{=}\frac{1}{2}t^{2}\phi_{ijkl}\times\int_{0<r<\rho<1}\big(\int_{0}^{r}\frac{dB_{u}^{t,i}}{1-u}\big)\big(\int_{0}^{\rho}\frac{dB_{v}^{t,j}}{1-v}\big)drd\rho\times\delta^{kl}.
\]
\textit{\uline{Result}}: 
\begin{equation}
\mathbb{E}[A_{II;P}]=\mathbb{E}[A_{P;II}]=\frac{1}{4}\delta^{ij}\delta^{kl}\phi_{ijkl}t^{2}+o(t^{2}).\label{eq:t210}
\end{equation}

\subsection{The $(IK;P)$ and $(P;IK)$ terms}

Target: 
\begin{align*}
A_{IK;P} & \triangleq-\frac{1}{2}t^{3/2}\int_{0}^{1}\big(\int_{0}^{\rho}\phi_{i}(X_{tr})\frac{X_{tr}^{i}}{1-r}dr\big)\otimes\phi_{j}(X_{t\rho})(\sigma(X_{t\rho})dB_{\rho}^{t})^{j}\\
 & \ \ \ \otimes\int_{0}^{1}\phi_{k}(X_{t\theta})\otimes\phi_{l}(X_{t\theta})a^{kl}(X_{t\theta})d\theta.
\end{align*}
Reduction:
\[
A_{IK;P}\stackrel{2}{=}-\frac{1}{2}t^{2}\phi_{ijkl}\times\int_{0}^{1}\big(\int_{0}^{\rho}\big(\int_{0}^{r}\frac{dB_{u}^{t,i}}{1-u}\big)dr\big)dB_{\rho}^{t,j}\times\delta^{kl}.
\]
\textit{\uline{Result}}: 
\begin{equation}
\mathbb{E}[A_{IK;P}]=\mathbb{E}[A_{P;IK}]=o(t^{2}).\label{eq:t211}
\end{equation}

\subsection{The $(KI;P)$ and $(P;KI)$ terms}

Target: 
\begin{align*}
A_{KI;P} & \triangleq-\frac{1}{2}t^{3/2}\int_{0}^{1}\big(\int_{0}^{\rho}\phi_{i}(X_{tr})(\sigma(X_{tr})dB_{r}^{t})^{i}\big)\otimes\phi_{j}(X_{t\rho})\frac{X_{t\rho}^{j}}{1-\rho}d\rho\\
 & \ \ \ \otimes\int_{0}^{1}\phi_{k}(X_{t\theta})\otimes\phi_{l}(X_{t\theta})a^{kl}(X_{t\theta})d\theta.
\end{align*}
Reduction:
\[
A_{KI;P}\stackrel{2}{=}-\frac{1}{2}t^{2}\phi_{ijkl}\times\int_{0}^{1}B_{\rho}^{t,i}\big(\int_{0}^{\rho}\frac{dB_{u}^{t,j}}{1-u}\big)d\rho\times\delta^{kl}.
\]
\textit{\uline{Result}}: 
\begin{equation}
\mathbb{E}[A_{KI;P}]=\mathbb{E}[A_{P;KI}]=-\frac{1}{2}\delta^{ij}\delta^{kl}\phi_{ijkl}t^{2}+o(t^{2}).\label{eq:t212}
\end{equation}

\subsection{The $(KK;P)$ and $(P;KK)$ terms}

Target: 
\begin{align*}
A_{KK;P} & \triangleq\frac{1}{2}t^{2}\int_{0}^{1}\big(\int_{0}^{\rho}\phi_{i}(X_{tr})(\sigma(X_{tr})dB_{r}^{t})^{i}\big)\otimes\phi_{j}(X_{t\rho})(\sigma(X_{t\rho})dB_{\rho}^{t})^{j}\\
 & \ \ \ \otimes\int_{0}^{1}\phi_{k}(X_{t\theta})\otimes\phi_{l}(X_{t\theta})a^{kl}(X_{t\theta})d\theta.
\end{align*}
Reduction:
\[
A_{KK;P}\stackrel{2}{=}-\frac{1}{2}t^{2}\phi_{ijkl}\times\int_{0}^{1}B_{\rho}^{t,i}dB_{\rho}^{t,j}\times\delta^{kl}.
\]
\textit{\uline{Result}}: 
\begin{equation}
\mathbb{E}[A_{KK;P}]=\mathbb{E}[A_{P;KK}]=o(t^{2}).\label{eq:t213}
\end{equation}

\subsection{The $(P;P)$ term}

Target:
\begin{align*}
A_{P;P} & \triangleq\frac{1}{4}t^{2}\int_{0}^{1}\phi_{i}(X_{t\rho})\otimes\phi_{j}(X_{t\rho})a^{ij}(X_{t\rho})d\rho\\
 & \ \ \ \otimes\int_{0}^{1}\phi_{k}(X_{t\theta})\otimes\phi_{l}(X_{t\theta})a^{kl}(X_{t\theta})d\theta.
\end{align*}
Reduction:
\[
A_{P;P}\stackrel{2}{=}\frac{1}{4}t^{2}\phi_{ijkl}\delta^{ij}\delta^{kl}.
\]
\textit{\uline{Result}}:
\begin{equation}
\mathbb{E}[A_{P;P}]=\frac{1}{4}\delta^{ij}\delta^{kl}\phi_{ijkl}t^{2}+o(t^{2}).\label{eq:t214}
\end{equation}

\section{Remaining cases for $t^3$-coefficient}\label{sec:t3}

In this appendix, we summarise the results for all remaining cases in the computation of the $t^3$-coefficient $\hat{\Xi}_x$. Recall that the function $\varphi$ is defined by (\ref{eq:varphi}).

\subsection{Total degree $=3$}

Here we discuss all remaining cases listed in (\ref{eq:D3List}).
We continue to use Notation \ref{Not:t3Eq} and \ref{not:PhiNot}. 

\subsubsection{The $(IJ;JI)$ and $(JI;IJ)$ terms}

Target:
\begin{align*}
B_{IJ;JI} & \triangleq t^{2}\int_{0}^{1}\big(\int_{0}^{\rho}\phi_{i}(X_{tr})\frac{X_{tr}^{i}}{1-r}dr\big)\otimes\varphi(X_{t\rho})d\rho\\
 & \ \ \ \otimes\int_{0}^{1}\big(\int_{0}^{\theta}\varphi(X_{t\delta})d\delta\big)\otimes\phi_{j}(X_{t\theta})\frac{X_{t\theta}^{j}}{1-\theta}d\theta.
\end{align*}
Reduction:
\begin{align*}
B_{IJ;JI} & \stackrel{3}{=}t^{3}\phi_{i}({\bf 0})\otimes\varphi({\bf 0})\otimes\varphi({\bf 0})\otimes\phi_{j}({\bf 0})\\
 & \ \ \ \times\int_{0<r<\rho<1}\big(\int_{0}^{r}\frac{dB_{u}^{t,i}}{1-u}\big)drd\rho\times\int_{0}^{1}\theta\big(\int_{0}^{\theta}\frac{dB_{v}^{t,j}}{1-v}\big)d\theta.
\end{align*}
\textit{\uline{Result}}: 
\begin{align*}
\mathbb{E}[B_{IJ;JI}] & =\frac{1}{24}\phi_{k|i,i|j,j|k}t^{3}+o(t^{3}),\\
\mathbb{E}[B_{JI;IJ}] & =\frac{1}{24}\phi_{i,i|kk|j,j}t^{3}+o(t^{3}).\ \ \ (\text{by symmetry})
\end{align*}

\subsubsection{The $(IJ;IJ)$ term}

Target:
\begin{align*}
B_{IJ;IJ}\triangleq & t^{2}\int_{0}^{1}\big(\int_{0}^{\rho}\phi_{i}(X_{tr})\frac{X_{tr}^{i}}{1-r}dr\big)\otimes\varphi(X_{t\rho})d\rho\\
 & \ \ \ \otimes\int_{0}^{1}\big(\int_{0}^{\theta}\phi_{j}(X_{t\delta})\frac{X_{t\delta}^{j}}{1-\delta}d\delta\big)\otimes\varphi(X_{t\theta})d\theta.
\end{align*}
Reduction:
\begin{align*}
B_{IJ;IJ}\stackrel{3}{=} & t^{3}\phi_{i}({\bf 0})\otimes\varphi({\bf 0})\otimes\phi_{j}({\bf 0})\otimes\varphi({\bf 0})\\
 & \ \ \ \times\int_{0<r<\rho<1}\big(\int_{0}^{r}\frac{dB_{u}^{t,i}}{1-u}\big)drd\rho\times\int_{0<\delta<\theta<1}\big(\int_{0}^{\delta}\frac{dB_{v}^{t,j}}{1-v}\big)d\delta d\theta.
\end{align*}
\textit{\uline{Result}}:
\[
\mathbb{E}[B_{IJ;IJ}]=\frac{1}{48}\phi_{k|i,i|k|j,j}t^{3}+o(t^{3}).
\]

\subsubsection{The $(JI;JK)$ and $(JK;JI)$ terms}

Target:
\begin{align*}
B_{JI;JK}\triangleq & -t^{5/2}\int_{0}^{1}\big(\int_{0}^{\rho}\varphi(X_{tr})dr\big)\otimes\phi_{i}(X_{t\rho})\frac{X_{t\rho}^{i}}{1-\rho}d\rho\\
 & \ \ \ \otimes\int_{0}^{1}\big(\int_{0}^{\theta}\varphi(X_{t\delta})d\delta\big)\otimes\phi_{j}(X_{t\theta})(\sigma(X_{t\theta})dB_{\theta}^{t})^{j}.
\end{align*}
Reduction:
\begin{align*}
B_{JI;JK}\stackrel{3}{=} & -t^{3}\varphi({\bf 0})\otimes\phi_{i}({\bf 0})\otimes\varphi({\bf 0})\otimes\phi_{j}({\bf 0})\times\int_{0}^{1}\rho\big(\int_{0}^{\rho}\frac{dB_{u}^{t,i}}{1-u}\big)d\rho\times\int_{0}^{1}\theta dB_{\theta}^{t,j}.
\end{align*}
\textit{\uline{Result}}:
\begin{align*}
\mathbb{E}[B_{JI;JK}] & =-\frac{5}{48}\phi_{i,i|k|j,j|k}t^{3}+o(t^{3}),\\
\mathbb{E}[B_{JK;JI}] & =-\frac{5}{48}\phi_{i,i|k|j,j|k}t^{3}+o(t^{3}).\ \ \ (\text{by symmetry}).
\end{align*}

\subsubsection{The $(IJ;JK)$ and $(JK;IJ)$ terms}

Target:
\begin{align*}
B_{IJ;JK}\triangleq & -t^{5/2}\int_{0}^{1}\big(\int_{0}^{\rho}\phi_{i}(X_{tr})\frac{X_{tr}^{i}}{1-r}dr\big)\otimes\varphi(X_{t\rho})d\rho\\
 & \ \ \ \otimes\int_{0}^{1}\big(\int_{0}^{\theta}\varphi(X_{t\delta})d\delta\big)\otimes\phi_{j}(X_{t\theta})(\sigma(X_{t\theta})dB_{\theta}^{t})^{j}.
\end{align*}
Reduction:
\begin{align*}
B_{IJ;JK}\stackrel{3}{=} & -t^{3}\phi_{i}({\bf 0})\otimes\varphi({\bf 0})\otimes\varphi({\bf 0})\otimes\phi_{j}({\bf 0})\\
 & \ \ \ \times\int_{0<r<\rho<1}\big(\int_{0}^{r}\frac{dB_{u}^{t,i}}{1-u}\big)drd\rho\times\int_{0}^{1}\theta dB_{\theta}^{t,j}.
\end{align*}
\textit{\uline{Result}}:
\begin{align*}
\mathbb{E}[B_{IJ;JK}] & =-\frac{1}{48}\phi_{k|i,i|j,j|k}t^{3}+o(t^{3}),\\
\mathbb{E}[B_{JK;IJ}] & =-\frac{1}{48}\phi_{i,i|kk|j,j}t^{3}+o(t^{3}).\ \ \ (\text{by symmetry})
\end{align*}

\subsubsection{The $(JI;KJ)$ and $(KJ;JI)$ terms}

Target:
\begin{align*}
B_{JI;KJ}\triangleq & -t^{5/2}\int_{0}^{1}\big(\int_{0}^{\rho}\varphi(X_{tr})dr\big)\otimes\phi_{i}(X_{t\rho})\frac{X_{t\rho}^{i}}{1-\rho}d\rho\\
 & \ \ \ \otimes\int_{0}^{1}\big(\int_{0}^{\theta}\phi_{j}(X_{t\delta})(\sigma(X_{t\delta})dB_{\delta}^{t})^{j}\big)\otimes\varphi(X_{t\theta})d\theta.
\end{align*}
Reduction:
\begin{align*}
B_{JI;KJ}\stackrel{3}{=} & -t^{3}\varphi({\bf 0})\otimes\phi_{i}({\bf 0})\otimes\phi_{j}({\bf 0})\otimes\varphi({\bf 0})\\
 & \ \ \ \int_{0}^{1}\rho\big(\int_{0}^{\rho}\frac{dB_{u}^{t,i}}{1-u}\big)d\rho\times\int_{0}^{1}\big(\int_{0}^{\theta}dB_{\delta}^{t,j}\big)d\theta.
\end{align*}
\textit{\uline{Result}}:
\begin{align*}
\mathbb{E}[B_{JI;KJ}] & =-\frac{1}{12}\phi_{i,i|kk|j,j}t^{3}+o(t^{3}),\\
\mathbb{E}[B_{KJ;JI}] & =-\frac{1}{12}\phi_{k|i,i|j,j|k}t^{3}+o(t^{3}).\ \ \ (\text{by symmetry})
\end{align*}

\subsubsection{The $(IJ;KJ)$ and $(KJ;IJ)$ terms}

Target:
\begin{align*}
B_{IJ;KJ}\triangleq & -t^{5/2}\int_{0}^{1}\big(\int_{0}^{\rho}\phi_{i}(X_{tr})\frac{X_{tr}^{i}}{1-r}dr\big)\otimes\varphi(X_{t\rho})d\rho\\
 & \ \ \ \otimes\int_{0}^{1}\big(\int_{0}^{\theta}\phi_{j}(X_{t\delta})(\sigma(X_{t\delta})dB_{\delta}^{t})^{j}\big)\otimes\varphi(X_{t\theta})d\theta.
\end{align*}
Reduction:
\begin{align*}
B_{IJ;KJ}\stackrel{3}{=} & -t^{3}\phi_{i}({\bf 0})\otimes\varphi({\bf 0})\otimes\phi_{j}({\bf 0})\otimes\varphi({\bf 0})\\
 & \ \ \ \times\int_{0<r<\rho<1}\big(\int_{0}^{r}\frac{dB_{u}^{t,i}}{1-u}\big)drd\rho\times\int_{0}^{1}\big(\int_{0}^{\theta}dB_{\delta}^{t,j}\big)d\theta.
\end{align*}
\textit{\uline{Result}}: 
\begin{align*}
\mathbb{E}[B_{IJ;KJ}] & =-\frac{1}{24}\phi_{k|i,i|k|j,j}t^{3}+o(t^{3}),\\
\mathbb{E}[B_{KJ;IJ}] & =-\frac{1}{24}\phi_{k|i,i|k|j,j}t^{3}+o(t^{3}).\ \ \ (\text{by symmetry})
\end{align*}

\subsubsection{The $(JK;JK)$ term}

Target:
\begin{align*}
B_{JK;JK}\triangleq & t^{3}\int_{0}^{1}\big(\int_{0}^{\rho}\varphi(X_{tr})dr\big)\otimes\phi_{i}(X_{t\rho})(\sigma(X_{t\rho})dB_{\rho}^{t})^{i}\\
 & \ \ \ \otimes\int_{0}^{1}\big(\int_{0}^{\theta}\varphi(X_{t\delta})d\delta\big)\otimes\phi_{j}(X_{t\theta})(\sigma(X_{t\theta})dB_{\theta}^{t})^{j}.
\end{align*}
Reduction:
\begin{align*}
B_{JK;JK}\stackrel{3}{=} & t^{3}\varphi({\bf 0})\otimes\phi_{i}({\bf 0})\otimes\varphi({\bf 0})\otimes\phi_{j}({\bf 0})\times\int_{0}^{1}\rho dB_{\rho}^{t,i}\times\int_{0}^{1}\theta dB_{\theta}^{t,j}.
\end{align*}
\textit{\uline{Result}}:
\[
\mathbb{E}[B_{JK;JK}]=\frac{1}{12}\phi_{i,i|k|j,j|k}t^{3}+o(t^{3}).
\]

\subsubsection{The $(JK;KJ)$ and $(KJ;JK)$ terms}

Target:
\begin{align*}
B_{JK;KJ}\triangleq & t^{3}\int_{0}^{1}\big(\int_{0}^{\rho}\varphi(X_{tr})dr\big)\otimes\phi_{i}(X_{t\rho})(\sigma(X_{t\rho})dB_{\rho}^{t})^{i}\\
 & \ \ \ \otimes\int_{0}^{1}\big(\int_{0}^{\theta}\phi_{j}(X_{t\delta})(\sigma(X_{t\delta})dB_{\delta}^{t})^{j}\big)\otimes\varphi(X_{t\theta})d\theta.
\end{align*}
Reduction: 
\begin{align*}
B_{JK;KJ}= & t^{3}\varphi({\bf 0})\otimes\phi_{i}({\bf 0})\otimes\phi_{j}({\bf 0})\otimes\varphi({\bf 0})\times\int_{0}^{1}\rho dB_{\rho}^{t,i}\times\int_{0}^{1}\big(\int_{0}^{\theta}dB_{\delta}^{t,j}\big)d\theta.
\end{align*}
\textit{\uline{Result}}:
\begin{align*}
\mathbb{E}[B_{JK;KJ}] & =\frac{1}{24}\phi_{i,i|kk|j,j}t^{3}+o(t^{3}),\\
\mathbb{E}[B_{KJ;JK}] & =\frac{1}{24}\phi_{k|i,i|j,j|k}t^{3}+o(t^{3}).\ \ \ (\text{by symmetry})
\end{align*}

\subsubsection{The $(KJ;KJ)$ term}

Target:
\begin{align*}
B_{KJ;KJ}\triangleq & t^{3}\int_{0}^{1}\big(\int_{0}^{\rho}\phi_{i}(X_{tr})(\sigma(X_{tr})dB_{r}^{t})^{i}\big)\otimes\varphi(X_{t\rho})d\rho\\
 & \ \ \ \otimes\int_{0}^{1}\big(\int_{0}^{\theta}\phi_{j}(X_{t\delta})(\sigma(X_{t\delta})dB_{\delta}^{t})^{j}\big)\otimes\varphi(X_{t\theta})d\theta.
\end{align*}
Reduction:
\begin{align*}
B_{KJ;KJ}\stackrel{3}{=} & t^{3}\phi_{i}({\bf 0})\otimes\varphi({\bf 0})\otimes\phi_{j}({\bf 0})\otimes\varphi({\bf 0})\\
 & \ \ \ \times\int_{0}^{1}\big(\int_{0}^{\rho}dB_{r}^{t,i}\big)dr\times\int_{0}^{1}\big(\int_{0}^{\theta}dB_{\delta}^{t,j}\big)d\theta.
\end{align*}
\textit{\uline{Result}}: 
\[
\mathbb{E}[B_{KJ;KJ}]=\frac{1}{12}\phi_{k|i,i|k|j,j}t^{3}+o(t^{3}).
\]

\subsubsection{The $(JJ;II)$ and $(II;JJ)$ terms}

Target:
\begin{align*}
B_{JJ;II}\triangleq & t^{2}\int_{0}^{1}\big(\int_{0}^{\rho}\varphi(X_{tr})dr\big)\otimes\varphi(X_{t\rho})d\rho\\
 & \ \ \ \otimes\int_{0}^{1}\big(\int_{0}^{\theta}\phi_{i}(X_{t\delta})\frac{X_{t\delta}^{i}}{1-\delta}d\delta\big)\otimes\phi_{j}(X_{t\theta})\frac{X_{t\theta}^{j}}{1-\theta}d\theta.
\end{align*}
Reduction:
\begin{align*}
B_{JJ;II}\stackrel{3}{=} & t^{3}\varphi({\bf 0})\otimes\varphi({\bf 0})\otimes\phi_{i}({\bf 0})\otimes\phi_{j}({\bf 0})\\
 & \ \ \ \times\int_{0}^{1}\rho d\rho\times\int_{0<\delta<\theta<1}\big(\int_{0}^{\delta}\frac{dB_{u}^{t,i}}{1-u}\big)\big(\int_{0}^{\theta}\frac{dB_{v}^{t,j}}{1-v}\big)d\delta d\theta.
\end{align*}
\textit{\uline{Result}}:
\begin{align*}
\mathbb{E}[B_{JJ;II}] & =\frac{1}{16}\phi_{i,i|j,j|kk}t^{3}+o(t^{3}),\\
\mathbb{E}[B_{JJ;II}] & =\frac{1}{16}\phi_{kk|i,i|j,j}t^{3}+o(t^{3}).\ \ \ (\text{by symmetry})
\end{align*}

\subsubsection{The $(JJ;IK)$, $(IK;JJ)$, $(JJ;KK)$ and $(KK;JJ)$ terms}

Target:
\begin{align*}
B_{JJ;IK}\triangleq & -t^{5/2}\int_{0}^{1}\big(\int_{0}^{\rho}\varphi(X_{tr})dr\big)\otimes\varphi(X_{t\rho})d\rho\\
 & \ \ \ \otimes\int_{0}^{1}\big(\int_{0}^{\theta}\phi_{i}(X_{t\delta})\frac{X_{t\delta}^{i}}{1-\delta}d\delta\big)\otimes\phi_{j}(X_{t\theta})(\sigma(X_{t\theta})dB_{\theta}^{t})^{j}.
\end{align*}
If one freezes $\varphi(X_{tr})$ and $\varphi(X_{t\rho})$ at the
origin, the resulting expression is an It\^o integral. The same pattern
occurs for the $(JJ;KK)$ term. As a result, one sees that $\mathbb{E}[B_{JJ;IK}],\mathbb{E}[B_{IK;JJ}],\mathbb{E}[B_{JJ;KK}],\mathbb{E}[B_{KK;JJ}]$
are all of order $o(t^{3})$.

\subsubsection{The $(JJ;KI)$ and $(KI;JJ)$ terms}

Target:
\begin{align*}
B_{JJ;KI}\triangleq & -t^{5/2}\int_{0}^{1}\big(\int_{0}^{\rho}\varphi(X_{tr})dr\big)\otimes\varphi(X_{t\rho})d\rho\\
 & \ \ \ \otimes\int_{0}^{1}\big(\int_{0}^{\theta}\phi_{i}(X_{t\delta})(\sigma(X_{t\delta})dB_{\delta}^{t})^{i}\big)\otimes\phi_{j}(X_{t\theta})\frac{X_{t\theta}^{j}}{1-\theta}d\theta.
\end{align*}
Reduction:
\begin{align*}
B_{JJ;KI}\stackrel{3}{=} & -t^{3}\varphi({\bf 0})\otimes\varphi({\bf 0})\otimes\phi_{i}({\bf 0})\otimes\phi_{j}({\bf 0})\\
 & \ \ \ \times\int_{0}^{1}\rho d\rho\times\int\big(\int_{0}^{\theta}dB_{\delta}^{t,i}\big)\big(\int_{0}^{\theta}\frac{dB_{u}^{t,j}}{1-u}\big)d\theta.
\end{align*}
\textit{\uline{Result}}:
\begin{align*}
\mathbb{E}[B_{JJ;KI}] & =-\frac{1}{8}\phi_{i,i|j,j|kk}t^{3}+o(t^{3}),\\
\mathbb{E}[B_{KI;JJ}] & =-\frac{1}{8}\phi_{kk|i,i|j,j}t^{3}+o(t^{3}).\ \ \ (\text{by symmetry})
\end{align*}

\subsubsection{The $(JJ;P)$ and $(P;JJ)$ terms}

Target:
\begin{align*}
B_{JJ;P}\triangleq & \frac{1}{2}t^{3}\int_{0}^{1}\big(\int_{0}^{\rho}\varphi(X_{tr})dr\big)\otimes\varphi(X_{t\rho})d\rho\\
 & \ \ \ \otimes\int_{0}^{1}\phi_{i}(X_{t\theta})\otimes\phi_{j}(X_{t\theta})a^{ij}(X_{t\theta})d\theta.
\end{align*}
Reduction:
\begin{align*}
B_{JJ;P}\stackrel{3}{=} & \frac{1}{2}t^{3}\varphi({\bf 0})\otimes\varphi({\bf 0})\otimes\phi_{i}({\bf 0})\otimes\phi_{j}({\bf 0})\times\int_{0}^{1}\rho d\rho\times\delta^{ij}\int_{0}^{1}d\theta.
\end{align*}
\textit{\uline{Result}}:
\begin{align*}
\mathbb{E}[B_{JJ;P}] & =\frac{1}{16}\phi_{i,i|j,j|kk}t^{3}+o(t^{3}),\\
\mathbb{E}[B_{JJ;P}] & =\frac{1}{16}\phi_{kk|i,i|j,j}t^{3}+o(t^{3}).\ \ \ (\text{by symmetry})
\end{align*}

\subsection{Total degree $=2.5$}

Here we present the results for all remaining cases listed in (\ref{eq:D2.5List}). The corresponding results for the permuted cases (i.e. interchanging the $(1,2)$ and $(3,4)$ tensor slots) are obtained directly by tensor permutation and  will not be displayed here.

The calculation for this part is computer-assisted by Wolfram Mathematica. We only present the final expressions; all source codes and documentation are provided at the link
\footnote{\url{https://github.com/DeepIntoStreams/ESig_BM_on_Manifold}}. The displayed tensors are always evaluated at the origin; for instance $$\partial_i \phi_i \otimes \varphi \otimes \phi_j \otimes \phi _j \triangleq \partial_i \phi_i ({\bf 0})\otimes \varphi ({\bf 0}) \otimes \phi_j ({\bf 0})\otimes \phi _j ({\bf 0}).$$

\subsubsection{The $(IJ;P)$ term}

Target:
\begin{align*}
 & -\frac{1}{2}t^{2}\int_{0}^{1}\big(\int_{0}^{\rho}\phi_{i}(X_{tr})\frac{X_{tr}^{i}}{1-r}dr\big)\otimes\varphi(X_{t\rho})d\rho\otimes\int_{0}^{1}\phi_{k}(X_{t\theta})\phi_{l}(X_{t\theta})a^{kl}(X_{t\theta})d\theta.
\end{align*}
The extra order of $\sqrt{t}$ comes from one of the following four possibilities: expanding $\phi_i$, $\varphi$, $\phi_k$ or $\phi_l$. Here the possibilities of expanding $\frac{X_{tr}^i}{1-r}$ or $a^{kl}$ are not considered because the extra order gained from further expanding these two terms is $t$, which exceeds the needed  $\sqrt{t}$.

\vspace{2mm}\noindent Expand $\phi_{i}$: the resulting expectation is that 
\[
-\frac{1}{12}\partial_{i}\phi_{i}\otimes\varphi\otimes\phi_{j}\otimes\phi_{j}t^3 + o(t^3).
\]In what follows, we will only display the $t^3$-coefficient and omit the symbols ``$\times t^3$'' and ``$+o(t^3)$''.

\vspace{2mm}\noindent Expand $\varphi:$
\[
-\frac{1}{24}\phi_{i}\otimes\partial_{i}\varphi\otimes\phi_{j}\otimes\phi_{j}.
\]
Expand $\phi_{k}:$
\[
-\frac{1}{24}\phi_{i}\otimes\varphi\otimes\partial_{i}\phi_{j}\otimes\phi_{j}.
\]
Expand $\phi_{l}$:
\[
-\frac{1}{24}\phi_{i}\otimes\varphi\otimes\phi_{j}\otimes\partial_{i}\phi_{j}.
\]

\subsubsection{The $(JK;P)$ term}

Target:
\[
\frac{1}{2}t^{5/2}\int_{0}^{1}\big(\int_{0}^{\rho}\varphi(X_{tr})dr\big)\otimes\phi_{i}(X_{t\rho})(\sigma dB^{t})_{\rho}^{i}\otimes\int_{0}^{1}\phi_{k}(X_{t\theta})\phi_{l}(X_{t\theta})a^{kl}(X_{t\theta})d\theta.
\]
Expand $\phi_{k}:$
\[
\frac{1}{24}\varphi\otimes\phi_{i}\otimes\partial_{i}\phi_{j}\otimes\phi_{j}.
\]
Expand $\phi_{l}:$
\[
\frac{1}{24}\varphi\otimes\phi_{i}\otimes\phi_{j}\otimes\partial_{i}\phi_{j}.
\]
Freeze both: the result is $0$.

\subsubsection{The $(KJ;P)$ term}

Target:
\[
\frac{1}{2}t^{5/2}\int_{0}^{1}\big(\int_{0}^{\rho}\phi_{i}(X_{tr})(\sigma dB^{t})_{r}^{i}\big)\otimes\varphi(X_{t\rho})d\rho\otimes\int_{0}^{1}\phi_{k}(X_{t\theta})\phi_{l}(X_{t\theta})a^{kl}(X_{t\theta})d\theta.
\]
Expand $\varphi:$
\[
\frac{1}{8}\phi_{i}\otimes\partial_{i}\varphi\otimes\phi_{j}\otimes\phi_{j}.
\]
Expand $\phi_{k}:$
\[
\frac{1}{12}\phi_{i}\otimes\varphi\otimes\partial_{i}\phi_{j}\otimes\phi_{j}.
\]
Expand $\phi_{l}$:
\[
\frac{1}{12}\phi_{i}\otimes\varphi\otimes\phi_{j}\otimes\partial_{i}\phi_{j}.
\]
Freeze all: the result is $0$. 

\subsubsection{The $(JI;II)$ term}

Target: 
\[
    -t\int_{0}^{1}\big(\int_{0}^{\rho}\varphi(X_{tr})dr\big)\otimes\phi_{j}(X_{t\rho})\frac{X_{t\rho}^{j}}{1-\rho}d\rho\otimes\int_{0}^{1}\big(\int_{0}^{\theta}\phi_{k}(X_{t\delta})\frac{X_{t\delta}^{k}}{1-\delta}d\delta\big)\otimes\phi_{l}(X_{t\theta})\frac{X_{t\theta}^{l}}{1-\theta}d\theta.
\]

\vspace{2mm}\noindent Expand $\varphi$: 
\[
    \partial_{i}\varphi\otimes\phi_{j}\otimes\phi_{k}\otimes\phi_{l}\times t^{3}\big(-\frac{25}{432}\delta^{il}\delta^{jk}-\frac{55}{432}\delta^{ik}\delta^{jl}-\frac{1}{12}\delta^{ij}\delta^{kl}\big) + o(t^3).
\]

\vspace{2mm}\noindent Expand $\phi_{j}$:
\[
    \varphi\otimes\partial_{i}\phi_{j}\otimes\phi_{k}\otimes\phi_{l}\times \big(-\frac{19}{216}\delta^{il}\delta^{jk}-\frac{19}{216}\delta^{ik}\delta^{jl}-\frac{1}{6}\delta^{ij}\delta^{kl}\big).
\]
Expand $\phi_{k}$:
\[
    \varphi\otimes\phi_{j}\otimes\partial_{i}\phi_{k}\otimes\phi_{l}\times \big(-\frac{7}{72}\delta^{il}\delta^{jk}-\frac{35}{144}\delta^{ik}\delta^{jl}-\frac{7}{72}\delta^{ij}\delta^{kl}\big).
\]
Expand $\phi_{l}$:
\[
    \varphi\otimes\phi_{j}\otimes\phi_{k}\otimes\partial_{i}\phi_{l}\times \big(-\frac{19}{144}\delta^{il}\delta^{jk}-\frac{1}{12}\delta^{ik}\delta^{jl}-\frac{1}{12}\delta^{ij}\delta^{kl}\big).
\]

\subsubsection{The $(IJ;II)$ term}

Target: 
\[
    -t\int_{0}^{1}\big(\int_{0}^{\rho}\phi_{i}(X_{tr})\frac{X_{tr}^{i}}{1-r}dr\big)\varphi(X_{t\rho})d\rho\otimes\int_{0}^{1}\big(\int_{0}^{\theta}\phi_{k}(X_{t\delta})\frac{X_{t\delta}^{k}}{1-\delta}d\delta\big)\otimes\phi_{l}(X_{t\theta})\frac{X_{t\theta}^{l}}{1-\theta}d\theta.
\]
Expand $\varphi$:
\[
    \phi_{i}\otimes\partial_{j}\varphi\otimes\phi_{k}\otimes\phi_{l}\times \big(-\frac{11}{432}\delta^{il}\delta^{jk}-\frac{17}{432}\delta^{ik}\delta^{jl}-\frac{1}{24}\delta^{ij}\delta^{kl}\big).
\]
Expand $\phi_{i}$:
\[
    \partial_{j}\phi_{i}\otimes\varphi\otimes\phi_{k}\otimes\phi_{l}\times \big(-\frac{1}{27}\delta^{il}\delta^{jk}-\frac{1}{27}\delta^{ik}\delta^{jl}-\frac{1}{12}\delta^{ij}\delta^{kl}\big).
\]
Expand $\phi_{k}$:
\[
    \phi_{i}\otimes\varphi\otimes\partial_{j}\phi_{k}\otimes\phi_{l}\times \big(-\frac{1}{16}\delta^{il}\delta^{jk}-\frac{1}{24}\delta^{ik}\delta^{jl}-\frac{1}{24}\delta^{ij}\delta^{kl}\big).
\]
Expand $\phi_{l}$:
\[
    \phi_{i}\otimes\varphi\otimes\phi_{k}\otimes\partial_{j}\phi_{l}\times \big(-\frac{1}{36}\delta^{il}\delta^{jk}-\frac{1}{16}\delta^{ik}\delta^{jl}-\frac{1}{36}\delta^{ij}\delta^{kl}\big).
\]

\subsubsection{The $(JK;II)$ term}
Target:
\[
    t^{3/2}\int_{0}^{1}\big(\int_{0}^{\rho}\varphi(X_{tr})dr\big)\otimes\phi_{i}(X_{t\rho})(\sigma dB^t)_{\rho}^{i}\otimes\int_{0}^{1}\big(\int_{0}^{\theta}\phi_{k}(X_{t\delta})\frac{X_{t\delta}^{k}}{1-\delta}d\delta\big)\otimes\phi_{l}(X_{t\theta})\frac{X_{t\theta}^{l}}{1-\theta}d\theta.
\]
Expand $\phi_{k}$:
\[
    \varphi\otimes\phi_{i}\otimes\partial_{j}\phi_{k}\otimes\phi_{l}\times \big(\frac{13}{72}\delta^{il}\delta^{jk}+\frac{1}{18}\delta^{ik}\delta^{jl}+\frac{1}{18}\delta^{ij}\delta^{kl}\big).
\]
Expand $\phi_{l}$:
\[
    \varphi\otimes\phi_{i}\otimes\phi_{k}\otimes\partial_{j}\phi_{l}\times \big(\frac{1}{18}\delta^{il}\delta^{jk}+\frac{5}{72}\delta^{ik}\delta^{jl}+\frac{1}{18}\delta^{ij}\delta^{kl}\big).
\]
Freeze both and expand $\varphi$: 
\begin{align*}
    \partial_{j}\varphi\otimes\phi_{i}\otimes\phi_{k}\otimes\phi_{l}\times \big(\frac{13}{144}\delta^{il}\delta^{jk} + \frac{1}{48}\delta^{ik}\delta^{jl}\big).
\end{align*}
Freeze both and expand $\phi_i$:
\begin{align*}
    \varphi\otimes\partial_{j}\phi_{i}\otimes\phi_{k}\otimes\phi_{l}\times \big(\frac{7}{72}\delta^{il}\delta^{jk} + \frac{1}{24}\delta^{ik}\delta^{jl}\big).
\end{align*}

\subsubsection{The $(KJ;II)$ term}
Target:
\[
    t^{3/2}\int_{0}^{1}\big(\int_{0}^{\rho}\phi_{i}(X_{tr})(\sigma dB^t)_{r}^{i}\big)\otimes\varphi(X_{t\rho})d\rho\otimes\int_{0}^{1}\big(\int_{0}^{\theta}\phi_{k}(X_{t\delta})\frac{X_{t\delta}^{k}}{1-\delta}d\delta\big)\otimes\phi_{l}(X_{t\theta})\frac{X_{t\theta}^{l}}{1-\theta}d\theta.
\]
Expand $\varphi$:
\[
    \phi_{i}\otimes\partial_{j}\varphi\otimes\phi_{k}\otimes\phi_{l}\times \big(\frac{1}{16}\delta^{il}\delta^{jk}+\frac{11}{144}\delta^{ik}\delta^{jl}+\frac{1}{8}\delta^{ij}\delta^{kl}\big).
\]
Expand $\phi_{k}$:
\[
    \phi_{i}\otimes\varphi\otimes\partial_{j}\phi_{k}\otimes\phi_{l}\times \big(\frac{1}{8}\delta^{il}\delta^{jk}+\frac{1}{12}\delta^{ik}\delta^{jl}+\frac{1}{12}\delta^{ij}\delta^{kl}\big).
\]
Expand $\phi_{l}$:
\[
    \phi_{i}\otimes\varphi\otimes\phi_{k}\otimes\partial_{j}\phi_{l}\times \big(\frac{1}{18}\delta^{il}\delta^{jk}+\frac{1}{8}\delta^{ik}\delta^{jl}+\frac{1}{18}\delta^{ij}\delta^{kl}\big).
\]
Freeze all and expand $\phi_i$:
\[
    \partial_{j}\phi_{i}\otimes\varphi\otimes\phi_{k}\otimes\phi_{l}\times \big(\frac{5}{72}\delta^{il}\delta^{jk} + \frac{1}{24}\delta^{ik}\delta^{jl}\big).
\]

\subsubsection{The $(IJ;IK)$ term}
Target:
\[
    t^{3/2}\int_{0}^{1}\big(\int_{0}^{\rho}\phi_{i}(X_{tr})\frac{X_{tr}^{i}}{1-r}dr\big)\otimes\varphi(X_{t\rho})d\rho\otimes\int_{0}^{1}\big(\int_{0}^{\theta}\phi_{k}(X_{t\delta})\frac{X_{t\delta}^{k}}{1-\delta}d\delta\big)\otimes\phi_{l}(X_{t\theta})(\sigma dB^t)_{\theta}^{l}.
\]
Expand $\varphi$:
\[
    \phi_{i}\otimes\partial_{j}\varphi\otimes\phi_{k}\otimes\phi_{l}\times \big(\frac{1}{216}\delta^{il}\delta^{jk}+\frac{1}{54}\delta^{ik}\delta^{jl}\big).
\]
Expand $\phi_{i}$:
\[
    \partial_{j}\phi_{i}\otimes\varphi\otimes\phi_{k}\otimes\phi_{l}\times \big(\frac{1}{108}\delta^{il}\delta^{jk}+\frac{1}{108}\delta^{ik}\delta^{jl}\big).
\]
Freeze both and expand $\phi_k$:
\[
    \phi_{i}\otimes\varphi\otimes\partial_{j}\phi_{k}\otimes\phi_{l}\times \big(\frac{1}{48}\delta^{il}\delta^{jk}\big).
\]
Freeze both and expand $\phi_l$:
\[
    \phi_{i}\otimes\varphi\otimes\phi_{k}\otimes\partial_{j}\phi_{l}\times \big(\frac{1}{72}\delta^{il}\delta^{jk}\big).
\]

\subsubsection{The $(JI;IK)$ term}
Target:
\[
    t^{3/2}\int_{0}^{1}\big(\int_{0}^{\rho}\varphi(X_{tr})dr\big)\otimes\phi_{i}(X_{t\rho})\frac{X_{t\rho}^{i}}{1-\rho}d\rho\otimes\int_{0}^{1}\big(\int_{0}^{\theta}\phi_{k}(X_{t\delta})\frac{X_{t\delta}^{k}}{1-\delta}d\delta\big)\otimes\phi_{l}(X_{t\theta})(\sigma dB^t)_{\theta}^{l}.
\]
Expand $\varphi$:
\[
    \partial_{j}\varphi\otimes\phi_{i}\otimes\phi_{k}\otimes\phi_{l}\times \big(\frac{17}{216}\delta^{il}\delta^{jk}+\frac{1}{108}\delta^{ik}\delta^{jl}\big).
\]
Expand $\phi_{i}$:
\[
    \varphi\otimes\partial_{j}\phi_{i}\otimes\phi_{k}\otimes\phi_{l}\times \big(\frac{7}{216}\delta^{il}\delta^{jk}+\frac{7}{216}\delta^{ik}\delta^{jl}\big).
\]
Freeze both and expand $\phi_k$:
\[
    \varphi\otimes\phi_{i}\otimes\partial_{j}\phi_{k}\otimes\phi_{l}\times \big(\frac{7}{48}\delta^{il}\delta^{jk}\big).
\]
Freeze both and expand $\phi_l$
\[
    \varphi\otimes\phi_{i}\otimes\phi_{k}\otimes\partial_{j}\phi_{l}\times \big(\frac{5}{72}\delta^{il}\delta^{jk}\big).
\]

\subsubsection{The $(KJ;IK)$ term}
Target:
\[
    -t^2\int_{0}^{1}\big(\int_{0}^{\rho}\phi_{i}(X_{tr})(\sigma dB^t)_{r}^{i}\big)\otimes\varphi(X_{t\rho})d\rho\otimes\int_{0}^{1}\big(\int_{0}^{\theta}\phi_{k}(X_{t\delta})\frac{X_{t\delta}^{k}}{1-\delta}d\delta\big)\otimes\phi_{l}(X_{t\theta})(\sigma dB^t)_{\theta}^{l}.
\]
Expand $\varphi$:
\[
    \phi_{i}\otimes\partial_{j}\varphi\otimes\phi_{k}\otimes\phi_{l}\times \big(-\frac{1}{36}\delta^{ik}\delta^{jl}-\frac{1}{72}\delta^{il}\delta^{jk}\big).
\]
Expand $\phi_i$:
\[
    \partial_{j}\phi_{i}\otimes\varphi\otimes\phi_{k}\otimes\phi_{l}\times \big(-\frac{1}{36}\delta^{il}\delta^{jk}\big).
\]
Expand $\phi_{k}$:
\[
    \phi_{i}\otimes\varphi\otimes\partial_{j}\phi_{k}\otimes\phi_{l}\times \big(-\frac{1}{24}\delta^{il}\delta^{jk}\big).
\]
Expand $\phi_{l}$:
\[
    \phi_{i}\otimes\varphi\otimes\phi_{k}\otimes\partial_{j}\phi_{l}\times \big(-\frac{1}{36}\delta^{il}\delta^{jk}\big).
\]

\subsubsection{The $(JK;IK)$ term}
Target:
\[
    -\delta^{il}t^{2}\int_{0}^{1}\big(\int_{0}^{\rho}\varphi(X_{tr})dr\big)\otimes\phi_{i}(X_{t\rho})\otimes\big(\int_{0}^{\rho}\phi_{k}(W_{t\eta})\frac{X_{t\eta}^{k}}{1-\eta}d\eta\big)\otimes\phi_{l}(X_{t\rho})d\rho.
\]
Expand $\varphi$:
\[
    \partial_{j}\varphi\otimes\phi_{i}\otimes\phi_{k}\otimes\phi_{l}\times \big(-\frac{5}{72}\delta^{il}\delta^{jk}\big).
\]
Expand $\phi_{i}$:
\[
    \varphi\otimes\partial_{j}\phi_{i}\otimes\phi_{k}\otimes\phi_{l}\times \big(-\frac{1}{18}\delta^{il}\delta^{jk}\big).
\]
Expand $\phi_{k}$:
\[
    \varphi\otimes\phi_{i}\otimes\partial_{j}\phi_{k}\otimes\phi_{l}\times \big(-\frac{1}{8}\delta^{il}\delta^{jk}\big).
\]
Expand $\phi_{l}$:
\[
    \varphi\otimes\phi_{i}\otimes\phi_{k}\otimes\partial_{j}\phi_{l}\times \big(-\frac{1}{18}\delta^{il}\delta^{jk}\big).
\]

\subsubsection{The $(IJ;KI)$ term}
Target:
\[
    t^{3/2}\int_{0}^{1}\big(\int_{0}^{\rho}\phi_{i}(X_{tr})\frac{X_{tr}^{i}}{1-r}dr\big)\otimes\varphi(X_{t\rho})d\rho\otimes\int_{0}^{1}\big(\int_{0}^{\theta}\phi_{k}(X_{t\delta})(\sigma dB^t)_{\delta}^{k}\big)\otimes\phi_{l}(X_{t\theta})\frac{X_{t\theta}^{l}}{1-\theta}d\theta.
\]
Expand $\varphi$:
\[
    \phi_{i}\otimes\partial_{j}\varphi\otimes\phi_{k}\otimes\phi_{l}\times \big(\frac{5}{108}\delta^{il}\delta^{jk}+\frac{13}{216}\delta^{ik}\delta^{jl}+\frac{1}{12}\delta^{ij}\delta^{kl}\big).
\]
Expand $\phi_{i}$:
\[
    \partial_{j}\phi_{i}\otimes\varphi\otimes\phi_{k}\otimes\phi_{l}\times \big(\frac{7}{108}\delta^{il}\delta^{jk}+\frac{7}{108}\delta^{ik}\delta^{jl}+\frac{1}{6}\delta^{ij}\delta^{kl}\big).
\]
Expand $\phi_{l}$:
\[
    \phi_{i}\otimes\varphi\otimes\phi_{k}\otimes\partial_{j}\phi_{l}\times \big(\frac{5}{72}\delta^{il}\delta^{jk}+\frac{5}{48}\delta^{ik}\delta^{jl}+\frac{5}{72}\delta^{ij}\delta^{kl}\big).
\]
Freeze all and expand $\phi_k$:
\[
    \phi_{i}\otimes\varphi\otimes\partial_{j}\phi_{k}\otimes\phi_{l}\times \big(\frac{1}{24}\delta^{ik}\delta^{jl} + \frac{1}{12}\delta^{ij}\delta^{kl}\big).
\]

\subsubsection{The $(JI;KI)$ term}
Target: 
\[
    t^{3/2}\int_{0}^{1}\big(\int_{0}^{\rho}\varphi(X_{tr})dr\big)\otimes\phi_{i}(X_{t\rho})\frac{X_{t\rho}^{i}}{1-\rho}d\rho\otimes\int_{0}^{1}\big(\int_{0}^{\theta}\phi_{k}(X_{t\delta})(\sigma dB^t)_\delta^{k}\big)\otimes\phi_{l}(X_{t\theta})\frac{X_{t\theta}^{l}}{1-\theta}d\theta.
\]
Expand $\varphi$:
\[
    \partial_{j}\varphi\otimes\phi_{i}\otimes\phi_{k}\otimes\phi_{l}\times \big(\frac{19}{108}\delta^{il}\delta^{jk}+\frac{23}{216}\delta^{ik}\delta^{jl}+\frac{1}{6}\delta^{ij}\delta^{kl}\big).
\]
Expand $\phi_{i}$:
\[
    \varphi\otimes\partial_{j}\phi_{i}\otimes\phi_{k}\otimes\phi_{l}\times \big(\frac{31}{216}\delta^{il}\delta^{jk}+\frac{31}{216}\delta^{ik}\delta^{jl}+\frac{1}{3}\delta^{ij}\delta^{kl}\big).
\]
Expand $\phi_{l}$:
\[
    \varphi\otimes\phi_{i}\otimes\phi_{k}\otimes\partial_{j}\phi_{l}\times \big(\frac{13}{72}\delta^{il}\delta^{jk}+\frac{11}{48}\delta^{ik}\delta^{jl}+\frac{13}{72}\delta^{ij}\delta^{kl}\big).
\]
Freeze all and expand $\phi_k$:
\[
    \varphi\otimes\phi_{i}\otimes\partial_{j}\phi_{k}\otimes\phi_{l}\times \big(\frac{1}{6}\delta^{ij}\delta^{kl} + \frac{1}{8}\delta^{ik}\delta^{jl}\big).
\]

\subsubsection{The $(KJ;KI)$ term}
Target:
\[
    -t^2\int_{0}^{1}\big(\int_{0}^{\rho}\phi_{i}(X_{tr})(\sigma dB^t)_{r}^{i}\big)\otimes\varphi(X_{t\rho})d\rho\otimes\int_{0}^{1}\big(\int_{0}^{\theta}\phi_{k}(X_{t\delta})(\sigma dB^t)_{\delta}^{k}\big)\otimes\phi_{l}(X_{t\theta})\frac{X_{t\theta}^{l}}{1-\theta}d\theta.
\]
Expand $\varphi$:
\[
    \phi_{i}\otimes\partial_{j}\varphi\otimes\phi_{k}\otimes\phi_{l}\times \big(-\frac{1}{9}\delta^{il}\delta^{jk}-\frac{1}{8}\delta^{ik}\delta^{jl}-\frac{1}{4}\delta^{ij}\delta^{kl}\big).
\]
Expand $\phi_{l}$:
\[
    \phi_{i}\otimes\varphi\otimes\phi_{k}\otimes\partial_{j}\phi_{l}\times \big(-\frac{5}{36}\delta^{il}\delta^{jk}-\frac{5}{24}\delta^{ik}\delta^{jl}-\frac{5}{36}\delta^{ij}\delta^{kl}\big).
\]
Freeze both and expand $\phi_i$:
\[
    \partial_{j}\phi_{i}\otimes\varphi\otimes\phi_{k}\otimes\phi_{l}\times \big(-\frac{1}{9}\delta^{il}\delta^{jk} - \frac{1}{12}\delta^{ik}\delta^{jl}\big).
\]
Freeze both and expand $\phi_k$:
\[
    \phi_{i}\otimes\varphi\otimes\partial_{j}\phi_{k}\otimes\phi_{l}\times \big(-\frac{1}{6}\delta^{ij}\delta^{kl} - \frac{1}{12}\delta^{ik}\delta^{jl}\big).
\]

\subsubsection{The $(JK;KI)$ term}
Target:
\[
    -t^2\int_{0}^{1}\big(\int_{0}^{\rho}\varphi(X_{tr})dr\big)\otimes\phi_{i}(X_{t\rho})(\sigma dB^t)_{\rho}^{i}\otimes\int_{0}^{1}\big(\int_{0}^{\theta}\phi_{k}(X_{t\delta})(\sigma dB^t)_{\delta}^{k}\big)\otimes\phi_{l}(X_{t\theta})\frac{X_{t\theta}^{l}}{1-\theta}d\theta.
\]
Expand $\phi_{l}$:
\[
    \varphi\otimes\phi_{i}\otimes\phi_{k}\otimes\partial_{j}\phi_{l}\times \big(-\frac{1}{9}\delta^{il}\delta^{jk}-\frac{1}{8}\delta^{ik}\delta^{jl}-\frac{1}{9}\delta^{ij}\delta^{kl}\big).
\]
Freeze $\phi_{l}$ and expand $\phi_i$:
\[
    \varphi\otimes\partial_{j}\phi_{i}\otimes\phi_{k}\otimes\phi_{l}\times \big(-\frac{5}{36}\delta^{il}\delta^{jk} - \frac{1}{12}\delta^{ik}\delta^{jl}\big)
\]
Freeze $\phi_{l}$ and expand $\phi_k$:
\[
    \varphi\otimes\phi_{i}\otimes\partial_{j}\phi_{k}\otimes\phi_{l}\times \big(-\frac{1}{12}\delta^{ij}\delta^{kl} - \frac{1}{12}\delta^{ik}\delta^{jl}\big)
\]
Freeze $\phi_{l}$ and expand $\varphi$:
\[
    \partial_{j}\varphi\otimes\phi_{i}\otimes\phi_{k}\otimes\phi_{l}\times \big(-\frac{1}{9}\delta^{il}\delta^{jk} - \frac{1}{24}\delta^{ik}\delta^{jl}\big)
\]

\subsubsection{The $(IJ;KK)$ term}
Target:
\[
    -t^2\int_{0}^{1}\big(\int_{0}^{\rho}\phi_{i}(X_{tr})\frac{X_{tr}^{i}}{1-r}dr\big)\otimes\varphi(X_{t\rho})d\rho\otimes\int_{0}^{1}\big(\int_{0}^{\theta}\phi_{k}(X_{t\delta})(\sigma dB^t)_\delta^{k}\big)\otimes\phi_{l}(X_{t\theta})(\sigma dB^t)_{\theta}^{l}.
\]
Expand $\phi_{i}$:
\[
    \partial_{j}\phi_{i}\otimes\varphi\otimes\phi_{k}\otimes\phi_{l}\times \big(-\frac{1}{27}\delta^{il}\delta^{jk}-\frac{1}{27}\delta^{ik}\delta^{jl}\big).
\]
Expand $\varphi$:
\[
    \phi_{i}\otimes\partial_{j}\varphi\otimes\phi_{k}\otimes\phi_{l}\times \big(-\frac{1}{54}\delta^{il}\delta^{jk}-\frac{5}{108}\delta^{ik}\delta^{jl}\big).
\]
Freeze both and expand $\phi_k$: the result is $0$.

\noindent
Freeze both and expand $\phi_l$:
\[
    \phi_{i}\otimes\varphi\otimes\phi_{k}\otimes\partial_{j}\phi_{l}\times \big(-\frac{1}{18}\delta^{il}\delta^{jk}\big).
\]

\subsubsection{The $(JI;KK)$ term}
Target:
\[
    -t^2\int_{0}^{1}\big(\int_{0}^{\rho}\varphi(X_{tr})dr\big)\otimes\phi_{i}(X_{t\rho})\frac{X_{t\rho}^{i}}{1-\rho}d\rho\otimes\int_{0}^{1}\big(\int_{0}^{\theta}\phi_{k}(X_{t\delta})(\sigma dB^t)_\delta^{k}\big)\otimes\phi_{l}(X_{t\theta})(\sigma dB^t)_{\theta}^{l}.
\]
Expand $\varphi$:
\[
    \partial_{j}\varphi\otimes\phi_{i}\otimes\phi_{k}\otimes\phi_{l}\times \big(-\frac{4}{27}\delta^{il}\delta^{jk}-\frac{1}{27}\delta^{ik}\delta^{jl}\big).
\]
Expand $\phi_{i}$:
\[
    \varphi\otimes\partial_{j}\phi_{i}\otimes\phi_{k}\otimes\phi_{l}\times \big(-\frac{19}{216}\delta^{il}\delta^{jk}-\frac{19}{216}\delta^{ik}\delta^{jl}\big).
\]
Freeze both and expand $\phi_k$: the result is $0$.

\noindent 
Freeze both and expand $\phi_l$:
\[
    \varphi\otimes\phi_{i}\otimes\phi_{k}\otimes\partial_{j}\phi_{l}\times \big(-\frac{7}{36}\delta^{il}\delta^{jk}\big).
\]

\subsubsection{The $(JK;KK)$ term}
Target: 
\[
    t^{5/2}\int_{0}^{1}\big(\int_{0}^{\rho}\varphi(X_{tr})dr\big)\otimes\phi_{i}(X_{t\rho})(\sigma dB^t)_{\rho}^{i}\otimes\int_{0}^{1}\big(\int_{0}^{\theta}\phi_{k}(X_{t\delta})(\sigma dB^t)_{\delta}^{k}\big)\otimes\phi_{l}(X_{t\theta})(\sigma dB^t)_{\theta}^{l}.
\]
Expand $\varphi$:
\[
    \partial_{j}\varphi\otimes\phi_{i}\otimes\phi_{k}\otimes\phi_{l}\times \big(\frac{1}{9}\delta^{il}\delta^{jk}\big)
\]
Expand $\phi_{i}$:
\[
    \varphi\otimes\partial_{j}\phi_{i}\otimes\phi_{k}\otimes\phi_{l}\times \big(\frac{5}{36}\delta^{il}\delta^{jk}\big)
\]
Expand $\phi_k$: the result is $0$.

\noindent 
Expand $\phi_{l}$:
\[
    \varphi\otimes\phi_{i}\otimes\phi_{k}\otimes\partial_{j}\phi_{l}\times \big(\frac{5}{36}\delta^{il}\delta^{jk}\big)
\]

\subsubsection{The $(KJ;KK)$ term}
Target:
\[
    t^{5/2}\int_{0}^{1}\big(\int_{0}^{\rho}\phi_{i}(X_{tr})(\sigma dB^t)_{r}^{i}\big)\otimes\varphi(X_{t\rho})d\rho\otimes\int_{0}^{1}\big(\int_{0}^{\theta}\phi_{k}(X_{t\delta})(\sigma dB^t)_{\delta}^{k}\big)\otimes\phi_{l}(X_{t\rho})(\sigma dB^t)_{\rho}^{l}.
\]
Expand $\varphi$:
\[
    \phi_{i}\otimes\partial_{j}\varphi\otimes\phi_{k}\otimes\phi_{l}\times \big(\frac{1}{18}\delta^{il}\delta^{jk}+\frac{1}{12}\delta^{ik}\delta^{jl}\big).
\]
Freeze $\varphi$ and expand $\phi_{i}$:
\[
    \partial_{j}\phi_{i}\otimes\varphi\otimes\phi_{k}\otimes\phi_{l}\times \big(\frac{1}{9}\delta^{il}\delta^{jk}\big).
\]
Freeze $\varphi$ and expand $\phi_{k}$: the result is $0$.

\noindent
Freeze $\varphi$ and expand $\phi_{l}$:
\[
    \phi_{i}\otimes\varphi\otimes\phi_{k}\otimes\partial_{j}\phi_{l}\times \big(\frac{1}{9}\delta^{il}\delta^{jk}\big).
\]

\subsection{Total degree $=2$}
Here we present the results for all cases listed in (\ref{eq:D2List}). The calculation for this part is also computer-assisted by Wolfram Mathematica. We only present the final expressions; all source codes and documentation are provided at the link
\footnote{\url{https://github.com/DeepIntoStreams/ESig_BM_on_Manifold}}.  Again, we only display the results for half of the list because the corresponding results for the permuted cases (i.e. interchanging the $(1,2)$ and $(3,4)$ tensor slots) are obtained directly by tensor permutation.

We first introduce some notation. We fix the following 
list of products of Kronecker deltas that goes through all  possible
different combinations of the indices $i,j,k,l,p,q$ in a specific order:
\begin{align*}
\Delta & \triangleq 
(\delta^{ij}\delta^{kl}\delta^{pq}, 
\delta^{ij}\delta^{kp}\delta^{lq},
\delta^{ij}\delta^{kq}\delta^{lp},
\delta^{ik}\delta^{jl}\delta^{pq},
\delta^{ik}\delta^{jp}\delta^{lq},
\delta^{ik}\delta^{jq}\delta^{lp},
\delta^{il}\delta^{jk}\delta^{pq},
\delta^{il}\delta^{jp}\delta^{kq},
\\
 & \ \ \ \ \ \ \ \ \ 
 \delta^{il}\delta^{jq}\delta^{kp},
 \delta^{ip}\delta^{jk}\delta^{lq},
 \delta^{ip}\delta^{jl}\delta^{kq},
 \delta^{ip}\delta^{jq}\delta^{kl},
 \delta^{iq}\delta^{jk}\delta^{lp},
 \delta^{iq}\delta^{jl}\delta^{kp},
 \delta^{iq}\delta^{jp}\delta^{kl}).
\end{align*}This list of product Kronecker deltas is regarded as a basis, so that one only needs to record the list of coefficients to represent the final results
instead of writing down full expressions. For instance,
the vector
\[
    (2,2,2,0,0,0,0,0,0,0,0,0,0,0,0),
\]
represents the expression
\[
    \Delta\cdot(2,2,2,0,0,0,0,0,0,0,0,0,0,0,0) = 2 \delta^{ij} \delta^{kl} \delta^{pq} + 2 \delta^{ij} \delta^{kp} \delta^{lq} + 2 \delta^{ij} \delta^{kq} \delta^{lp}.
\]We will also use Notation \ref{not:PhiNot} exclusively. Recall that $\bar{b}$ is the function defined by (\ref{eq:barb}).

\subsubsection{The $(II;II)$ term}
Target:
\small
\[
    \int_{0}^{1}\big(\int_{0}^{\rho}\phi_{i}(X_{tr})\frac{X_{tr}^{i}}{1-r}dr\big)\otimes\phi_{j}\big(X_{t\rho}\big)\frac{X_{t\rho}^{j}}{1-\rho}d\rho\otimes\int_{0}^{1}\big(\int_{0}^{\theta}\phi_{k}(X_{t\delta})\frac{X_{t\delta}^{k}}{1-\delta}d\delta\big)\otimes\phi_{l}(X_{t\theta})\frac{X_{t\theta}^{l}}{1-\theta}d\theta.
\]
\normalsize
Expand $\phi_{i}$ to order $t$:
\scriptsize
\[
    \phi_{i,pq|jkl}\times\big(\frac{1}{48},\frac{37}{3456},\frac{37}{3456},\frac{533}{17280},\frac{37}{3456},\frac{37}{3456},\frac{257}{17280},\frac{37}{3456},\frac{37}{3456},\frac{257}{17280},\frac{533}{17280},\frac{1}{48},\frac{257}{17280},\frac{533}{17280},\frac{1}{48}\big).
\]
\normalsize
Expand $\phi_{j}$ to order $t$:
\scriptsize
\[
    \phi_{i|j,pq|kl}\times\big(\frac{1}{72},\frac{1}{128},\frac{1}{128},\frac{253}{17280},\frac{253}{17280},\frac{253}{17280},\frac{157}{17280},\frac{157}{17280},\frac{157}{17280},\frac{1}{128},\frac{1}{128},\frac{1}{72},\frac{1}{128},\frac{1}{128},\frac{1}{72}\big).
\]
\normalsize
Expand $\phi_{k}$ to order $t$:
\scriptsize
\[
    \phi_{ij|k,pq|l}\times\big(\frac{1}{48},\frac{1}{48},\frac{1}{48},\frac{533}{17280},\frac{37}{3456},\frac{37}{3456},\frac{257}{17280},\frac{257}{17280},\frac{37}{3456},\frac{533}{17280},\frac{37}{3456},\frac{37}{3456},\frac{37}{3456},\frac{533}{17280},\frac{37}{3456}\big).
\]
\normalsize
Expand $\phi_{l}$ to order $t$:
\scriptsize
\[
    \phi_{ijk|l,pq}\times\big(\frac{1}{72},\frac{1}{72},\frac{1}{72},\frac{253}{17280},\frac{253}{17280},\frac{253}{17280},\frac{157}{17280},\frac{1}{128},\frac{1}{128},\frac{157}{17280},\frac{1}{128},\frac{1}{128},\frac{157}{17280},\frac{1}{128},\frac{1}{128}\big).
\]
\normalsize
Expand both $\phi_{i}$ and $\phi_{j}$ to order $\sqrt{t}$:
\scriptsize
\[
    \phi_{i,p|j,q|kl}\times\big(\frac{1}{48},\frac{347}{17280},\frac{233}{17280},\frac{347}{17280},\frac{347}{17280},\frac{101}{3456},\frac{233}{17280},\frac{233}{17280},\frac{101}{3456},\frac{115}{3456},\frac{115}{3456},\frac{1}{16},\frac{233}{17280},\frac{347}{17280},\frac{1}{48}\big).
\]
\normalsize
Expand both $\phi_{i}$ and $\phi_{k}$ to order $\sqrt{t}$:
\scriptsize
\[
    \phi_{i,p|j|k,q|l}\times\big(\frac{1}{45},\frac{1}{45},\frac{3}{80},\frac{11}{270},\frac{1}{45},\frac{1}{54},\frac{1}{54},\frac{3}{80},\frac{1}{54},\frac{3}{80},\frac{37}{360},\frac{3}{80},\frac{1}{54},\frac{11}{270},\frac{1}{45}\big).
\]
\normalsize
Expand both $\phi_{i}$ and $\phi_{l}$ to order $\sqrt{t}$:
\scriptsize
\[
    \phi_{i,p|jk|l,q}\times\big(\frac{23}{1440},\frac{23}{1440},\frac{23}{1440},\frac{11}{576},\frac{23}{720},\frac{11}{576},\frac{1}{80},\frac{23}{1440},\frac{11}{576},\frac{1}{20},\frac{97}{2880},\frac{97}{2880},\frac{1}{80},\frac{11}{576},\frac{23}{1440}\big).
\]
\normalsize
Expand both $\phi_{j}$ and $\phi_{k}$ to order $\sqrt{t}$:
\scriptsize
\[
    \phi_{i|j,p|k,q|l}\times\big(\frac{23}{1440},\frac{23}{1440},\frac{97}{2880},\frac{11}{576},\frac{23}{720},\frac{11}{576},\frac{1}{80},\frac{1}{20},\frac{1}{80},\frac{23}{1440},\frac{97}{2880},\frac{23}{1440},\frac{11}{576},\frac{11}{576},\frac{23}{720}\big).
\]
\normalsize
Expand both $\phi_{j}$ and $\phi_{l}$ to order $\sqrt{t}$:
\scriptsize
\[
    \phi_{i|j,p|k|l,q}\times\big(\frac{7}{540},\frac{7}{320},\frac{7}{540},\frac{19}{1080},\frac{17}{360},\frac{19}{1080},\frac{1}{96},\frac{7}{320},\frac{1}{96},\frac{7}{320},\frac{7}{540},\frac{7}{540},\frac{1}{96},\frac{1}{96},\frac{7}{320}\big).
\]
\normalsize
Expand both $\phi_{k}$ and $\phi_{l}$ to order $\sqrt{t}$:
\scriptsize
\[
    \phi_{ij|k,p|l,q}\times\big(\frac{1}{48},\frac{1}{16},\frac{1}{48},\frac{347}{17280},\frac{101}{3456},\frac{347}{17280},\frac{233}{17280},\frac{233}{17280},\frac{115}{3456},\frac{101}{3456},\frac{347}{17280},\frac{347}{17280},\frac{233}{17280},\frac{115}{3456},\frac{233}{17280}\big).
\]
\normalsize
Expand $\frac{X_{tr}^{i}}{1-r}$ or $\frac{X_{t\rho}^{j}}{1-\rho}$ or $\frac{X_{t\delta}^{k}}{1-\delta}$
or $\frac{X_{t\theta}^{l}}{1-\theta}$ to order $t^{1.5}$: 
\scriptsize
\[
    \phi_{ijkl}\times \bigg[\partial_{p}\bar{b}^{i}\big(\frac{1}{24}\delta^{pj}\delta^{kl}+\frac{19}{216}\delta^{jl}\delta^{kp}+\frac{7}{216}\delta^{jk}\delta^{lp}\big)+\partial_{p}\bar{b}^{j}\big(\frac{1}{12}\delta^{ip}\delta^{kl}+\frac{7}{108}\delta^{il}\delta^{kp}+\frac{7}{108}\delta^{ik}\delta^{lp}\big)
\]
\[
    +\partial_{p}\bar{b}^{k}\big(\frac{19}{216}\delta^{ip}\delta^{jl}+\frac{7}{216}\delta^{il}\delta^{jp}+\frac{1}{24}\delta^{ij}\delta^{lp}\big)+\partial_{p}\bar{b}^{l}\big(\frac{7}{108}\delta^{ip}\delta^{jk}+\frac{7}{108}\delta^{ik}\delta^{jp}+\frac{1}{12}\delta^{ij}\delta^{kp}\big) \bigg].
\]
\normalsize
Expand $a$ to order $t$:
\scriptsize
\begin{align*}
        &
        \phi_{ijkl}\times\bigg[\partial_{pq}^{2}a^{ij}\big(\frac{1}{108}\delta^{kq}\delta^{lp}+\frac{1}{108}\delta^{kp}\delta^{lq}+\frac{1}{96}\delta^{kl}\delta^{pq}\big)+\partial_{pq}^{2}a^{kl}\frac{1}{96}\delta^{ij}\delta^{pq} \\
        &
        +\partial_{pq}^{2}a^{ik}\big(\frac{1}{240}\delta^{jq}\delta^{lp}+\frac{1}{240}\delta^{jp}\delta^{lq}+\frac{1}{240}\delta^{jl}\delta^{pq}\big)+\partial_{pq}^{2}a^{jl}\frac{1}{288}\delta^{ik}\delta^{pq} \\
        &
        +\partial_{pq}^{2}a^{il}\big(\frac{23}{2880}\delta^{jq}\delta^{kp}+\frac{23}{2880}\delta^{jp}\delta^{kq}+\frac{23}{2880}\delta^{jk}\delta^{pq}\big)+\partial_{pq}^{2}a^{jk}\frac{1}{144}\delta^{il}\delta^{pq} \\
        &
        +\partial_{pq}^{2}a^{jk}\big(\frac{23}{2880}\delta^{iq}\delta^{lp}+\frac{23}{2880}\delta^{ip}\delta^{lq}+\frac{23}{2880}\delta{il}\delta^{pq}\big)+\partial_{pq}^{2}a^{il}\frac{1}{144}\delta^{jk}\delta^{pq} \\
        &
        +\partial_{pq}^{2}a^{jl}\big(\frac{73}{4320}\delta^{iq}\delta^{kp}+\frac{73}{4320}\delta^{ik}\delta^{kq}+\frac{31}{1440}\delta^{ik}\delta^{pq}\big)+\partial_{pq}^{2}a^{ik}\frac{7}{288}\delta^{jl}\delta^{pq} \\
        &
        +\partial_{pq}^{2}a^{kl}\big(\frac{1}{108}\delta^{iq}\delta^{jp}+\frac{1}{108}\delta^{ip}\delta^{jq}+\frac{1}{96}\delta^{ij}\delta^{pq}\big)+\partial_{pq}^{2}a^{ij}\frac{1}{96}\delta^{kl}\delta^{pq}\bigg].
\end{align*}
\normalsize
\subsubsection{The $(II;IK)$ term}
Target:
\small
\[
    -\sqrt{t}\int_{0}^{1}\big(\int_{0}^{\rho}\phi_{i}(X_{tr})\frac{X_{tr}^{i}}{1-r}dr\big)\otimes\phi_{j}(X_{t\rho})\frac{X_{t\rho}^{j}}{1-\rho}d\rho\otimes\int_{0}^{1}\big(\int_{0}^{\theta}\phi_{k}(X_{t\delta})\frac{X_{t\delta}^{k}}{1-\delta}d\delta\big)\otimes\phi_{l}(X_{\theta})(\sigma dB^t)_{\theta}^{l}.
\]
\normalsize
Expand $\phi_{i}$ to order $t$:
\scriptsize
\[
    \phi_{i,pq|jkl}\times\big(0,-\frac{53}{17280},-\frac{53}{17280},-\frac{329}{17280},-\frac{53}{17280},-\frac{53}{17280},-\frac{53}{17280},-\frac{53}{17280},-\frac{53}{17280},-\frac{53}{17280},-\frac{329}{17280},0,-\frac{53}{17280},-\frac{329}{17280},0\big).
\]
\normalsize
Expand $\phi_{j}$ to order $t$:
\scriptsize
\[
    \phi_{i|j,pq|kl}\times\big(0,-\frac{47}{17280},-\frac{47}{17280},-\frac{221}{51840},-\frac{221}{51840},-\frac{221}{51840},-\frac{53}{51480},-\frac{53}{51840},-\frac{53}{51840},-\frac{47}{17280},-\frac{47}{17280},0,-\frac{47}{17280},-\frac{47}{17280},0\big).
\]
\normalsize
Expand $\phi_{k}$ to order $t$:
\scriptsize
\[
    \phi_{ij|k,pq|l}\times\big(0,0,0,-\frac{29}{1440},0,0,-\frac{1}{240},-\frac{1}{240},-\frac{1}{240},0,-\frac{29}{1440},0,0,-\frac{29}{1440},0\big).
\]
\normalsize
Expand $\phi_{l}$ to order $t$:
\scriptsize
\[
    \phi_{ijk|l,pq}\times\big(0,0,0,-\frac{13}{960},0,0,-\frac{11}{2880},-\frac{11}{2880},-\frac{11}{2880},0,-\frac{77}{8640},0,0,-\frac{77}{8640},0\big).
\]
\normalsize
Expand both $\phi_{i}$ and $\phi_{j}$ to order $\sqrt{t}$:
\scriptsize
\[
    \phi_{i,p|j,q|kl}\times\big(0,-\frac{167}{17280},-\frac{53}{17280},-\frac{167}{17280},-\frac{167}{17280},-\frac{53}{17280},-\frac{53}{17280},-\frac{53}{17280},-\frac{53}{17280},-\frac{227}{17280},-\frac{227}{17280},0,-\frac{53}{17280},-\frac{167}{17280},0\big).
\]
\normalsize
Expand both $\phi_{i}$ and $\phi_{k}$ to order $\sqrt{t}$:
\scriptsize
\[
    \phi_{i,p|j|k,q|l}\times\big(0,0,-\frac{11}{720},-\frac{7}{270},0,-\frac{1}{270},-\frac{1}{270},-\frac{11}{720},-\frac{1}{270},0,-\frac{47}{720},0,-\frac{1}{270},-\frac{7}{270},0\big).
\]
\normalsize
Expand both $\phi_{i}$ and $\phi_{l}$ to order $\sqrt{t}$:
\scriptsize
\[
    \phi_{i,p|jk|l,q}\times\big(0,0,-\frac{7}{720},-\frac{79}{4320},0,-\frac{11}{2160},-\frac{11}{2160},-\frac{7}{720},-\frac{11}{2160},0,-\frac{43}{1440},0,-\frac{11}{2160},-\frac{79}{4320},0\big).
\]
\normalsize
Expand both $\phi_{j}$ and $\phi_{k}$ to order $\sqrt{t}$:
\scriptsize
\[
    \phi_{i|j,p|k,q|l}\times\big(0,0,-\frac{17}{960},-\frac{73}{8640},0,-\frac{73}{8640},-\frac{1}{540},-\frac{13}{720},-\frac{1}{540},0,-\frac{17}{960},0,-\frac{73}{8640},-\frac{73}{8640},0\big).
\]
\normalsize
Expand both $\phi_{j}$ and $\phi_{l}$ to order $\sqrt{t}$:
\scriptsize
\[
    \phi_{i|j,p|k|l,q}\times\big(0,0,-\frac{43}{4320},-\frac{53}{4320},0,-\frac{53}{4320},-\frac{11}{4320},-\frac{17}{1440},-\frac{11}{4320},0,-\frac{43}{4320},0,-\frac{11}{1440},-\frac{11}{1440},0\big).
\]
\normalsize
Expand both $\phi_{k}$ and $\phi_{l}$ to order $\sqrt{t}$:
\scriptsize
\[
    \phi_{ij|k,p|l,q}\times\big(0,0,0,-\frac{3}{160},0,0,-\frac{1}{180},-\frac{1}{180},-\frac{1}{80},0,-\frac{3}{160},0,0,-\frac{47}{1440},0\big).
\]
\normalsize
Expand $\frac{X_{tr}^{i}}{1-r}$ or $\frac{X_{t\rho}^{j}}{1-\rho}$ or $\frac{X_{t\delta}^{k}}{1-\delta}$
to order $t^{1.5}$:
\scriptsize
\begin{align*}
    &
    \phi_{ijkl}\times\bigg[\partial_{p}\bar{b}^{i}\big(-\frac{13}{216}\delta^{jl}\delta^{kp}-\frac{1}{216}\delta^{jk}\delta^{lp}\big)
    +\partial_{p}\bar{b}^{j}\big(-\frac{5}{216}\delta^{il}\delta^{kp}-\frac{5}{216}\delta^{ik}\delta^{lp}\big) \\
    & \qquad
    +\partial_{p}\bar{b}^{k}\big(-\frac{1}{144}\delta^{il}\delta^{kp}-\frac{7}{144}\delta^{ik}\delta^{jl}\big)\bigg].
\end{align*}
\normalsize
Expand $a$ to order $t$:
\scriptsize
\begin{align*}
    &
    \phi_{ijkl}\times\bigg[\partial_{pq}a^{il}\big(-\frac{11}{2880}\delta^{jq}\delta^{kp}-\frac{11}{2880}\delta^{jp}\delta^{kq}-\frac{11}{2880}\delta^{jk}\delta^{pq}\big)+\partial_{pq}a^{jk}(\big(-\frac{1}{288}\delta^{il}\delta^{pq}\big) \\
    & \qquad
    +\partial_{pq}a^{jl}\big(-\frac{77}{8640}\delta^{iq}\delta^{kq}-\frac{77}{8640}\delta^{ip}\delta^{kq}-\frac{13}{960}\delta^{ik}\delta^{pq}\big) \\
    & \qquad
    +\partial_{pq}a^{ik}\big(-\frac{5}{288}\delta^{jl}\delta^{pq}\big)+\partial_{pq}a^{ij}\big(-\frac{1}{432}\delta^{kq}\delta^{lp}-\frac{1}{432}\delta^{kp}\delta^{lq}\big)\bigg].
\end{align*}
\normalsize

\subsubsection{The $(II;KI)$ term}
Target:
\small
\[
    -\sqrt{t}\int_{0}^{1}\big(\int_{0}^{\rho}\phi_{i}(X_{tr})\frac{X_{tr}^{i}}{1-r}dr\big)\otimes\phi_{j}\big(X_{t\rho}\big)\frac{X_{t\rho}^{j}}{1-\rho}d\rho\otimes\int_{0}^{1}\big(\int_{0}^{\theta}\phi_{k}(X_{t\delta})(\sigma dB^t)_{\delta}^{k}\big)\otimes\phi_{l}(X_{t\theta})\frac{X_{t\theta}^{l}}{1-\theta}d\theta.
\]
\normalsize
Expand $\phi_{i}$ to order $t$:
\tiny
\[
    \phi_{i,pq|jkl}\times\big(-\frac{1}{24},-\frac{317}{17280},-\frac{317}{17280},-\frac{737}{17280},-\frac{317}{17280},-\frac{317}{17280},-\frac{461}{17280},-\frac{317}{17280},-\frac{317}{17280},-\frac{461}{17280},-\frac{737}{17280},-\frac{1}{24},-\frac{461}{17280},-\frac{737}{17280},-\frac{1}{24}\big).
\]
\normalsize
Expand $\phi_{j}$ to order $t$:
\tiny
\[
    \phi_{i|j,pq|kl}\times\big(-\frac{1}{36},-\frac{223}{17280},-\frac{223}{17280},-\frac{379}{17280},-\frac{379}{17280},-\frac{379}{17280},-\frac{283}{17280},-\frac{283}{17280},-\frac{283}{17280},-\frac{223}{17280},-\frac{223}{17280},-\frac{1}{36},-\frac{223}{17280},-\frac{223}{17280},-\frac{1}{36}\big).
\]
\normalsize
Expand $\phi_{k}$ to order $t$:
\scriptsize
\[
    \phi_{ij|k,pq|l}\times\big(-\frac{1}{24},0,0,-\frac{19}{480},-\frac{1}{120},-\frac{1}{120},-\frac{19}{720},0,0,-\frac{23}{1440},0,-\frac{1}{54},-\frac{23}{1440},0,-\frac{1}{54}\big).
\]
\normalsize
Expand $\phi_{l}$ to order $t$:
\scriptsize
\[
    \phi_{ijk|l,pq}\times\big(-\frac{5}{144},-\frac{5}{144},-\frac{5}{144},-\frac{73}{2880},-\frac{73}{2880},-\frac{73}{2880},-\frac{19}{960},-\frac{1}{54},-\frac{1}{54},-\frac{19}{960},-\frac{1}{54},-\frac{1}{54},-\frac{19}{960},-\frac{1}{54},-\frac{1}{54}\big).
\]
\normalsize
Expand both $\phi_{i}$ and $\phi_{j}$ to order $\sqrt{t}$:
\tiny
\[
    \phi_{i,p|j,q|kl}\times\big(-\frac{1}{24},-\frac{527}{17280},-\frac{413}{17280},-\frac{527}{17280},-\frac{527}{17280},-\frac{877}{17280},-\frac{413}{17280},-\frac{413}{17280},-\frac{877}{17280},-\frac{923}{17280},-\frac{923}{17280},-\frac{1}{8},-\frac{413}{17280},-\frac{527}{17280},-\frac{1}{24}\big).
\]
\normalsize
Expand both $\phi_{i}$ and $\phi_{k}$ to order $\sqrt{t}$:
\scriptsize
\[
    \phi_{i,p|j|k,q|l}\times\big(-\frac{1}{24},-\frac{1}{40},0,-\frac{13}{270},-\frac{1}{40},-\frac{11}{540},-\frac{17}{540},0,-\frac{11}{540},-\frac{1}{20},0,-\frac{1}{16},-\frac{17}{540},-\frac{13}{270},-\frac{1}{24}\big).
\]
\normalsize
Expand both $\phi_{i}$ and $\phi_{l}$ to order $\sqrt{t}$:
\scriptsize
\[
    \phi_{i,p|jk|l,q}\times\big(-\frac{11}{288},-\frac{13}{240},-\frac{11}{288},-\frac{293}{8640},-\frac{13}{240},-\frac{293}{8640},-\frac{59}{2160},-\frac{11}{288},-\frac{293}{8640},-\frac{7}{80},-\frac{41}{576},-\frac{41}{576},-\frac{59}{2160},-\frac{293}{8640},-\frac{11}{288}\big).
\]
\normalsize
Expand both $\phi_{j}$ and $\phi_{k}$ to order $\sqrt{t}$:
\scriptsize
\[
    \phi_{i|j,p|k,q|l}\times\big(-\frac{1}{36},-\frac{29}{1440},0,-\frac{37}{2160},-\frac{1}{30},-\frac{37}{2160},-\frac{77}{4320},0,-\frac{77}{4320},-\frac{29}{1440},0,-\frac{1}{36},-\frac{137}{4320},-\frac{137}{4320},-\frac{1}{16}\big).
\]
\normalsize
Expand both $\phi_{j}$ and $\phi_{l}$ to order $\sqrt{t}$:
\scriptsize
\[
    \phi_{i|j,p|k|l,q}\times\big(-\frac{25}{864},-\frac{109}{2880},-\frac{25}{864},-\frac{61}{2160},-\frac{19}{240},-\frac{61}{2160},-\frac{91}{4320},-\frac{31}{576},-\frac{91}{4320},-\frac{109}{2880},-\frac{25}{864},-\frac{25}{864},-\frac{91}{4320},-\frac{91}{4320},-\frac{31}{576}\big).
\]
\normalsize
Expand both $\phi_{k}$ and $\phi_{l}$ to order $\sqrt{t}$:
\scriptsize
\[
    \phi_{ij|k,p|l,q}\times\big(-\frac{1}{24},0,-\frac{1}{24},-\frac{1}{45},-\frac{7}{240},-\frac{1}{45},-\frac{11}{480},-\frac{11}{432},0,-\frac{73}{1440},-\frac{17}{432},-\frac{17}{432},-\frac{11}{480},0,-\frac{11}{432}\big).
\]
\normalsize
Expand $\frac{X_{tr}^{i}}{1-r}$ or $\frac{X_{t\rho}^{j}}{1-\rho}$ or $\frac{X_{t\theta}^{l}}{1-\theta}$
to order $t^{1.5}$:
\scriptsize
\begin{align*}
    &
    \phi_{ijkl}\times\bigg[\partial_{p}\bar{b}^{i}(\big(-\frac{1}{12}\delta^{jp}\delta^{kl}-\frac{25}{216}\delta^{jl}\delta^{kp}-\frac{13}{216}\delta^{jk}\delta^{lp}\big) + \partial_{p}\bar{b}^{j}\big(-\frac{1}{6}\delta^{ip}\delta^{kl}-\frac{23}{216}\delta^{il}\delta^{kp}-\frac{23}{216}\delta^{ik}\delta^{lp}\big) \\
    & \qquad \qquad
    +\partial_{p}\bar{b}^{l}\big(-\frac{5}{48}\delta^{ip}\delta^{jk}-\frac{13}{144}\delta^{ik}\delta^{jp}-\frac{1}{8}\delta^{ij}\delta^{kp}\big)\bigg].
\end{align*}
\normalsize
Expand $a$ to order $t$:
\scriptsize
\begin{align*}
    &
    \phi_{ijkl}\times\bigg[\partial_{pq}^{2}a^{ik}\big(-\frac{1}{120}\delta^{jq}\delta^{lp}-\frac{1}{120}\delta^{jp}\delta^{lq}-\frac{1}{120}\delta^{jl}\delta^{pq}\big) +\partial_{pq}^{2}a^{jk}\big(-\frac{23}{1440}\delta^{iq}\delta^{lp}--\frac{23}{1440}\delta^{ip}\delta^{lq}-\frac{23}{1440}\delta^{il}\delta^{pq}\big) \\
    & \qquad
    +\partial_{pq}^{2}a^{lk}\big(-\frac{1}{54}\delta^{iq}\delta^{jp}--\frac{1}{54}\delta^{ip}\delta^{jq}-\frac{1}{48}\delta^{ij}\delta^{pq}\big) + \partial_{pq}^{2}a^{ij}\big(-\frac{7}{432}\delta^{kq}\delta^{lp}--\frac{7}{432}\delta^{kp}\delta^{lq}-\frac{1}{48}\delta^{kl}\delta^{pq}\big) \\
    & \qquad
    +\partial_{pq}^{2}a^{il}\big(-\frac{7}{576}\delta^{jq}\delta^{kp}--\frac{7}{576}\delta^{jp}\delta^{kq}-\frac{7}{576}\delta^{jk}\delta^{pq}\big) + \partial_{pq}^{2}a^{jl}\big(-\frac{43}{1728}\delta^{iq}\delta^{kp}--\frac{43}{1728}\delta^{ip}\delta^{kq}-\frac{17}{576}\delta^{ik}\delta^{pq}\big) \\
    & \qquad
    +\partial_{pq}^{2}a^{jl}\big(-\frac{1}{144}\delta^{ik}\delta^{pq}\big)+\partial_{pq}^{2}a^{il}\cdot\big(-\frac{1}{72}\delta^{jk}\delta^{pq}\big)+\partial_{pq}^{2}a^{ij}\big(-\frac{1}{48}\delta^{lk}\delta^{pq}\big)+\partial_{pq}^{2}a^{lk}\big(-\frac{1}{48}\delta^{ij}\delta^{pq}\big) \\
    & \qquad
    +\partial_{pq}^{2}a^{jk}\big(-\frac{1}{96}\delta^{il}\delta^{pq}\big)+\cdot\partial_{pq}^{2}a^{ik}\big(-\frac{1}{32}\delta^{jl}\delta^{pq}\big)\bigg].
\end{align*}
\normalsize

\subsubsection{The $(II;KK)$ term}
Target:
\small
\[
    t\int_{0}^{1}\big(\int_{0}^{\rho}\phi_{i}(X_{tr})\frac{X_{tr}^{i}}{1-r}dr\big)\otimes\phi_{j}(X_{t\rho})\frac{X_{t\rho}^{j}}{1-\rho}d\rho\otimes\int_{0}^{1}\big(\int_{0}^{\theta}\phi_{k}(X_{t\delta})(\sigma dB^t)_{\delta}^{k}\big)\otimes\phi_{l}(X_{\theta})(\sigma dB^t)_{\theta}^{l}.
\]
Expand $\phi_{i}$ to order $t$:
\scriptsize
\[
    \phi_{i,pq|jkl}\times\big(0,\frac{37}{3456},\frac{37}{3456},\frac{121}{3456},\frac{37}{3456},\frac{37}{3456},\frac{37}{3456},\frac{37}{3456},\frac{37}{3456},\frac{37}{3456},\frac{121}{3456},0,\frac{37}{3456},\frac{121}{3456},0\big).
\]
\normalsize
Expand $\phi_{j}$ to order $t$:
\scriptsize
\[
    \phi_{i|j,pq|kl}\times\big(0,\frac{1}{128},\frac{1}{128},\frac{59}{3456},\frac{59}{3456},\frac{59}{3456},\frac{23}{3456},\frac{23}{3456},\frac{23}{3456},\frac{1}{128},\frac{1}{128},0,\frac{1}{128},\frac{1}{128},0\big).
\]
\normalsize
Expand $\phi_{k}$ to order $t$:
\scriptsize
\[
    \phi_{ij|k,pq|l}\times\big(0,0,0,\frac{1}{32},0,0,\frac{1}{96},0,0,0,0,0,0,0,0\big).
\]
\normalsize
Expand $\phi_{l}$ to order $t$:
\scriptsize
\[
    \phi_{ijk|l,pq}\times\big(0,0,0,\frac{17}{576},0,0,\frac{7}{576},0,0,\frac{7}{576},\frac{43}{1728},0,\frac{7}{576},\frac{43}{1728},0\big).
\]
\normalsize
Expand both $\phi_{i}$ and $\phi_{j}$ to order $\sqrt{t}$:
\scriptsize
\[
    \phi_{i,p|j,q|kl}\times\big(0,\frac{79}{3456},\frac{79}{3456},\frac{79}{3456},\frac{79}{3456},\frac{101}{3456},\frac{37}{3456},\frac{37}{3456},\frac{101}{3456},\frac{115}{128},\frac{115}{128},0,\frac{37}{3456},\frac{79}{3456},0\big).
\]
\normalsize
Expand both $\phi_{i}$ and $\phi_{k}$ to order $\sqrt{t}$:
\scriptsize
\[
    \phi_{i,p|j|k,q|l}\times\big(0,0,0,\frac{1}{27},0,\frac{1}{108},\frac{1}{108},0,\frac{1}{108},0,0,0,\frac{1}{108},\frac{1}{27},0\big).
\]
\normalsize
Expand both $\phi_{i}$ and $\phi_{l}$ to order $\sqrt{t}$:
\scriptsize
\[
    \phi_{i,p|jk|l,q}\times\big(0,0,\frac{5}{144},\frac{35}{864},0,\frac{7}{432},\frac{7}{432},\frac{5}{144},\frac{7}{432},0,\frac{23}{288},0,\frac{7}{432},\frac{35}{864},0\big).
\]
\normalsize
Expand both $\phi_{j}$ and $\phi_{k}$ to order $\sqrt{t}$:
\scriptsize
\[
    \phi_{i|j,p|k,q|l}\times\big(0,0,0,\frac{5}{432},0,\frac{5}{432},\frac{1}{216},0,\frac{1}{216},0,0,0,\frac{1}{54},\frac{1}{54},0\big).
\]
\normalsize
Expand both $\phi_{j}$ and $\phi_{l}$ to order $\sqrt{t}$:
\scriptsize
\[
    \phi_{i|j,p|k|l,q}\times\big(0,0,\frac{13}{432},\frac{11}{432},0,\frac{11}{432},\frac{7}{864},\frac{13}{288},\frac{7}{864},0,\frac{13}{432},0,\frac{1}{48},\frac{1}{48},0\big).
\]
\normalsize
Expand both $\phi_{k}$ and $\phi_{l}$ to order $\sqrt{t}$:
\scriptsize
\[
    \phi_{ij|k,p|l,q}\times\big(0,0,0,\frac{1}{36},0,0,\frac{1}{72},\frac{1}{72},0,0,\frac{1}{24},0,0,0,0\big).
\]
\normalsize
Expand $\frac{X_{tr}^{i}}{1-r}$ or $\frac{X_{t\rho}^{j}}{1-\rho}$ to order
$t^{1.5}$:
\scriptsize
\[
    \phi_{ijkl}\times\Big[\partial_{p}\bar{b}^{i}\big(\frac{11}{108}\delta^{jl}\delta^{kp}+\frac{1}{54}\delta^{jk}\delta^{lp}\big)+\partial_{p}\bar{b}^{j}\big(\frac{7}{108}\delta^{il}\delta^{kp}+\frac{7}{108}\delta^{ik}\delta^{ip}\big)\Big].
\]
\normalsize
Expand $a$ to order $t$:
\scriptsize
\begin{align*}
    &
    \phi_{ijkl}\times\bigg[\partial_{pq}a^{il}\big(\frac{7}{576}\delta^{jq}\delta^{kp}+\frac{7}{576}\delta^{jp}\delta^{kq}+\frac{7}{576}\delta^{jk}\delta^{pq}\big)+\partial_{pq}a^{jk}\frac{1}{96}\delta^{il}\delta^{pq} \\
    & \qquad
    +\partial_{pq}a^{jl}\big(\frac{43}{1728}\delta^{iq}\delta^{kp}+\frac{43}{1728}\delta^{ip}\delta^{kq}+\frac{17}{576}\delta^{ik}\delta^{pq}\big)+\partial_{pq}a^{ik}\big(\frac{1}{32}\delta^{jl}\delta^{pq}\big) \\
    & \qquad
    +\partial_{pq}a^{ij}\big(\frac{1}{108}\delta^{kq}\delta^{lp}+\frac{1}{108}\delta^{kp}\delta^{lq}\big)\bigg].
\end{align*}
\normalsize

\subsubsection{The $(IK;IK)$ term}
Target:
\small
\[
    t \int_{0}^{1}\big(\int_{0}^{\rho}\phi_{i}(X_{tr})\frac{X_{tr}^{i}}{1-r}dr\big)\otimes\phi_{j}(X_{t\rho})(\sigma dB^t)_{\rho}^{j}\otimes\int_{0}^{1}\big(\int_{0}^{\theta}\phi_{k}(X_{t\delta})\frac{X_{t\delta}^{k}}{1-\delta}d\delta\big)\otimes\phi_{l}(X_{t\theta})(\sigma dB^t)_{\theta}^{l}.
\]
\normalsize
Expand $\phi_{i}$ to order $t$:
\scriptsize
\[
    \phi_{i,pq|jkl}\times\big(0,0,0,\frac{23}{1440},0,0,0,0,0,0,\frac{23}{1440},0,0,\frac{23}{1440},0\big).
\]
\normalsize
Expand $\phi_{j}$ to order $t$:\scriptsize
\[
    \phi_{i|j,pq|kl}\times\big(0,0,0,\frac{7}{720},0,0,0,0,0,0,\frac{11}{2160},0,0,\frac{11}{2160},0\big).
\]
\normalsize Expand $\phi_{k}$ to order $t$: \scriptsize
\[
    \phi_{ij|k,pq|l}\times\big(0,0,0,\frac{23}{1440},0,0,0,0,0,0,\frac{23}{1440},0,0,\frac{23}{1440},0\big).
\]
\normalsize Expand $\phi_{l}$ to order $t$: \scriptsize
\[
    \phi_{ijk|l,pq}\times\big(0,0,0,\frac{7}{720},0,0,0,0,0,0,\frac{11}{2160},0,0,\frac{11}{2160},0\big).
\]
\normalsize Expand both $\phi_{i}$ and $\phi_{j}$ to order $\sqrt{t}$: \scriptsize
\[
    \phi_{i,p|j,q|kl}\times\big(0,0,0,\frac{19}{1440},0,0,0,0,0,0,\frac{29}{1440},0,0,\frac{19}{1440},0\big).
\]
\normalsize Expand both $\phi_{i}$ and $\phi_{k}$ to order $\sqrt{t}$: \scriptsize
\[
    \phi_{i,p|j|k,q|l}\times\big(0,0,0,\frac{1}{45},0,0,0,0,0,0,\frac{1}{20},0,0,\frac{1}{45},0\big).
\]
\normalsize Expand both $\phi_{i}$ and $\phi_{l}$ to order $\sqrt{t}$: \scriptsize
\[
    \phi_{i,p|jk|l,q}\times\big(0,0,0,\frac{19}{1440},0,0,0,0,0,0,\frac{29}{1440},0,0,\frac{19}{1440},0\big).
\]
\normalsize Expand both $\phi_{j}$ and $\phi_{k}$ to order $\sqrt{t}$: \scriptsize
\[
    \phi_{i|j,p|k,q|l}\times\big(0,0,0,\frac{19}{1440},0,0,0,0,0,0,\frac{29}{1440},0,0,\frac{19}{1440},0\big).
\]
\normalsize Expand both $\phi_{j}$ and $\phi_{l}$ to order $\sqrt{t}$: \scriptsize
\[
    \phi_{i|j,p|k|l,q}\times\big(0,0,0,\frac{7}{360},0,0,0,0,0,0,\frac{11}{1080},0,0,\frac{11}{1080},0\big).
\]
\normalsize Expand both $\phi_{k}$ and $\phi_{l}$ to order $\sqrt{t}$: \scriptsize
\[
    \phi_{ij|k,p|l,q}\times\big(0,0,0,\frac{19}{1440},0,0,0,0,0,0,\frac{19}{1440},0,0,\frac{29}{1440},0\big).
\]
\normalsize Expand $\frac{X_{tr}^{i}}{1-r}$ or $\frac{X_{t\delta}^{k}}{1-\delta}$ to order
$t^{1.5}$: \scriptsize
\[
    \phi_{ijkl}\times\bigg[\partial_{p}\bar{b}^{i}\frac{1}{24}\delta^{jl}\delta^{kp}+\partial_{p}\bar{b}^{k}\frac{1}{24}\delta^{ip}\delta^{jl}\bigg].
\]
\normalsize Expand $a$ to order $t$: \scriptsize
\[
    \phi_{ijkl}\times\bigg[\partial_{pq}^{2}a^{jl}\big(\frac{11}{2160}\delta^{iq}\delta^{kp}+\frac{11}{2160}\delta^{ip}\delta^{kq}+\frac{7}{720}\delta^{ik}\delta^{pq}\big)+\partial_{pq}^{2}a^{ik}\frac{1}{72}\delta^{jl}\delta^{pq}\bigg]
\]
\normalsize

\subsubsection{The $(IK;KI)$ term}
Target:
\small
\[
    t\int_{0}^{1}\big(\int_{0}^{\rho}\phi_{i}(X_{tr})\frac{X_{tr}^{i}}{1-r}du\big)\otimes\phi_{j}(X_{t\rho})(\sigma dB^t)_{\rho}^{j}\otimes\int_{0}^{1}\big(\int_{0}^{\theta}\phi_{k}(X_{t\delta})(\sigma dB^t)_{\delta}^{k}\big)\otimes\phi_{l}(X_{t\theta})\frac{X_{t\theta}^{l}}{1-\theta}d\theta.
\]
\normalsize
Expand $\phi_{i}$ to order $t$: \scriptsize
\[
    \phi_{i,pq|jkl}\times\big(0,0,0,\frac{7}{288},0,0,\frac{1}{120},0,0,\frac{1}{120},\frac{7}{288},0,\frac{1}{120},\frac{7}{288},0\big).
\]
\normalsize Expand $\phi_{j}$ to order $t$:\scriptsize
\[
    \phi_{i|j,pq|kl}\times\big(0,0,0,\frac{5}{288},0,0,\frac{11}{1440},0,0,\frac{11}{1440},\frac{11}{864},0,\frac{11}{1440},\frac{11}{864},0\big).
\]
\normalsize Expand $\phi_{k}$ to order $t$: \scriptsize
\[
    \phi_{ij|k,pq|l}\times\big(0,0,0,\frac{1}{48},0,0,\frac{11}{1440},0,0,\frac{11}{1440},0,\frac{1}{216},\frac{11}{1440},0,\frac{1}{216}\big).
\]
\normalsize Expand $\phi_{l}$ to order $t$: \scriptsize
\[
    \phi_{ijk|l,pq}\times\big(0,0,0,\frac{1}{96},\frac{1}{96},\frac{1}{96},\frac{7}{1440},\frac{5}{864},\frac{5}{864},\frac{7}{1440},\frac{5}{864},\frac{5}{864},\frac{7}{1440},\frac{5}{864},\frac{5}{864}\big).
\]
\normalsize Expand both $\phi_{i}$ and $\phi_{j}$ to order $\sqrt{t}$: \scriptsize
\[
    \phi_{i,p|j,q|kl}\times\big(0,0,0,\frac{7}{288},0,0,\frac{1}{90},0,0,\frac{1}{40},\frac{13}{288},0,\frac{1}{90},\frac{7}{288},0\big).
\]
\normalsize Expand both $\phi_{i}$ and $\phi_{k}$ to order $\sqrt{t}$: \scriptsize
\[
    \phi_{i,p|j|k,q|l}\times\big(0,0,0,\frac{1}{36},0,0,\frac{1}{90},0,0,\frac{1}{40},0,\frac{1}{48},\frac{1}{90},\frac{1}{36},0\big).
\]
\normalsize Expand both $\phi_{i}$ and $\phi_{l}$ to order $\sqrt{t}$: \scriptsize
\[
    \phi_{i,p|jk|l,q}\times\big(0,0,0,\frac{7}{576}0,\frac{7}{576},\frac{1}{180},0,\frac{7}{576},\frac{1}{30},\frac{19}{576},\frac{19}{576},\frac{1}{180},\frac{7}{576},0\big).
\]
\normalsize Expand both $\phi_{j}$ and $\phi_{k}$ to order $\sqrt{t}$: \scriptsize
\[
    \phi_{i|j,p|k,q|l}\times\big(0,0,0,\frac{1}{72},0,0,\frac{11}{720},0,0,\frac{11}{720},0,\frac{1}{72},\frac{11}{720},\frac{1}{36},0\big).
\]
\normalsize Expand both $\phi_{j}$ and $\phi_{l}$ to order $\sqrt{t}$: \scriptsize
\[
    \phi_{i|j,p|k|l,q}\times\big(0,0,0,\frac{5}{288},0,\frac{5}{288},\frac{11}{1440},0,\frac{11}{864},\frac{31}{1440},\frac{11}{864},\frac{11}{864},\frac{11}{1440},\frac{11}{864},0\big).
\]
\normalsize Expand both $\phi_{k}$ and $\phi_{l}$ to order $\sqrt{t}$: \scriptsize
\[
    \phi_{ij|k,p|l,q}\times\big(0,0,0,\frac{1}{144},0,\frac{1}{144},\frac{11}{1440},\frac{1}{216},0,\frac{31}{1440},\frac{1}{54},\frac{1}{54},\frac{11}{1440},0,\frac{1}{216}\big).
\]
\normalsize Expand $\frac{X_{tr}^{i}}{1-r}$ or $\frac{X_{t\theta}^{l}}{1-\theta}$ to order
$t^{1.5}$: \scriptsize
\[
    \phi_{ijkl}\times\Big[\partial_{p}\bar{b}^{i}\big(\frac{1}{72}\delta^{jk}\delta^{lp}+\frac{1}{18}\delta^{jl}\delta^{kp}\big)+\partial_{p}\bar{b}^{l}\big(\frac{1}{24}\delta^{ip}\delta^{jk}+\frac{1}{36}\delta^{ik}\delta^{jp}\big)\Big].
\]
\normalsize Expand $a$ to order $t$: \scriptsize
\begin{align*}
    &
    \phi_{ijkl}\times\bigg[\partial_{pq}^{2}a^{jk}\big(\frac{11}{1440}\delta^{iq}\delta^{lp}+\frac{11}{1440}\delta^{ip}\delta^{lq}+\frac{11}{1440}\delta^{il}\delta^{pq}\big)+\partial_{pq}^{2}a^{il}\frac{1}{144}\delta^{jk}\delta^{pq} \\
    & \qquad
    +\partial_{pq}^{2}a^{jl}\big(\frac{11}{864}\delta^{iq}\delta^{kp}+\frac{11}{864}\delta^{ip}\delta^{kq}+\frac{5}{288}\delta^{ik}\delta^{pq}\big)+\partial_{pq}^{2}a^{ik}\frac{1}{48}\delta^{jl}\delta^{pq} \\
    & \qquad
    +\partial_{pq}^{2}a^{kl}\big(\frac{1}{216}\delta^{iq}\delta^{jp}+\frac{1}{216}\delta^{ip}\delta^{jq}\big)\bigg].
\end{align*}
\normalsize

\subsubsection{The $(IK;KK)$ term}
Target:
\small
\[
    -t^{3/2}\int_{0}^{1}\big(\int_{0}^{\rho}\phi_{i}(X_{tr})\frac{X_{tr}^{i}}{1-r}dr\big)\otimes\phi_{j}(X_{t\rho})(\sigma dB^t)_{\rho}^{j}\otimes\int_{0}^{1}\big(\int_{0}^{\theta}\phi_{k}(X_{t\delta})(\sigma dB)_{\delta}^{k}\big)\otimes\phi_{l}(X_{t\theta})(\sigma dB^t)_{\theta}^{l}.
\]
\normalsize
Expand $\phi_{i}$ to order $t$: \scriptsize
\[
    \phi_{i,pq|jkl}\times\big(0,0,0,-\frac{7}{288},0,0,0,0,0,0,-\frac{7}{288},0,0,-\frac{7}{288},0\big).
\]
\normalsize Expand $\phi_{j}$ to order $t$: \scriptsize
\[
    \phi_{i|j,pq|kl}\times\big(0,0,0,-\frac{5}{288},0,0,0,0,0,0,-\frac{11}{864},0,0,-\frac{11}{864},0\big).
\]
\normalsize Expand $\phi_{k}$ to order $t$: \scriptsize
\[
    \phi_{ij|k,pq|l}\times\big(0,0,0,-\frac{1}{48},0,0,0,0,0,0,0,0,0,0,0\big).
\]
\normalsize Expand $\phi_{l}$ to order $t$: \scriptsize
\[
    \phi_{ijk|l,pq}\times\big(0,0,0,-\frac{5}{288},0,0,0,0,0,0,-\frac{11}{864},0,0,-\frac{11}{864},0\big).
\]
\normalsize Expand both $\phi_{i}$ and $\phi_{j}$ to order $\sqrt{t}$: \scriptsize
\[
    \phi_{i,p|j,q|kl}\times\big(0,0,0,-\frac{7}{288},0,0,0,0,0,0,-\frac{13}{288},0,0,-\frac{7}{288},0\big).
\]
\normalsize Expand both $\phi_{i}$ and $\phi_{k}$ to order $\sqrt{t}$: \scriptsize
\[
    \phi_{i,p|j|k,q|l}\times\big(0,0,0,-\frac{1}{36},0,0,0,0,0,0,0,0,0,-\frac{1}{36},0\big).
\]
\normalsize Expand both $\phi_{i}$ and $\phi_{l}$ to order $\sqrt{t}$: \scriptsize
\[
    \phi_{i,p|jk|l,q}\times\big(0,0,0,-\frac{7}{288},0,0,0,0,0,0,-\frac{13}{288},0,0,-\frac{7}{288},0\big).
\]
\normalsize Expand both $\phi_{j}$ and $\phi_{k}$ to order $\sqrt{t}$: \scriptsize
\[
    \phi_{i|j,p|k,q|l}\times\big(0,0,0,-\frac{1}{72},0,0,0,0,0,0,0,0,0,-\frac{1}{36},0\big).
\]
\normalsize Expand both $\phi_{j}$ and $\phi_{l}$ to order $\sqrt{t}$: \scriptsize
\[
    \phi_{i|j,p|k|l,q}\times\big(0,0,0,-\frac{5}{144},0,0,0,0,0,0,-\frac{11}{432},0,0,-\frac{11}{432},0\big).
\]
\normalsize Expand both $\phi_{k}$ and $\phi_{l}$ to order $\sqrt{t}$: \scriptsize
\[
    \phi_{ij|k,p|l,q}\times\big(0,0,0,-\frac{1}{72},0,0,0,0,0,0,-\frac{1}{36},0,0,0,0\big).
\]
\normalsize Expand $\frac{X_{tr}^{i}}{1-r}$ to order $t^{1.5}$: \scriptsize
\[
    \phi_{ijkl}\times\Big[\partial_{p}\bar{b}^{i}\big(-\frac{1}{18}\delta^{jl}\delta^{kp}\big)\Big].
\]
\normalsize Expand $a$ to order $t$: \scriptsize
\[
    \phi_{ijkl}\times\bigg[\partial_{pq}^{2}a^{jl}\big(-\frac{11}{864}\delta^{iq}\delta^{kp}-\frac{11}{864}\delta^{ip}\delta^{kq}-\frac{5}{288}\delta^{ik}\delta^{pq}\big)+\partial_{pq}^{2}a^{ik}\big(-\frac{1}{48}\delta^{jl}\delta^{pq}\big)\bigg].
\]
\normalsize

\subsubsection{The $(KI;KI)$ term}
Target:
\small
\[
    t\int_{0}^{1}\big(\int_{0}^{\rho}\phi_{i}(X_{tr})(\sigma dB^t)_{r}^{i}\big)\otimes\phi_{j}(X_{t\rho})\frac{X_{t\rho}^{j}}{1-\rho}dv\otimes\int_{0}^{1}\big(\int_{0}^{\theta}\phi_{k}(X_{t\delta})(\sigma dB^t)_{\delta}^{k}\big)\otimes\phi_{l}(X_{t\theta})\frac{X_{t\theta}^{l}}{1-\theta}d\theta.
\]
\normalsize
Expand $\phi_{i}$ to order $t$: \scriptsize
\[
    \phi_{i,pq|jkl}\times\big(\frac{1}{12},\frac{7}{216},\frac{7}{216},\frac{7}{120},\frac{1}{60},\frac{1}{60},\frac{13}{288},\frac{7}{288},\frac{7}{288},0,0,0,0,0,0\big).
\]
\normalsize Expand $\phi_{j}$ to order $t$: \scriptsize
\[
    \phi_{i|j,pq|kl}\times\big(\frac{5}{72},\frac{1}{32},\frac{1}{32},\frac{29}{720},\frac{29}{720},\frac{29}{720},\frac{5}{144},\frac{5}{144},\frac{5}{144},\frac{1}{32},\frac{1}{32},\frac{5}{72},\frac{1}{32},\frac{1}{32},\frac{5}{72}\big).
\]
\normalsize Expand $\phi_{k}$ to order $t$: \scriptsize
\[
    \phi_{ij|k,pq|l}\times\big(\frac{1}{12},0,0,\frac{7}{120},\frac{1}{60},\frac{1}{60},\frac{13}{288},0,0,\frac{7}{288},0,\frac{7}{216},\frac{7}{288},0,\frac{7}{216}\big).
\]
\normalsize Expand $\phi_{l}$ to order $t$: \scriptsize
\[
    \phi_{ijk|l,pq}\times\big(\frac{5}{72},\frac{5}{72},\frac{5}{72},\frac{29}{720},\frac{29}{720},\frac{29}{720},\frac{5}{144},\frac{1}{32},\frac{1}{32},\frac{5}{144},\frac{1}{32},\frac{1}{32},\frac{5}{144},\frac{1}{32},\frac{1}{32}\big).
\]
\normalsize Expand both $\phi_{i}$ and $\phi_{j}$ to order $\sqrt{t}$: \scriptsize
\[
    \phi_{i,p|j,q|kl}\times\big(\frac{1}{12},\frac{13}{216},\frac{5}{108},\frac{3}{80},\frac{3}{80},\frac{7}{120},\frac{11}{288},\frac{11}{288},\frac{23}{288},0,0,0,\frac{5}{108},\frac{13}{216},\frac{1}{12}\big).
\]
\normalsize Expand both $\phi_{i}$ and $\phi_{k}$ to order $\sqrt{t}$: \scriptsize
\[
    \phi_{i,p|j|k,q|l}\times\big(\frac{1}{12},\frac{1}{24},0,\frac{7}{60},\frac{1}{30},\frac{1}{30},\frac{1}{18},0,\frac{1}{36},0,0,0,\frac{1}{36},0,\frac{1}{24}\big).
\]
\normalsize Expand both $\phi_{i}$ and $\phi_{l}$ to order $\sqrt{t}$: \scriptsize
\[
    \phi_{i,p|jk|l,q}\times\big(\frac{5}{72},\frac{5}{48},\frac{5}{72},\frac{3}{80},\frac{7}{120},\frac{3}{80},\frac{11}{288},\frac{13}{288},\frac{5}{96},0,0,0,\frac{11}{288},\frac{5}{96},\frac{13}{288}\big).
\]
\normalsize Expand both $\phi_{j}$ and $\phi_{k}$ to order $\sqrt{t}$: \scriptsize
\[
    \phi_{i|j,p|k,q|l}\times\big(\frac{5}{72},\frac{13}{288},0,\frac{3}{80},\frac{7}{120},\frac{3}{80},\frac{11}{288},0,\frac{11}{288},\frac{13}{288},0,\frac{5}{72},\frac{5}{96},\frac{5}{96},\frac{5}{48}\big).
\]
\normalsize Expand both $\phi_{j}$ and $\phi_{l}$ to order $\sqrt{t}$: \scriptsize
\[
    \phi_{i|j,p|k|l,q}\times\big(\frac{29}{432},\frac{53}{576},\frac{29}{432},\frac{1}{20},\frac{2}{15},\frac{1}{20},\frac{37}{864},\frac{53}{576},\frac{37}{864},\frac{53}{576},\frac{29}{432},\frac{29}{432},\frac{37}{864},\frac{37}{864},\frac{53}{576}\big).
\]
\normalsize Expand both $\phi_{k}$ and $\phi_{l}$ to order $\sqrt{t}$: \scriptsize
\[
    \phi_{ij|k,p|l,q}\times\big(\frac{1}{12},0,\frac{1}{12},\frac{3}{80},\frac{7}{120},\frac{3}{80},\frac{11}{288},\frac{5}{108},0,\frac{23}{288},\frac{13}{216},\frac{13}{216},\frac{11}{288},0,\frac{5}{108}\big).
\]
\normalsize Expand $\frac{X_{t\rho}^{j}}{1-\rho}$ or $\frac{X_{t\theta}^{l}}{1-\theta}$ to order
$t^{1.5}$: \scriptsize
\[
    \phi_{ijkl}\times\bigg[\partial_{p}\bar{b}^{j}\big(\frac{1}{4}\delta^{ip}\delta^{kl}+\frac{1}{6}\delta^{ip}\delta^{kl}+\frac{11}{72}\delta^{ik}\delta^{lp}\big)+\partial_{p}\bar{b}^{l}\big(\frac{1}{6}\delta^{ip}\delta^{jk}+\frac{11}{72}\delta^{ik}\delta^{jp}+\frac{1}{4}\delta^{ij}\delta^{kp}\big)\bigg].
\]
\normalsize Expand $a$ to order $t$:
\scriptsize
\begin{align*}
    &
    \phi_{ijkl}\times\bigg[\partial_{pq}^{2}a^{ij}\big(\frac{7}{216}\delta^{kq}\delta^{lp}+\frac{7}{216}\delta^{kp}\delta^{lq}+\frac{1}{24}\delta^{kl}\delta^{pq}\big)+\partial_{pq}^{2}a^{ik}\big(\frac{1}{60}\delta^{jq}\delta^{lp}+\frac{1}{60}\delta^{jp}\delta^{lq}+\frac{1}{60}\delta^{jl}\delta^{pq}\big) \\
    & \qquad
    +\partial_{pq}^{2}a^{il}\big(\frac{7}{288}\delta^{jq}\delta^{kp}+\frac{7}{288}\delta^{jp}\delta^{kq}+\frac{7}{288}\delta^{jk}\delta^{pq}\big)+\partial_{pq}^{2}a^{jk}\big(\frac{7}{288}\delta^{iq}\delta^{lp}+\frac{7}{288}\delta^{ji}\delta^{lq}+\frac{7}{288}\delta^{il}\delta^{pq}\big) \\
    & \qquad
    +\partial_{pq}^{2}a^{jl}\big(\frac{1}{27}\delta^{iq}\delta^{kp}+\frac{1}{27}\delta^{ip}\delta^{kq}+\frac{1}{24}\delta^{ik}\delta^{pq}\big)+\partial_{pq}^{2}a^{kl}\big(\frac{7}{216}\delta^{iq}\delta^{jp}+\frac{7}{216}\delta^{ip}\delta^{jq}+\frac{1}{24}\delta^{ij}\delta^{pq}\big) \\
    & \qquad
    +\partial_{pq}^{2}a^{kl}\big(\frac{1}{24}\delta^{ij}\delta^{pq}\big)+\partial_{pq}^{2}a^{jl}\big(\frac{1}{72}\delta^{ik}\delta^{pq}\big)+\partial_{pq}^{2}a^{jk}\big(\frac{1}{48}\delta^{il}\delta^{pq}\big) \\
    & \qquad
    +\partial_{pq}^{2}a^{il}\big(\frac{1}{48}\delta^{jk}\delta^{pq}\big)+\partial_{pq}^{2}a^{ik}\big(\frac{1}{24}\delta^{jl}\delta^{pq}\big)+\partial_{pq}^{2}a^{ij}\big(\frac{1}{24}\delta^{kl}\delta^{pq}\big)\bigg].
\end{align*}
\normalsize

\subsubsection{The $(KI;KK)$ term}
Target:
\small
\[
    -t^{3/2}\int_{0}^{1}\big(\int_{0}^{\rho}\phi_{i}(X_{tr})(\sigma dB^t)_{r}^{i}\big)\otimes\phi_{j}(X_{t\rho})\frac{X_{t\rho}^{j}}{1-\rho}d\rho\otimes\int_{0}^{1}\big(\int_{0}^{\theta}\phi_{k}(X_{t\delta})(\sigma dB^t)_{\delta}^{k}\big)\otimes\phi_{l}(X_{t\theta})(\sigma dB^t)_{\theta}^{l}.
\]
\normalsize
Expand $\phi_{i}$ to order $t$: \scriptsize
\[
    \phi_{i,pq|jkl}\times\big(0,-\frac{1}{54},-\frac{1}{54},-\frac{1}{24},0,0,-\frac{7}{288},-\frac{7}{288},-\frac{7}{288},0,0,0,0,0,0\big).
\]
\normalsize Expand $\phi_{j}$ to order $t$: \scriptsize
\[
    \phi_{i|j,pq|kl}\times\big(0,-\frac{1}{54},-\frac{1}{54},-\frac{1}{36},-\frac{1}{36},-\frac{1}{36},-\frac{5}{288},-\frac{5}{288},-\frac{5}{288},-\frac{1}{54},-\frac{1}{54},0,-\frac{1}{54},-\frac{1}{54},0\big).
\]
\normalsize Expand $\phi_{k}$ to order $t$: \scriptsize
\[
    \phi_{ij|k,pq|l}\times\big(0,0,0,-\frac{1}{24},0,0,-\frac{1}{48},0,0,0,0,0,0,0,0\big).
\]
\normalsize Expand $\phi_{l}$ to order $t$: \scriptsize
\[
    \phi_{ijk|l,pq}\times\big(0,0,0,-\frac{1}{24},0,0,-\frac{7}{288},-\frac{7}{288},-\frac{7}{288},0,-\frac{1}{27},0,0,-\frac{1}{27},0\big).
\]
\normalsize Expand both $\phi_{i}$ and $\phi_{j}$ to order $\sqrt{t}$: \scriptsize
\[
    \phi_{i,p|j,q|kl}\times\big(0,-\frac{1}{108},-\frac{1}{54},-\frac{1}{48},-\frac{1}{48},0,-\frac{7}{288},-\frac{7}{288},-\frac{23}{288},0,0,0,-\frac{1}{54},-\frac{5}{108},0\big).
\]
\normalsize Expand both $\phi_{i}$ and $\phi_{k}$ to order $\sqrt{t}$: \scriptsize
\[
    \phi_{i,p|j|k,q|l}\times\big(0,0,0,-\frac{1}{12},0,0,-\frac{1}{36},0,-\frac{1}{36},0,0,0,0,0,0\big).
\]
\normalsize Expand both $\phi_{i}$ and $\phi_{l}$ to order $\sqrt{t}$: \scriptsize
\[
    \phi_{i,p|jk|l,q}\times\big(0,0,-\frac{1}{18},-\frac{1}{24},0,0,-\frac{7}{288},-\frac{7}{288},-\frac{7}{288},0,0,0,0,-\frac{1}{18},0\big).
\]
\normalsize Expand both $\phi_{j}$ and $\phi_{k}$ to order $\sqrt{t}$: \scriptsize
\[
    \phi_{i|j,p|k,q|l}\times\big(0,0,0,-\frac{1}{48},0,-\frac{1}{48},0,0,0,0,0,0,-\frac{1}{36},-\frac{1}{36},0\big).
\]
\normalsize Expand both $\phi_{j}$ and $\phi_{l}$ to order $\sqrt{t}$: \scriptsize
\[
    \phi_{i|j,p|k|l,q}\times\big(0,0,-\frac{7}{108},-\frac{1}{24},0,-\frac{1}{24},-\frac{7}{288},-\frac{23}{288},-\frac{7}{288},0,-\frac{7}{108},0,-\frac{1}{27},-\frac{1}{27},0\big).
\]
\normalsize Expand both $\phi_{k}$ and $\phi_{l}$ to order $\sqrt{t}$: \scriptsize
\[
    \phi_{ij|k,p|l,q}\times\big(0,0,0,-\frac{1}{24},0,0,-\frac{1}{36},-\frac{1}{36},0,0,-\frac{1}{18},0,0,0,0\big).
\]
\normalsize Expand $\frac{X_{t\rho}^{j}}{1-\rho}$ to order $t^{1.5}$: \scriptsize
\[
    \phi_{ijkl}\times\bigg[\partial_{p}\bar{b}^{j}\big(-\frac{1}{9}\delta^{il}\delta^{kp}-\frac{1}{12}\delta^{ik}\delta^{lp}\big)\bigg].
\]
\normalsize Expand $a$ to order $t$:
\scriptsize
\begin{align*}
    &
    \phi_{ijkl}\times\bigg[\partial_{pq}^{2}a^{ij}\big(-\frac{1}{54}\delta^{kq}\delta^{lp}-\frac{1}{54}\delta^{kp}\delta^{lq}\big)+\partial_{pq}^{2}a^{il}\big(-\frac{7}{288}\delta^{jq}\delta^{kp}-\frac{7}{288}\delta^{jp}\delta^{kq}-\frac{7}{288}\delta^{jk}\delta^{pq}\big) \\
    & \qquad
    +\partial_{pq}^{2}a^{jl}\big(-\frac{1}{27}\delta^{iq}\delta^{kp}-\frac{1}{27}\delta^{ip}\delta^{kq}-\frac{1}{24}\delta^{ik}\delta^{pq}\big)+\partial_{pq}^{2}a^{jk}\big(-\frac{1}{24}\delta^{il}\delta^{pq}\big)+\partial_{pq}^{2}a^{ik}\big(-\frac{1}{12}\delta^{jl}\delta^{pq}\big)\bigg].
\end{align*}
\normalsize

\subsubsection{The $(KK;KK)$ term}
Target:
\small
\[
    t^2\int_{0}^{1}\big(\int_{0}^{\rho}\phi_{i}(X_{tr})(\sigma dB^t)_{r}^{i}\big)\otimes\phi_{j}(X_{t\rho})(\sigma dB^t)_{\rho}^{j}\otimes\int_{0}^{1}\big(\int_{0}^{\theta}\phi_{k}(X_{t\delta})(\sigma dB^t)_{\delta}^{k}\big)\otimes\phi_{l}(X_{t\theta})(\sigma dB^t)_{\theta}^{l}. 
\]
\normalsize
\normalsize Expand $\phi_{i}$ to order $t$: \scriptsize
\[
    \phi_{i,pq|jkl}\times\big(0,0,0,\frac{1}{24},0,0,0,0,0,0,0,0,0,0,0\big).
\]
\normalsize Expand $\phi_{j}$ to order $t$: \scriptsize
\[
    \phi_{i|j,pq|kl}\times\big(0,0,0,\frac{1}{24},0,0,0,0,0,0,\frac{1}{27},0,0,\frac{1}{27},0\big).
\]
\normalsize Expand $\phi_{k}$ to order $t$: \scriptsize
\[
    \phi_{ij|k,pq|l}\times\big(0,0,0,\frac{1}{24},0,0,0,0,0,0,0,0,0,0,0,0,0\big).
\]
\normalsize Expand $\phi_{l}$ to order $t$: \scriptsize
\[
    \phi_{ijk|l,pq}\times\big(0,0,0,\frac{1}{24},0,0,0,0,0,0,\frac{1}{27},0,0,\frac{1}{27},0\big).
\]
\normalsize Expand both $\phi_{i}$ and $\phi_{j}$ to order $\sqrt{t}$: \scriptsize
\[
    \phi_{i,p|j,q|kl}\times\big(0,0,0,\frac{1}{24},0,0,0,0,0,0,0,0,0,\frac{1}{18},0\big).
\]
\normalsize Expand both $\phi_{i}$ and $\phi_{k}$ to order $\sqrt{t}$: \scriptsize
\[
    \phi_{i,p|j|k,q|l}\times\big(0,0,0,\frac{1}{12},0,0,0,0,0,0,0,0,0,0,0\big).
\]
\normalsize Expand both $\phi_{i}$ and $\phi_{l}$ to order $\sqrt{t}$: \scriptsize
\[
    \phi_{i,p|jk|l,q}\times\big(0,0,0,\frac{1}{24},0,0,0,0,0,0,0,0,0,\frac{1}{18},0\big).
\]
\normalsize Expand both $\phi_{j}$ and $\phi_{k}$ to order $\sqrt{t}$: \scriptsize
\[
    \phi_{i|j,p|k,q|l}\times\big(0,0,0,\frac{1}{24},0,0,0,0,0,0,\frac{1}{18},0,0,0,0\big).
\]
\normalsize Expand both $\phi_{j}$ and $\phi_{l}$ to order $\sqrt{t}$: \scriptsize
\[
    \phi_{i|j,p|k|l,q}\times\big(0,0,0,\frac{1}{12},0,0,0,0,0,0,\frac{2}{27},0,0,\frac{2}{27},0\big).
\]
\normalsize Expand both $\phi_{k}$ and $\phi_{l}$ to order $\sqrt{t}$: \scriptsize
\[
    \phi_{ij|k,p|l,q}\times\big(0,0,0,\frac{1}{24},0,0,0,0,0,0,\frac{1}{18},0,0,0,0\big).
\]
\normalsize Expand $a$ to order $t$: \scriptsize
\[
    \phi_{ijkl}\times\bigg[\partial_{pq}^{2}a^{jl}\big(\frac{1}{27}\delta^{iq}\delta^{kp}+\frac{1}{27}\delta^{ip}\delta^{kq}+\frac{1}{24}\delta^{ik}\delta^{pq}\big)+\partial_{pq}^{2}a^{ik}\frac{1}{24}\delta^{jl}\delta^{pq}\bigg].
\]
\normalsize

\subsubsection{The $(II;P)$ term}
Target:
\small
\[
    t\int_{0}^{1}\big(\int_{0}^{\rho}\phi_{i}(X_{tr})\frac{X_{tr}^{i}}{1-r}dr\big)\otimes\phi_{j}(X_{t\rho})\frac{X_{t\rho}^{j}}{1-\rho}d\rho\otimes\int_{0}^{1}\frac{1}{2}\phi_{k}\big(X_{t\theta}\big)\otimes\phi_{l}\big(X_{t\theta}\big)a^{kl}\big(X_{t\theta}\big)d\theta.
\]
\normalsize
Expand $\phi_{i}$ to order $t$: \scriptsize
\[
    \phi_{i,pq|jkl}\times\big(\frac{1}{48},0,0,0,0,0,0,0,0,0,0,\frac{1}{48},0,0,\frac{1}{48}\big).
\]
\normalsize Expand $\phi_{j}$ to order $t$: \scriptsize
\[
    \phi_{i|j,pq|kl}\times\big(\frac{1}{72},0,0,0,0,0,0,0,0,0,0,\frac{1}{72},0,0,\frac{1}{72}\big).
\]
\normalsize Expand $\phi_{k}$ to order $t$: \scriptsize
\[
    \phi_{ij|k,pq|l}\times\big(\frac{1}{48},0,0,0,0,0,0,0,0,0,0,\frac{1}{108},0,0,\frac{1}{108}\big).
\]
\normalsize Expand $\phi_{l}$ to order $t$: \scriptsize
\[
    \phi_{ijk|l,pq}\times\big(\frac{1}{48},0,0,0,0,0,0,0,0,0,0,\frac{1}{108},0,0,\frac{1}{108}\big).
\]
\normalsize Expand both $\phi_{i}$ and $\phi_{j}$ to order $\sqrt{t}$: \scriptsize
\[
    \phi_{i,p|j,q|kl}\times\big(\frac{1}{48},0,0,0,0,0,0,0,0,0,0,\frac{1}{16},0,0,\frac{1}{48}\big).
\]
\normalsize Expand both $\phi_{i}$ and $\phi_{k}$ to order $\sqrt{t}$: \scriptsize
\[
    \phi_{i,p|j|k,q|l}\times\big(\frac{1}{48},0,0,0,0,0,0,0,0,0,0,\frac{1}{32},0,0,\frac{1}{48}\big).
\]
\normalsize Expand both $\phi_{i}$ and $\phi_{l}$ to order $\sqrt{t}$: \scriptsize
\[
    \phi_{i,p|jk|l,q}\times\big(\frac{1}{48},0,0,0,0,0,0,0,0,0,0,\frac{1}{32},0,0,\frac{1}{48}\big).
\]
\normalsize Expand both $\phi_{j}$ and $\phi_{k}$ to order $\sqrt{t}$: \scriptsize
\[
    \phi_{i|j,p|k,q|l}\times\big(\frac{1}{72},0,0,0,0,0,0,0,0,0,0,\frac{1}{72},0,0,\frac{1}{32}\big).
\]
\normalsize Expand both $\phi_{j}$ and $\phi_{l}$ to order $\sqrt{t}$: \scriptsize
\[
    \phi_{i|j,p|k|l,q}\times\big(\frac{1}{72},0,0,0,0,0,0,0,0,0,0,\frac{1}{72},0,0,\frac{1}{32}\big).
\]
\normalsize Expand both $\phi_{k}$ and $\phi_{l}$ to order $\sqrt{t}$: \scriptsize
\[
    \phi_{ij|k,p|l,q}\times\big(\frac{1}{24},0,0,0,0,0,0,0,0,0,0,\frac{1}{54},0,0,\frac{1}{54}\big).
\]
\normalsize Expand $\frac{X_{tr}^{i}}{1-r}$ or $\frac{X_{t\rho}^{j}}{1-\rho}$ to order
$t^{1.5}$: \scriptsize
\[
    \phi_{ijkl}\times\Big[\partial_{p}\bar{b}^{i}\frac{1}{24}\delta^{jp}\delta^{kl}+\partial_{p}\bar{b}^{j}\frac{1}{12}\delta^{ip}\delta^{kl}\Big].
\]
\normalsize Expand $a$ to order $t$: \scriptsize
\[
    \phi_{ijkl}\times\bigg[\partial_{pq}^{2}a^{kl}\big(\frac{1}{108}\delta^{iq}\delta^{jp}+\frac{1}{108}\delta^{ip}\delta^{jq}+\frac{1}{48}\delta^{ij}\delta^{pq}\big)+\partial_{pq}^{2}a^{ij}\big(\frac{1}{48}\delta^{kl}\delta^{pq}\big)\bigg].
\]
\normalsize

\subsubsection{The $(IK;P)$ term}
Target:
\small
\[
    -t^{3/2}\int_{0}^{1}\big(\int_{0}^{\rho}\phi_{i}(X_{tr})\frac{X_{tr}^{i}}{1-r}dr\big)\otimes\phi_{j}(X_{t\rho})(\sigma dB^t)_{\rho}^{j}\otimes\int_{0}^{1}\frac{1}{2}\phi_{k}\big(X_{t\theta}\big)\otimes\phi_{l}\big(X_{t\theta}\big)a^{kl}\big(X_{t\theta}\big)d\theta.
\]
\normalsize
Expand $\phi_{k}$ to order $t$: \scriptsize
\[
    \phi_{ij|k,pq|l}\times\big(0,0,0,0,0,0,0,0,0,0,0,-\frac{1}{432},0,0,-\frac{1}{432}\big).
\]
\normalsize Expand $\phi_{l}$ to order $t$: \scriptsize
\[
    \phi_{ijk|l,pq}\times\big(0,0,0,0,0,0,0,0,0,0,0,-\frac{1}{432},0,0,-\frac{1}{432}\big). 
\]
\normalsize Expand both $\phi_{i}$ and $\phi_{k}$ to order $\sqrt{t}$: \scriptsize
\[
    \phi_{i,p|j|k,q|l}\times\big(0,0,0,0,0,0,0,0,0,0,0,-\frac{1}{96},0,0,0\big).
\]
\normalsize Expand both $\phi_{i}$ and $\phi_{l}$ to order $\sqrt{t}$: \scriptsize
\[
    \phi_{i,p|jk|l,q}\times\big(0,0,0,0,0,0,0,0,0,0,0,-\frac{1}{96},0,0,0\big).
\]
\normalsize Expand both $\phi_{j}$ and $\phi_{k}$ to order $\sqrt{t}$: \scriptsize
\[
    \phi_{i|j,p|k,q|l}\times\big(0,0,0,0,0,0,0,0,0,0,0,-\frac{1}{144},0,0,0\big).
\]
\normalsize Expand both $\phi_{j}$ and $\phi_{l}$ to order $\sqrt{t}$: \scriptsize
\[
    \phi_{i|j,p|k|l,q}\times\big(0,0,0,0,0,0,0,0,0,0,0,-\frac{1}{144},0,0,0\big).
\]
\normalsize Expand both $\phi_{k}$ and $\phi_{l}$ to order $\sqrt{t}$: \scriptsize
\[
    \phi_{ij|k,p|l,q}\times\big(0,0,0,0,0,0,0,0,0,0,0,-\frac{1}{216},0,0,-\frac{1}{216}\big).
\]
\normalsize Expand $a$ to order $t$: \scriptsize
\[
    \phi_{ijkl}\times\Big[\partial_{pq}^{2}a^{kl}\big(-\frac{1}{432}\delta^{iq}\delta^{jp}-\frac{1}{432}\delta^{ip}\delta^{jq}\big)\Big].
\]
\normalsize

\subsubsection{The $(KI;P)$ term}
Target:
\small
\[
    -t^{3/2}\int_{0}^{1}\big(\int_{0}^{\rho}\phi_{i}(X_{tr})(\sigma dB^t)_{r}^{i}\big)\otimes\phi_{j}(X_{t\rho})\frac{X_{t\rho}^{j}}{1-\rho}d\rho\otimes\int_{0}^{1}\frac{1}{2}\phi_{k}\big(X_{t\theta}\big)\otimes\phi_{l}\big(X_{t\theta}\big)a^{kl}\big(X_{t\theta}\big)d\theta.
\]
\normalsize
Expand $\phi_{i}$ to order $t$: \scriptsize
\[
    \phi_{i,pq|jkl}\times\big(-\frac{1}{24},0,0,0,0,0,0,0,0,0,0,0,0,0,0\big).
\]
\normalsize Expand $\phi_{j}$ to order $t$: \scriptsize
\[
    \phi_{i|j,pq|kl}\times\big(-\frac{5}{144},0,0,0,0,0,0,0,0,0,0,-\frac{5}{144},0,0,-\frac{5}{144}\big).
\]
\normalsize Expand $\phi_{k}$ to order $t$: \scriptsize
\[
    \phi_{ij|k,pq|l}\times\big(-\frac{1}{24},0,0,0,0,0,0,0,0,0,0,-\frac{7}{432},0,0,-\frac{7}{432}\big).
\]
\normalsize Expand $\phi_{l}$ to order $t$: \scriptsize
\[
    \phi_{ijk|l,pq}\times\big(-\frac{1}{24},0,0,0,0,0,0,0,0,0,0,-\frac{7}{432},0,0,-\frac{7}{432}\big).
\]
\normalsize Expand both $\phi_{i}$ and $\phi_{j}$ to order $\sqrt{t}$: \scriptsize
\[
    \phi_{i,p|j,q|kl}\times\big(-\frac{1}{24},0,0,0,0,0,0,0,0,0,0,0,0,0,-\frac{1}{24}\big).
\]
\normalsize Expand both $\phi_{i}$ and $\phi_{k}$ to order $\sqrt{t}$: \scriptsize
\[
    \phi_{i,p|j|k,q|l}\times\big(-\frac{1}{24},0,0,0,0,0,0,0,0,0,0,0,0,0,-\frac{1}{48}\big).
\]
\normalsize Expand both $\phi_{i}$ and $\phi_{l}$ to order $\sqrt{t}$: \scriptsize
\[
    \phi_{i,p|jk|l,q}\times\big(-\frac{1}{24},0,0,0,0,0,0,0,0,0,0,,0,0,-\frac{1}{48}\big).
\]
\normalsize Expand both $\phi_{j}$ and $\phi_{k}$ to order $\sqrt{t}$: \scriptsize
\[
    \phi_{i|j,p|k,q|l}\times\big(-\frac{5}{144},0,0,0,0,0,0,0,0,0,0,-\frac{5}{144},0,0,-\frac{5}{96}\big).
\]
\normalsize Expand both $\phi_{j}$ and $\phi_{l}$ to order $\sqrt{t}$: \scriptsize
\[
    \phi_{i|j,p|k|l,q}\times\big(-\frac{5}{144},0,0,0,0,0,0,0,0,0,0,-\frac{5}{144},0,0,-\frac{5}{96}\big).
\]
\normalsize Expand both $\phi_{k}$ and $\phi_{l}$ to order $\sqrt{t}$: \scriptsize
\[
    \phi_{ij|k,p|l,q}\times\big(-\frac{1}{12},0,0,0,0,0,0,0,0,0,0,-\frac{7}{216},0,0,-\frac{7}{216}\big).
\]
\normalsize Expand $\frac{X_{t\rho}^{j}}{1-\rho}$ to order $t^{1.5}$: \scriptsize
\[
    \phi_{ijkl}\times\Big[\partial_{p}\bar{b}^{j}\cdot\big(-\frac{1}{8}\delta^{ip}\delta^{kl}\big)\Big].
\]
\normalsize Expand $a$ to order $t$: \scriptsize
\[
    \phi_{ijkl}\times\bigg[\partial_{pq}^{2}a^{kl}\big(-\frac{7}{432}\delta^{iq}\delta^{jp}-\frac{7}{432}\delta^{ip}\delta^{jq}-\frac{1}{24}\delta^{ij}\delta^{pq}\big)+\partial_{pq}^{2}a^{ij}\big(-\frac{1}{24}\delta^{kl}\delta^{pq}\big)\bigg].
\]
\normalsize

\subsubsection{The $(KK;P)$ term}
Target:
\small
\[
    t^2\int_{0}^{1}\big(\int_{0}^{\rho}\phi_{i}(X_{tr})(\sigma dB^t)_{r}^{i}\big)\otimes\phi_{j}(X_{t\rho})(\sigma dB^t)_{\rho}^{j}\otimes\int_{0}^{1}\frac{1}{2}\phi_{k}\big(X_{t\theta}\big)\otimes\phi_{l}\big(X_{t\theta}\big)a^{kl}\big(X_{t\theta}\big)d\theta.
\]
\normalsize
Expand $\phi_{k}$ to order $t$: \scriptsize
\[
    \phi_{ij|k,pq|l}\times\big(0,0,0,0,0,0,0,0,0,0,0,\frac{1}{108},0,0,\frac{1}{108}\big).
\]
\normalsize Expand $\phi_{l}$ to order $t$: \scriptsize
\[
    \phi_{ijk|l,pq}\times\big(0,0,0,0,0,0,0,0,0,0,0,\frac{1}{108},0,0,\frac{1}{108}\big).
\]
\normalsize Expand both $\phi_{j}$ and $\phi_{k}$ to order $\sqrt{t}$: \scriptsize
\[
    \phi_{i|j,p|k,q|l}\times\big(0,0,0,0,0,0,0,0,0,0,0,\frac{1}{36},0,0,0\big).
\]
\normalsize Expand both $\phi_{j}$ and $\phi_{l}$ to order $\sqrt{t}$: \scriptsize
\[
    \phi_{i|j,p|k|l,q}\times\big(0,0,0,0,0,0,0,0,0,0,0,\frac{1}{36},0,0,0\big).
\]
\normalsize Expand both $\phi_{k}$ and $\phi_{l}$ to order $\sqrt{t}$: \scriptsize
\[
    \phi_{ij|k,p|l,q}\times\big(0,0,0,0,0,0,0,0,0,0,0,\frac{1}{54},0,0,\frac{1}{54}\big).
\]
\normalsize Expand $a$ to order $t$: \scriptsize
\[
    \phi_{ijkl}\times\bigg[\partial_{pq}^{2}a^{kl}\big(\frac{1}{108}\delta^{iq}\delta^{jp}+\frac{1}{108}\delta^{ip}\delta^{jq}\big)\bigg].
\]
\normalsize

\subsubsection{The $(P;P)$ term}
Target:
\small
\[
    t^2\int_{0}^{1}\frac{1}{2}\phi_{i}\big(X_{t\rho}\big)\otimes\phi_{j}\big(X_{t\rho}\big)a^{ij}\big(X_{t\rho}\big)d\rho\otimes\int_{0}^{1}\frac{1}{2}\phi_{k}\big(X_{t\theta}\big)\otimes\phi_{l}\big(X_{t\theta}\big)a^{kl}\big(X_{t\theta}\big)d\theta.
\]
\normalsize Expand $\phi_{i}$ to order $t$: \scriptsize
\[
    \phi_{i,pq|jkl}\times\big(\frac{1}{48},0,0,0,0,0,0,0,0,0,0,0,0,0,0\big).
\]
\normalsize Expand $\phi_{j}$ to order $t$: \scriptsize
\[
    \phi_{i|j,pq|kl}\times\big(\frac{1}{48},0,0,0,0,0,0,0,0,0,0,0,0,0,0\big).
\]
\normalsize Expand $\phi_{k}$ to order $t$: \scriptsize
\[
    \phi_{ij|k,pq|l}\times\big(\frac{1}{48},0,0,0,0,0,0,0,0,0,0,0,0,0,0\big).
\]
\normalsize Expand $\phi_{l}$ to order $t$: \scriptsize
\[
    \phi_{ijk|l,pq}\times\big(\frac{1}{48},0,0,0,0,0,0,0,0,0,0,0,0,0,0\big).
\]
\normalsize Expand both $\phi_{i}$ and $\phi_{j}$ to order $\sqrt{t}$: \scriptsize
\[
    \phi_{i,p|j,q|kl}\times\big(\frac{1}{24},0,0,0,0,0,0,0,0,0,0,0,0,0,0\big).
\]
\normalsize Expand both $\phi_{i}$ and $\phi_{k}$ to order $\sqrt{t}$: \scriptsize
\[
    \phi_{i,p|j|k,q|l}\times\big(\frac{1}{48},0,0,0,0,0,0,0,0,0,0,0,0,0,0\big).
\]
\normalsize Expand both $\phi_{i}$ and $\phi_{l}$ to order $\sqrt{t}$: \scriptsize
\[
    \phi_{i,p|jk|l,q}\times\big(\frac{1}{48},0,0,0,0,0,0,0,0,0,0,0,0,0,0\big).
\]
\normalsize Expand both $\phi_{j}$ and $\phi_{k}$ to order $\sqrt{t}$: \scriptsize
\[
    \phi_{i|j,p|k,q|l}\times\big(\frac{1}{48},0,0,0,0,0,0,0,0,0,0,0,0,0,0\big).
\]
\normalsize Expand both $\phi_{j}$ and $\phi_{l}$ to order $\sqrt{t}$: \scriptsize
\[
    \phi_{i|j,p|k|l,q}\times\big(\frac{1}{48},0,0,0,0,0,0,0,0,0,0,0,0,0,0\big).
\]
\normalsize Expand both $\phi_{k}$ and $\phi_{l}$ to order $\sqrt{t}$: \scriptsize
\[
    \phi_{ij|k,p|l,q}\times\big(\frac{1}{24},0,0,0,0,0,0,0,0,0,0,0,0,0,0\big).
\]
\normalsize Expand $a$ to order $t$: \scriptsize
\[
    \phi_{ijkl}\times\bigg[\partial_{pq}^{2}a^{ij}\frac{1}{48}\delta^{kl}\delta^{pq}+\partial_{pq}^{2}a^{kl}\frac{1}{48}\delta^{ij}\delta^{pq}\bigg].
\]

\end{appendices}

\end{document}